\documentclass[a4paper]{article}

\usepackage[authoryear]{natbib}
\usepackage[left=3.6cm,right=3.6cm,bottom=3cm,top=3cm]{geometry}

\usepackage[utf8]{inputenc} \usepackage{lmodern} \usepackage[ngerman, english]{babel} \usepackage{csquotes}

\usepackage{emptypage} \usepackage{xcolor}
\usepackage{mdframed}
\usepackage{appendix}

\usepackage{url}            \usepackage{multicol,multirow,booktabs}       \usepackage{nicematrix} \usepackage{array}          \usepackage{amsfonts}       \usepackage{nicefrac}       \usepackage{microtype}      \usepackage{xcolor}    
\usepackage{subcaption}       \usepackage{pdflscape}      

\usepackage{wrapfig} \usepackage{tikz, tikz-cd, pgfplots} \usetikzlibrary{
    fit,
    backgrounds,
    positioning,
    arrows.meta,
    calc
}
\pgfplotsset{compat=1.18}

\usepackage[ruled,linesnumbered]{algorithm2e}


\usepackage{mathtools} \usepackage{amsthm, amssymb} \usepackage{keytheorems} \usepackage{aligned-overset} \usepackage{mathrsfs} 

\usepackage{cancel}

\usepackage{enumitem} \usepackage[unicode, colorlinks]{hyperref}  \hypersetup{
	pdftitle={User-friendly approximation theory for statistical regression},
	pdfauthor={Felix Benning},
}
\mdfdefinestyle{default}{
    innertopmargin=0pt,
}

\theoremstyle{plain}\newtheorem{prop}{Proposition}[section]
\newtheorem{lemma}[prop]{Lemma}
\newtheorem{corollary}[prop]{Corollary}
\newtheorem{theorem}[prop]{Theorem}

\theoremstyle{definition}
\newtheorem{definition}[prop]{Definition}
\newtheorem{example}[prop]{Example}

\newtheorem{remark}[prop]{Remark}

\newcommand{\ul}[1]{\underline{#1}}

\newcommand*{\dims}{d}
\DeclarePairedDelimiter{\abs}{\lvert}{\rvert} \DeclarePairedDelimiter{\set}{\{}{\}} \DeclarePairedDelimiter{\norm}{\|}{\|} \DeclarePairedDelimiter{\floor}{\lfloor}{\rfloor}
\DeclarePairedDelimiter{\ceil}{\lceil}{\rceil}

\newcommand*{\ones}{\mathbf{1}}
\NewDocumentCommand{\ind}{s o m}{\IfBooleanTF{#1}
    {\ones_{\set[#2]{#3}}} {\ones_{#3}}           }

\NewDocumentEnvironment{myproof}{o}{
  \pushQED{\qed}\medbreak\noindent \IfNoValueTF{#1}{\textit{Proof.}}{\textit{#1.}}}{\popQED \par\medskip}

\newcommand*{\qedHere}{\pushQED{\qed}\qedhere\popQED }

\NewDocumentCommand{\opnorm}{sO{}m}{\IfBooleanTF{#1}{\left|\opnormkern\left|\opnormkern\left|
    #3
    \right|\opnormkern\right|\opnormkern\right|
  }{
    \mathopen{#2|\opnormkern #2|\opnormkern #2|}
    #3
    \mathclose{#2|\opnormkern #2|\opnormkern #2|}
  }}
\newcommand{\opnormkern}{\mkern-2.7mu\relax}

\newcommand*{\funSpace}{\mathbb{F}}
\newcommand*{\domain}{\mathbb{X}}
\newcommand*{\real}{\mathbb{R}}
\newcommand*{\field}{\mathbb{K}}
\newcommand*{\nat}{\mathbb{N}}
\newcommand*{\integer}{\mathbb{Z}}
\newcommand*{\complex}{\mathbb{C}}

\newcommand*{\hilbert}{\mathcal{H}}
\newcommand*{\sphere}[1][\dims-1]{\mathbb{S}^{#1}}
\newcommand*{\Ball}[1][\dims]{\mathbb{B}^{#1}}
\newcommand*{\torus}{\mathbb{T}}
\DeclarePairedDelimiterXPP{\ball}[2]{\mathrm{B}}{(}{)}{}{#1, #2}

\newcommand*{\Lip}{\mathrm{Lip}}
\newcommand*{\TV}{\mathrm{TV}}
\newcommand*{\diam}{\mathrm{diam}}

\newcommand*{\metric}{\mathrm{d}}

\newcommand*{\metricProj}{\mathsf{proj}} \newcommand*{\linProj}{\mathrm{P}} \newcommand*{\regOp}[1][X]{\mathrm{L}_{#1}} \newcommand*{\regMin}{\mathrm{M}_X} \newcommand*{\eval}{\pi} 

\newcommand*{\error}{\mathcal{E}}
\newcommand*{\modCont}{\mathcal{\omega}}
 \DeclarePairedDelimiterXPP{\lebfun}[3]{\lambda_{#1}}{(}{)}{}{#2; #3}
\DeclarePairedDelimiterXPP{\lebconst}[2]{\Lambda_{#1}}{(}{)}{}{#2}
 \newcommand*{\Gram}{G}

\newcommand*{\lever}{k} \newcommand*{\reprK}{K} 

\newcommand*{\modelSpace}{\Phi}
\newcommand*{\modelSpaceW}{\Psi}
\newcommand*{\model}{\varphi}
\newcommand*{\mBasis}{\phi}
\newcommand*{\mBasisW}{\psi}
\newcommand*{\mBasisVec}{\phi}
\newcommand*{\mBasisVecW}{\psi}

\newcommand*{\modelVander}{\mathbf{\Phi}_X}

\newcommand*{\weight}{w}

\newcommand*{\pSym}{\mathsf{P}} \DeclarePairedDelimiterXPP{\prob}[1]{\pSym}{\{}{\}}{}{#1}
\newcommand{\eSym}{\mathsf{E}} \DeclarePairedDelimiterXPP{\E}[1]{\eSym}{[}{]}{}{#1}

\DeclareMathOperator{\supp}{supp}

\DeclarePairedDelimiterXPP{\normal}[2]{\mathcal{N}}{(}{)}{}{#1, #2}
\DeclarePairedDelimiterXPP{\unif}[1]{\mathrm{Unif}}{(}{)}{}{#1}
\DeclarePairedDelimiterXPP{\binomial}[2]{\mathrm{Bin}}{(}{)}{}{#1, #2}
\DeclarePairedDelimiterXPP{\poisson}[1]{\mathrm{Poi}}{(}{)}{}{#1}
\DeclarePairedDelimiterXPP{\bernoulli}[1]{\mathrm{Ber}}{(}{)}{}{#1}
\DeclarePairedDelimiterXPP{\exponential}[1]{\mathrm{Exp}}{(}{)}{}{#1}

\DeclarePairedDelimiterXPP{\subG}[1]{\mathrm{SGauss}}{(}{)}{}{#1}
\DeclarePairedDelimiterXPP{\subExp}[1]{\mathrm{SExp}}{(}{)}{}{#1}
\DeclarePairedDelimiterXPP{\subWeibull}[1]{\mathrm{SWeibull}}{(}{)}{}{#1}

\newcommand*{\varSym}{\mathsf{var}} \DeclarePairedDelimiterXPP{\var}[1]{\varSym}{[}{]}{}{#1}
\newcommand*{\covSym}{\mathsf{cov}} \DeclarePairedDelimiterXPP{\cov}[2]{\covSym}{[}{]}{}{#1, #2}
\newcommand*{\CovSym}{\mathsf{Cov}} \DeclarePairedDelimiterXPP{\Cov}[1]{\CovSym}{[}{]}{}{#1}
\newcommand*{\mSym}{\mathsf{M}} \DeclarePairedDelimiterXPP{\median}[1]{\mSym}{[}{]}{}{#1}

 \DeclareMathOperator{\phase}{phase} \DeclareMathOperator{\Span}{span} \DeclareMathOperator{\tr}{tr} \DeclareMathOperator{\rank}{rank} \DeclareMathOperator{\diag}{diag} \DeclareMathOperator{\Image}{im}  \DeclarePairedDelimiterXPP{\scp}[2]{}{\langle}{\rangle}{}{#1, #2} \newcommand{\identity}{\mathbb{I}}
\newcommand{\id}{\mathrm{id}}
\newcommand*{\transpose}{\mathsf T} \newcommand*{\conT}{\mathsf{H}} \newcommand*{\nullspace}{\mathscr{N}}

\DeclareMathOperator*{\argmin}{arg\,min}

\DeclareMathOperator*{\esssup}{ess\,sup}

\newcommand*{\noise}{\xi}
\newcommand*{\noiseIndep}{\varsigma}

\renewcommand*{\Re}{\operatorname{Re}}
\renewcommand*{\Im}{\operatorname{Im}}

\newcommand*{\bigO}{\mathcal{O}}

\newcommand*{\favConst}{\mathcal{K}}

\newcommand*{\ang}{\theta}

\newcommand*{\dirichlet}{D}
\newcommand*{\dct}[1]{\textnormal{DCT-#1}}
\newcommand*{\F}{\mathrm{F}}
\newcommand*{\Trig}{\mathrm{Trig}}

\newcommand*{\cheb}{T}

\newcommand*{\poly}{\Pi}

\newcommand*{\inter}{ \mathchoice{\mskip1.5mu{:}\mskip1.5mu}{\mskip1.5mu{:}\mskip1.5mu}{{:}}{{:}}
}

\newlist{steps}{enumerate}{3}
\setlist[steps]{
    wide=0pt,
}
\setlist[steps,1]{
    label=\textbf{Step~\arabic*:},
    ref={Step~\arabic*},
}
\setlist[steps,2]{
    labelindent=0.75em,
    leftmargin=1.5em,
    label=\emph{Step~\arabic{stepsi}.\arabic*:},
    ref={Step~\arabic{stepsi}.\arabic*},
}
\setlist[steps,3]{
    label=\emph{Step~\arabic{stepsi}.\arabic{stepsii}.\arabic*:},
    ref={Step~\arabic{stepsi}.\arabic{stepsii}.\arabic*},
}

\newlist{casebycase}{enumerate}{3}
\setlist[casebycase]{
    wide=0pt,
}
\setlist[casebycase,1]{
    label=\textbf{Case \arabic*:},
    ref={Case \arabic*},
}
\setlist[casebycase,2]{
    label=\emph{Case \arabic{casebycasei}.\arabic*:},
    ref={Case \arabic{casebycasei}.\arabic*},
}
\setlist[casebycase,3]{
    label=\emph{Case \arabic{casebycasei}.\arabic{casebycaseii}.\arabic*:},
    ref={Case \arabic{casebycasei}.\arabic{casebycaseii}.\arabic*},
}

\newenvironment{talign*}
 {\let\displaystyle\textstyle\csname align*\endcsname}
 {\endalign}

\makeatletter
\newcounter{constants}
 \newcommand*{\defConst}[1]{\stepcounter{constants}\edef\const@temp{\arabic{constants}}\hypertarget{const:#1}{}\immediate\write\@auxout{\string\newlabel{const: #1}{{\noexpand\ensuremath{c_{\const@temp}}}{\thepage}{\noexpand\ensuremath{c_{\const@temp}}}{const:#1}{}}}\ensuremath{c_{\const@temp}}}
\makeatother

\makeatletter
\newcommand\newtag[2]{#1\def\@currentlabel{#1}\label{#2}}
\makeatother

\definecolor{primary}{RGB}{0,48,86}

\colorlet{primary10}{primary!10}
\colorlet{primary25}{primary!25}
\colorlet{primary40}{primary!40}
\colorlet{primary55}{primary!55}
\colorlet{primary70}{primary!70}
\colorlet{primary85}{primary!85}

\definecolor{accent}{RGB}{179,182,185}

\definecolor{magenta}{HTML}{D81B66}
\definecolor{red}{HTML}{c71818}
\definecolor{blue}{HTML}{11548f}
\definecolor{yellow}{HTML}{FFC107}
\definecolor{green}{HTML}{04593a}

\newcommand*{\magenta}[1]{\textcolor{magenta}{#1}}
\newcommand*{\green}[1]{\textcolor{green}{#1}}
\newcommand*{\blue}[1]{\textcolor{blue}{#1}}

\title{User-friendly approximation theory for statistical regression}
\author{Felix Benning}

\begin{document}
	\maketitle 

	\begin{abstract}
	The error of regression using noisy function values naturally decomposes into
an approximation error from the choice of model space and stochastic error
from the noise on the observations. Statistical theory is mostly concerned with the stochastic error
while approximation theory focuses on the approximation error, often subject
to the assumption of noiseless observations. We combine the tools
from approximation theory and statistics into a framework for bounding both
errors in generalized least squares regression. The results cover weighted
sup-norms and \(L^2(\mu)\)-norms, real- and complex-valued functions, and
correlated noise under a range of tail and moment assumptions.
Weighted Lebesgue functions and constants control the amplification of the
best approximation error, while concentration inequalities provide bounds
on the stochastic error. We organize these results into reusable bounds
with explicit constants and establish sharpness where possible.

For the common strategy of randomized regression we prove concentration of
the Lebesgue function and constant around those of the corresponding
\(L^2(\mu)\)-projection.

For Fourier and Chebyshev regression on standard sampling grids, we obtain
explicit, asymptotically tight \(\bigO(\log m)\) bounds on the Lebesgue
constants where \(m\) is the model dimension. This improves upon
the previous \(\bigO(\sqrt{m})\) bound obtained for the regression setting and
matches the bounds known for interpolation.

A catalogue of approximation bounds and worked examples illustrates how to
turn these results into explicit regression guarantees.

 	\end{abstract}

	{ 
    \setcounter{tocdepth}{2}
    \hypersetup{hidelinks}
    \tableofcontents 
}
\setcounter{tocdepth}{3}

\section*{Introduction}
\label{sec: introduction}

Linear regression using feature functions from a model space is one of the most fundamental and widely
used techniques in statistics
\citep[e.g.][]{pukelsheimOptimalDesignExperiments1993}. Its error naturally separates
into the stochastic error due to noise on the data and the approximation error incurred
by the choice of model space. The main focus of statistical theory
is usually the stochastic error, whereas approximation theory has long studied
the approximation error with the help of the best approximation error and the
``Lebesgue constant''---the operator
norm of projections onto the model space
\citep[e.g.][]{jacksonTheoryApproximation1930,devoreConstructiveApproximation1993,trefethenApproximationTheoryApproximation2019,lubinskySurveyWeightedPolynomial2007}.

\paragraph*{Our contributions.}
The main contribution of this paper is to marry the techniques of approximation
theory with the statistical analysis of regression with the goal to build a
bridge between the two communities. Since we integrate results from both communities,
many of these results already exist in some form in the literature. This paper is
however not merely a literature review as many results require generalization.
Among other things, we generalize the setting from interpolation to regression,
ordinary least squares to generalized least squares, from unweighted sup-norms to
weighted sup-norms, real- to complex-valued functions, independent noise to
correlated noise and widen the menu of possible tail assumptions about the
noise. We aim for
sharp results, usually with explicit constants and sometimes matching lower
bounds. We also ensure that all these results are reusable, user-friendly,
modular building blocks and provide a flow chart for
easier navigation (Fig.~\ref{fig: flowchart uniform error bounds}).

\noindent Beyond the distillation and generalization of existing results, we also 
\begin{itemize}[nosep]
    \item explain why the \(L^2(\mu)\)-norm
    is more difficult to attack with approximation-theoretic techniques than the sup-norm
    with an impossibility result (Thm.~\ref{thm: Lp approximation is hopeless}),
    \item prove a concentration result with explicit constants for the Lebesgue
    function and Lebesgue constant of randomized regression around those of
    \(L^2\)-projection (Thm.~\ref{thm: lebesgue constant concentration}),
    \item prove explicit, asymptotically tight \(\bigO(\log m)\) bounds on the Lebesgue constants of Fourier and Chebyshev regression
    (Sec.~\ref{sec: lebesgue constants of regression} and \ref{sec: examples, trigonometric and chebyshev regression}).
\end{itemize}

\paragraph*{Related work.} 
With the observation that regression may be viewed as a type of projection, the
operator theoretic formulation of the Lebesgue constant is immediately
applicable \citetext{e.g.\ Appendix~\ref{sec: approximation theory} or
\citealp[Prop.~4.1]{devoreConstructiveApproximation1993}}.
This has already been observed by \citet{barelGoodPointSets,belloniNewAsymptoticTheory2015,maranLeastSquaresUniform2025}
but has remained fairly obscure in the statistical community.
Their results are different from ours for various reasons: they do not consider
generalized least squares or weighted sup-norms, most results are only
asymptotic with unknown constants and \citet{maranLeastSquaresUniform2025}
mostly consider finite domains.
That weighted sup-norms are necessary to obtain useful bounds for interpolation
problems with unbounded domains is well-known in approximation theory
\citep[e.g.][]{lubinskySurveyWeightedPolynomial2007}, which inspired
this particular generalization.

Bounds on the \(L^2(\mu)\) error using the best
approximation error in sup-norm were developed by
\citet{cohenStabilityAccuracyLeast2013,cohenOptimalWeightedLeastsquares2017,adcockOptimalSamplingLeastsquares2025}
and we generalize these mixed norm bounds to generalized least squares and weighted sup-norms
in Thm.~\ref{thm: mixed norm bound}. They also proved concentration of the
Gram matrix in the random regression setting, which we state as a reusable lemma
(Lem.~\ref{lem: gram concentration}). The spectral bound characterizations
of the Gram matrix (Lem. \ref{lem: gram bound equivalence} and Thm.~\ref{thm:
approximately identity gram matrix}) were also heavily inspired by their proofs.

For the special case of Chebyshev polynomials uniform bounds on the regression
error were developed by \citet{matsudaPolynomialApproximationNoisy2025}.
Their proofs inspired our uniform stochastic error bounds, in particular Lemma \ref{lem: pointwise to uniform conversion},
which allows for the conversion of pointwise stochastic error bounds to uniform bounds. 
As they remark themselves, their application of the Cauchy-Schwarz inequality
results in an overestimate of the approximation error of order \(\sqrt{m}\).
Translated to our framework they essentially use the leverage function bound on
the Lebesgue function (Thm.~\ref{thm: leverage function
bound}). We improve upon this for Fourier and Chebyshev regression by replacing the
\(\sqrt{m}\) bound with an asymptotically sharp
\(\bigO(\log m)\) bound (Section~\ref{sec: lebesgue constants of regression}).

\paragraph*{Reader's guide.} In Section \ref{sec: regression} we recapitulate the basics of regression
to establish notation and to make the paper accessible to readers with an
approximation theoretic background. In this section we also show how to separate the analysis
of the approximation error from the stochastic error of the regression.
These are then treated separately in the following sections
(Sec.~\ref{sec: bias bounds} and \ref{sec: noise bounds}). Both sections
are split into two parts: bounding the error in terms of the weighted-sup norm
and bounding it in terms of the \(L^2(\mu)\)-norm. The sup-norm approximation
error bounds (Sec.~\ref{sec: relative weighted sup-norm bounds}) draw heavily
from approximation theory and this is where we
introduce the Lebesgue function and constant and strategies to bound them. Readers
that want to understand the underlying approximation theory behind these
bounds should read the corresponding appendix (Sec.~\ref{sec: approximation
theory}), where we recapitulate the operator theoretic tools from approximation theory
that are used to prove these bounds. 
In Section \ref{sec: relative Lp bounds} we first establish the necessity of
mixed norm bounds for the \(L^2(\mu)\) approximation error before stating their
extension to generalized least squares. Section \ref{sec: noise bounds}
develops the corresponding bounds on the stochastic error.

In Section \ref{sec: strategies} we discuss the results from the previous
sections, including a flowchart on how to arrive at sup-norm bounds and
two common strategies for regression design. In Section \ref{sec: catalogue}
we catalogue known approximation bounds with a careful eye for constants
and list the new Fourier and Chebyshev regression bounds on the Lebesgue
constants (Section~\ref{sec: lebesgue constants of regression}). In Section
\ref{sec: examples} we apply our results to several examples.

 \section{Regression}
\label{sec: regression}

The classic (controlled) function estimation setting aims to estimate a function
\(f\colon \domain \to \field\) for \(\field \in \set{\real, \complex}\) by
choosing evaluation points \(x_1, \dots, x_n\) in the domain \(\domain\) and
observing noisy function values
\begin{equation}
    \label{eq: regression data}    
    y_i = f(x_i) + \noise_i,
\end{equation}
where \(\noise_i\) is centered noise on the \(i\)-th evaluation. In the uncontrolled
setting there simply exist datapoints \((x_i, y_i)_{i=1}^n\), and the goal is to estimate
the function \(f(x) \coloneq \E{y_i\mid x_i=x}\). By conditioning on the points \(x_i\) 
this problem reduces to the controlled setting except that the
points \(x_i\) cannot be chosen optimally.

\begin{definition}[Regression Setting]
    \label{def: regression setting}
    Let the \emph{model space} \(\modelSpace\) be a linear subspace of a function space \(\funSpace\) of
    functions from \(\domain\) to \(\field\), let \(X=(x_1, \dots, x_n)\) be the
    \emph{evaluation points} in the domain \(\domain\) and \(\eval_X\colon \funSpace \to \field^n\) the evaluation functional at the points \(X\),
    that is \(\eval_X[f] = (f(x_1), \dots, f(x_n))\). 
    Let \(\noiseIndep\coloneq (\noiseIndep_1,\dots, \noiseIndep_n)\) be centered, independent,
    \(\field\)-valued innovations such that
    \[
        y_i = f(x_i) + \noise_i, \qquad\text{with}\qquad \noise = A\noiseIndep
    \]
    for some \(f \in \funSpace\) and \(A \in \field^{n\times n}\).
    Equivalently, \(y=\eval_X[f] + \noise\).
    Finally, we define \(\Sigma \coloneq AA^\conT\).
\end{definition}

Throughout this paper we use many different types of (semi)-norms. While we will
introduce each norm as needed, we also provide an overview in Section~\ref{sec: notation for norms}. In general we use the notation \(\abs{\cdot}\) for
semi-norms in finite dimension like \(\field^\dims\), \(\norm{\cdot}\) for
semi-norms on function spaces and \(\opnorm{\cdot}\) for operator norms.
Recall that a matrix \(\Omega_n \in \field^{n\times n}\) is called
positive semi-definite denoted by \(\Omega_n\succeq0\) if \(a^\conT\Omega_n a\ge0\) for every \(a\in\field^n\) and in the real
case \(\Omega_n\) is assumed to be symmetric. In the complex case the Hermitian property of \(\Omega_n\)
follows from positive definiteness
\citep[Lem.~3]{berlinetReproducingKernelHilbert2004}.
Such a positive semi-definite matrix \(\Omega_n \in \field^{n\times n}\) defines the
following semi-norms
\begin{equation}
\label{eq: omega norms}
\begin{tabular}{l l l}
    \toprule
    space & semi-norm & induced by
    \\
    \midrule
    \(\field^n\) & \(\abs{\cdot}_{\Omega_n}\) & \(\scp{a}{b}_{\Omega_n} \coloneq \frac1n b^\conT \Omega_n a\)
    \\[0.3ex]
    \(\funSpace\) & \(\norm{\cdot}_{\Omega_n}\) & \(\scp{f}{g}_{\Omega_n} \coloneq \scp{\eval_X[f]}{\eval_X[g]}_{\Omega_n}\)
    \\
    \bottomrule
\end{tabular}
\end{equation}

\begin{definition}[Least squares]
    \label{def: least squares}
    Assume the standard regression setting (Def.~\ref{def: regression setting})
    let \(\Omega_n\) be a positive semi-definite matrix and \(\abs{\cdot}_{\Omega_n}\) as
    defined in \eqref{eq: omega norms}.
    \begin{enumerate}[label=(\roman*),noitemsep]
        \item\label{it: general least squares} \textbf{Generalized least squares (GLS)}
        defines the operator
        \begin{equation}
            \label{eq: definition general least squares}    
            \regMin(y) \coloneqq \argmin_{\model \in \modelSpace} \abs[\big]{\eval_X[\model] - y}_{\Omega_n}^2.
        \end{equation}
        
        \item\label{it: weighted least squares} \textbf{Weighted least squares (WLS)} is a special
        case of GLS, where \(\Omega_n\) is a diagonal matrix with weights \((\Omega_n)_{ii} = \omega_i \ge 0\).
        Note that since \(\Omega_n\) is Hermitian, its diagonal entries are always real
        and they are always non-negative due to positive semi-definiteness.
        In this case
        \[
            \regMin(y) = \argmin_{\model \in \modelSpace} \frac1n\sum_{i=1}^n \omega_i \abs{\model(x_i) - y_i}^2
            \qquad \forall y\in \field^n.
        \]
        \item \textbf{Ordinary least squares (OLS)} is a special case of WLS with \(\omega_i = 1\) for all \(i\).
    \end{enumerate}
\end{definition}

These minimizers have well known closed form solutions (Theorem \ref{thm: least
squares solution}). And the solutions are given in terms of a basis of the
model space. Let \(\mBasis_1, \dots, \mBasis_m\) be any basis of
\(\modelSpace\). Then
\[
    \modelVander \coloneq \begin{pmatrix}
        \mBasis_1(x_1) & \dots & \mBasis_m(x_1) \\
        \vdots & \ddots & \vdots \\
        \mBasis_1(x_n) & \dots & \mBasis_m(x_n)
    \end{pmatrix}
    \in \field^{n\times m}
\]
is called the \emph{design matrix}. The \emph{Gram matrix} \(\Gram \in \field^{m\times m}\) is defined entry-wise as
\[
    \Gram_{kl}
    \coloneq \overline{\scp{\mBasis_k}{\mBasis_l}_{\Omega_n}}
    = \scp{\mBasis_l}{\mBasis_k}_{\Omega_n}
    = \tfrac1n(\modelVander^\conT \Omega_n \modelVander)_{kl}.
\]

\begin{theorem}[Least squares characterization]
    \label{thm: least squares solution}
    The set of GLS minimizers \(\regMin(y)\) defined in \eqref{eq: definition
    general least squares} is never empty. Indeed, \(\regOp(y) \in
    \regMin(y)\), for the GLS linear operator \(\regOp\colon \field^n \to
    \modelSpace\) defined by 
    \begin{equation}
        \label{eq: definition least squares reg op}    
        \regOp(y) \coloneq \tfrac1n\mBasisVec(\cdot)^\transpose\Gram^\dagger \modelVander^\conT \Omega_n y,
    \end{equation}
    where \(\mBasisVec \coloneq (\mBasis_1, \dots, \mBasis_m)^\transpose\) and 
    \(\Gram^\dagger\) is the Moore-Penrose pseudoinverse \citep{penroseGeneralizedInverseMatrices1955}
    of \(\Gram\) 
    that coincides with \(\Gram^{-1}\) if \(\Gram\) is invertible.
    The linear operators \(\linProj \coloneq \regOp \circ \eval_X\colon
    \funSpace \to \modelSpace\) and \(\eval_X \circ \regOp\colon \field^n \to
    \field^n\) are linear projections,\footnote{
        We recall the definition of linear projections and other projection types in Definition \ref{def:
        projection types}.
} and \(\linProj\) is a metric projection with respect to
    \(\norm{\cdot}_{\Omega_n}\), that is
    \begin{equation}
        \label{eq: LS is metric projection}
        \linProj f \in \argmin_{\model \in \modelSpace} \norm{\model - f}_{\Omega_n}^2
        \qquad \text{for } f\in \funSpace.
    \end{equation}
    The following are further equivalent:
    \begin{enumerate}[label=(\roman*),noitemsep]
        \item\label{it: least squares unique}
        \(\regMin(y)=\{\regOp(y)\}\), that is \(\regMin(y)\) is a singleton for all \(y\in \field^n\),
        \item\label{it: norm positive definite}
        The semi-norm \(\norm{\cdot}_{\Omega_n}\) is a norm on the
        model space \(\modelSpace\) (positive definite),
        \item\label{it: gram positive definite}
        the Gram matrix \(\Gram\) is invertible (positive definite),
        \item\label{it: lin proj surjective} \(\linProj\) maps onto \(\modelSpace\) (is surjective).
    \end{enumerate}
    If any (all) of the above conditions hold then the minimizer in \eqref{eq:
    LS is metric projection} is unique.
\end{theorem}

\begin{remark}[Basis dependence]
    Observe that \ref{it: least squares unique}, \ref{it: norm positive
    definite} and \ref{it: lin proj surjective} are inherently independent of the basis. The Gram matrix in
    \ref{it: gram positive definite} depends on the choice of basis, but since
    the equivalence holds for any choice, it is sufficient to have
    \ref{it: gram positive definite} for one choice of basis to get it
    for all choices of basis.
\end{remark}

Let \(y = \eval_X[f] + \noise\) be the observed data as introduced 
Definition \ref{def: regression setting}. Linearity
of \(\linProj = \regOp \circ \eval_X\) defined in Theorem \ref{thm: least squares solution}
thus implies the following error decomposition of regression into approximation, or misspecification, error and
stochastic error.
\begin{mdframed}[style=default]
\begin{corollary}[Error decomposition of regression]
\label{cor: bias noise decomposition} Let \(\regOp\) be a regression operator as
defined in \eqref{eq: definition least squares reg op} or more generally in
Definition \ref{def: regression}, then
    \[
        \regOp(y) = f + \underbrace{(\linProj f-f)}_{\mathrm{approx.\ error}} + \underbrace{\regOp(\noise)}_{\mathrm{stoch.\ }\mathrlap{\mathrm{error}}}.
    \]
\end{corollary}
\end{mdframed}
Linearity of \(\regOp\) and \(\linProj\) is the key structural assumption
for the error decomposition and thereby motivates the following abstract
definition of regression. This definition only enforces
the practicality of the linear projection \(\linProj\). That is, \(\linProj\)
must be feasible to implement and may therefore only depend on finitely
many evaluations of the target function \(f\).

\begin{definition}[Abstract regression]
    \label{def: regression}
    The linear map \(\regOp\colon \field^n \to \modelSpace\) associated to evaluation points
    \(X=(x_1, \dots, x_n) \in \domain^n\)
    is called a \emph{regression operator} if
    the linear map \(\linProj \colon \funSpace\to \modelSpace\) defined as
    \begin{equation}
        \label{eq: least squares projection}    
        \linProj \coloneq \regOp \circ \eval_X
    \end{equation}
    is a linear projection (cf.~Def.~\ref{def: projection types}) \underline{onto} the model
    space \(\modelSpace\). Then 
    \(\linProj\) is called a \emph{regression projection}.
\end{definition}

\begin{corollary}[Least squares is an abstract regression]
    \label{cor: least squares is an abstract regression}
    If the Gram matrix \(\Gram\) is invertible, then the GLS operator \(\regOp\) defined in \eqref{eq: definition least squares reg op} is a regression operator.
\end{corollary}
\begin{proof}
    Theorem \ref{thm: least squares solution} shows that \(\linProj\) is a linear projection onto \(\modelSpace\). 
\end{proof}

\begin{remark}[Non-invertible case]
    \label{rem: non-invertible gram matrix}
    If \(\Gram\) is not invertible, then \(\regOp\) is not surjective by Theorem
    \ref{thm: least squares solution} \ref{it: gram positive definite}
    \(\Leftrightarrow\) \ref{it: lin proj surjective}.
    It therefore does not satisfy the definition of a regression operator.
    Redefining \(\modelSpace\) as the image of this linear map results in a
    regression operator. A Gram matrix constructed with a basis of this
    image is then invertible by Thm.~\ref{thm: least squares solution}
    due to surjectivity
    while the regression operator remains unchanged.  In other words: The
    over-parametrized case \(m > n\) reduces to the classic setting with
    \(m=\rank \Gram \le n\) after selecting a (\(\rank \Gram\))-dimensional
    subspace of \(\modelSpace\) that is the image of the linear map \(\regOp\).
\end{remark}
 	\section{Bounds on the approximation error}
\label{sec: bias bounds}

To quantify the magnitude of the approximation error \(\linProj f - f\) in Corollary \ref{cor:
bias noise decomposition}, it is necessary to introduce a notion of magnitude.
For this purpose let \(\norm{\cdot}_*\) be a (semi-)norm on the function
space \(\funSpace\). The minimum error that is achievable using functions from
the model space \(\modelSpace\)
is thus given by the \emph{best approximation error} (or \emph{metric projection error})
\begin{equation}
    \label{eq: metric projection error}
    \error_\modelSpace^*[f] \coloneq \inf_{\model\in \modelSpace} \norm{\model-f}_*.
\end{equation}
Clearly, the approximation error must be at least as large as the metric projection error. Since
investigating all pairs of regression operators and function spaces is
a daunting task, it turns out to be fruitful to split this problem
into two tasks:
\begin{enumerate}[nosep]
    \item Bound the metric projection error \(\error_\modelSpace^*[f]\) for various function spaces \(\funSpace\), model spaces \(\modelSpace\) and norms \(\norm{\cdot}_*\)
    (Section \ref{sec: metric projection error}).
    \item Bound the approximation error of a specific regression operator relative to the best approximation
    error.
\end{enumerate}
Since the best approximation error remains the same whether one uses
regression or interpolation all the results on the metric projection error
from approximation theory are directly applicable. These results
are reviewed in Section \ref{sec: metric projection error}. Otherwise
we focus on \emph{relative} approximation error bounds.

\subsection{Relative weighted sup-norm bounds}
\label{sec: relative weighted sup-norm bounds}

In this section we will consider the approximation error in the weighted
sup-norm
\[
    \norm{f}_\weight \coloneq \norm{f\weight}_\infty = \sup_{x\in \domain} \abs{f(x)} \weight(x)
\]
for weight functions \(\weight \colon \domain \to (0, \infty)\), which we always implicitly
assume to be measurable. Observe that
\(\weight \equiv 1\) reduces to the usual sup-norm. To emulate \(\weight(x)=0\)
one may drop \(x\) from the domain \(\domain\).
On unbounded domains weights are vital for approximation to succeed as
it is typically impossible to approximate functions uniformly. For example, bounded
continuous functions cannot be approximated uniformly by polynomials on
unbounded domains, as nonconstant polynomials are unbounded on unbounded domains. 
For \(\domain \subseteq \field^\dims\) the typical form of such weight functions
are the ``Freud weights'' \citep{lubinskySurveyWeightedPolynomial2007}
\[
    \weight(x) = \exp(-Q(x))
\quad \text{with}\quad Q(x) \propto \abs{x}_2^\alpha
\]
for \(\abs{\cdot}_p\) the \(p\)-norm and
\(\alpha\in[1,\infty)\).\footnote{
Convergence in the norms induced by Freud weights implies convergence in the
    `compact-open' topology which requires uniform convergence on all compact
    sets. Assuming \(\domain\) to be locally compact, this topology is the
    unique `acceptable' topology to topologists \citep[Prop.~2.6.12 (iv) and
    Thm.~3.4.3]{engelkingGeneralTopologyRevised1989} due to its `proper' and `admissible'
    interaction with evaluation mappings. The coarser topology of pointwise
    convergence is not metrizable for uncountable domains \(\domain\) and generally not
    `admissible', since \((f,x) \mapsto f(x)\) is generally not continuous.
    Convergence in the finer topology of uniform convergence is often impossible
    to achieve on unbounded
    domains and generally not `proper', since \(z\mapsto f(z, \cdot)\) is not necessarily
    continuous for continuous \(f\). It may therefore be unsurprising to
    topologists that the unweighted sup-norm is typically only useful for
    approximation on compact domains \(\domain\) where the compact-open topology
    coincides with the uniform topology.
}
For the optimal bounds on the approximation error relative to the best approximation
error \(\error_\modelSpace^w[f]\), we require a few concepts inspired by
approximation theory \citep[e.g.][]{trefethenApproximationTheoryApproximation2019,devoreConstructiveApproximation1993}: `Lagrange functions' motivated by the Lagrange
polynomials (cf.~Example~\ref{ex: lagrange polynomial}), 
`Lebesgue functions' and `Lebesgue constants'.

\begin{definition}[Lagrange function]
    \label{def: lagrange function}
    Let \(\regOp\) be a regression operator.
    The \emph{Lagrange function} \(\ell\) maps \(x\in \domain\) to the unique
    vector \(\ell(x) \in \field^n\) such that
    \[
         \ell(x)^\conT y = \pi_x \circ \regOp(y) \qquad \forall y\in \field^n,
    \]
    where \(\pi_x(f) = f(x)\). In other words, \(\ell(x)\) is the unique Riesz
    representation of the linear functional \(\pi_x \circ \regOp\colon
    \field^n \to \field\).
\end{definition}

Observe that the above definition of the Lagrange function is independent of the
choice of basis of the model space. For a concrete choice of basis
Theorem~\ref{thm: least squares solution} immediately implies the
following explicit representation for least squares regression operators.

\begin{example}[Lagrange function of least squares]
    \label{ex: lagrange function of least squares}
    The Lagrange function of the GLS operator \(\regOp\) defined in \eqref{eq:
    definition least squares reg op} is given by 
    \begin{equation}
        \label{eq: lagrange function least squares}    
        \ell(x) = \tfrac1n\Omega_n \modelVander \Gram^\dagger \overline{\mBasisVec(x)}.
    \end{equation}
    \textbf{Warning:} Recall that a regression projection must map \underline{onto} \(\modelSpace\)
    (Definition~\ref{def: regression}).
    If \(\Gram\) is not invertible \(\modelSpace\) must therefore be redefined to be the image of \(\linProj =
    \regOp \circ \eval_X\) (cf.~Remark~\ref{rem: non-invertible gram
    matrix}). This changes
    the metric projection error \(\error_\modelSpace^\weight[f]\) defined in \eqref{eq: metric projection error}.
    And after this redefinition, the Gram matrix \(\Gram\) is invertible (Thm.~\ref{thm: least squares solution}).
\end{example}

\begin{definition}[Lebesgue function and constant]
    \label{def: lebesgue function and constant}
    Let \(\ell\colon \domain \to \field^n\) be the Lagrange function of a
    regression projection \(\linProj\), then the \emph{Lebesgue function} and
    \emph{Lebesgue constant} are given by
    \[
        \lebfun{\weight}{\linProj}{x} \coloneq \weight(x)\sum_{i=1}^n \frac{\abs{\ell_i(x)}}{\weight(x_i)}
        \qquad\text{and}\qquad
        \lebconst{\weight}{\linProj} \coloneq \sup_{x\in \domain} \lambda_\weight(\linProj; x).
    \]
\end{definition}
In particular for \(\weight \equiv 1\), the Lebesgue function is simply the
\(1\)-norm \(\abs{\ell(x)}_1\) of the Lagrange function.
In this unweighted case, we write \(\lebfun{\infty}{\linProj}{x}\) and
\(\lebconst{\infty}{\linProj}\), using the sup-norm index \(\infty\).

\begin{mdframed}[style=default]
\begin{theorem}[Optimal approximation error bound in weighted sup-norm]
    \label{thm: optimal bias bounds in the weighted sup-norm}
    Let \(\regOp\) be a regression operator and \(\linProj\) be the corresponding regression projection.
    \begin{enumerate}[label=(\roman*)]
        \item\label{it: (weighted) pointwise bound}\textbf{Pointwise bound.} Let \(\lebfun{\weight}{\linProj}{x}\) be the Lebesgue function, then
        \[
            \abs{\linProj f(x) - f(x)}\weight(x)
            \le (1+\lebfun{\weight}{\linProj}{x}) \error_\modelSpace^\weight[f],
        \]
        \underline{Optimal:} If the function space \(\funSpace\) is
        sufficiently rich (Def.~\ref{def: sufficiently rich function space}) and
        the \(x_i\) in \(X=(x_1, \dots, x_n)\) distinct, then 
        for any \(x \notin \set{x_1, \dots, x_n}\) there exists \(f\in
        \funSpace\) such that the bound holds with equality. In particular,
        there is no constant smaller than \(1+\lebfun{\weight}{\linProj}{x}\)
        that holds for all \(f\in \funSpace\).

        \item\label{it: (weighted) uniform bound}\textbf{Uniform bound.}
        Let \(\lebconst{\weight}{\linProj}\) be the Lebesgue constant,
        then
        \[
            \norm{f - \linProj f}_\weight
            \le (1+\lebconst{\weight}{\linProj}) \error_\modelSpace^\weight[f],
        \]
        \underline{Optimal:} If \(\funSpace\) is sufficiently rich
        (Def.~\ref{def: sufficiently rich function space}), the \(x_i\) in
        \(X=(x_1, \dots, x_n)\) are distinct and not isolated points of \(\domain\), \(\modelSpace \subseteq C(\domain;
        \field)\) and \(\weight \in C(\domain; \real)\), then there is no
        constant smaller than \(1+\lebconst{\weight}{\linProj}\) that holds for all
        \(f\in \funSpace\).
    \end{enumerate}
\end{theorem}
\end{mdframed}

The intuition behind this theorem lies in the operator theoretic view
of approximation theory discussed in Appendix~\ref{sec: approximation theory}.

Observe that the bounds are uniform in the functions \(f\in \funSpace\).
Selecting \(\funSpace\) to be tiny would therefore open the door to better bounds.
The requirement that the function space is sufficiently rich is therefore
natural, which is formally defined in the following definition.

\begin{definition}[Sufficiently rich]
    \label{def: sufficiently rich function space}
    The function space \(\funSpace\) is called \emph{sufficiently rich} if for
    all \(n\in \nat\), all \(X=(x_1, \dots, x_n) \in \domain^n\) such that the
    \(x_i\) are distinct and all \(y\in \field^n\) there exists \(f\in
    \funSpace\) such that \(\eval_X[f] = y\) and \(\norm{f}_\weight \le
    \abs{\eval_X[f\weight]}_\infty\) (i.e.\ interpolation without overshoot).
\end{definition}

\begin{example}
    If \(\weight\in C^\infty(\real^\dims; (0,\infty))\), then the function space
    \(\funSpace = C_c^\infty(\real^\dims; \field)\) of infinitely differentiable functions
    with compact support is sufficiently rich.
\end{example}
\begin{proof}
    Using bump functions, construct a function
    \(g \in C^\infty_c(\real^\dims; \field)\) which interpolates \(y\odot
    \eval_X[\weight] = (y_1\weight(x_1), \dots, y_n \weight(x_n))\) and assumes
    its maximum at these points such that \(\norm{g}_\infty \le
    \abs{\eval_X[g]}_\infty\). Then \(f \coloneq g/\weight
    \in C^\infty_c(\real^\dims;\field)\) does the job.
\end{proof}

\subsubsection{Reproducing kernel representation}

The definition of the Lagrange function is fairly abstract and the concrete representation
of the Lagrange function for least squares regression (Example \ref{ex: lagrange function of least squares})
appears to be basis dependent. The representation of the Lagrange function in
terms of the `reproducing kernel' of the model space \(\modelSpace\), will make
it much easier to derive bounds on the Lebesgue function and constant.

\begin{lemma}[Reproducing kernel]
    \label{lem: reproducing kernel}
    For any inner product \(\scp{\cdot}{\cdot}_*\) over \(\field\) on the
    \(m\)-dimensional model space \(\modelSpace\), there exists a unique
    \emph{reproducing kernel} \(K_\modelSpace^*\colon \domain
    \times \domain \to \field\) with the defining properties:
    \begin{enumerate}[label=(\roman*), noitemsep]
        \item \(K_\modelSpace^*(\cdot, y) \in \modelSpace\) for all \(y\in \domain\),
        \item the kernel has the reproducing property, that is for all \(y\in \domain\) and all \(\model \in \modelSpace\)
        \[
            \scp{\model}{K_\modelSpace^*(\cdot, y)}_* = \model(y).
        \]
    \end{enumerate}
    Let \(\mBasis_1, \dots, \mBasis_{m}\) be any orthonormal basis of
    \(\modelSpace\) with respect to the inner product \(\scp{\cdot}{\cdot}_*\).
    Then the reproducing kernel has the representation
    \[
        K_\modelSpace^*(x, y) = \sum_{k=1}^{m} \mBasis_k(x)\overline{\mBasis_k(y)}.
    \]
    In particular the kernel is positive semi-definite.
\end{lemma}

\begin{remark}[Name and relationship to the Christoffel function]
    In approximation theory, \(\reprK_\modelSpace^*\) is sometimes referred to as the Christoffel-Darboux
    kernel due to the Christoffel-Darboux formula
    \citep[Theorem~1.32]{gautschiOrthogonalPolynomialsComputation2004} that
    efficiently computes \(\reprK_\modelSpace^*\) when \(\modelSpace\) is the space of polynomials.
    The well-studied \emph{Christoffel function} is defined as
    \(\eta_\modelSpace^*(x) \coloneq 1/{\reprK_\modelSpace^*(x,x)}\).
    Lower bounds on the Christoffel function \citep[e.g.][]{levinChristoffelFunctionsOrthogonal1992} thus
    translate to upper bounds on the reproducing kernel.
\end{remark}

\begin{theorem}[Reproducing kernel representation]
    \label{thm: reproducing kernel representation}
    Let \(\regOp\) be a GLS regression operator as defined in \eqref{eq: definition least squares reg op}
    using the matrix \(\Omega_n\), with regression projection \(\linProj\) mapping onto the model space \(\modelSpace\).
    Then the semi-norm \(\norm{\cdot}_{\Omega_n}\) is a norm on \(\modelSpace\) and
    the regression operator \(\regOp\) and Lebesgue function have the representations
    \begin{equation}
        \label{eq: regression operator reproducing kernel representation}
        \regOp(y)
        = \scp[\Bigg]{y}{
            \begin{pmatrix}
                \reprK_\modelSpace^{\Omega_n}(x_1,\cdot)\\
                \vdots\\
                \reprK_\modelSpace^{\Omega_n}(x_n,\cdot)
            \end{pmatrix}}_{\Omega_n}
        \qquad\text{and}\qquad
        \ell(x) = \frac1n\Omega_n
            \begin{pmatrix}
            \reprK_\modelSpace^{\Omega_n}(x_1, x)\\
            \vdots \\
            \reprK_\modelSpace^{\Omega_n}(x_n, x)
        \end{pmatrix}.
    \end{equation}
    In particular, in the WLS case \(\Omega_n = \diag(\omega_1, \dots, \omega_n)\), the Lebesgue function is given by
    \[
        \lebfun{\weight}{\linProj}{x} = \weight(x)\frac1n\sum_{i=1}^n \frac{\omega_i}{\weight(x_i)}\abs{K_\modelSpace^{\Omega_n}(x_i, x)}.
    \]
\end{theorem}
\begin{remark}[Non-negative reproducing kernels]
    \label{rem: positive reproducing kernels}
    In the WLS setting assume \(\ones \in \modelSpace\) and
    \(\reprK_\modelSpace^{\Omega_n}(x,y) \ge 0\) for all \(x,y\in \domain\) (in particular \(\reprK_\modelSpace^{\Omega_n}(x,y)\in \real\)), then 
    \[
        \lebfun{\infty}{\linProj}{x} = \frac1n\sum_{i=1}^n \omega_i\abs{K_\modelSpace^{\Omega_n}(x_i, x)}
        = \scp{\ones}{\reprK_\modelSpace^{\Omega_n}(\cdot, x)}_{\Omega_n}
        \overset{\text{repr.}}= \ones(x) = 1.
    \]
    This setting will be satisfied for piecewise constant functions for example (cf.~Section~\ref{sec: lebesgue constants of regression}).
\end{remark}

\subsubsection{Leverage function bound}

The following `leverage function' measures how big the change of \(\model(x)\) may be
per change of \(\model\) in the norm \(\norm{\cdot}_*\) (cf.~Lemma \ref{lem:
variational characterization}).  That is, how much leverage the norm
\(\norm{\cdot}_*\) has at point \(x\). This function is also closely related
to statistical leverage scores \citep[p.~2003]{adcockOptimalSamplingLeastsquares2025}.
\citet{adcockOptimalSamplingLeastsquares2025} refers to the square leverage as the (reciprocal) Christoffel function.

\begin{definition}[Leverage function]
    \label{def: leverage function}
    The root of the diagonal of the reproducing kernel
    \(\lever_\modelSpace^*(x) \coloneq \sqrt{K_\modelSpace^*(x,x)}\) is the
    \emph{leverage function}.
\end{definition}

\begin{lemma}[Variational characterization] 
    \label{lem: variational characterization}
    The leverage function is equal to the operator norm of
    the evaluation map \(\eval_x[f] = f(x)\) restricted to the model space \(\modelSpace\), that is
    \begin{equation}
        \label{eq: pointwise variational repro kernel bound} 
        \lever^*_\modelSpace(x)
        = \sup_{\model\in \modelSpace, \norm{\model}_* \le 1} \abs{\model(x)}
        \overset{\mathrm{def.}}= \opnorm{{\eval_x}_{\restriction\modelSpace}}_{*}.
    \end{equation}
    The weighted supremum over \(x\) is simply the operator norm
    of the identity embedding the space \((\modelSpace, \scp{\cdot}{\cdot}_*)\)
    into the space \((\modelSpace, \norm{\cdot}_\weight)\), that is
    \begin{equation}
        \label{eq: variational weighted sup bound}        
        \norm{\lever_\modelSpace^*}_\weight
        \overset{\mathrm{def.}}= \sup_{x\in \domain} \weight(x)\lever^*_\modelSpace(x)
        = \sup_{\model \in \modelSpace, \norm{\model}_* \le 1}  \norm{\model}_\weight
        \overset{\mathrm{def.}}= \opnorm{\id_{\restriction\modelSpace}}_{*}^\weight.
    \end{equation}
\end{lemma}

If a small change of the model \(\model\) measured by \(\norm{\cdot}_{\Omega_n}\)
causes a big change of \(\model(x)\), it is intuitive that the approximation may
not be very good at \(x\), since least squares regression is a minimizer in
\(\norm{\cdot}_{\Omega_n}\). And indeed, the Lebesgue function may be bounded
using the weighted leverage function \(\lever_\modelSpace^{\Omega_n}\weight\).

\begin{mdframed}[style=default]
\begin{theorem}[Leverage function bounds]
    \label{thm: leverage function bound}
    The Lebesgue function is bounded in terms of the weighted leverage function, that is
    \begin{equation}
        \lebfun{\weight}{\linProj}{x}
        \label{eq: rk bound Omega}
        \le A_\weight(\Omega_n) \lever_\modelSpace^{\Omega_n}(x)\weight(x),
        \qquad
        A_\weight(\Omega_n)
        \coloneq \sqrt{
            \frac1n\sum_{i,j=1}^n \frac{\abs{(\Omega_n)_{ij}}}{\weight(x_i)\weight(x_j)}
        }.
    \end{equation}
    Let \((\mBasis_k^*)_{k=1}^m\) be an orthonormal basis of \(\modelSpace\) with respect to an inner
    product \(\scp{\cdot}{\cdot}_*\) on \(\modelSpace\) that also induces the reproducing kernel \(K_\modelSpace^*\)
    and leverage function \(\lever_\modelSpace^*\). Then,
    \[
        \lever_\modelSpace^{\Omega_n}(x) \le \sqrt{\opnorm{\Gram_*^{-1}}_2}
        \lever_\modelSpace^*(x)
        \qquad\text{for}\qquad
        (\Gram_*)_{ij} \coloneq \scp{\mBasis_j^*}{\mBasis_i^*}_{\Omega_n},
    \]
    where \(\opnorm{\Gram_*^{-1}}_2\) is the spectral norm of the inverse Gram
    matrix, and consequently
    \begin{equation}
        \label{eq: rk bound general}
        \lebfun{\weight}{\linProj}{x}
        \le A_\weight(\Omega_n)\sqrt{\opnorm{\Gram_*^{-1}}_2} \lever_\modelSpace^*(x)\weight(x).
    \end{equation}
\end{theorem}
\end{mdframed}

\begin{remark}[Notes on \(A_\weight(\Omega_n)\)]
    The absolute value around \((\Omega_n)_{ij}\) in \(A_\weight(\Omega_n)\) is necessary in the GLS setting
    (cf.~Section~\ref{sec: absolute value around Omega is necessary}). In the WLS setting:
    \begin{itemize}[nosep]
        \item \(A_\weight(\Omega_n) = \norm{\tfrac1\weight}_{\Omega_n}\), since
        the diagonal entries of \(\Omega_n\) are non-negative, and
        \item \(\omega_i = \weight^2(x_i)\) implies \(A_\weight(\Omega_n) = 1\) (e.g.\ OLS with \(\weight\equiv 1\)).
    \end{itemize}
\end{remark}

\begin{remark}[Scale invariance]
    \label{rem: scale invariance}
    The bound \eqref{eq: rk bound general} is invariant to scaling of \(\Omega_n\) (contained
    in both \(A_\weight(\Omega_n)\) and \(\Gram\)), scaling of the inner product
    \(\scp{\cdot}{\cdot}_*\) (influencing \(\Gram_*\) and \(\lever_\modelSpace^*\) proportionally),
    and scaling of the weight function \(\weight\).
\end{remark}

Observe that \eqref{eq: rk bound Omega} is a special case of \eqref{eq: rk bound
general} with \(\scp{\cdot}{\cdot}_* = \scp{\cdot}{\cdot}_{\Omega_n}\) and \(\Gram_*
= \identity\). This suggests that a well-conditioned Gram matrix is desirable.
However, rather than using \(\scp{\cdot}{\cdot}_{\Omega_n}\) directly, it is generally easier to use a well studied
inner product that is \emph{similar} to \(\scp{\cdot}{\cdot}_{\Omega_n}\) in the
sense of Lemma~\ref{lem: gram bound equivalence}. Specifically, the following lemma shows that the similarity
of the inner product is intimately related to the spectral norm of the inverse Gram matrix
\(\opnorm{\Gram_*^{-1}}_2\). In particular, the bound only depends on the
choice of inner product and not the selection of the orthonormal basis that
induces the Gram matrix.

\begin{lemma}[Equivalent lower spectral bounds]
    \label{lem: gram bound equivalence}
    Using the Gram matrix \(\Gram_*\) from Theorem~\ref{thm: leverage function bound}
    and \(\gamma \in (0,\infty)\) the following are equivalent:
    \begin{enumerate}[label=(\roman*),noitemsep]
        \item\label{it: (gram bound) inverse}
        Inverse Gram bound: \(\opnorm{\Gram_*^{-1}}_2 \le \gamma\)
        \item\label{it: (gram bound) spectral}
        Spectral Gram bound: \(\identity \precsim \gamma\Gram_*\)
        \item\label{it: (gram bound) induced norm}
        Induced norm bound: \(\norm{\model}_* \le \sqrt{\gamma} \norm{\model}_{\Omega_n}\)
        for all \(\model \in \modelSpace\).
    \end{enumerate}
\end{lemma}

If one strengthens the lower bound on the spectrum to a symmetric upper and lower bound
on the spectrum, the norms become equivalent on the model space and similarity
of the inner products also becomes an equivalent condition (cf.~Theorem \ref{thm: approximately identity gram matrix}).

\paragraph*{Discussion of the bounds.} 

Bounds using the leverage function are often easier to work with.
For example, a stationary kernel always has a constant leverage function as can be seen
directly from its definition. The bound \eqref{eq: rk bound general}
further decouples the choice of leverage function from the choice of evaluation
points, which then only influence the Gram matrix and the weights in
\(A_\weight(\Omega_n)\). The bounds \eqref{eq: rk bound Omega} and \eqref{eq: rk
bound general} are essentially
Cauchy-Schwarz bounds and we will see in Lemma~\ref{lem: optimality of constant weighted leverage bound} that the price to pay is
roughly of order \(\sqrt{m}\). An intuitive explanation for the \(\sqrt{m}\)-order
of the constant is that for \(x\in \field^m\) we have \(\abs{x}_1 \le \sqrt{m}\abs{x}_2\) by Cauchy-Schwarz,
and the model space \(\modelSpace\) is \(m\)-dimensional.

To understand Lemma~\ref{lem: optimality of constant weighted leverage bound} recall that the Lebesgue constant
is defined as the supremum of the Lebesgue function, that is  \(\lebconst{\weight}{\linProj}= \sup_{x\in \domain}\lebfun{\weight}{\linProj}{x}\).
Consequently, the bounds \eqref{eq: rk bound Omega} and \eqref{eq: rk bound general} are typically maximized over \(x\in \domain\)
to obtain
\begin{equation}
    \label{eq: reproducing kernel bound on the lebesgue constant}    
    \lebconst{\weight}{\linProj}
    \le A_\weight(\Omega_n) \sqrt{\opnorm{\Gram_*^{-1}}_2} \sup_{x\in \domain} \weight(x) \lever_\modelSpace^*(x)
    = A_\weight(\Omega_n) \sqrt{\opnorm{\Gram_*^{-1}}_2} \norm{\lever_\modelSpace^*}_\weight.
\end{equation}
The following well known result \citep[e.g.][]{narayanChristoffelFunctionWeighted2016} shows
that \(\norm{\lever_\modelSpace^*}_\weight\) is at least of order \(\sqrt{m}\).
As the Gram matrix may be selected to be of order \(1\) just like
\(A_\weight(\Omega_n)\) (e.g.\ \(A_\weight(\Omega_n) = 1\) in the OLS setting
with \(\weight \equiv 1\)) this implies the upper bound \eqref{eq:
reproducing kernel bound on the lebesgue constant} on the Lebesgue constant is
of the same order. That is, the upper bound on the Lebesgue constant \eqref{eq: reproducing kernel bound on the lebesgue constant} using the leverage
function is at best of order \(\sqrt{m}\) and this order is achieved by constant weighted leverage functions.

\begin{lemma}[Leverage lower bound]
    \label{lem: optimality of constant weighted leverage bound}
    Let \(\mu\) be a finite measure on \(\domain\) inducing
    the inner product \(\scp{\cdot}{\cdot}_\mu\) on \(\modelSpace\) and thereby the
    reproducing kernel \(K^\mu_\modelSpace\). Then for \(\nu(dx) \coloneq \frac1{\weight^2(x)}\mu(dx)\) and \(m=\dim \modelSpace\)
    \[
        \norm{\lever_\modelSpace^\mu}_\weight
        \ge \sqrt{\frac{m}{\nu(\domain)}}.
    \]
    \begin{itemize}[nosep]
        \item If \(\lever^\mu_\modelSpace(x)\weight(x)\) is constant in \(x\), then the above is an equality,
        \item if \(\lever^\mu_\modelSpace(x)\weight(x)\) is not \(\nu\)-almost surely constant in \(x\), then the
        above is a strict inequality
    \end{itemize}
    For continuous kernels \(K^\mu_\modelSpace\) and weights \(\weight\) where \(\supp \nu = \domain\), the
    inequality is therefore an equality if and only if the weighted leverage function is
    constant.
\end{lemma}

\subsubsection{Restricted range inequalities for unbounded domains}
\label{sec: restricted range inqualities}

Restricted range inequalities are a very useful tool from polynomial
approximation theory \citep[e.g.][]{lubinskySurveyWeightedPolynomial2007}, which may be generalized to arbitrary model spaces. The
core idea is that if the weight \(\weight\) decays sufficiently quickly, then
it is plausible that the weighted norm may be determined on a compact set. 
And if that is the case, one may reduce approximation problems on unbounded
domains to problems on compact domains, which are generally more tractable.

\begin{definition}[Restricted range inequality]
    \label{def: restricted range inequality}
    A semi-norm \(\norm{\cdot}_*\) obeys a \emph{restricted range inequality}
    with constant \(C>0\), if there exists a compact set \(S_\modelSpace\subseteq \domain\) that
    may depend on the model space \(\modelSpace\) such
    that
    \[
        \norm{\model}_* \le C \norm{\model \ind{S_\modelSpace}}_{*}
        \quad \forall \model \in \modelSpace.
    \]
    A restricted range inequality is called \emph{exact} if \(C=1\).
\end{definition}

Restricted range inequalities are already available for polynomials model spaces
combined with exponential weights (cf.~Theorem~\ref{thm: restricted range inequalities for polynomials} or \citep[Thm.~4.1]{levinOrthogonalPolynomialsExponential2001}).
They have proven useful for bounding the metric projection
error of polynomial spaces \citep[see e.g.][]{lubinskySurveyWeightedPolynomial2007}.
In the following we show that it is possible to calculate the Lebesgue constant
by maximizing the Lebesgue function only over the compact set \(S_\modelSpace\)
of the restricted range inequality.

\begin{mdframed}[style=default]
\begin{lemma}[Restricted Lebesgue constant]
    \label{lem: restricted lebesgue constant}
    Let \(\weight\colon \domain \to (0,\infty)\) be a weight function and \(\modelSpace\)
    a model space such that the weighted sup-norm satisfies a restricted range
    inequality with constant \(C>0\) and compact set \(S_\modelSpace\subseteq
    \domain\) (Definition
    \ref{def: restricted range inequality}).
    For any \(\linProj\colon \funSpace \to \modelSpace\) assume either
    \begin{itemize}[nosep]
        \item \(\linProj\) is an (abstract) regression projection onto \(\modelSpace\), or
        \item \(\linProj\) is an \(L^2(\mu)\)-projection onto \(\modelSpace\) for some \(\sigma\)-finite measure \(\mu\) and \(\funSpace=L^2(\mu)\).
    \end{itemize}
    Then we have
    \[
        \lebconst{\weight}{\linProj}
        \overset{\mathrm{def.}}= \sup_{x\in\domain} \lebfun{\weight}{\linProj}{x}
        \le C \sup_{x\in S_\modelSpace} \lebfun{\weight}{\linProj}{x}.
    \]
    If the restricted range inequality is exact (i.e.\ \(C=1\)), then the inequality is an equality.
\end{lemma}
\end{mdframed}

It is therefore a natural strategy to identify a minimal compact set
\(S_\modelSpace\) such that a restricted range inequality holds and
focus the attention on this set. For example it is usually enough to choose
\(\mu\) with support on \(S_\modelSpace\) to get good approximation properties
for the \(L^2(\mu)\)-projection, which may be translated to regression with
exact quadrature or random sampling (cf. Section \ref{sec: strategies}).

\subsection{Relative \texorpdfstring{\(L^p(\mu)\)}{Lp(µ)}-norm bounds}
\label{sec: relative Lp bounds}

We denote the \(L^p(\mu)\)-norm by
\(\norm{f}_{L^p(\mu)} \coloneq \bigl(\int \abs{f(x)}^p
\mu(dx)\bigr)^{\frac1p}\), which defines a semi-norm on the space of
functions \(\mathcal{L}^p(\mu)\), where this norm is finite,
and a norm on the space of equivalence classes \(L^p(\mu)\).
In the case \(p=2\) this norm is induced by the inner product
\(\scp{\cdot}{\cdot}_\mu\) and we denote it by \(\norm{\cdot}_\mu\).
The best approximation error \(\error_\modelSpace^{L^p(\mu)}[f]\) with respect to
the \(L^p(\mu)\)-norms is unfortunately much less useful. Indeed the following theorem shows it is impossible to bound the
approximation error in \(L^p(\mu)\) using \(\error_\modelSpace^{L^q(\mu)}[f]\) if the evaluation points
\(x_i\) are not atoms of \(\mu\).

\begin{mdframed}[style=default]
\begin{theorem}[Adversarial \(L^p(\mu)\)-approximation is hopeless]
    \label{thm: Lp approximation is hopeless}
    Let \(\mu\) be a finite measure on \(\domain\subseteq \real^\dims\) with
    \(\mu(\set{x_1,\dots, x_n}) = 0\) for the distinct evaluation points \(x_1, \dots, x_n\)
    of a regression projection \(\linProj\) onto \(\modelSpace\) and for \(\green{p},\blue{q}\in [1, \infty)\) let \(C_c^\infty(\domain; \field)
    \subseteq \funSpace \subseteq \mathcal{L}^p(\mu)\cap \mathcal{L}^q(\mu)\). Then we have either
    \[
        \inf_{c>0} \set[\Big]{\forall f\in \funSpace:\green{\norm{\linProj f - f}_{L^p(\mu)}} \le c\, \blue{\error_\modelSpace^{L^q(\mu)}[f]}}
        =\infty
    \]
    or \(\green{\norm{\model}_{L^p(\mu)}} = 0\) for all \(\model \in \modelSpace\)
    and
    \[
        \green{\norm{\linProj f - f}_{L^p(\mu)}}
        = \green{\norm{f}_{L^p(\mu)}}
        = \green{\error_\modelSpace^{L^p(\mu)}[f]}.
    \]
    For \(p=q\) the above remains true for \(\mu\) that are 
    finite on bounded measurable sets.
\end{theorem}
\end{mdframed}

\begin{remark}[Discussion]
    The case \(\norm{\cdot}_{L^p(\mu)} = 0\) on the entire model space is 
    a degenerate case since then the zero map \(\linProj \equiv 0\) would do
    just as well. The assumption \(C_c^\infty(\domain;\field)
    \subseteq \funSpace\) is necessary for the construction of adversarial
    functions using bump functions. Other spaces that are sufficiently rich
    may be considered. The structural assumption is that the \(x_i\) may not be
    atoms. This assumption is violated by \(\mu_n = \frac1n\sum_{i=1}^n \omega_i \delta_{x_i}\)
    for example, which corresponds to the weighted least squares setting. Since the
    WLS regression projection is the metric projection with respect to the
    \(L^2(\mu_n)\)-norm (Theorem \ref{thm: least squares solution} \eqref{eq: LS is
    metric projection}) one would trivially obtain
    \[
        \norm{\linProj f - f}_{\mu_n} = \error_\modelSpace^{\mu_n}[f],
    \]
    which is a bound with constant \(c=1\) for all \(f\in \funSpace\). 
    Adding the evaluation points \(x_i\) as atoms to \(\mu\) could therefore
    ``fix'' the impossibility of Theorem \ref{thm: Lp approximation is hopeless}.
    However, the goal of bounding the approximation error in terms of the best approximation
    error was to avoid considering all pairs of regression operators and
    function spaces. If the measure \(\mu\) depends on the
    evaluation points \(x_i\) of the regression, this ``fix'' sacrifices the very reason this bound is desirable in
    the first place.
\end{remark}

What we have not ruled out in the Theorem above is the case \(q=\infty\). And indeed,
it is possible to bound the approximation error in the \(L^p(\mu)\)-norm using
the best approximation error \(\error_\modelSpace^\weight[f]\) of the 
(weighted) sup-norm. These mixed norm bounds, where the \(L^2(\mu)\)-norm
is bounded using the sup-norm have received a lot of attention in the past
decade \citep[e.g.][]{adcockOptimalSamplingLeastsquares2025,cohenOptimalWeightedLeastsquares2017,cohenStabilityAccuracyLeast2013}.
The following theorem is such a mixed bound 
extracted from a proof of \citet{cohenOptimalWeightedLeastsquares2017}
with a few generalizations.

\begin{theorem}[Mixed norm bound]
    \label{thm: mixed norm bound}
    Let \(\mu\) be a finite measure on \(\domain\) and \(\linProj\) a GLS
    regression projection onto the model space \(\modelSpace\) using matrix
    \(\Omega_n\). If there exists a constant \(\gamma>0\) such that
    \begin{equation}
        \label{eq: mixed norm bound assumption}
        \norm{\model}_{L^p(\mu)} \le \sqrt{\gamma} \norm{\model}_{\Omega_n}
        \qquad \forall\model\in \modelSpace,
    \end{equation}
    then with \(A_\weight(\Omega_n) = \sqrt{\frac1n\sum_{i,j=1}^n \frac{\abs{(\Omega_n)_{ij}}}{\weight(x_i)\weight(x_j)}}\)
    \begin{equation}
        \label{eq: mixed norm bound}
        \green{\norm{\linProj f - f}_{L^p(\mu)}}
        \le \Bigl(\norm{\tfrac1\weight}_{L^p(\mu)} + \sqrt{\gamma} A_\weight(\Omega_n) \Bigr) \magenta{\error_\modelSpace^\weight[f]}
        \qquad \forall f\in \funSpace.
    \end{equation}
\end{theorem}

Recall that for \(p=2\), \eqref{eq: mixed norm bound assumption} is equivalent to a spectral
lower bound on the Gram matrix (Lem.~\ref{lem: gram bound equivalence}). This yields
the following corollary.

\begin{mdframed}[style=default]
\begin{corollary}[Mixed norm bound for \(p=2\)]
    \label{cor: mixed norm bound for p=2}
    Let \(\mu\) be a finite measure on \(\domain\) and \(\linProj\) a GLS
    regression projection onto the model space \(\modelSpace\) using matrix
    \(\Omega_n\). Let \(\mBasis_1, \dots, \mBasis_m\) be an orthonormal basis of
    \(\modelSpace\) with respect to \(\scp{\cdot}{\cdot}_\mu\) and
    \((\Gram_\mu)_{kl} = \scp{\mBasis_l}{\mBasis_k}_{\Omega_n}\) the associated
    Gram matrix. Then with \(A_\weight(\Omega_n)\) as in Theorem~\ref{thm: mixed norm bound}
    \[
        \green{\norm{\linProj f - f}_{\mu}}
        \le \Bigl(\norm{\tfrac1\weight}_{\mu} + A_\weight(\Omega_n)\sqrt{\opnorm{\Gram_\mu^{-1}}_2}  \Bigr) \magenta{\error_\modelSpace^\weight[f]}.
    \]
    In the WLS setting \(A_\weight(\Omega_n) =
    \norm{\tfrac1\weight}_{\Omega_n}\). If further \(\nu(dx)=\frac1{\weight^2(x)}\mu(dx)\)  
    is a probability measure and \(\omega_i = \weight^2(x_i)\)
    for all \(i\) then the bound simplifies to
    \[
        \green{\norm{\linProj f - f}_{\mu}}
        \le \Bigl(1 + \sqrt{\opnorm{\Gram_\mu^{-1}}_2}\Bigr) \magenta{\error_\modelSpace^\weight[f]}.
    \]
\end{corollary}
\end{mdframed}

\begin{remark}[On the probability measure assumption]
    For \(\norm{\frac1\weight}_\mu < \infty\) it is necessary that
    \[
        \nu(dx) = \tfrac1{\weight^2(x)}\mu(dx)
    \]
    is a finite measure. By scaling the weight
    function \(\weight\) it is therefore possible to turn \(\nu\) into a probability measure.  This
    scaling does not influence the mixed norm bound as
    \(\error_\modelSpace^\weight[f]\) scales proportionally to \(\weight\) and
    \(\norm{\frac1\weight}_\mu\) scales inversely proportional to \(\weight\).
\end{remark}

\begin{remark}[Invariance to scaling]
    The bounds of Thm.~\ref{thm: mixed norm bound} and Cor.~\ref{cor: mixed norm
    bound for p=2} are invariant to scaling of \(\Omega_n\), since positive
    scaling of \(\Omega_n\) decreases \(\gamma\) (resp.\
    \(\opnorm{\Gram_\mu^{-1}}_2\)) proportionally.
\end{remark}

Another way to work around the impossibility of adversarial \(L^p(\mu)\) bounds 
is to relax the requirement that the bound needs to hold uniformly for all \(f\in \funSpace\).
For example, if for a fixed \(f\in \funSpace\) the evaluation points \(x_i\) are
selected randomly, then the \underline{expected} \(L^2(\mu)\)-norm approximation
error can be bounded using the \(L^2(\mu)\)-metric projection error.
Specifically, for a randomized regression projection \(\linProj\)
(cf.~Definition \ref{def: randomized regression})
\citet{cohenOptimalWeightedLeastsquares2017,belloniNewAsymptoticTheory2015}
provide bounds of the form
\[
    \E{\norm{f - \linProj f}_\mu^2}
    \lesssim (\error_\modelSpace^{\mu}[f])^2
    + \epsilon.
\]

 	\section{Bounds on the stochastic error}
\label{sec: noise bounds}

Obtaining bounds on the stochastic error is comparatively simpler than obtaining bounds on the
approximation error. The following Theorem \ref{thm: pointwise concentration} lists a
selection of pointwise inequalities depending on the assumptions on the noise.
These pointwise inequalities may be combined with Lemma~\ref{lem: pointwise to
uniform conversion} to obtain uniform bounds on the stochastic error.
We will also provide \(L^2(\mu)\)-bounds in a later section.
We write \(\pSym\) for the underlying probability measure and
\(\norm{Z}_{L^p(\pSym)} \coloneq \E{\abs{Z}^p}^{1/p}\) for
the moment norm of a random variable \(Z\), where \(1\le p<\infty\).

\begin{mdframed}[style=default]
\begin{theorem}[Pointwise concentration]
    \label{thm: pointwise concentration}
    In the standard regression setting (Definition \ref{def: regression setting}) with \(y = \eval_X[f] + \noise\)
    and \(\noise = A \noiseIndep\) for some \(A \in \field^{n\times n}\) with \(\Sigma= AA^\conT\),
    let \(\regOp\) be a GLS regression operator (Thm.~\ref{thm: least squares solution} \eqref{eq: definition least squares reg op}).
    Let \(\mBasis_1,\dots, \mBasis_m\)
    be an orthonormal basis of \(\modelSpace\) with respect to \(\scp{\cdot}{\cdot}_*\)
    and \((\Gram_*)_{kl} = \scp{\mBasis_k}{\mBasis_l}_{\Omega_n}\) the associated invertible Gram matrix.
    Then for any \(x\in \domain\) we have with probability at least \(1-\delta(\varepsilon)\)
    the following inequalities
    \begin{align*}
        \abs[\big]{\regOp(\noise)(x) \weight(x)}
        &\le \tfrac{\varepsilon}{\sqrt{n}} \sqrt{\opnorm{\Gram_*^{-1}}_2 \opnorm{\Omega_n^{\frac12}\Sigma \Omega_n^{\frac12}}_2} \lever_\modelSpace^*(x) \weight(x)
        \tag{pointwise error bound}
        \\
        &\le \tfrac{\varepsilon}{\sqrt{n}} \sqrt{\opnorm{\Gram_*^{-1}}_2 \opnorm{\Omega_n^{\frac12}\Sigma \Omega_n^{\frac12}}_2} \norm{\lever_\modelSpace^*}_\weight
        \tag{uniform error bound}
    \end{align*}
    with \(\delta(\varepsilon)\) listed below subject to various assumptions (cf.~Definition \ref{def: concentration families}). Note that \(\opnorm{\Omega_n^{\frac12}\Sigma \Omega_n^{\frac12}}_2 \le \opnorm{\Omega_n\Sigma}_2\)
    (Lemma \ref{lem: bound on operator norm of Omega Sigma}).
    \begin{center}
        \begin{NiceTabular}{l p{5.65cm} l}[tabularnote={
    \footnotesize    
    \(c_\alpha>0\) is an unoptimized constant stated before \eqref{eq: sub-Weibull constant}, 
    \(C_{p,\real}\coloneq \frac{14.7p}{\ln p}\) and
    \(C_{p,\complex}\coloneq\sqrt2\,C_{p,\real}\)
    }]
    \toprule
    & Assumption (decreasing strength \(\downarrow\)) & \(\delta(\varepsilon)\)
    \\
    \midrule
    \newtag{\normalfont{(i)}}{it: (concentration) subgaussian}\tabularnote{Assume \(\field = \real\).\label{footnote: K=R}}
    & \(\noiseIndep_i \sim \subG{\sigma^2}\)
    & \(\displaystyle 2\exp\bigl(-\tfrac{\varepsilon^2}{2\sigma^2}\bigr)\)
    \\[1ex]
    \newtag{\normalfont{(ii)}}{it: (concentration) subexponential}\ref*{footnote: K=R}
    & \(\noiseIndep_i \sim \subExp{\nu, \alpha}\)
    & \(\displaystyle 2\exp\bigl(-\tfrac12\min\set{\tfrac{\varepsilon^2}{\nu^2}, \tfrac\varepsilon{\alpha}}\bigr)\)
    \\[1ex]
    \newtag{\normalfont{(iii)}}{it: (concentration) sub-Weibull}\ref*{footnote: K=R}
    & \(\noiseIndep_i \sim \subWeibull{\alpha}\), \(\norm{\noiseIndep_i}_{\psi_\alpha}\le K\), \(\alpha\le 2\)
    & \(\displaystyle 2 \exp\bigl(-c_\alpha\min\set{(\tfrac{\varepsilon}{K})^2, (\tfrac{\varepsilon}{K})^\alpha}\bigr)\)
    \\[1ex]
    \newtag{\normalfont{(iv)}}{it: (concentration) finite p-th moment}
    & \(\E{\abs{\noiseIndep_i}^2}\le \sigma^2\) and \(\E{\abs{\noiseIndep_i}^p} \le \kappa^p\), \(p> 2\)
    & \(\displaystyle \bigl(C_{p,\field} \tfrac{\max\set{\kappa, \sigma}}{\varepsilon}\bigr)^{p}\)
    \\[1ex]
    \newtag{\normalfont{(v)}}{it: (concentration) finite second moment}
    & \(\E{\abs{\noiseIndep_i}^2} \le \sigma^2\)
    & \(\displaystyle\tfrac{\sigma^2}{\varepsilon^2}\)
    \\
    \bottomrule
\end{NiceTabular}
     \end{center}
\end{theorem}
\end{mdframed}

\begin{remark}
    \(\opnorm{\Omega_n^{\frac12}\Sigma \Omega_n^{\frac12}}_2\)
    encourages \(\Omega_n = \Sigma^{-1}\) up to scaling. And since \(\Omega_n\)
    is both contained in \(\Gram_*\) and this term, the bound is invariant to
    scaling of \(\Omega_n\) such that \(\Omega_n = \Sigma^{-1}\) is encouraged.
    This choice of \(\Omega_n\) is generally a nuisance for the approximation error
    however resulting in a bias-noise tradeoff if the noise is correlated.
\end{remark}

The proof of the individual cases is fairly short. All of them use Lemma \ref{lem: bounds
on beta}, which may be easily reused for other assumptions on the noise.
By linearity of the regression operator \(\regOp\) it is further possible to
decompose the noise \(\xi = \xi_1 + \xi_2\) and analyze the two components
separately. This may be helpful to decorrelate the noise for example.

Below we list the definitions of the concentration families used in Theorem \ref{thm: pointwise concentration}.
As they are usually stated in terms of real valued random variables, we assume \(\field=\real\) for
simplicity. In the complex case, similar results follow from identifying \(\complex\)
with \(\real^2\).
Observe that the sub-Weibull family is a fairly recent generalization \citep{vladimirovaSubWeibullDistributionsGeneralizing2020} of the sub-Gaussian and sub-exponential families
with \(\alpha=2\) and \(\alpha=1\) respectively.  Since the constant
\(c_\alpha\) in the sub-Weibull case is not optimized, it only makes sense to
use the sub-Weibull family for \(\alpha<1\) in practice, which explains its position in the table.

\begin{definition}[Concentration families]
    \label{def: concentration families}
    A real valued, centered random variable \(X\) is said to be
    \begin{enumerate}[label=(\roman*)]
        \item
        \textbf{Sub-Gaussian}, denoted by \(X\sim \subG{\sigma^2}\) \citep[Def.~2.2]{wainwrightHighDimensionalStatisticsNonAsymptotic2019} for \(\sigma^2 > 0\), if 
        \[
            \E{\exp(tX)} \le \exp\bigl(\tfrac{\sigma^2 t^2}{2}\bigr) \quad \forall t\in \real.
        \]
        The parameter \(\sigma^2\) is known as the variance proxy.

        \item
        \textbf{Sub-exponential}, denoted by \(X\sim \subExp{\nu, \alpha}\) \citep[Def.~2.7]{wainwrightHighDimensionalStatisticsNonAsymptotic2019} for \(\nu, \alpha > 0\), if
        \[
            \E{\exp(tX)} \le \exp\bigl(\tfrac{\nu^2 t^2}{2}\bigr) \quad \forall \abs{t} < \tfrac1\alpha.
        \]
        \item
        \textbf{Sub-Weibull} denoted by
        \(X\sim \subWeibull{\alpha}\)
        for \(\alpha>0\)
        \citep[Def.~2]{kuchibhotlaMovingSubGaussianityHighdimensional2022}
        (aka \(\alpha\)-sub-exponential
        \citep[e.g.][]{gotzeConcentrationInequalitiesPolynomials2021}), if
        \begin{equation}
            \label{eq: sub-Weibull orlicz norm}    
            \norm{X}_{\psi_\alpha} \coloneq \inf\set[\Big]{t>0 : \E[\big]{\psi_\alpha\bigl(\tfrac{\abs{X}}t\bigr)} \le 1} < \infty
        \end{equation}
        with \(\psi_\alpha(x) \coloneq \exp(x^\alpha) - 1\) for \(x\ge 0\).
    \end{enumerate}
\end{definition}

The sub-Weibull family naturally arises from products/powers of sub-Gaussian or
sub-exponential random variables and consists of distributions with infinitely
many finite moments, but not necessarily a finite moment generating function.
Observe that `Orlicz norm' defined in \eqref{eq: sub-Weibull orlicz norm} is
only a norm for \(\alpha\ge 1\) and only a quasi-norm for \(0 < \alpha < 1\) \citep[Prop.~A.5]{kuchibhotlaMovingSubGaussianityHighdimensional2022}.

\subsection{Uniform bounds on the stochastic error}

The following lemma, based on a proof by \citet{matsudaPolynomialApproximationNoisy2025}
for the special case of Chebyshev regression, allows for the simple conversion
of pointwise to uniform bounds.
Indeed, using the uniform error bound from Theorem \ref{thm: pointwise concentration}
as \(\epsilon\)
in the following Lemma~\ref{lem: pointwise to uniform conversion} yields a uniform bound on the stochastic error.

\begin{lemma}[Pointwise to uniform conversion]
    \label{lem: pointwise to uniform conversion}
    Let \(\regOp\colon \field^n \to \modelSpace\) be an (abstract) regression operator
    (Def.~\ref{def: regression})
    and \(\epsilon, \delta > 0\), then
    \begin{equation*}
\prob[\big]{\abs{\regOp(\noise)(x)\weight(x)} > \epsilon} \le \delta
        \quad \forall x\in\domain
        \quad \implies\quad
        \prob[\Big]{\norm{\regOp(\noise)}_\weight > \inf_{\linProj_\star} \lebconst{\weight}{\linProj_\star} \epsilon} \le m\delta,
    \end{equation*}
    where the infimum is taken over all regression projections \(\linProj_\star
    \colon \funSpace \to \modelSpace\) with \(m\) distinct regression points.
\end{lemma}

While we usually assume \(m=\dim \modelSpace\) by convention, this is not necessary
for the above lemma. However, a regression projection must be surjective by definition,
which requires \(m\ge \dim\modelSpace\). As larger \(m\) loosen the probability bound, it
is therefore still natural to assume \(m=\dim\modelSpace\). Observe that this
implies the number of regression points \(m\) of \(\linProj_\star\) coincide
with the dimension,
which implies that \(\linProj_\star\) is an interpolation operator. When
the Lebesgue constant of interpolation is much larger than the Lebesgue constant
of regression it may be beneficial to deviate from this convention and choose \(m > \dim\modelSpace\).

Lemma \ref{lem: pointwise to uniform conversion} combined with Theorem \ref{thm: pointwise concentration} 
then directly implies the following Corollary.

\begin{mdframed}[style=default]
\begin{corollary}[Uniform stochastic error bound]
    \label{cor: uniform noise bound}
    Assume the setting of Theorem \ref{thm: pointwise concentration}.
    Then with probability at least \(1-\delta\)
    \[
        \norm{\regOp(\noise)}_\weight
        \le \varepsilon(\delta)\inf_{\linProj_\star} \lebconst{\weight}{\linProj_\star}\sqrt{\opnorm{\Gram_*^{-1}}_2 \opnorm{\Omega_n^{\frac12}\Sigma \Omega_n^{\frac12}}_2} \frac{\norm{\lever_\modelSpace^*}_\weight}{\sqrt{n}}  ,
    \]
    where the infimum is taken over all regression projections \(\linProj_\star\colon \funSpace \to \modelSpace\) with \(m\) distinct regression points
    and \(\varepsilon(\delta)\) is listed below subject to various assumptions (cf.~Def.~\ref{def: concentration families}).
    \begin{center}
    \makebox[\textwidth][c]{
    \begin{NiceTabular}{l m{6cm} m{5cm}}[tabularnote={
        \footnotesize 
    \(C_\alpha, C_{\alpha}'>0\) are unoptimized constants given in \eqref{eq: sub-Weibull constants}, \(C_{p,\real}\coloneq \frac{14.7p}{\ln p}\) and
    \(C_{p,\complex}=\sqrt2\,C_{p,\real}\).
    }]
        \toprule
        & Assumption (decreasing strength \(\downarrow\)) &  \(\varepsilon(\delta)\)
        \\
        \midrule
        \newtag{\normalfont{(i)}}{it: (uniform concentration) subgaussian}\tabularnote{\footnotesize Assume \(\field=\real\).\label{footnote: K=R, corollary}}\!\!\!
        & \(\noiseIndep_i \sim \subG{\sigma^2}\)
        & \(
            \sigma\sqrt{2\ln(\frac{2m}\delta)}
        \)
        \\[1ex]
        \newtag{\normalfont{(ii)}}{it: (uniform concentration) exponential}\ref*{footnote: K=R, corollary}\!\!\!
        & \(\noiseIndep_i \sim \subExp{\nu, \alpha}\)
        & \(
            \max\set[\Big]{\sqrt{2\nu^2\ln(\frac{2m}\delta)}, 2\alpha\ln(\frac{2m}\delta)}
        \)
        \\[1ex]
        \newtag{\normalfont{(iii)}}{it: (uniform concentration) sub-Weibull}\ref*{footnote: K=R, corollary}\!\!\!
        & \(\noiseIndep_i \sim \subWeibull{\alpha}\), \(\norm{\noiseIndep_i}_{\psi_\alpha}\le K\), \(\alpha\le 2\)
        & \(
            K\Bigl(C_\alpha \sqrt{\ln(\tfrac{2m}{\delta})} + C_{\alpha}' \ln(\tfrac{2m}{\delta})^{\frac1\alpha}\Bigr)
        \)
        \\[1ex]
        \newtag{\normalfont{(iv)}}{it: (uniform concentration) finite p-th moment}
        & \(\E{\abs{\noiseIndep_i}^2}\le \sigma^2\) and \(\E{\abs{\noiseIndep_i}^p} \le \kappa^p\), \(p> 2\)
        & \(
            C_{p, \field}\max\set{\sigma, \kappa}\bigl(\frac{m}{\delta}\bigr)^{\frac1p}
        \)
        \\[1ex]
        \newtag{\normalfont{(v)}}{it: (uniform concentration) finite second moment}
        & \(\E{\abs{\noiseIndep_i}^2} \le \sigma^2\)
        & \(
            \sigma\sqrt{\frac{m}{\delta}}
        \)
        \\
        \bottomrule
    \end{NiceTabular}
    }
    \end{center}
\end{corollary}
\end{mdframed}

The bound above is tightest for \(\ul{\weight} \coloneq
\frac{\sqrt{m}}{\lever_\modelSpace^*}\), as
\(\norm{\lever_\modelSpace^*}_{\ul{\weight}} = \sqrt{m}\) is minimal by Lemma \ref{lem:
optimality of constant weighted leverage bound}. Even if \(\ul{\weight}\) is not the
desired weight function, one can obtain the bound in any other weight function
from this bound without loss of sharpness. Indeed, we have
\[
    \norm{f}_\weight = \sup_{x\in\domain} \abs{f(x)\weight(x)}
    \le \sup_{x\in \domain} \frac{\weight(x)}{\ul{\weight}(x)} \sup_{y\in \domain} f(y) \ul{\weight}(y)
    = \tfrac{\norm{\lever_\modelSpace^*}_\weight}{\sqrt{m}} \norm{f}_{\ul{\weight}}.
\]
And consequently \(\norm{\regOp(\noise)}_\weight \le 
\frac{\norm{\lever_\modelSpace^*}_\weight}{\norm{\lever_\modelSpace^*}_{\ul{\weight}}}
\norm{\regOp(\noise)}_{\ul{\weight}}\). In this sense, \(\ul{\weight}\) is the natural
weight function adapted to the regression problem, choosing any other weight incurs the cost
of \(\sup_{x\in \domain}\frac{\weight(x)}{\ul{\weight}(x)} = \frac{\norm{\lever_\modelSpace^*}_\weight}{\sqrt{m}} \ge 1\).

\subsection{\texorpdfstring{\(L^2(\mu)\)}{L2(µ)}-bounds on the stochastic error}

Finally, we provide \(L^2(\mu)\)-bounds on the stochastic error. Since the results
are not obtained through a pointwise to uniform conversion, we can
also provide moment bounds.

\begin{mdframed}[style=default]
\begin{theorem}[\(L^2(\mu)\)-bounds on the stochastic error]
    \label{thm: L2 concentration}
    Assume the standard regression setting (Definition \ref{def: regression setting}) with \(y = \eval_X[f] + \noise\)
    and \(\noise = A \noiseIndep\) for some \(A \in \field^{n\times n}\) with \(\Sigma= AA^\conT\),
    and let \(\regOp\) be the GLS regression operator in Thm.~\ref{thm: least squares solution} \eqref{eq: definition least squares reg op}.
    Let \(\mBasis_1, \dots, \mBasis_m\) be an orthonormal basis of
    \(\modelSpace\) with respect to \(\scp{\cdot}{\cdot}_\mu\) and 
    \((\Gram_\mu)_{ij} = \scp{\mBasis_i}{\mBasis_j}_{\Omega_n}\) the associated Gram matrix.
    Define
    \[
        \rho_\Sigma^2 \coloneq \opnorm{\Omega_n^{\frac12}\Sigma \Omega_n^{\frac12}}_2
        \overset{\mathrm{Lem.~\ref{lem: bound on operator norm of Omega Sigma}}}\le \opnorm{\Omega_n\Sigma}_2 \le \opnorm{\Omega_n}_2 \opnorm{\Sigma}_2.
    \]
    Then for \(\delta \in (0,1)\) we have the following bounds on the stochastic error:
    \begin{center}
    \makebox[\textwidth][c]{
    \begin{NiceTabular}{l m{5.65cm} m{5.8cm}}[tabularnote={
        \footnotesize
        \(C_\alpha>0\) is a constant that only depends on \(\alpha\) and \(C_{p,\real}=\frac{14.7p}{\ln p}\), \(C_{p,\complex}=\sqrt2\,C_{p,\real}\).
    }]
        \toprule
        & Assumption (decreasing strength \(\downarrow\)) & With probability at least \(1-\delta\)\newline \(\norm{\regOp(\noise)}_\mu \le \dots\)
        \\
        \midrule
        \newtag{\normalfont{(i)}}{it: (L2 concentration) subgaussian}\tabularnote{\footnotesize Assume \(\field=\real\).\label{footnote: K=R, L2}}\!\!
        & \(\noiseIndep_i \sim \subG{\sigma^2}\)
        & \(\sigma\frac{\rho_\Sigma}{\sqrt{n}}\Bigl(\sqrt{\tr(\Gram_\mu^\dagger)} + \sqrt{2\opnorm{\Gram_\mu^\dagger}_2 \ln(\tfrac1\delta)}\Bigr)\)
        \\[1ex]
        \newtag{\normalfont{(ii)}}{it: (L2 concentration) sub-Weibull}\ref*{footnote: K=R, L2}\!\!
        & \(\noiseIndep_i \sim \subWeibull{\alpha}\), \(\norm{\noiseIndep_i}_{\psi_\alpha}\le K\), \(\alpha\le 2\) 
        & \(C_\alpha K \frac{\rho_\Sigma}{\sqrt{n}}\Bigl(\sqrt{\tr(\Gram_\mu^\dagger)} + \sqrt{\opnorm{\Gram_\mu^\dagger}_2} \ln(\frac2\delta)^{\frac1\alpha}\Bigr)\)
        \\[1ex]
        \newtag{\normalfont{(iii)}}{it: (L2 concentration) finite p-th moment}
        & \(\E{\abs{\noiseIndep_i}^2}\le \sigma^2\) and \(\E{\abs{\noiseIndep_i}^p} \le \kappa^p\), \(p> 2\)
        & \(C_{p,\field} \max\set{\kappa, \sigma}\frac{\rho_\Sigma}{\sqrt{n}}\bigl(\frac{1}{\delta}\bigr)^{\frac1p}\sqrt{\tr(\Gram_\mu^\dagger)}\)
        \\[1ex]
        \newtag{\normalfont{(iv)}}{it: (L2 concentration) finite second moment}
        & \(\E{\abs{\noiseIndep_i}^2} \le \sigma^2\)
        & \(\sigma\rho_\Sigma\sqrt{\tfrac{1}{n\delta}\tr(\Gram_\mu^\dagger)}\).
        \\
        \bottomrule
    \end{NiceTabular}
    }
    \end{center}
    Moreover, the following moment bounds hold
    \begin{align}
        \label{eq: (L2 concentration) p-th moment}
        \mathrm{\ref{it: (L2 concentration) finite p-th moment}}
        &\quad\Rightarrow \qquad
        \E{\norm{\regOp(\noise)}_\mu^p}
        \le C_{p,\field}^p \max\set{\kappa^p, \sigma^p}\rho_\Sigma^p\bigl(\tfrac1n\tr(\Gram_\mu^\dagger)\bigr)^{\frac{p}2},
        \\
        \mathrm{\ref{it: (L2 concentration) finite second moment}}
        &\quad\Rightarrow \qquad
        \E{\norm{\regOp(\noise)}_\mu^2}
        \le \sigma^2 \rho_\Sigma^2 \frac{\tr(\Gram_\mu^\dagger)}n.
    \end{align}
    Finally, note that \(\tr(\Gram_\mu^\dagger) \le m \opnorm{\Gram_\mu^\dagger}_2\).
\end{theorem}
\end{mdframed}

The core of the proof of Theorem \ref{thm: L2 concentration} is the observation
that \(\norm{\regOp(\noise)}_\mu^2\) may be expressed as a quadratic form, see \eqref{eq: parseval identity}.
This leads to a sharper bound on the \(L^2\)-norm using the trace and operator norm
of the matrix defining the quadratic form (Thm.~\ref{thm: L2 concentration sharper}).
This trace and operator norm may in turn be bounded using the term
\(\rho_\Sigma = \opnorm{\Omega_n^{\frac12} \Sigma \Omega_n^{\frac12}}_2\) and the Gram matrix
(Lemma \ref{lem: (L2-noise bound) trace and op-norm bounds}). The individual cases
are simply applications of available concentration inequalities for quadratic forms
in the literature.

\begin{remark}[Optimizing \(p\) in the finite \(p\)-th moment case]
    Even if \(p\) moments are available, it may be beneficial to use a smaller \(p\) to get
    concentration bounds from the Markov inequality using the moments in \eqref{eq: (L2 concentration) p-th moment}.
    To see this observe that \(C_{p, \field} \propto p\) up to
    logarithmic factors. The \(L^p\)-norm bound \(\kappa\)
    may also grow with \(p\) depending on the distribution of the noise.
    For example in the sub-Weibull case, \(\kappa = \norm{\noiseIndep_i}_{L^p(\pSym)} \le K p^{\frac1\alpha}\) for some constant \(K\) \citep[Thm.~2.1]{vladimirovaSubWeibullDistributionsGeneralizing2020}.
    Collecting all the \(p\)-dependent terms in the bound we therefore have
    \[
        C_{p,\field} \max\set{\kappa, \sigma} \delta^{-\frac1p}
        \lesssim p^{1 + \frac1\alpha} \delta^{-\frac1p}.
    \]
    Optimizing over \(p\) yields \(p_* \asymp \frac{\log(\frac1\delta)}{1+\frac1\alpha}\),
    which grows only logarithmically with \(\frac1\delta\).
\end{remark}
 	\section{Strategies for regression design}
\label{sec: strategies}

This section is concerned with strategies for regression design that minimize
the bounds on the approximation and stochastic error that we derived in the previous sections.
Any such strategy needs to start by setting a goal: What type of bound is desired?
An \(L^2(\mu)\)-bound, a weighted sup-norm bound or simply a pointwise bound at
a specific \(x\in \domain\)?

\begin{enumerate}
    \item \textbf{\(L^2(\mu)\)-bounds:} The stochastic error bound of Theorem \ref{thm: L2 concentration}
    contains two terms that are influenced by the design decision, specifically
    \(\opnorm{\Omega_n^{\frac12}\Sigma\Omega_n^{\frac1{2}}}_2\)
    and \(\opnorm{\Gram_\mu^\dagger}_2\).
    Minimizing the former encourages \(\Omega_n\) to be proportional to \(\Sigma^{-1}\),
    but the structure of the noise \(\Sigma\) both heavily depends on the
    problem and is also typically unknown in practice. So we are going to ignore
    this term and focus on the Gram term \(\opnorm{\Gram_\mu^\dagger}_2\).  This
    term also appears in the mixed norm approximation error bound of Corollary
    \ref{cor: mixed norm bound for p=2}. And in the discussion following that
    corollary we explained how the other terms in that bound may be simplified
    such that everything boils down to control of the Gram term.
    We will see later that the choice of the weight function can help
    control this Gram term.

    \item \textbf{Weighted sup-norm bounds and pointwise bounds:}
    The Gram term \(\opnorm{\Gram_*^{-1}}_2\) may also appear in the pointwise
    or uniform stochastic error bound (Theorem \ref{thm: pointwise concentration},
    Cor.~\ref{cor: uniform noise bound}) if a leverage function other than
    \(\lever_\modelSpace^{\Omega_n}\) is used. This combination of the Gram
    term and the leverage function may also be used
    to bound the approximation error (Theorem \ref{thm: leverage function bound}).
    Tight approximation error bounds, however, require tight bounds on the Lebesgue function
    (Theorem \ref{thm: optimal
    bias bounds in the weighted sup-norm}) which may be represented using the
    reproducing kernel (Theorem \ref{thm: reproducing kernel representation}).
    Since there are multiple strategies to control the uniform error, we
    provide an overview in Figure~\ref{fig: flowchart uniform error bounds}.
\end{enumerate}
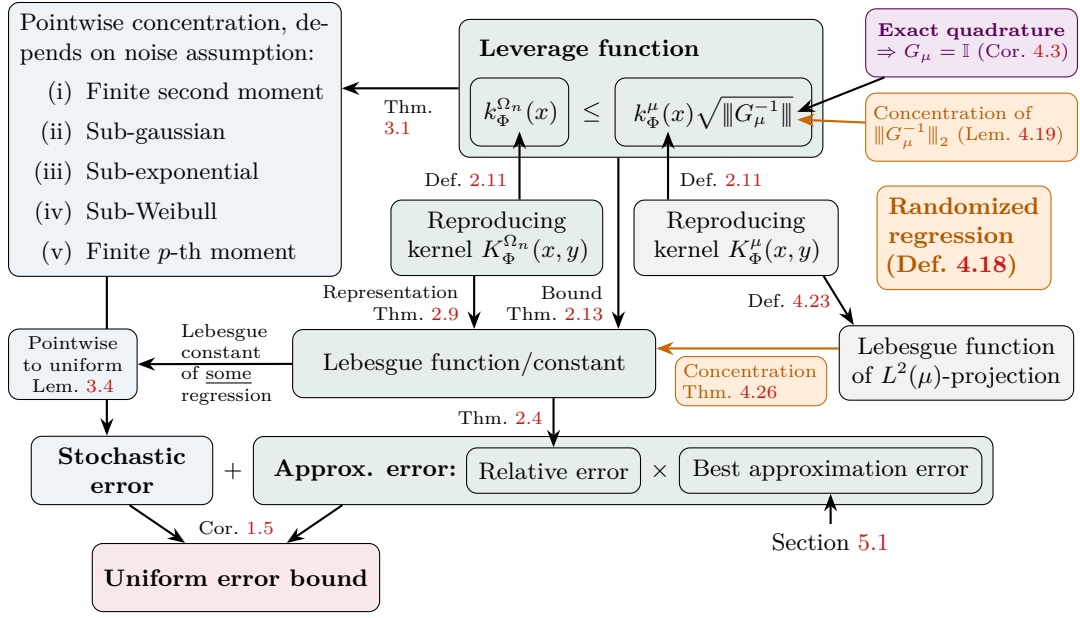
\begin{figure}
    \makebox[\textwidth][c]{
    \begin{tikzpicture}[
    >=Stealth,
    node distance=8mm and 12mm,
    every node/.style={font=\small},
    box/.style={
        draw,
        rounded corners,
        align=center,
        inner sep=1.5mm,
        minimum height=9mm,
        text width=20mm
    },
    noise/.style={box, fill=blue!7},
    bias/.style={box, fill=green!10},
    rand/.style={fill=orange!14, rounded corners, text=orange!70!black, draw=orange!80!black, inner sep=1.5mm},
    quad/.style={box, fill=violet!10!white, text=violet!70!black, draw=violet!70!black},
    result/.style={box, fill=red!10, align=center, text width=35mm, font=\bfseries\small},
    aux/.style={box, fill=gray!10, align=center},
    line/.style={->, thick, fill=none},
    lab/.style={midway,inner sep=1mm, font=\scriptsize, text width=20mm, thin}
]

\node[result] (goal) {Uniform error bound};

\node[noise, above left=5mm and -4mm of goal, anchor=south, font=\bfseries\small, text width=21mm] (noise) {Stochastic error};
\draw[line] ([xshift=5mm]noise.south -| goal.north west) -- (goal);

\node[above right=6.5mm and -2mm of goal, minimum height=6mm, anchor=south, font=\bfseries\small] (bias) {Approx. error:};
\node[bias, right=0mm of bias, minimum height=6mm] (rel-error) {Relative error};
\node[bias, right=5mm of rel-error, text width=37mm,  minimum height=6mm] (metric-projection) {Best approximation error};
\node[font=\bfseries\small] at ($(rel-error.east)!0.5!(metric-projection.west)$) {\(\times\)};

\node[below=4mm of metric-projection] (metric-proj-ref) {Section \ref{sec: metric projection error}};
\draw[line] (metric-proj-ref.north) -- (metric-projection.south);

\begin{scope}[on background layer]
  \node[
    bias,
    draw,
    rounded corners,
    fit=(metric-projection)(rel-error)(bias),
    inner sep=4pt,
  ] (bias-group) {};
\end{scope}

\node[font=\bfseries\small] (sum) at ($(noise.east)!0.5!(bias-group.west)$) {\(+\)};

\draw[line] ([xshift=-5mm]bias-group.south -| goal.north east) -- (goal);
\node[lab, below=-8.5mm of sum, align=center] {Cor.~\ref{cor: bias noise decomposition}};

\node[noise, xshift=7mm, above=21mm of noise, text width=41mm, align=left]
(pointwise) {
    Pointwise concentration, depends on noise assumption:
    \begin{enumerate}[label=(\roman*)]
        \item Finite second moment
        \item Sub-gaussian
        \item Sub-exponential
        \item Sub-Weibull
        \item Finite \(p\)-th moment
    \end{enumerate}
};

\draw[line] ([xshift=-2mm]noise.north |- pointwise.south)
-- node[noise, lab, pos=0.55, xshift=-4.5mm, text width=15mm] (pt-to-uniform) {Pointwise to uniform Lem.~\ref{lem: pointwise to uniform conversion}} 
([xshift=-2mm]noise.north);

\node[bias, xshift=-10.5mm, above=6.5mm of rel-error, text width=45mm] (lebesgue) {Lebesgue function/constant};
\draw[line] (rel-error.north |- lebesgue.south) -- node[lab, pos=0.4, left=0.5mm, text width=11mm] {Thm.~\ref{thm: optimal bias bounds in the weighted sup-norm}} (rel-error.north);

\draw[line] (lebesgue)
-- node[lab, pos=0.4, text width=13mm] {Lebesgue constant of \ul{some} regression}
(pt-to-uniform);

\node[bias, xshift=3mm, above=7mm of lebesgue, text width=25mm] (reprK) {Reproducing kernel \(\reprK^{\Omega_n}_\modelSpace(x,y)\)};
\draw[line] ([xshift=-3mm]reprK.south) -- node[lab, pos=0.5, left=1mm, text width=18mm, align=right] {Representation Thm.~\ref{thm: reproducing kernel representation}} ([xshift=-3mm]reprK.south |- lebesgue.north);

\node[bias, xshift=3mm, above=7mm of reprK, text width=10mm] (lever-omega) {\(\lever_\modelSpace^{\Omega_n}(x)\)};
\node[bias, right=6mm of lever-omega, text width=23mm] (lever-mu) {\(\lever_\modelSpace^\mu(x)\sqrt{\opnorm{\Gram_\mu^{-1}}}\)};
\node[font=\bfseries\small] at ($(lever-omega.east)!0.5!(lever-mu.west)$) {\(\le\)};
\node[font=\bfseries\small,above=1mm of lever-omega.north west, anchor=south west, text width=30mm] (lever) {Leverage function};

\draw[line] (lever-omega.south |- reprK.north) -- node[lab, pos=0.3, left=0.5mm, text width=11mm]{Def.~\ref{def: leverage function}} ([yshift=2mm]lever-omega.south);

\begin{scope}[on background layer]
  \node[
    bias,
    draw,
    rounded corners,
    fit=(lever-omega)(lever-mu)(lever),
    inner sep=4pt,
  ] (lever-wrapper) {};
\end{scope}

\draw[line] ([xshift=-5mm] lebesgue.east |- lever-wrapper.south)
  -- node[lab,pos=0.85,left=1mm, align=right, text width=13mm]{Bound Thm.~\ref{thm: leverage function bound}}
([xshift=-5mm] lebesgue.east |- lebesgue.north);

\draw[line] (lever-wrapper.west) -- node[lab, below=0.5mm, pos=0.45, align=left,text width=5.7mm] {Thm.\ \ref{thm: pointwise concentration}} (lever-wrapper.west -| pointwise.east);

\node[quad, inner sep=1mm, font=\scriptsize, above right=0mm and 7mm of lever-mu, text width=26mm] (exact-quadrature) {\textbf{Exact quadrature} \(\Rightarrow\Gram_\mu=\identity\) (Cor.~\ref{cor: exact quadrature})};
\draw[line] (exact-quadrature) -- ([xshift=-2mm, yshift=0mm]lever-mu.east);

\node[
  aux, xshift=3mm, below=7mm of lever-mu, text width=24mm
] (reprK-mu) {Reproducing kernel \(\reprK^\mu_\modelSpace(x,y)\)};
\draw[line] ([xshift=-6mm] lever-mu.south |- reprK-mu.north) 
-- node[lab, pos=0.3, right=0.5mm, text width=11mm]{Def.~\ref{def: leverage function}}
([xshift=-6mm, yshift=2mm] lever-mu.south);

\node[aux, below right=6.5mm and 0mm of reprK-mu, text width=28mm] (lebesgue-mu) {Lebesgue function of \(L^2(\mu)\)-projection};
\draw[line] ([xshift=-2mm]reprK-mu.south east)
-- node[lab, left=-1mm, text width=12mm]{Def.~\ref{def: lebesgue function and constant of L2 projection}}
([xshift=2mm]lebesgue-mu.north west);

\node[rand, font=\bfseries\small, right=5mm of reprK-mu, text width=20mm] {Randomized regression (Def.~\ref{def: randomized regression})};

\node[
  rand, inner sep=1mm, font=\scriptsize, right=7mm of lever-mu, text width=26mm, yshift=-2mm, minimum height=9mm
] (gram-concentration)
{Concentration of \(\opnorm{\Gram_\mu^{-1}}_2\) (Lem.~\ref{lem: gram concentration})};

\draw[rand, line] (gram-concentration) -- ([xshift=-2mm, yshift=-1mm]lever-mu.east);

\draw[rand, line] ([yshift=2mm]lebesgue.east -| lebesgue-mu.west)
-- node[rand, lab, pos=0.48, text width=18mm, below=1mm] {Concentration Thm.~\ref{thm: lebesgue constant concentration}}
([yshift=2mm]lebesgue.east);

\end{tikzpicture}     }
\caption{Routes to uniform error bounds for regression. Colors distinguish the
stochastic error (blue) from the approximation error (green) and the strategies to control them in violet and yellow.}
\label{fig: flowchart uniform error bounds}
\end{figure}

In both cases, the Gram matrix is a central object to control. While
one could technically avoid it for the pointwise and uniform bounds,
the leverage function \(\lever_\modelSpace^\mu\) of a measure \(\mu\) is
typically easier to control than the leverage function 
\(\lever_{\modelSpace}^{\Omega_n}\) of a specific regression design.
Since the leverage function is a simpler object, it is generally easier to find
a regression design that has approximately constant weighted leverage and then
bound the Lebesgue function for this specific design, rather than directly
optimizing the Lebesgue function over designs.

Generally, existing regression design strategies boil down to the following two steps:
\begin{enumerate}[noitemsep]
    \item Select a measure \(\mu\) where the \(L^2(\mu)\)-projection
    \(\linProj_\mu\) has good approximation properties. If the error
    is measured by the \(L^2\) norm, then the choice of measure \(\mu\) is obvious.
    In the case of the weighted sup-norm, which measure \(\mu\) to choose
    becomes an optimization problem. The quantity to optimize will be the Lebesgue
    function/constant of projection (cf.~Def.~\ref{def: lebesgue function and constant of L2 projection} and Section~\ref{sec: approximation theory of Lp projection}).

    \item Choose evaluation points \(x_1,\dots, x_n\) and \(\Omega_n\) such that
    \(\scp{\cdot}{\cdot}_{\Omega_n}\) is approximately equal to
    \(\scp{\cdot}{\cdot}_\mu\) on the model space \(\modelSpace\), which
    essentially ensures that the performance of the regression projection
    \(\linProj\) is close to that of the \(L^2(\mu)\)-projection
    \(\linProj_\mu\).
\end{enumerate}

In the following we will see that this similarity of inner products
is directly related to symmetric spectral bounds of the Gram matrix.

\begin{theorem}[Symmetric spectral bounds]
    \label{thm: approximately identity gram matrix}
    Let \(\scp{\cdot}{\cdot}_*\) be an inner product on the model space \(\modelSpace\), \(\mBasis_1, \dots, \mBasis_m\)
    an orthonormal basis of \(\modelSpace\) with respect to this inner product \(\scp{\cdot}{\cdot}_*\)
    and \(\Gram_*\) the associated Gram matrix with entries \((\Gram_*)_{ij} = \scp{\mBasis_j}{\mBasis_i}_{\Omega_n}\).
    For a given \(\epsilon \in (0, 1)\) the following are equivalent:
    \begin{enumerate}[label=(\roman*)]
        \item\label{it: approximately identity Gram matrix}
        \textbf{Approximately identity Gram matrix}
        \(\opnorm{\Gram_* - \identity}_2 \le \epsilon\),
        \item\label{it: approximately equal inner product}
        \textbf{Approximately equal inner product}
        on the model space \(\modelSpace\), that is
        \[
            \abs{\scp{\model}{\tilde\model}_{\Omega_n} - \scp{\model}{\tilde\model}_*}
            \le \epsilon \norm{\model}_* \norm{\tilde\model}_*
            \qquad \forall \model,\tilde\model \in \modelSpace,
        \]
        \item\label{it: approximately equal norms}
        \textbf{Approximately equal norms}
        on the model space \(\modelSpace\), that is
        \[
            \abs[\big]{\norm{\model}_{\Omega_n}^2 - \norm{\model}_*^2} \le \epsilon \norm{\model}_*^2
            \qquad \forall\model \in \modelSpace,
        \]
        or equivalently
        \(
            (1-\epsilon) \norm{\model}_*^2 \le \norm{\model}_{\Omega_n}^2 \le (1+\epsilon) \norm{\model}_*^2
        \) for all \(\model \in \modelSpace\).
    \end{enumerate}
\end{theorem}

Observe that the equivalent properties \ref{it: approximately equal inner product}
and \ref{it: approximately equal norms} show that \ref{it: approximately identity Gram matrix}
is not basis dependent despite the definition of the Gram matrix \(\Gram_*\).

We will see in the following sections how to obtain such symmetric bounds.
But before let us convince ourselves that this allows for
bounds on the inverse:

\begin{lemma}[Bounds on the inverse Gram matrix]
    \label{lem: bounds on inverse gram matrix}
    With the same notation as in Theorem \ref{thm: approximately identity gram matrix}, if \(\opnorm{\Gram_* - \identity}_2 \le \epsilon\) for some \(\epsilon \in [0, 1)\), then
    \[
        \opnorm{\Gram_*^{-1}}_2 \le \tfrac1{1-\epsilon}
        \qquad \text{and} \qquad
        \opnorm{\Gram_*^{-1} - \identity}_2 \le \tfrac{\epsilon}{1-\epsilon}.
    \]
\end{lemma}

In the following we will see two approaches to obtain a well conditioned
Gram matrix for weighted least squared regression (WLS).

\subsection{Exact quadrature/cubature}

Why settle for approximately equal inner products when exactly equal
inner products are available? This exact equality is characterized by exact quadrature
and explains why the roots of orthogonal polynomials are often used
for interpolation. We start with this characterization. Then we consider the
feasibility of such exact quadrature and finally provide examples of how exact
quadrature simplifies the approximation and stochastic error bounds we derived
in the previous sections.

\begin{corollary}[Exact quadrature]
    \label{cor: exact quadrature}
    Let \(\mu\) be a finite measure on \(\domain\) inducing an inner product
    \(\scp{\cdot}{\cdot}_\mu\) on \(\modelSpace\) and assume the WLS setting, where \(\Omega_n\) is diagonal with weights \((\Omega_n)_{ii} = \omega_i\)
    inducing the discrete measure \(\mu_n = \frac1n\sum_{i=1}^n \omega_i \delta_{x_i}\).
    Then the following are equivalent
    \begin{enumerate}[label=(\roman*)]
        \item\label{it: identity gram matrix} 
        \textbf{Identity Gram matrix} \(\Gram_\mu = \identity\).
        \item\label{it: equal inner product}
        \textbf{Equal inner product} of \(\mu\) and \(\mu_n\) on the model space \(\modelSpace\), that is
        \[
            \scp{\model}{\tilde\model}_{\mu_n} = \scp{\model}{\tilde\model}_\mu
            \qquad \forall\model,\tilde\model \in \modelSpace.
        \] 
        \item\label{it: exact quadrature}\textbf{Exact quadrature}
        on the product space \(\modelSpace \cdot \overline{\modelSpace} = \Span\set{\model_1 \cdot \overline{\model_2} : \model_1,\model_2 \in \modelSpace}\). That is
        \[
            \frac1n\sum_{i=1}^n \omega_i h(x_i)  \overset{\mathrm{def.}}= \int h(x) d\mu_n(x)= \int h(x) d\mu(x)
            \qquad \forall h\in \modelSpace \cdot \overline{\modelSpace}.
        \]
    \end{enumerate}
\end{corollary}
\begin{proof}
    \ref{it: identity gram matrix} \(\Leftrightarrow\) \ref{it: equal inner product}:
    Special case of Theorem \ref{thm: approximately identity gram matrix}.
    \ref{it: equal inner product} \(\Leftrightarrow\) \ref{it: exact quadrature}:
    Definition of \(\modelSpace \cdot \overline{\modelSpace}\).
\end{proof}

\begin{remark}[Leverage and reproducing kernel under exact quadrature]
    \label{rem: (exact quadrature) leverage and reproducing kernel}
    Since the inner products coincide on the model space, an orthonormal basis with
    respect to one is also an orthonormal basis with respect to the other. This also
    implies the reproducing kernel and leverage functions coincide.
\end{remark}

\subsubsection{Feasibility of exact quadrature}

Before we compose error bounds for exact quadrature we
discuss its feasibility.

\begin{theorem}[Feasibility of exact quadrature {[Richter-Tchakaloff]}]
    \label{thm: feasibility of exact quadrature}
    Let \(\norm{\cdot}_\mu\) be a norm on \(\modelSpace\) induced by the measure \(\mu\).
    Let \(n_\mathrm{exact}\) be the minimum number of evaluation points
    needed for exact quadrature of \(\mu\) with non-negative weights to be possible on the product space \(\modelSpace \cdot \overline{\modelSpace}\).
    Then
    \[
        m\coloneq \dim\modelSpace
        \le n_\mathrm{exact}
        \le \dim(\modelSpace \cdot \overline{\modelSpace})
        \le \begin{cases}
            m^2 & \field = \complex
            \\
            \frac{m(m+1)}2 & \field = \real
        \end{cases}.
    \]
\end{theorem}

In the following we will list a few examples that show that these bounds are tight.
Proofs can be found in the appendix as always.

\begin{example}[\(m= n_\mathrm{exact} = \dim(\modelSpace \cdot \overline{\modelSpace})\)]
    The space of piecewise constant functions on disjoint sets \(A_i\subseteq \domain\)
    with \(0<\mu(A_i)<\infty\) for \(i=1,\dots, m\) given by
    \[
        \modelSpace \coloneq \Span\set{\ind{A_1}, \dots, \ind{A_m}}
    \]
    is an example of a space with \(\modelSpace \cdot \overline{\modelSpace} = \modelSpace\)
    and consequently \(n_\mathrm{exact} = \dim(\modelSpace)\).
\end{example}

\begin{example}[\(m=n_\mathrm{exact}\) and  \(\dim(\modelSpace\cdot\overline{\modelSpace}) = m^2\)]
    \label{ex: exact quadrature for monomials of degree one}
    Let the domain be \(\domain = \complex^m\) and the model space be the space
    of monomials of degree exactly one, that is
    \begin{equation}
        \label{eq: model space of monomials of degree one}    
        \modelSpace \coloneq \Span\set{x\mapsto x^{(i)}: i=1,\dots,m}.
    \end{equation}
    Assume that \(\mu\) induces a norm \(\norm{\cdot}_\mu\) on \(\modelSpace\).
    Then we have \(\dim(\modelSpace\cdot \overline{\modelSpace}) = m^2\) and 
    \(n_\mathrm{exact} = m\).
\end{example}
\begin{example}[\(n_\mathrm{exact} = \dim(\modelSpace\cdot\overline{\modelSpace}) = m^2\)]
    \label{ex: exact quadrature for monomials of degree one, reduced domain}
    We will use the same model space of monomials of degree one as
    in the previous example \eqref{eq: model space of monomials of degree one} but
    with the reduced domain
    \[
        \domain = \set[\big]{e_k, e_k + e_l, e_k + i e_l: k,l=1,\dots,m, k < l} \subseteq \complex^m
    \]
    Let \(\mu\) be a measure with positive weight on each \(x \in \domain\),
    then \(n_\mathrm{exact} = \dim(\modelSpace\cdot\overline{\modelSpace}) = m^2\).
\end{example}

Recent works go beyond the existence of exact quadrature and propose
\textbf{constructive algorithms} to find such quadrature points and weights. \citet[Thm.~3 and 6]{miglioratiStableHighorderRandomized2022}
show that modifying the weights of randomized regression (Section \ref{sec: randomized regression})
can yield exact quadrature with positive weights with high probability. 
While concentration results for approximate quadrature only require \(
    n\succeq \norm{\lever_\modelSpace^\mu}_\weight^2\log(m) \ge m \log(m)
\)
as we will see in
Section \ref{sec: randomized regression}, the ability to strengthen this to exact quadrature
with positive weights requires
\(
    n \succeq m^2
\)
up to logarithmic factors with the assumption of
\(\norm{\lever_\modelSpace^\mu}_\weight^2 = m\)
\citet[Thm.~6]{miglioratiStableHighorderRandomized2022}. Subsequently, \citet{glaubitzConstructionApplicationProvable2023}
proposed a deterministic algorithm to obtain exact quadrature
and also observed that \(n \asymp m^2\) is necessary experimentally.
Given that in the worst case, \(n_\mathrm{exact} = m^2\), it is perhaps not
surprising that general algorithms for exact quadrature require a
quadratic number of evaluation points.

In the following we will discuss special cases that are exceptions to this rule.
We begin with a general characterization of the interpolation case
that becomes relevant if \(n_\mathrm{exact} = m\) as will be the case
for the univariate polynomials.

\begin{prop}[The interpolation case \(n=m\)]
    \label{prop: interpolation case}
    Let \(\modelSpace\) be an \(m\)-dimensional model space,
    \(X=(x_1,\dots, x_n)\in \domain\) be evaluation points and \(\omega_1, \dots, \omega_n \ge 0\)
    the diagonal entries of the weight matrix \(\Omega_n\) of the WLS regression
    operator \(\regOp\).
    Then for \(n=m\) the restriction of the evaluation operator \(\eval_X\) to \(\modelSpace\), denoted by \(\eval_X^\modelSpace\), is bijective and the components of the Lagrange function \(\ell_i\) (Def.~\ref{def: lagrange function}) satisfy
    the interpolation properties
    \[
        \ell_i \overset{\mathrm{def.}}=\regOp(e_i) = (\eval_X^\modelSpace)^{-1}[e_i]
        \qquad \text{and}\qquad
        \ell_i(x_j) = \delta_{ij}.
    \]
    For a measure \(\mu\) on \(\domain\) inducing a positive definite inner product on \(\modelSpace\), the following are further equivalent:
    \begin{enumerate}[label={(\roman*)}, noitemsep]
        \item\label{it: (interpolation) exact quadrature} \textbf{Exact quadrature}, that is \(\scp{\cdot}{\cdot}_\mu = \scp{\cdot}{\cdot}_{\Omega_n}\) on \(\modelSpace\).
        \item\label{it: (interpolation) lagrange} The lagrange functions \(\ell_i\) are orthogonal with respect to \(\scp{\cdot}{\cdot}_\mu\) and \(\norm{\ell_i}_\mu = \sqrt{\frac{\omega_i}{m}}\).
        \item\label{it: (interpolation) kernel, weights} We have
        \[
            \reprK_\modelSpace^\mu(x_i,x_j) = 0 \qquad \text{for all } i\neq j
            \qquad\text{and}\qquad
            \omega_i = \tfrac{m}{\lever_\modelSpace^\mu(x_i)^2}.
        \]
        \item\label{it: (interpolation) orthonormal kernel} The functions \(\mBasis_i \coloneq \sqrt{\frac{\omega_i}{m}}\reprK_\modelSpace^\mu(\cdot, x_i)\)
        form an orthonormal basis of \(\modelSpace\) w.r.t.\ \(\scp{\cdot}{\cdot}_\mu\).
    \end{enumerate}
    Moreover, if any (all) of the above conditions hold, then \(\ell_i = \frac{\omega_i}{m} \reprK_\modelSpace^\mu(\cdot, x_i)\)
\end{prop}

\begin{remark}[Uniqueness of weights \(\omega_i\)]
    Since \ref{it: (interpolation) exact quadrature} does not make
    any assumptions about the weights \(\omega_i\ge 0\), the explicit weights in \ref{it: (interpolation) kernel, weights}
    are unique. In other words, if, for a given set of evaluation points \(x_1, \dots, x_m\), exact quadrature is possible
    at all, then the weights \(\omega_i\) \emph{must} be the weights given in \ref{it: (interpolation) kernel, weights}.
\end{remark}

Exact quadrature for \(x_1, \dots, x_m\) is therefore possible iff
\(\reprK_\modelSpace^\mu(x_i,x_j) = 0\) for all \(i\neq j\).

\paragraph*{Exact quadrature for polynomials}

Observe that if \(\modelSpace\) is the space of univariate polynomials up to degree \(k\) denoted
by \(\poly_M\), then \(\overline{\modelSpace} = \modelSpace\) and \(\modelSpace
\cdot \overline{\modelSpace} = \poly_{2M}\). 
The reason that the roots of orthogonal polynomials are such good candidates for
exact quadrature is the following: Let \(p(x) = \prod_{i=1}^n (x-x_i)\) be the polynomial of evaluation nodes,
which is zero at all evaluation points \(x_i\) and
let \(r=2M-n\). Then \(pq\in \poly_{n+r} = \poly_{2M}\) for any \(q\in \poly_{r}\) and therefore
\[
    \scp{p}{q}_\mu = \int p(x) q(x) \,\mu(dx)
    \overset{\substack{\text{exact}\\ \text{quadrature}}}
    = \sum_{i=1}^n \omega_i p(x_i) q(x_i) \overset{q(x_i)=0}{=} 0.
\]
In other words: the node polynomial \(p\) must be orthogonal to \(\poly_{r}\) with respect to \(\scp{\cdot}{\cdot}_\mu\) if we want exact quadrature 
on \(\poly_{2M}\). \citet[Thm.~1.45]{gautschiOrthogonalPolynomialsComputation2004}
essentially shows that this is not only a necessary but also a sufficient condition for exact quadrature.
And since \(r=2M-n\) increases with \(M\), we generally want the node polynomial \(p\)
to be orthogonal to as many polynomials of lower degree as possible. 
Observe that \(m = \dim \poly_M = M+1\) such that for \(n=m\)
we have \(r=2(m-1) -m = m - 2\). This shows exact quadrature for \(n=m\) is
possible if the node polynomial (of degree \(m\)) is orthogonal to all
polynomials of degree \(\poly_{m-2}\). In principle the node polynomial
may therefore be any mixture of the \((m-1)\)-th and \(m\)-th orthogonal polynomial
that has \(m\) roots---the simplest choice are the \(m\) roots of the
\(m\)-th polynomial.

Once the polynomials become multivariate, exact quadrature for \(n = m\)
is typically impossible \citep[Corollary 3.8.8]{dunklOrthogonalPolynomialsSeveral2014}
and \(n>m\) becomes necessary to achieve exact quadrature. 
The study of exact quadrature in the case of multivariate polynomials on the sphere
is referred to as spherical design \citep[e.g.][]{anSurveySphericalDesigns2026}
and an active area of research.

\subsubsection{Uniform bounds subject to exact quadrature}

In the following example you can see how exact quadrature simplifies the uniform
approximation and stochastic error bounds we derived in the previous sections.

\begin{example}[Uniform bounds for exact quadrature]
    \label{ex: uniform bounds for exact quadrature}
    Let \(\regOp\) be a WLS regression operator in the standard regression
    setting (Def.~\ref{def: regression setting}) and assume an identity Gram
    matrix \(\Gram_\mu=\identity\) for some measure \(\mu\) on \(\domain\)
    (exact quadrature).

    \begin{itemize}[noitemsep]
        \item If \(\ones \in \modelSpace\) we have with probability at least \(1-\delta\)
        \begin{align}
            \label{eq: (exact quadrature) uniform bound}
            \norm{\regOp(y) -f}_\infty
            &\le \green{\norm{\linProj f - f}_\infty} + \blue{\norm{\regOp(\noise)}_\infty}
            \\
            \nonumber
            &\le \green{(1+\sqrt{\mu(\domain)}\norm{\lever_\modelSpace^\mu}_\infty) \error_\modelSpace^\infty[f]} +
            \blue{\varepsilon(\delta)\inf_{\linProj_\star} \lebconst{\infty}{\linProj_\star}\sqrt{\opnorm{\Omega_n^{\frac12}\Sigma \Omega_n^{\frac12}}_2} \tfrac{\norm{\lever_\modelSpace^\mu}_\infty}{\sqrt{n}}},
        \end{align}

        \item For \(\weight(x) = \frac{\sqrt{m}}{\lever_\modelSpace^\mu(x)}\) we have with probability at least \(1-\delta\)
        \begin{equation}
            \label{eq: (exact quadrature) inverse leverage uniform bound}
            \norm{\regOp(y) -f}_\weight
            \le \green{(1+ \sqrt{m}) \error_\modelSpace^\weight[f]} + \blue{\varepsilon(\delta)\inf_{\linProj_\star} \lebconst{\weight}{\linProj_\star}\sqrt{\opnorm{\Omega_n^{\frac12}\Sigma \Omega_n^{\frac12}}_2\tfrac{m}{n}}},
        \end{equation}
    \end{itemize}
    The quantity \(\varepsilon(\delta)\) is listed in Corollary \ref{cor: uniform noise bound} for different assumptions on the noise,
    and the infimum is taken over regression projections \(\linProj_\star\colon \funSpace \to \modelSpace\) with \(m\) distinct regression points.
\end{example}
\begin{proof}
    Using \(\Gram_\mu = \identity\), the approximation error bounds follow from the leverage
    function bounds on the Lebesgue function (Theorem \ref{thm: leverage
    function bound}) combined with Lemma \ref{lem: A_w and A_1 for exact
    quadrature} to simplify \(A_\weight(\Omega_n)\). The stochastic error bound follows
    from Corollary \ref{cor: uniform noise bound}  and
    \(\norm{\lever_\modelSpace^{\mu}}_\weight = \sqrt{m}\)
    in the case of \(\weight(x) = \frac{\sqrt{m}}{\lever_\modelSpace^\mu(x)}\).
\end{proof}

\begin{remark}[Inverse leverage as weight]
    Observe that for \(\weight(x) = \sqrt{m}/\lever_\modelSpace^\mu(x)\), we have
    \begin{align*}
        \norm{\regOp(y) -f}_\infty
        &= \sup_{x\in \domain}\abs{\regOp(y)(x) -f(x)}\tfrac{\weight(x)}{\weight(x)}
        \\
        &\le \norm{\regOp(y) -f}_\weight \sup_{x\in \domain}\tfrac{1}{\weight(x)}
        = \norm{\regOp(y) -f}_\weight \tfrac{\norm{\lever_\modelSpace^\mu}_\infty}{\sqrt{m}}.
    \end{align*}
    If we now insert the bound \eqref{eq: (exact quadrature) inverse leverage uniform bound},
    we obtain the same stochastic error bound as in \eqref{eq: (exact quadrature) uniform
    bound} up to the change of \(\inf_{\linProj_\star} \lebconst{\infty}{\linProj_\star}\) to \(\inf_{\linProj_\star} \lebconst{\weight}{\linProj_\star}\),
    and the approximation error bound is also identical up to the factor
    \[
        \frac{\error_\modelSpace^\weight[f]}{\sqrt{\mu(\domain)}\error_\modelSpace^\infty[f]}.
    \]
    We conjecture that this factor is typically smaller than one based on this polynomial
    example: As polynomials of degree greater zero always diverge to infinity, they
    cannot uniformly approximate bounded functions on unbounded domains. This changes,
    if we consider the weighted sup-norm with the inverse leverage as weight. This is because the
    reproducing kernel contains all basis functions and the leverage function
    therefore grows like the highest order polynomial as \(x\to \infty\). This
    allows for the use of higher order polynomials for approximation in weighted
    sup-norm.

    If that conjecture is true, then better unweighted sup-norm bounds can be obtained using \(\error_\modelSpace^\weight[f]\)
    instead of \(\error_\modelSpace^\infty[f]\). It therefore seems that the
    inverse leverage weight---which depends on \(\mu\)---is in some sense the
    `natural' weight if the regression design is based on exact quadrature with
    respect to \(\mu\).
\end{remark}

If the inverse leverage weight is indeed the natural weight function, then
for a desired weight \(\weight\), the measure \(\mu\) must be chosen such that
\(\weight(x) = \sqrt{m}/\lever_\modelSpace^\mu(x)\). This is a difficult inverse
problem that is partially addressed by the theory on equilibrium measures 
\citetext{see \citealp{narayanChristoffelFunctionWeighted2016} and references therein}

\begin{remark}[Sample complexity]
    \label{rem: sample complexity discussion uniform bounds}
    In the setting \(\weight(x) = \sqrt{m}/\lever_\modelSpace^\mu(x)\)
    assume \(\error_\modelSpace^\weight[f] \in \bigO(m^{-\alpha})\) for some
    \(\alpha>0\). If \(\funSpace\) is the space of Lipschitz continuous functions
    we often have \(\alpha=1\) for example. Then the approximation error is of order \(\bigO(m^{\frac12-\alpha})\).
    And since \(\varepsilon(\delta) \in \bigO(m^{\frac1p})\) if the noise has \(p\)-th moments (cf.~Corollary \ref{cor: uniform noise bound}),
    the stochastic error is of order \(\bigO(m^{\frac12+\frac1p} n^{-\frac12})\) up to a logarithmic
    factor assuming \(\inf_{\linProj_\star} \lebconst{\weight}{\linProj_\star} \in \bigO(\log(m))\). A logarithmic
    Lebesgue constant is feasible in the case of polynomial regression for example. Balancing
    these two terms gives (up to logarithmic terms)
    \[
        m \asymp n^{\frac1{2(\alpha+\frac1p)}}
        \quad\implies\quad 
        \norm{\regOp(y) -f}_\weight \in \bigO(n^\gamma) \quad\text{with} \quad \gamma = \tfrac{\frac12 - \alpha}{2(\alpha+\frac1p)}
    \]
    For \(\alpha=1\) and \(p=2\) this gives \(\gamma= - \frac18\) and for \(\alpha, p \to \infty\) we get \(\gamma \to -\frac12\).
    The rate for large \(\alpha\) and \(p\) is optimal, since the stochastic error has this rate even for
    constant \(m\). The rate \(\gamma\) for small \(\alpha\) and \(p\) is however suboptimal in
    general. This is because 
    \begin{enumerate}[noitemsep]
        \item we used the leverage function bound on the Lebesgue constant, which
        is lower bounded by \(\sqrt{m}\) (Lem.~\ref{lem: optimality of constant
        weighted leverage bound}) while the Lebesgue constant may be constant or
        logarithmic in \(m\) (see Section~\ref{sec: lebesgue constants of regression}).
        \item the dependence on \(p\) is due to the union bound in the
        pointwise to uniform conversion (Lemma~\ref{lem: pointwise to uniform
        conversion}). It is unclear to us if this is unavoidable or an artifact of the proof.
    \end{enumerate}
\end{remark}

\begin{remark}[Towards tighter bounds on the Lebesgue constant]
    The example above uses the leverage function bound on the Lebesgue function.
    In the following we discuss a potential strategy to obtain tighter bounds on the Lebesgue function.
    Recall that due to exact quadrature, the reproducing kernel and leverage function of
    \(\mu\) and \(\mu_n\) coincide (Rem.~\ref{rem: (exact quadrature) leverage and reproducing kernel}).
    The Lebesgue function of the WLS regression
    is therefore given by (Thm.~\ref{thm: reproducing kernel representation})
    \[
        \lebfun{\weight}{\linProj}{x}
        = \weight(x) \frac1n \sum_{i=1}^n \frac{\omega_i}{\weight(x_i)} \abs{\reprK_\modelSpace^\mu(x, x_i)}
        = \int f_x(y) \,\mu_n(dy),
    \]
    for \(f_x(y) = \frac{\weight(x)}{\weight(y)} \abs{\reprK_\modelSpace^\mu(x, y)}\). Now if we had \(f_x \in \modelSpace\overline{\modelSpace}\),
    then the exact quadrature property would ensure that the Lebesgue function of regression
    coincides with the Lebesgue function of the \(L^2(\mu)\)-projection \(\linProj_\mu\), which is given by
    (Definition \ref{def: lebesgue function and constant of L2 projection})
    \[
        \lebfun{\weight}{\linProj_\mu}{x}
        = \weight(x) \int \frac1{\weight(y)} \abs{\reprK_\modelSpace^\mu(x, y)} \,\mu(dy).
    \]
    However, \(f_x\) is typically not in \(\modelSpace\overline{\modelSpace}\). If \(\ones\in\modelSpace\) such that \(\modelSpace\subseteq \modelSpace\overline{\modelSpace}\) a way to attack this problem may be
    to approximate \(f_x\) by some \(\model \in \modelSpace\) for which the integrals coincide such that
    \begin{align*}
        \abs{\lebfun{\weight}{\linProj}{x}
        - \lebfun{\weight}{\linProj_\mu}{x}}
        &\le \inf_{\model \in \modelSpace} \int \abs{f_x(y) - \model(y)} \,\mu_n(dy) + \int \abs{\model(y) - f_x(y)} \,\mu(dy)
        \\
        &\le (\mu_n(\domain) + \mu(\domain)) \underbrace{\inf_{\model \in \modelSpace} \norm{f_x - \model}_\infty}_{\error_\modelSpace^\infty[f_x]}.
    \end{align*}
    It is unclear how to bound \(\error_\modelSpace^\infty[f_x]\) in general however. The absolute value contained
    in \(f_x\) also destroys smoothness properties of the reproducing kernel beyond Lipschitz continuity.
\end{remark}

\subsubsection{\texorpdfstring{\(L^2\)}{L2}-bounds subject to exact quadrature}

Exact quadrature similarly simplifies the \(L^2\)-bounds as
shown in the next example.

\begin{example}[\(L^2\)-bounds for exact quadrature]
    \label{ex: L2 bounds for exact quadrature}
    Let \(\regOp\) be a WLS regression operator in the standard
    regression setting (Def.~\ref{def: regression setting}) and assume an identity Gram
    matrix \(\Gram_\mu=\identity\) for some measure \(\mu\) on \(\domain\) (exact quadrature).
    Then we have the mixed-norm \(L^2\) approximation error bounds
    \begin{center}
    \begin{tabular}{l l}
        \toprule
        Assumption & Bound 
        \\
        \midrule
        \(\ones \in \modelSpace\)
        &  \(\green{\norm{\linProj f - f}_\mu}\le 2\sqrt{\mu(\domain)} \blue{\error_\modelSpace^\infty[f]}\)
        \\
        \(\weight(x) = \frac{\sqrt{m}}{\lever_\modelSpace^\mu(x)}\)
        & \(
            \green{\norm{\linProj f - f}_\mu}
            \le 2\magenta{\error_\modelSpace^\weight[f]}
        \)
        \\
        \bottomrule
    \end{tabular}
    \end{center}
    and the \(L^2\) stochastic error bounds using \(\rho_\Sigma^2 \coloneq \opnorm{\Omega_n^{\frac12}\Sigma \Omega_n^{\frac12}}_2\) given by
    \begin{center}
    \makebox[\textwidth][c]{
    \begin{NiceTabular}{l m{5.65cm} m{4.8cm}}[tabularnote={\footnotesize \(C_\alpha >0\) is some constant,
    \(C_{p,\real}=\frac{14.7p}{\ln p}\), and
    \(C_{p,\complex}=\sqrt2\,C_{p,\real}\).
    }]
        \toprule
        & Assumption (decreasing strength \(\downarrow\)) & With probability at least \(1-\delta\)\newline \(\norm{\regOp(\noise)}_\mu \le \dots\)
        \\
        \midrule
        \newtag{\normalfont{(i)}}{it: (L2 exact quadrature) subgaussian}\tabularnote{Assume \(\field=\real\).\label{footnote: K=R, exact quadrature}}\!
        & \(\noiseIndep_i \sim \subG{\sigma^2}\)
        & \(\sigma\rho_\Sigma\frac{\sqrt{m} + \sqrt{2\ln(\frac1\delta)}}{\sqrt{n}}\)
        \\[1ex]
        \newtag{\normalfont{(ii)}}{it: (L2 exact quadrature) sub-Weibull}\ref*{footnote: K=R, exact quadrature}\!
        & \(\noiseIndep_i \sim \subWeibull{\alpha}\), \(\norm{\noiseIndep_i}_{\psi_\alpha}\le K\), \(\alpha\le 2\)
        & \(C_\alpha K \rho_\Sigma\frac{\sqrt{m} + \ln(\frac2\delta)^{\frac1\alpha}}{\sqrt{n}}\)
        \\[1ex]
        \newtag{\normalfont{(iii)}}{it: (L2 exact quadrature) finite p-th moment}
        & \(\E{\abs{\noiseIndep_i}^2}\le \sigma^2\) and \(\E{\abs{\noiseIndep_i}^p} \le \kappa^p\), \(p> 2\)
        & \(C_{p,\field} \max\set{\kappa, \sigma}\rho_\Sigma\sqrt{\frac{m}{n}}\bigl(\frac{1}{\delta}\bigr)^{\frac1p}\)
        \\[1ex]
        \newtag{\normalfont{(iv)}}{it: (L2 exact quadrature) finite second moment}
        & \(\E{\abs{\noiseIndep_i}^2} \le \sigma^2\)
        & \(\sigma\rho_\Sigma\sqrt{\tfrac{m}{n\delta}}\).
        \\
        \bottomrule
    \end{NiceTabular}
    }
    \end{center}
\end{example}
\begin{proof}
    Using \(\Gram_\mu = \identity\), the approximation error bounds follow from Corollary \ref{cor: mixed norm bound for p=2} combined with
    Lemma~\ref{lem: A_w and A_1 for exact quadrature} to simplify
    \(A_\weight(\Omega_n)\) and \(\norm{\frac1\weight}_\mu\). The stochastic error bound follows from Theorem \ref{thm: L2 concentration}.
\end{proof}

Observe that for a fixed \(\delta \in (0,1)\), the stochastic error bounds are of the form
\(\bigO(\sqrt{\frac{m}n})\). Balancing this rate with the convergence rate of the
approximation error term yields the following corollary.

\begin{corollary}[Optimal model dimension to minimize \(L^2\)-error bounds]
    Consider the regression setting of Example \ref{ex: L2 bounds for exact
    quadrature} or Theorem \ref{thm: specialized L2 noise bounds for exact
    quadrature}.  If the best approximation error
    \(\error_\modelSpace^\infty[f]\) or
    \(\error_\modelSpace^\weight[f]\) with
    \(\weight(x)=\frac{\sqrt{m}}{\lever_\modelSpace^\mu(x)}\) is of the order
    \(\bigO(m^{-\alpha})\) for some \(\alpha>0\) and \(\frac1\delta\in
    \bigO(1)\) and \(\rho_\Sigma \in \bigO(1)\), then
    the optimal choice of model dimension \(m=\dim \modelSpace\) is
    \[
        m \asymp n^{\frac1{2\alpha+1}},
        \qquad \text{which leads to} \qquad
        \norm{\regOp(y) -f}_\mu \in \bigO(n^{-\frac{\alpha}{2\alpha+1}}).
    \]
\end{corollary}
    
In the case of polynomial approximation, the best approximation
error is of order \(\bigO(m^{-k})\) where \(k\) is the number of
bounded derivatives of the function \(f\) (see Section \ref{sec: metric projection error}).
Since Lipschitz continuity is already sufficient for \(\alpha=1\), the
error bound is then of order \(\bigO(n^{-\frac13})\) and it improves
towards \(\bigO(n^{-\frac12})\) as the smoothness of \(f\) increases.

Since constant factors are ignored by the asymptotic statements of the corollary above,
it also applies to the following theorem. This theorem
replaces some occurrences of \(\rho_\Sigma^2\) by the term \(\opnorm{\Sigma}_2\).
This may be desirable if \(\Omega_n\) cannot be easily adapted to \(\Sigma\), which is typically unknown.
In this case \(\opnorm{\Sigma}_2\) is generally smaller than the bound
\(
    \rho_\Sigma^2 
\le \opnorm{\Sigma}_2 \opnorm{\Omega_n}_2.
\)
This replacement of \(\rho_\Sigma^2\) by \(\opnorm{\Sigma}_2\) is however only
possible for a specific choice of \(\Omega_n\). This \(\Omega_n\) already
made an appearance in Proposition \ref{prop: interpolation case} as the
unique choice of weights \(\omega_i\ge0\) for which exact quadrature is possible
in the interpolation setting \(n=m\).

\begin{theorem}[Specialized \(L^2\) stochastic error bounds for exact quadrature]
    \label{thm: specialized L2 noise bounds for exact quadrature}
    In the setting of Example \ref{ex: L2 bounds for exact quadrature}, 
    further assume \(\Omega_n = \diag[\omega_1, \dots, \omega_n]\) with
    \(\omega_i = \frac{m}{\lever_\modelSpace^\mu(x_i)^2}\), \(\ones \in
    \modelSpace\) and \(\mu(\domain)=1\). Then we have the alternative bounds on the stochastic error term \(\norm{\regOp(\noise)}_\mu\):
    \begin{center}
    \makebox[\textwidth][c]{
    \begin{NiceTabular}{l m{6.15cm} m{4.8cm}}[tabularnote={\footnotesize \(C_\alpha >0\) is some constant,
    \(C_{p,\real}=\frac{14.7p}{\ln p}\), and
    \(C_{p,\complex}=\sqrt2\,C_{p,\real}\).
    }]
        \toprule
        & Assumption (decreasing strength \(\downarrow\)) & With probability at least \(1-\delta\)\newline \(\norm{\regOp(\noise)}_\mu \le \dots\)
        \\
        \midrule
        \newtag{\normalfont{(i)}}{it: (L2 exact quadrature 2) subgaussian}\tabularnote{Assume \(\field=\real\).\label{footnote: K=R, exact quadrature, L2 omega}}\!
        & \(\noiseIndep_i \sim \subG{\sigma^2}\)
        & \(\frac{\sigma}{\sqrt{n}}\Bigl(\sqrt{\opnorm{\Sigma}_2m} + \rho_\Sigma\sqrt{2\ln(\tfrac1\delta)}\Bigr)\)
        \\[1ex]
        \newtag{\normalfont{(ii)}}{it: (L2 exact quadrature 2) sub-Weibull}\ref*{footnote: K=R, exact quadrature, L2 omega}\!
        & \(\noiseIndep_i \sim \subWeibull{\alpha}\), \(\alpha\le 2\), \(\norm{\noiseIndep_i}_{\psi_\alpha}\le K\)
        & \(\frac{C_\alpha K}{\sqrt{n}}\Bigl(\sqrt{\opnorm{\Sigma}_2m} + \rho_\Sigma\ln(\frac2\delta)^{\frac1\alpha}\Bigr)\)
        \\[1ex]
        \newtag{\normalfont{(iii)}}{it: (L2 exact quadrature 2) finite p-th moment}
        & \(\E{\abs{\noiseIndep_i}^2}\le \sigma^2\) and \(\E{\abs{\noiseIndep_i}^p} \le \kappa^p\) for \(p> 2\)
        & \(C_{p,\field} \max\set{\kappa, \sigma}\sqrt{\opnorm{\Sigma}_2\frac{m}{n}}\bigl(\frac{1}{\delta}\bigr)^{\frac1p}\)
        \\[1ex]
        \newtag{\normalfont{(iv)}}{it: (L2 exact quadrature 2) finite second moment}
        & \(\E{\abs{\noiseIndep_i}^2} \le \sigma^2\)
        & \(\sigma\sqrt{\opnorm{\Sigma}_2\tfrac{m}{n\delta}}\).
        \\
        \bottomrule
    \end{NiceTabular}
    }
    \end{center}
\end{theorem}

\subsection{Randomized regression}
\label{sec: randomized regression}

Since the search for exact quadrature is a difficult, problem-specific
task, it is often easier to settle for approximately equal inner products, which
is the topic of this section. There are two reasons to consider randomized
regression. First, it is a common and general approach to obtain approximately
identity Gram matrices, that is easier to work with than exact quadrature
\citep[e.g.][]{cohenStabilityAccuracyLeast2013,cohenOptimalWeightedLeastsquares2017,narayanChristoffelFunctionWeighted2016}.
The second reason is that randomized regression naturally models the
uncontrolled setting, where the evaluation points \(x_i\) are random samples
with sample distribution \(\nu\) out of the user's control.

In the previous section, we have already seen that 
\begin{itemize}[noitemsep]
    \item the regression weights
    \(\omega_i = \frac{m}{\lever_\modelSpace^\mu(x_i)^2}\) have special properties
    (Prop.~\ref{prop: interpolation case} and Thm.~\ref{thm: specialized L2 noise bounds for exact quadrature})
    \item the weight function \(\weight(x) = \frac{\sqrt{m}}{\lever_\modelSpace^\mu(x)}\)
    minimizes \(\norm{\lever_\modelSpace^\mu}_\weight^2\) (Lemma~\ref{lem: optimality of constant weighted leverage bound}).
\end{itemize}
The relationship \(\omega_i = \weight^2(x_i)\) suggested by these two observations has nice
properties by itself, for example the quantity \(A_\weight(\Omega_n)\) used in Thm.~\ref{thm: leverage function bound} and Cor.~\ref{cor: mixed norm bound for p=2}
simplifies to
\[
\textstyle
    A_\weight(\Omega_n) = \norm{\frac1\weight}_{\Omega_n}^2 =  \frac1n\sum_{i=1}^n \frac{\omega_i}{\weight^2(x_i)} = 1.
\]
In this section we will typically assume this relationship. 
In the following definition there is no mention of a sup-norm yet, such
that one could in principle view this as an odd collision of notation and simply
treat \(\weight^2\) as an arbitrary density. However, we will soon see that this
choice is not an accident.

\begin{definition}[Randomized regression]
    \label{def: randomized regression}
    For a measure \(\mu\), sampling probability measure \(\nu\) on \(\domain\) and a weight function \(\weight\colon \domain \to (0,\infty)\)
    assume that two out of the three parameters are given and the third
    determined by the relation
    \[
    \mu(dx) = \weight^2(x) \nu(dx).
    \]
    Randomized regression samples the evaluation points \(x_1, \dots, x_n\)
    iid from \(\nu\) and applies WLS to the obtained data with weights \(\omega_i =
    \weight^2(x_i)\).
\end{definition}

The following Lemma was first shown by \citet{cohenStabilityAccuracyLeast2013} for
the unweighted OLS case and later generalized to the WLS case \citet{cohenOptimalWeightedLeastsquares2017}.
Both papers contained a small mistake, which was only corrected in the first paper \citep{cohenCorrectionStabilityAccuracy2019}.
Here \(\weight\) first appears as the weight function of a
sup-norm.

\begin{lemma}[Gram concentration {\citep[Thm.~1]{cohenStabilityAccuracyLeast2013}}]
    \label{lem: gram concentration} 
    Assume we are in the randomized regression setting of Definition \ref{def: randomized regression}
    producing a diagonal \(\Omega_n\) (the WLS case). Let \((\Gram_\mu)_{ij} =
    \scp{\mBasis_j^\mu}{\mBasis_i^\mu}_{\Omega_n}\) be the Gram
    matrix of an orthonormal basis \((\mBasis_k^\mu)_{k=1}^m\) of \(\modelSpace\)
    with respect to an inner product \(\scp{\cdot}{\cdot}_\mu\). Then
    for \(\epsilon \in [0,1)\)
    \[
        \prob{\opnorm{\Gram_\mu - \identity}_2 > \epsilon}
        \le 2m \exp\bigl( - \tfrac{c(\epsilon) }{\norm{\lever_\modelSpace^\mu}_\weight^2}n\bigr)
        \le 2m \exp\bigl( - C \tfrac{\epsilon^2}{\norm{\lever_\modelSpace^\mu}_\weight^2}n\bigr)
    \]
    with \(c(\epsilon) \coloneq (1+\epsilon)\ln(1+\epsilon) - \epsilon\) and \(C \coloneq 2\ln 2 - 1 \approx 0.38\).
\end{lemma}
How may one apply this Lemma? Well, for \(\epsilon=\frac12\) we have \(c(\frac12) \approx 0.108 \ge 0.1\) 
and since \(\opnorm{\Gram_\mu-\identity}_2\le \frac12\) implies \(\opnorm{\Gram_\mu^{-1}} \le 2\) by Lemma~\ref{lem: bounds on inverse gram matrix}
we get the following corollary:
\begin{corollary}[Reasonably conditioned Gram matrix]
    \label{cor: reasonably conditioned Gram matrix}
    In the setting of Lemma \ref{lem: gram concentration}, if 
    \begin{equation}
        \label{eq: sufficient samples for well-conditioned gram matrix}
        n \ge 10\norm{\lever_\modelSpace^\mu}_\weight^2 \ln(\tfrac{2m}{\delta}),
    \end{equation}
    then
    \(
        \prob[\big]{\opnorm{\Gram_\mu - \identity}_2 \le \frac12, \opnorm{\Gram_\mu^{-1}} \le 2} \ge 1-\delta
    \).
\end{corollary}

Recall that \(\weight(x) = \frac{\sqrt{m}}{\lever_\modelSpace^\mu(x)}\) results
in the minimal weighted leverage maximum
\(\norm{\lever_\modelSpace^\mu}_\weight^2 = m\) by Lemma \ref{lem: optimality of
constant weighted leverage bound} such that \(n \in \bigO(m \log(m/\delta))\).
If the goal is an \(L^2(\mu)\)-bound, then the choice of \(\weight\) is
a choice of the sampling measure \(\nu(dx) = \frac1{\weight^2(x)}\mu(dx)\)
and the inverse leverage weight minimizes the number of samples \(n\) necessary
to get good results. Beyond this sampling criterion, the weight function
\(\weight\) does not necessarily have to appear in the \(L^2(\mu)\)-norm bounds.
In the following example it appears in the approximation error, but we also state an
alternative unweighted approximation error bound that does not depend on \(\weight\).

\begin{mdframed}[style=default]
\begin{example}[\(L^2(\mu)\)-bounds for randomized regression]
    \label{ex: L2 bounds for randomized regression}
    In the randomized regression setting of Definition \ref{def: randomized regression} assume \eqref{eq: sufficient samples for well-conditioned gram matrix}
    holds. Then with probability at least \(1-2\delta\)
    \[
        \norm{\regOp(y) - f}_\mu
        \le \green{\underbrace{(1+\sqrt{2})\error_\modelSpace^\weight[f]}_{\text{approximation error}}}
        + \blue{\underbrace{\sqrt{\tfrac1n\opnorm{\Omega_n^\frac12 \Sigma \Omega_n^\frac12}_2}\epsilon(\delta)}_{\text{stochastic error}}}
    \]
    with \(\epsilon(\delta)\) listed in the following table depending on the assumptions on the noise.
        \begin{center}
    \makebox[\textwidth][c]{
    \begin{NiceTabular}{l m{5.65cm} m{5.8cm}}[tabularnote={
        \footnotesize
        \(C_\alpha>0\) is a constant that only depends on \(\alpha\) and \(C_{p,\real}=\frac{14.7p}{\ln p}\), \(C_{p,\complex}=\sqrt2\,C_{p,\real}\).
    }]
        \toprule
        & Assumption (decreasing strength \(\downarrow\)) & \(\epsilon(\delta)\)
        \\
        \midrule
        \newtag{\normalfont{(i)}}{it: (L2 concentration) subgaussian, randomized}\tabularnote{\footnotesize Assume \(\field=\real\).\label{footnote: K=R, L2, randomized}}\!\!
        & \(\noiseIndep_i \sim \subG{\sigma^2}\)
        & \(\sigma\Bigl(\sqrt{2m} + 2\sqrt{\ln(\tfrac1\delta)}\Bigr)\)
        \\[1ex]
        \newtag{\normalfont{(ii)}}{it: (L2 concentration) sub-Weibull, randomized}\ref*{footnote: K=R, L2, randomized}\!\!
        & \(\noiseIndep_i \sim \subWeibull{\alpha}\), \(\norm{\noiseIndep_i}_{\psi_\alpha}\le K\), \(\alpha\le 2\) 
        & \(C_\alpha K \Bigl(\sqrt{2m} + \sqrt{2} \ln(\frac2\delta)^{\frac1\alpha}\Bigr)\)
        \\[1ex]
        \newtag{\normalfont{(iii)}}{it: (L2 concentration) finite p-th moment, randomized}
        & \(\E{\abs{\noiseIndep_i}^2}\le \sigma^2\) and \(\E{\abs{\noiseIndep_i}^p} \le \kappa^p\), \(p> 2\)
        & \(C_{p,\field} \max\set{\kappa, \sigma}\bigl(\frac{1}{\delta}\bigr)^{\frac1p}\sqrt{2m}\)
        \\[1ex]
        \newtag{\normalfont{(iv)}}{it: (L2 concentration) finite second moment, randomized}
        & \(\E{\abs{\noiseIndep_i}^2} \le \sigma^2\)
        & \(\sigma\sqrt{\tfrac{2m}{\delta}}\).
        \\
        \bottomrule
    \end{NiceTabular}
    }
    \end{center}
    If we further assume \(\weight(x) = \frac{\sqrt{m}}{\lever_\modelSpace^\mu(x)}\) in the randomized regression setting of Def.~\ref{def: randomized regression}
    then for \(n \ge 10 m \ln(\frac{2m}{\delta})\) we have with probability at least \(1-2\delta\) for \(\rho_\Sigma^2 \coloneq \opnorm{\Omega_n^{\frac12}\Sigma \Omega_n^{\frac12}}_2\)
    \[
        \norm{\regOp(y) - f}_\mu
        \le \green{\underbrace{(1+\sqrt{2})\error_\modelSpace^\weight[f]}_{\text{approximation error}}}
        + \blue{\begin{cases}
            2\sigma\Bigl(\sqrt{\opnorm{\Sigma}_2 \frac{m}n} + \sqrt{\frac{\rho_\Sigma^2}n\ln(\frac1\delta)}\Bigr)
            & \text{if \ref{it: (L2 concentration) subgaussian, randomized}}
            \\
            C_\alpha K \Bigl(2\sqrt{\opnorm{\Sigma}_2\tfrac{m}n} + \sqrt{2\frac{\rho_\Sigma^2}n}\ln(\frac2\delta)^{\frac1\alpha}\Bigr)
            & \text{if \ref{it: (L2 concentration) sub-Weibull, randomized}}
            \\[1ex]
            2C_{p,\field} \max\set{\kappa, \sigma}(\frac1\delta)^{\frac1p}\sqrt{\opnorm{\Sigma}_2 \frac{m}n}
            & \text{if \ref{it: (L2 concentration) finite p-th moment, randomized}}
            \\[1ex]
            2\sigma \sqrt{\opnorm{\Sigma}_2 \frac{m}{n\delta}}
            & \text{if \ref{it: (L2 concentration) finite second moment, randomized}}
        \end{cases}
        }
    \]
    Alternatively, if \(\mu\) is a probability measure and \(\ones\in \modelSpace\), then subject to the same assumptions
    the approximation error bound may in both cases be replaced by the following bound in terms of the unweighted best approximation error 
    \[
        \tag{alternative approx. error bound}
        \norm{\linProj f -f}_\mu \le \green{(1+\sqrt{3})\error_\modelSpace^\infty[f]}.
    \]
\end{example}
\end{mdframed}
\begin{proof}
    The probability \(1-2\delta\) comes from the union bound that ensures we are
    on the \(1-\delta\) probability Gram concentration event of Corollary
    \ref{cor: reasonably conditioned Gram matrix} and the \(1-\delta\) probability noise
    concentration. 
    The weighted approximation error bound follows directly from Corollary \ref{cor: mixed norm
    bound for p=2}
    using \(\opnorm{\Gram_\mu^{-1}}_2 \le 2\).
    The alternative approximation error bound also follows from Cor.~\ref{cor: mixed norm bound for p=2}
    using \[
        A_\ones(\Omega_n) \le \sqrt{1+1/2} \norm{\ones}_\mu =\sqrt{3/2}
    \]
    that follows from \(\opnorm{\Gram-\identity}_2\le \frac12\), \(\norm{\ones}_\mu = \mu(\domain)=1\) and Lemma~\ref{lem: A_w and A_1 for exact quadrature} \ref{it: A_1 approximate Gram identity}.
    The first stochastic error bounds follow from Theorem \ref{thm: L2 concentration} using \(\opnorm{\Gram_\mu^{-1}}_2 \le 2\) and \(\tr(\Gram_\mu^{-1})\le m \opnorm{\Gram_\mu^{-1}}_2 \le 2m\).
    The second stochastic error bounds follow from the inequalities \(\tr(H_\mu) \le \frac{4}n\opnorm{\Sigma}_2\)  and
    \(\opnorm{H_\mu}_2\le \frac2n \rho_\Sigma^2\) by Lemma~\ref{lem: (L2-noise bound) trace and op-norm bounds} put
    into Theorem \ref{thm: L2 concentration sharper}.
\end{proof}

If the goal is to obtain uniform bounds, then the weight function \(\weight\) is
inherent to the problem and cannot be chosen freely. On the other hand, the
measure \(\mu\) is now a free parameter that can be optimized (equivalently the
sample distribution \(\nu\)). Again we would like to choose \(\weight(x) =
\frac{\sqrt{m}}{\lever_\modelSpace^\mu(x)}\) however, this is a much harder
inverse problem than in the \(L^2\) case as the dependence of the leverage
function on \(\mu\) is more complicated. For research in this direction see
\citet{narayanChristoffelFunctionWeighted2016} and references therein.

\begin{example}[Uniform bounds for randomized regression]
    \label{ex: uniform bounds for randomized regression}
    In the randomized regression setting of Definition \ref{def: randomized regression} assume 
    the sample assumption \eqref{eq: sufficient samples for well-conditioned gram matrix}
    holds. Then with probability at least \(1-2\delta\)
    \[
        \norm{\regOp(y) - f}_\weight
        \le \green{\underbrace{(1+\sqrt{2}\norm{\lever_\modelSpace^\mu}_\weight)\error_\modelSpace^\weight[f]}_{\text{approximation error}}}
        + \blue{\underbrace{\varepsilon(\delta)\inf_{\linProj_\star}\lebconst{\weight}{\linProj_\star}\sqrt{\opnorm{\Omega_n^\frac12 \Sigma \Omega_n^\frac12}_2\tfrac{\norm{\lever_\modelSpace^\mu}_\weight^2}n}}_{\text{noise bound}}},
    \]
    where the infimum is taken over all regression projections \(\linProj_\star\colon \funSpace \to \modelSpace\) with \(m\) distinct regression points
    and \(\varepsilon(\delta)\) listed in the following table depending on the
    assumptions on the noise
    \begin{center}
    \makebox[\textwidth][c]{
    \begin{NiceTabular}{l m{6cm} m{5cm}}[tabularnote={
        \footnotesize 
    \(C_\alpha, C_{\alpha}'>0\) are unoptimized constants given in \eqref{eq: sub-Weibull constants}, \(C_{p,\real}\coloneq \frac{14.7p}{\ln p}\) and
    \(C_{p,\complex}=\sqrt2\,C_{p,\real}\).
    }]
        \toprule
        & Assumption (decreasing strength \(\downarrow\)) &  \(\varepsilon(\delta)\)
        \\
        \midrule
        \newtag{\normalfont{(i)}}{it: (uniform concentration) subgaussian, randomized}\tabularnote{\footnotesize Assume \(\field=\real\).\label{footnote: K=R, randomized}}\!\!\!
        & \(\noiseIndep_i \sim \subG{\sigma^2}\)
        & \(
            \sigma\sqrt{2\ln(\frac{2m}\delta)}
        \)
        \\[1ex]
        \newtag{\normalfont{(ii)}}{it: (uniform concentration) exponential, randomized}\ref*{footnote: K=R, randomized}\!\!\!
        & \(\noiseIndep_i \sim \subExp{\nu, \alpha}\)
        & \(
            \max\set[\Big]{\sqrt{2\nu^2\ln(\frac{2m}\delta)}, 2\alpha\ln(\frac{2m}\delta)}
        \)
        \\[1ex]
        \newtag{\normalfont{(iii)}}{it: (uniform concentration) sub-Weibull, randomized}\ref*{footnote: K=R, randomized}\!\!\!
        & \(\noiseIndep_i \sim \subWeibull{\alpha}\), \(\norm{\noiseIndep_i}_{\psi_\alpha}\le K\), \(\alpha\le 2\)
        & \(
            K\Bigl(C_\alpha \sqrt{\ln(\tfrac{2m}{\delta})} + C_{\alpha}' \ln(\tfrac{2m}{\delta})^{\frac1\alpha}\Bigr)
        \)
        \\[1ex]
        \newtag{\normalfont{(iv)}}{it: (uniform concentration) finite p-th moment, randomized}
        & \(\E{\abs{\noiseIndep_i}^2}\le \sigma^2\) and \(\E{\abs{\noiseIndep_i}^p} \le \kappa^p\), \(p> 2\)
        & \(
            C_{p, \field}\max\set{\sigma, \kappa}\bigl(\frac{m}{\delta}\bigr)^{\frac1p}
        \)
        \\[1ex]
        \newtag{\normalfont{(v)}}{it: (uniform concentration) finite second moment, randomized}
        & \(\E{\abs{\noiseIndep_i}^2} \le \sigma^2\)
        & \(
            \sigma\sqrt{\frac{m}{\delta}}
        \)
        \\
        \bottomrule
    \end{NiceTabular}
    }
    \end{center}
\end{example}
\begin{proof}
    The proof follows from the Leverage function bound (Theorem~\ref{thm:
    leverage function bound}) using \(A_\weight(\Omega_n)=1\) and
    \(\opnorm{\Gram_\mu^{-1}}_2 \le 2\) on the \(1-\delta\) probability Gram
    concentration event of Corollary~\ref{cor: reasonably conditioned Gram
    matrix} together with the uniform noise concentration of Corollary~\ref{cor:
    uniform noise bound}.
\end{proof}

Similar pointwise bounds can be obtained from the mentioned Theorems. And the
same sample complexity discussion as in Remark~\ref{rem: sample complexity discussion uniform bounds}
applies. Again, the leverage function bound is suboptimal. But while we could only
sketch a path to tighter bounds on the Lebesgue constant in the case of
exact quadrature, we will prove an actual concentration result in the following section.

\subsubsection{Concentration toward the \texorpdfstring{\(L^2(\mu)\)}{L2(µ)}-projection}

In the following we will show that the Lebesgue function 
\(\lebfun{\weight}{\linProj}{x}\) of the randomized regression projection \(\linProj\)
concentrates on the Lebesgue function of the \(L^2(\mu)\)-projection
\(\linProj_\mu\) as the sample size \(n\) increases. We start with a
definition of the Lebesgue function and constant of the \(L^2(\mu)\)-projection

\begin{definition}[Lebesgue function and constant of the \(L^2(\mu)\)-projection]
    \label{def: lebesgue function and constant of L2 projection}
    Let \(\linProj_\mu\) be the \(L^2(\mu)\)-projection onto \(\modelSpace\) and
    \(\weight\colon \domain \to (0,\infty)\) a weight function. Then
    the Lebesgue function and constant of \(\linProj_\mu\) are defined as
    \[
        \lebfun{\weight}{\linProj_\mu}{x}
        \coloneq \weight(x) \int \frac{\abs{K_\modelSpace^\mu(x,y)}}{\weight(y)} \mu(dy)
        \qquad\text{and}\qquad
        \lebconst{\weight}{\linProj_\mu} \coloneq \sup_{x\in \domain} \lebfun{\weight}{\linProj_\mu}{x}.
    \]
\end{definition}
\begin{remark}[Non-negative kernel]
    \label{rem: non-negative kernel}
    If \(\reprK_\modelSpace^\mu(x,y) \ge 0\) and \(\ones\in \modelSpace\), then by
    the same argument as in Remark \ref{rem: positive reproducing kernels} we have \(\lebfun{\infty}{\linProj_\mu}{x} = 1 = \lebconst{\infty}{\linProj_\mu}\).
\end{remark}

In Section \ref{sec: approximation theory} we will outline a joint generalization
of the Lebesgue function for regression and \(L^2(\mu)\)-projection that justifies this
name collision.  The
above definition of the Lebesgue function is essentially the reproducing kernel
representation of Theorem \ref{thm: reproducing kernel representation} combined
with the definition of the Lebesgue function of regression (Def.~\ref{def: lebesgue function and constant}).

In the following theorem we prove concentration of the Lebesgue function of
randomized regression around that of \(L^2(\mu)\)-projection, which is one of
our main original contributions.

\begin{remark}[Related work]
During the finalization of this manuscript we became aware that
\citet[Prop~8; see Prop.~21 for its proof]{maranLeastSquaresUniform2025} state a similar concentration result for
Lebesgue constants in the real-valued setting. However the proof appears to bound the variance for a fixed \(x\in \domain\)
using Bernstein's inequality (p.~40), and they thereby bound the Lebesgue function.
The Lebesgue constant is however a supremum over all \(x\in \domain\).
In their setting of finite domains \(\domain\) this may be addressed with
a union bound over \(x\) at the cost of a factor \(\sqrt{\log\abs{\domain}}\) in
the constant. For infinite domains \(\domain\) a different approach is needed.
They also obtain slightly worse rates, essentially
\(\sqrt{m}\norm{\lever_\modelSpace^\mu}_\weight^2\) instead of
\(\norm{\lever_\modelSpace^\mu}_\weight^2\sqrt{\log m}\) and they only state the
asymptotic rate without explicit, optimized constants.
Their result is more general in that they allow for ridge regularization.
\end{remark}

\begin{mdframed}[style=default]
\begin{theorem}[Lebesgue function and constant concentration]
    \label{thm: lebesgue constant concentration}
    Let \(\linProj\) be a randomized regression projection (Definition \ref{def: randomized regression})
    and \(\linProj_\mu\) the \(L^2(\mu)\) projection to \(\modelSpace\). 
    \begin{enumerate}[label=(\roman*), noitemsep]
    \item \textbf{Lebesgue function concentration.}
    Let \(m\in \nat\) and \(\delta \in (0,1)\) and assume
    \[
        n \ge 9\norm{\lever_\modelSpace^\mu}_\weight^2 \ln(2\tfrac{m+1}\delta)
    \]
    then for every fixed \(x\in
    \domain\) with probability at least \(1-\delta\) we have
    \[
        \abs{\lebfun{\weight}{\linProj}{x} - \lebfun{\weight}{\linProj_\mu}{x}}
        \le 3\norm{\lever_\modelSpace^\mu}_\weight\sqrt{\tfrac{1}{n}\ln(2\tfrac{m+1}\delta)} \lever_\modelSpace^\mu(x)\weight(x).
    \]
    \item \textbf{Lebesgue constant concentration.}
    Let \(m\in \nat\) and \(\delta \in (0,1)\) and assume
    \[
        n \ge C_{\field}^2\norm{\lever_\modelSpace^\mu}_\weight^2 \ln(2\tfrac{m+1}\delta)
    \]
    for \(C_\real \coloneq 3\sqrt{3}\approx 5.2\) and \(C_\complex \coloneq 4\sqrt{3}\approx 6.9\). Then with probability at least \(1-\delta\) we have
    \[
        \abs{\lebconst{\weight}{\linProj} - \lebconst{\weight}{\linProj_\mu}}
        \le C_{\field} \norm{\lever_\modelSpace^\mu}_\weight^2 \sqrt{\tfrac1n\ln(2\tfrac{m+1}{\delta})},
    \]
\end{enumerate}
\end{theorem}
\end{mdframed}
\begin{remark}[Tail bounds]
    If you prefer tail bounds to high probability bounds, see \eqref{eq:
    pointwise lebesgue function concentration} and \eqref{eq: uniform leb
    constant concentration} respectively.
\end{remark}
\begin{remark}[Constants]
    The proof shows that the constant \(C_\field\) can be asymptotically improved
    to \(\sqrt{12}\) as \(m\to \infty\). While we tried our best to optimize
    constants, it is difficult to say how far this is still from optimal.
\end{remark}

Observe that we already need \(n\gtrsim \norm{\lever_\modelSpace^\mu}_\weight^2
\log(m/\delta)\) samples for the Gram concentration in Corollary~\ref{cor:
reasonably conditioned Gram matrix}. If \(n\) has exactly this
order, then
\[
    \abs{\lebfun{\weight}{\linProj}{x} - \lebfun{\weight}{\linProj_\mu}{x}}
    \lesssim \lever_\modelSpace^\mu(x)\weight(x)
    \qquad\text{and}\qquad
    \abs{\lebconst{\weight}{\linProj} - \lebconst{\weight}{\linProj_\mu}}
    \lesssim \norm{\lever_\modelSpace^\mu}_\weight.
\]
This therefore only yields the leverage function bounds. However if \(n\) is
increasing faster, then this concentration result can offer an improvement over
this leverage function bound. For example if we choose
\(n \asymp \norm{\lever_\modelSpace^\mu}_\weight^4 \log(m/\delta)\), then
\[
    \abs{\lebfun{\weight}{\linProj}{x} - \lebfun{\weight}{\linProj_\mu}{x}}
    \in \bigO(1) 
    \qquad\text{and}\qquad
    \abs{\lebconst{\weight}{\linProj} - \lebconst{\weight}{\linProj_\mu}}
    \in \bigO(1).
\]
Recall that in the optimal case of constant weighted leverage,
\(\norm{\lever_\modelSpace^\mu}_\weight^2 = m\), such that this is the difference
between \(n \asymp m \log(m/\delta)\) and \(n \asymp m^2 \log(m/\delta)\).
If \(\lebconst{\weight}{\linProj_\mu} \in \bigO(\log m)\), this is a significant improvement
over the leverage function bound for example, yielding
\[
    \norm{\linProj f - f}_\weight
    \le \bigl(1 + \lebconst{\weight}{\linProj}\bigr) \error_\modelSpace^\weight[f]
    \le \Bigl(1+ \underbrace{\lebconst{\weight}{\linProj_\mu}}_{\in \bigO(\log m)} + \underbrace{\abs[\big]{\lebconst{\weight}{\linProj} - \lebconst{\weight}{\linProj_\mu}}}_{\in \bigO(1)}\Bigr) \error_\modelSpace^\weight[f].
\]
Whereas the leverage function bound is at best of order \(\bigO(\sqrt{m})\)
(Lemma \ref{lem: optimality of constant weighted leverage bound}). 	\section{Catalogues}
\label{sec: catalogue}

\subsection{Metric projection errors}
\label{sec: metric projection error}

The metric projection error 
\(\error_\modelSpace^*[f]
    \coloneq \inf_{\model\in \modelSpace} \norm{f-\model}_*
\)
is a well-studied object in Approximation theory.
Bounds on this error are often referred to as `Jackson and Bernstein theorems' as they
proved the first bounds of this type for trigonometric and algebraic polynomials
\citep[e.g.][]{jacksonUeberGenauigkeitAnnaeherung1911,jacksonApproximationTrigonometricSums1912,bernstein1912recherches,bernsteinLordreMeilleureApproximation1912}.
Since the bounds are the same whether one considers regression or interpolation all
bounds found in the approximation theory literature
\citep[e.g.][]{lubinskySurveyWeightedPolynomial2007,devoreConstructiveApproximation1993,jacksonTheoryApproximation1930,achieserTheoryApproximation2013}
may be reused. To state a list of examples, we make use of the modulus of continuity
\[
    \modCont(f; \delta) \coloneq \sup_{\substack{\metric(x,y) \le \delta \\ x,y \in \domain}} \abs{f(x) - f(y)},
\]
where \(\metric\) is the metric on \(\domain\) and \(\delta > 0\). If \(f\) is
\(\alpha\)-Hölder continuous, then \(\modCont(f; \delta) \le c\delta^\alpha\) for example.

\begin{enumerate}[wide=0pt]
    \item \textbf{Piecewise constant functions.}
    Let \((\domain, \metric)\) be a metric space and 
    \((A_i)_{i\in I}\) a finite collection of disjoint subsets of \(\domain\) with \(\biguplus_{i\in
    I} A_i =\domain\) defining the space of piecewise constant functions \(\modelSpace =
    \Span\set{\ind{A_i} : i\in I}\). Choosing the midpoint between maximum and
    minimum values of \(f\) on each \(A_i\), i.e., \((\sup_{A_i} f +
    \inf_{A_i}f)/2\) as the value of the piecewise constant function on \(A_i\)
    immediately yields
    \[
        \tag{general decomposition}
        \error_\modelSpace^\infty[f]
        \le \tfrac12 \max_{i\in I} \Bigl(\sup_{A_i} f - \inf_{A_i}f\Bigr)
        \le \tfrac12 \modCont\Bigl(f; \max_{i\in I} \mathrm{diam}(A_i)\Bigr).
    \]
    For example let \(\domain = [a,b]^\dims\) equipped with metric
    \(\metric(x,y) = \abs{x-y}_\infty\) and let \(A_i\) be a square tiling of
    side-length \(\frac{b-a}{M}\) such that \(m\coloneq \dim \modelSpace = \# I =
    M^\dims\). Then
    \[
        \tag{square tiling}
        \error_\modelSpace^\infty[f]
        \le \tfrac12 \modCont\bigl(f; \tfrac{b-a}{m^{1/\dims}}\bigr).
    \]
    For example, if \(f\) is \(\alpha\)-Hölder continuous, then \(\error_\modelSpace^\infty[f] \in \bigO(m^{-\frac\alpha\dims})\),
    which shows a curse of dimensionality.

    To see that these bounds are sharp for the tiling of \(\domain=[a,b]^\dims\), consider
    linear functions, which are Lipschitz continuous with some constant \(L\) such that \(\omega(f; \delta) = L\delta\). 
    It is straightforward to show that the upper bound above is then attained by a linear
    function with maximum and minimum in opposite corners of the tiles. Since the worst-case
    examples are already linear functions, the bounds cannot be improved with the
    assumption of higher order smoothness than Lipschitz continuity. This is in
    contrast to the case of polynomials following below.

    \item\textbf{Algebraic polynomials.} 
    Let \(\domain \subseteq \real^\dims\) and let \(\modelSpace_M \coloneq
    \poly_{\set{\abs{\lambda}_1 \le M}}\) be the span of multivariate
    monomials \(x_1^{\lambda_1}\cdots x_\dims^{\lambda_\dims}\) for \(\lambda
    \in \nat_0^\dims\) with total degree \(\abs{\lambda}_1 \le M\).
    Let \(C^k\) denote the space of \(k\)-times continuously differentiable
    functions and \(f^{(k)}\) the \(k\)-th derivative of \(f\). Let
    \((x)_r \coloneq x(x-1)\cdots(x-r+1)\) denote the falling factorial, \(\sphere\)
    the unit sphere and \(\Ball\) the unit ball. Let \(\TV(f)\) be the total
    variation of \(f\); if \(f\) is absolutely continuous, then \(\TV(f) =
    \norm{f'}_{L^1}\). In the following table we always assume \(M\in \nat\) and \(k\in \nat_0\).
    \begin{center}\makebox[\linewidth][c]{
    \begin{NiceTabular}{c c m{3.4cm} m{6cm} c}[
        notes/label-in-tabular={\scriptsize #1},
        notes/style=\arabic{#1}
    ]
        \toprule
        \(\domain\) & \(\weight\) & Assumptions  & Bound \(\error_\modelSpace^\weight[f] \le\dots\) & Ref.
        \\
        \midrule
        \([a,b]\)
        & \(1\)
        & \(f\) analytic
        & \(o(\rho^{-M})\)\hfill for some \(\rho>1\)
        & \tabularnote{\footnotesize\citet[Table on p.~257]{bernstein1912recherches} quoting \citet{bernsteinLordreMeilleureApproximation1912}.
        See \citet[Thm.~8.2]{trefethenApproximationTheoryApproximation2019} with Rem.~\ref{rem: change of domain} for a quantitative result
        that also shows Chebyshev \(L^2\)-projection achieves this bound.
        }
        \\[2ex]
        \([a,b]\)
        & \(1\)
        & \(f^{(k)}\) exists, \(M>k \ge 1\)
        & \(\frac{2}{\pi k} \bigl(\frac{b-a}{2(M-k)}\bigr)^k \TV(f^{(k)})\)
        & \tabularnote{\footnotesize\citet[Thm.~7.2]{trefethenApproximationTheoryApproximation2019} with Rem.~\ref{rem: change of domain}.
        He moreover shows Chebyshev \(L^2\)-projection achieves this bound.
        }
        \\[2ex]
        \([a,b]\)
        & \(1\)
        & \(f\in C^{k}\), \(M \ge k\)
        & \(\frac{c_k}{2^k}\frac{(b-a)^k}{(m)_k} \modCont(f^{(k)}; \frac{b-a}{m-k})\)
        \hfill \(0.6< c_k \le 2.02\)
        & \tabularnote{
            \footnotesize 
            Corollary~\ref{cor: close to optimal jackson bound} based on \citet{sinwelUniformApproximationDifferentiable1981}
            and \citet{jacksonUeberGenauigkeitAnnaeherung1911}. The bound does not hold for \(c_k \le 0.6\).
        }
\\[2ex]
        \(\real\)
        & \(\exp(-\abs{x}^\alpha)\)
        & \(f^{(k-1)}\) abs.\ cont.,\par \(\alpha > 1\), \(M\ge k \ge 1\)
        & \(c_{\alpha,k} M^{k(\frac1\alpha - 1)}\norm{f^{(k)}}_\weight\)
        \hfill \(c_{\alpha,k} >0\)
        & \tabularnote{
            \footnotesize\citet[Cor.~3.2]{lubinskySurveyWeightedPolynomial2007} (+ Rem.~\ref{rem: MRS number} on the MRS number),
            Modulus of continuity estimates also exist but require an awkward redefinition of the modulus of continuity \citep[Thm.~3.5]{lubinskySurveyWeightedPolynomial2007}.
        }
        \\[2ex]
        \(\sphere\)
        & \(1\)
        & \(f\in C^0\)
        & \(c\, \modCont(f; \frac{\dims}M)\) \hfill \(c>0\)
        & \tabularnote{\footnotesize\citet[Thm.~2]{newmanJacksonsTheoremHigher1964}.}
        \\[2ex]
        \(\sphere\)
        & \(1\)
        & \(f\in C^k, \dims\ge 2\)
        & \(c_{\dims}^{k+1} M^{-k}\modCont(f^{(k)}; \frac1M)\) \hfill \(c_{\dims} > 0\)
        & \tabularnote{\footnotesize \citet[Thm.~3.3]{ragozinConstructivePolynomialApproximation1971} with inductively defined
            \(\modCont(f^{(k+1)};\delta) \coloneq \sum_{1\le i<j\le \dims} \modCont(D_{ij} f^{(k)}; \delta)\),
            where the operators \(D_{ij} = x_i\partial_j - x_j\partial_i\) span the Lie algebra of rotations.
            And the metric on \(\domain=\sphere\) is the geodesic distance.
        }
        \\[2ex]
        \(\Ball\)
        & \(1\)
        & \(f\in C^0\)
        & \(c\, \modCont(f; \frac{\dims}M)\) \hfill \(c>0\)
        & \tabularnote{\footnotesize\citet[Thm.~3]{newmanJacksonsTheoremHigher1964}, see also \citet[Thm.~3.4]{ragozinConstructivePolynomialApproximation1971} for a bound involving derivatives.}
        \\[2ex]
        \([-1,1]^\dims\)
        & \(1\)
        & \(f\in C^0\)
        & \(c\, \modCont(f; \frac{\dims^{3/2}}M)\) \hfill \(c>0\)
        & \tabularnote{\footnotesize\citet[Thm.~4]{newmanJacksonsTheoremHigher1964}.}
        \\[\defaultaddspace]
        \bottomrule
    \end{NiceTabular}
    }
    \end{center}
    \begin{remark}[Curse of dimensionality]
        In one dimension, \(\dims=1\), we have \(\dim \modelSpace_M = M+1\).
        For \(\domain = \sphere\) with \(\dims\ge 2\) however, we have
        \begin{equation}
            \label{eq: dimension of polynomials on the sphere} 
            \textstyle
            m \coloneq \dim \modelSpace_M = \binom{M+\dims}{\dims} - \binom{M+\dims -2}{\dims}\sim \frac{2M^{\dims-1}}{(\dims-1)!}
            \asymp M^{\dims-1},
        \end{equation}
        where the asymptotic regime assumes \(M\to \infty\) and \(\dims\) fixed.
        In particular, for \(f\in C^k\) with \(\alpha\)-Hölder continuous \(k\)-th derivative
        \[
            \error_\modelSpace^\infty[f] \in \bigO(M^{-(k+\alpha)}) = \bigO(m^{-\frac{k+\alpha}{\dims-1}}).
        \]
        That is, the polynomial growth \(m\asymp M^{\dims-1}\) of the dimension \(m=\dim \modelSpace_M\)
        in the degree cripples the approximation rate just as in the case of
        piecewise constant functions. The smoothness of the function
        \(f\) mitigates the curse of dimensionality however.
    \end{remark}

    \begin{remark}[Weighted norms]
        Weighted norms are a comparatively recent development and not as well
        explored as the unweighted case. To gain an intuition for the weighted
        bound, recall we already mentioned the idea of restricted range
        inequalities in Section~\ref{sec: restricted range inqualities}.
        Such a restricted range inequality is available for the univariate polynomials
        with weight \(\weight(x) = \exp(-\abs{x}^\alpha)\) (see Theorem~\ref{thm: restricted range inequalities for polynomials}). Specifically, the
        `relevant domain' is \([-a_M, a_M]\) with \(a_M\asymp M^{\frac1\alpha}\)
        the MRS number (see Rem.~\ref{rem: MRS number}). Best unweighted approximation on the
        interval \([-a_M, a_M]\) is of order \(\bigO(\frac{(2a_M)^k}{M^{k}} \modCont(f^{(k)}; \frac{2a_M}{M})) = \bigO(M^{k(\frac1\alpha-1)}\modCont(f^{(k)}; M^{\frac1\alpha-1}))\).
        This explains where the \(M^{k(\frac1\alpha - 1)}\) factor in the bound for weighted polynomials comes from.
        The actual proof is more complicated, since the restricted range inequality
        only applies to the functions in the model space and not the function \(f\).
    \end{remark}

    Many of these bounds are sharp and converse theorems exist that prove the
    convergence rate of approximation forces smoothness assumptions on \(f\) 
    \citep[see e.g.][Table on p.~257]{bernstein1912recherches}. Even the constants in the list above
    are fairly optimized. To obtain these bounds, approximation theorists identified
    a number of tricks to assume \(\domain=[-1,1]\) without loss of generality and that
    the function \(f\) is Lipschitz continuous. Many of the higher-order smoothness assumptions
    can then be bootstrapped from this case. The bounds for this simple case
    are often obtained constructively from Chebyshev projection (Fourier
    projection and transformation to algebraic polynomials). In the following we
    outline the tricks used to make the assumptions above. We also use them to
    prove the close to optimal modulus of continuity bound based on a result by
    \citet{sinwelUniformApproximationDifferentiable1981}.

    \begin{remark}[Trick 1: Change of domain]
        \label{rem: change of domain}
        Instead of the domain \([a,b]\) many authors consider the domain \([-1,1]\) without loss of generality.
        Let \(\modelSpace\) be the space of polynomials up to degree \(M\) restricted 
        to the domain \([a,b]\) and \(\modelSpace^T \coloneq \set{\model
        \circ T: \model \in \modelSpace}\) the space of polynomials on \([-1,1]\) with 
        \(T\colon [-1,1] \to [a,b]\) the linear bijection defined by \(T(x) = \frac{b-a}{2}x + \frac{a+b}{2}\). 
        Then
        \[
            \error_{\modelSpace}^\infty[f]
            = \inf_{\model\in \modelSpace}\norm{f-\model}_{\infty, [a,b]} 
            = \inf_{\model\in \modelSpace}\norm{f\circ T - \model\circ T}_{\infty, [-1,1]}
            = \error_{\modelSpace^T}^\infty[f\circ T].
        \]
        Using \((f\circ T)^{(k)} = (\frac{b-a}{2})^k
        f^{(k)}\circ T\), the observations \(\TV(cf\circ T) = c \TV(f)\) for \(c>0\)
        and \(\modCont(c f\circ T; \delta) = c \modCont(f; \frac{b-a}{2}\delta)\)
        yield the bounds on \([a,b]\) from those on \([-1,1]\).
    \end{remark}

    \begin{remark}[Trick 2: Derivative lifting]
        \label{rem: derivative lifting}
        Assume \(\modelSpace_M\coloneq \poly_M\) is the space of univariate
        polynomials.  Then the order of convergence for \(f\in C^k\)
        can be bootstrapped from a bound for Lipschitz continuous functions. 
        To see this, let \(\error_M^\infty \coloneq \error_{\modelSpace_M}^\infty\)
        and assume that we already have the bound
        \begin{equation}
            \label{eq: derivative lifting assumption}    
            \error_M^\infty[f] \le \eta_M \norm{f'}_\infty.
        \end{equation}
        Using the invariance of the best approximation error under translation by \(p\in \poly_M\) we get
        \[
            \error_M^\infty[f]
            = \inf_{p\in \poly_M}\error_M^\infty[f-p]
            \overset{\eqref{eq: derivative lifting assumption}}\le \eta_M \inf_{p\in \poly_M}\norm{f'-p'}_\infty
            = \eta_M \error_{M-1}^\infty[f'],
        \]
        where the last equality uses that any polynomial \(p'\in \poly_{M-1}\)
        has an anti-derivative \(p\) in \(\poly_M\) such that minimization over
        \(p\in \poly_M\) is equivalent to minimization over \(p'\in
        \poly_{M-1}\).
        Iterating this argument \(k\) times yields
        \(\error_M^\infty[f] \le \eta_M\cdots\eta_{M-k+1} \error_{M-k}^\infty[f^{(k)}]\).

        This recursive inequality is sometimes called the `Jackson-Favard inequality' and
        this approach also works for weighted norms \citep[e.g.][(3.11)]{lubinskySurveyWeightedPolynomial2007}.
        It breaks down for multivariate polynomials because not every vector
        field has an antiderivative. It also results in suboptimal constants
        whose disadvantage to the optimal constants grows exponentially in
        \(k\).  It is useful however, to gain an intuition where the falling
        products \((M)_k=M(M-1)\cdots(M-k+1)\) come from in the bounds for
        \(f\in C^k\) and to get the right order of convergence.
    \end{remark}

\citet[Theorem~1]{sinwelUniformApproximationDifferentiable1981} identified the optimal
    constants and shows (with application of Rem.~\ref{rem: change of domain})
    for \(\error_M^\infty \coloneq \error_{\poly_M}^\infty\) with \(m=\dim \poly_M=M+1\) that
    \begin{equation}
        \label{eq: sinwell favard bound}
        \error_M^\infty[f] \le \favConst_k \frac{(b-a)^k}{2^k(m)_k}\norm{f^{(k)}}_\infty
        \qquad\text{for}\qquad \favConst_k \coloneq \frac{4}{\pi}\sum_{j=0}^\infty \frac{(-1)^{j(k+1)}}{(2j+1)^{k+1}}.
    \end{equation}
    The Favard constants \(\favConst_k\)---named after \citet{favardMeilleursProcedesDapproximation1937}, who proved
    the corresponding tight result for trigonometric polynomials---cannot be reduced and satisfy
    \citep[p.~220]{favardMeilleursProcedesDapproximation1937}
\begin{equation}
        \label{eq: favard constants} 
        \tfrac{\pi^2}8 = \favConst_2 \le \favConst_4 \le \dots \le \tfrac{4}{\pi} < \dots \le \favConst_3 \le \favConst_1 = \tfrac{\pi}{2}
        \qquad \text{and}\qquad \lim_{k\to \infty} \favConst_k = \tfrac{4}{\pi}.
    \end{equation}

    \begin{remark}[Trick 3: Modulus of continuity from Lipschitz approximation]
        \label{rem: modulus of continuity from lipschitz approximation}
        A Lipschitz bound can always be translated into a modulus of continuity bound
        by approximation of \(f\colon [a,b] \to \real\) with a Lipschitz continuous function. 
        This trick was already used by
        \citet{jacksonUeberGenauigkeitAnnaeherung1911}, who used piecewise
        linear interpolation. For better constants, we use
        Lemma~\ref{lem: lipschitz approximation}
        to obtain a Lipschitz continuous function \(q\) with the properties
        \begin{equation}
            \label{eq: lipschitz approximation}
            \norm{f-q}_\infty \le \tfrac12\modCont(f; \delta)
            \qquad\text{and}\qquad
            \norm{q'}_\infty \le \tfrac{\modCont(f; \delta)}{\delta}.
        \end{equation}
        Assuming we already have the bound \(\error_M^\infty[q] \le \eta_M
        \norm{q'}_\infty\) for Lipschitz continuous functions \(q\), we may apply the
        triangle inequality to the approximation error to get
        \[
            \error_M^\infty[f]
            \le \error_M^\infty[q] + \norm{f-q}_\infty
            \overset{\text{assmpt.}}\le \eta_M \norm{q'}_\infty + \norm{f-q}_\infty
            \overset{\eqref{eq: lipschitz approximation}}\le (\tfrac12 + \tfrac{\eta_M}{\delta})\modCont(f; \delta).
        \]
        Observe that this modulus of continuity trick is not limited to polynomials and can be used for any model space.
    \end{remark}

    Using the tricks above, we can now prove a close to optimal Jackson bound
    for univariate polynomials using the optimal bounds from
    \citet{sinwelUniformApproximationDifferentiable1981} for Lipschitz
    continuous \(f^{(k)}\).

    \begin{corollary}[Close to optimal algebraic bound]
        \label{cor: close to optimal jackson bound}
        Assume \(\domain = [a,b]\), \(\error_M^\infty \coloneq \error_{\poly_M}^\infty\) with \(m=\dim \poly_M = M+1\). Then
        for \(f\in C^k\) with \(k\le M\)
        \[
            \error_M^\infty[f]
            \le \tfrac{c_k}{2^k} \tfrac{(b-a)^k}{(m)_k}\modCont(f^{(k)}; \tfrac{b-a}{m-k})
            \qquad\text{for}\qquad c_k \coloneq (\tfrac12 + \tfrac{\pi}{4})\favConst_k
            \le (\tfrac12 + \tfrac{\pi}{4})\tfrac{\pi}{2}
            \le 2.02.
        \]
        The inequality does not hold, if \(c_k\) is replaced by a value smaller than \(\frac{\favConst_{k+1}}{2} \ge \frac{\pi^2}{16} > 0.6\).
    \end{corollary}

    \item\textbf{Trigonometric polynomials.} We consider the space of
    trigonometric polynomials 
    \[
        \modelSpace_M^\Trig
        \coloneq \set[\Big]{a_0 + 2\sum_{k=1}^M a_k \cos(k\ang) + 2\sum_{k=1}^M b_k \sin(k\ang): a_k, b_k \in \real}
    \]
    of degree at most \(M\). The domain \(\domain\) is the circle
    \(\torus \coloneq \real/(2\pi\integer)\) such that
    the functions \(f\colon \torus \to \real\) may be viewed as \(2\pi\)-periodic
    functions \(f\colon \real\to \real\). Since weights are mostly used for unbounded
    domains, we only consider the unweighted case \(\weight\equiv 1\) here.
    \begin{center}\makebox[\linewidth][c]{
    \begin{NiceTabular}{m{5.5cm} m{6.5cm} c}[
        notes/label-in-tabular={\scriptsize #1},
        notes/style=\arabic{#1},
    ]
        \toprule
        Assumptions  & Bound \(\error_\modelSpace^\infty[f] \le\dots\) & Ref.
        \\
        \midrule
        \(f\) is analytic on \(S\) with \(\abs{f(z)}\le B\),
        \par \(S\coloneq\set{x+iy: \abs{y} \le \rho, \abs{x} \le \pi }\)
        & \(\frac{2B}{e^\rho -1} e^{-\rho M}\)
        & \tabularnote{\footnotesize
        \citet[Theorem~4.2]{wrightExtensionChebfunPeriodic2015}, 
        and the bound is achieved by \(L^2(\torus)\)-projection onto \(\modelSpace_M^\Trig\).
        }
        \\[3ex]
        \(f^{(k)}\) exists, \(M, k\in \nat\)
        & \(\frac{1}{\pi k} M^{-k} \TV(f^{(k)})\)
        & \tabularnote{\footnotesize \citet[Theorem~4.2]{wrightExtensionChebfunPeriodic2015}, 
        and the bound is achieved by \(L^2(\torus)\)-projection onto \(\modelSpace_M^\Trig\).
        }
        \\[2ex]
        \(f\in C^{k}\), \(M \ge 1\)
        & \(\frac{c_k}{(M+1)^k} \modCont(f^{(k)}; \frac\pi{M+1})\)
        \hfill \(\frac{\pi}{8} \le c_k \le \favConst_k \le \frac{\pi}{2}\)
        & \tabularnote{
            \footnotesize 
            Favard's inequality \citep[e.g.][(19.5)]{arestovExpositionLecturesSB2011} combined with
            \citet[between Eq.~(1) and (2)]{korneichukExactConstantJackson1962}, see Rem.~\ref{rem: favard's inequality} for details.
            The bound does not hold for \(c_k < \frac{\pi}{8}\).
        }
\\[\defaultaddspace]
        \bottomrule
    \end{NiceTabular}
    }
    \end{center}

    \begin{remark}[Favard's inequality]
        \label{rem: favard's inequality}
        For \(\error_M^\infty \coloneq \error_{\modelSpace_M^\Trig}^\infty\)
        \citet[Théorème, p.~219]{favardMeilleursProcedesDapproximation1937} proved
        \begin{equation}
            \label{eq: favard's inequality}
            \error_M^\infty[f] \le \tfrac{\favConst_k}{(M+1)^{k}}\norm{f^{(k)}}_\infty
        \end{equation}
        and showed the Favard constants \(\favConst_k\), cf.~\eqref{eq: favard
        constants}, to be optimal. Unfortunately, we can not apply Trick 2
        (Rem.~\ref{rem: derivative lifting}) since the antiderivative of the
        constant function included in \(\modelSpace_M^\Trig\) is not in
        \(\modelSpace_K^\Trig\) for any \(K\). It was later realized that the
        proof of Favard can be modified to yield the following `Favard's inequality' directly
        \citep[e.g][(19.5)]{arestovExpositionLecturesSB2011}
        \[
            \error_M^\infty[f] \le \tfrac{\favConst_k}{(M+1)^{k}}\error_M^\infty[f^{(k)}].
        \]
        Notice, that this inequality does not decrease \(M\) as in the case of algebraic polynomials,
        intuitively because the antiderivative of most trigonometric basis functions
        is contained in the same model space. With the optimal bound
        \(\error_M^\infty[q]  \le \modCont(q; \frac\pi{M+1})\) for
        continuous \(q\) by \citep[between Eq.~(1) and (2)]{korneichukExactConstantJackson1962}
        this implies the bound in the table. For the lower bound on \(c_k\)
        again consider \(f\) with \(\modCont(f^{(k)}; \delta) \le \delta \norm{f^{(k+1)}}_\infty\)
        and use the optimality of \eqref{eq: favard's inequality} to get \(c_k \pi \ge \favConst_{k+1} \ge \frac{\pi^2}{8}\).
    \end{remark}

\end{enumerate}

\subsection{Leverage function bounds}
\label{sec: christoffel function and reproducing kernel bounds}

\begin{enumerate}[wide=0pt]
    \item \textbf{Piecewise constant functions.} Let \((\domain, \metric)\) be a metric space and 
    \((A_i)_{i\in I}\) a finite collection of disjoint subsets of \(\domain\)
    with \(\biguplus_{i\in I} A_i =\domain\) defining the
    space of piecewise constant functions \(\modelSpace = \Span\set{\ind{A_i} :
    i\in I}\). Let \(\mu\) be a probability measure on \(\domain\). Since
    the function \(\ind{A_i}\) are orthogonal with respect to \(\mu\) with
    \(\norm{\ind{A_i}}_\mu^2 = \mu(A_i)\), the orthonormal basis is given by
    \(\mu(A_i)^{-\frac12}\ind{A_i}\) and thus
    \[
        \norm{\lever_\modelSpace^\mu}_\infty^2
        = \sup_{x\in \domain}\sum_{i\in I} \tfrac{\ind{A_i}(x)}{\mu(A_i)}
        = \sup_{i\in I} \tfrac{1}{\mu(A_i)}.
    \]
    This is clearly minimized to \(m\) by the uniform partition \(\mu(A_i) = \frac1{m}\) for all \(i\in I\) with \(m= \dim \modelSpace = \#I\).

\item \textbf{Polynomials.} The most studied model space is the space of univariate polynomials \(\modelSpace_m \coloneq \poly_M\) up to degree
\(M=m-1\) with \(\dim \modelSpace_m = m\). While the degrees \(M\in\nat_0\) of a univariate monomial \(x^M\)
have a natural total order, the vector valued degrees \(\lambda = (\lambda_1,\dots, \lambda_\dims) \in \nat_0^\dims\)
of a multivariate monomial \(x_1^{\lambda_1}\cdots x_\dims^{\lambda_\dims}\)
do not. To define `up to degree \(M\)', one could choose the maximum degree \(\abs{\lambda}_\infty\) or the
`total degree' \(\abs{\lambda}_1\).  A well known generalization of `up to
degree \(M\)', which can capture both of these notions is a downward closed set
of degrees.
\begin{definition}[Downward closed]
    \(\Lambda \subset \nat_0^\dims\) is downward closed, if for \(\lambda\in \Lambda\) and \(\tilde\lambda\in\nat_0^\dims\)
    the assumption \(\tilde{\lambda}_i \le \lambda_i\) for all \(i\in \set{1,\dots, \dims}\) implies 
    \(\tilde\lambda\in \Lambda\).
\end{definition}
In the following table we always assume
\[
    \modelSpace_\Lambda = \Span\set[\Big]{ x_1^{\lambda_1}\cdots x_\dims^{\lambda_\dims} : \lambda \in \Lambda},
    \qquad\text{with some \(\Lambda\subset \nat_0^\dims\) downward closed}
\]
and we use the notation \(m \coloneq \dim \modelSpace_{\Lambda} = \# \Lambda\) for the dimension, \(\lever_m^* \coloneq \lever_{\modelSpace_\Lambda}^*\) and \(\reprK_m^*\coloneq \reprK_{\modelSpace_\Lambda}^*\).

\begin{center}\makebox[\textwidth][c]{
\begin{NiceTabular}{m{1cm} m{5.5cm} m{4.7cm} c}[
        notes/label-in-tabular=\scriptsize{#1},
        notes/style=\arabic{#1}
    ]
    \toprule
    \(\domain\) & Assumptions &  Bound & Ref.
    \\
    \midrule
    \([-1,1]\)
    &``Legendre'': 
    \(\mu\sim \unif{[-1,1]}\)
    & \(\norm{\lever_m^\mu}_\infty = m\) 
    & \tabularnote{\citealp[Sec.~3, p.~830-831]{cohenStabilityAccuracyLeast2013}, their \(K(m)\) is \(\norm{\lever_m^\mu}_\infty^2\)}
    \\[3ex]
    \([-1,1]\)
    &``Chebyshev'':
    \(\mu(dx) = \frac{dx}{\pi\sqrt{1-x^2}}(x)\)
    & \(\norm{\lever_m^\mu}_\infty = \sqrt{2m-1}\) 
    & \tabularnote{\citealp[Sec.~3, p.~831]{cohenStabilityAccuracyLeast2013}}
    \\[3ex]
    \([-1,1]\)
    &\(\mu(dx) = W(x) dx\),\par + technical assmpt.\ on \(W\)
    & \(\reprK_m^\mu(x,x) \sim \frac{m}{W(x)}\frac{1}{\pi \sqrt{1-x^2}}\)
    & \tabularnote{\citealp[Thm. 4.5.2, 4.5.5]{nevaiGezaFreudOrthogonal1986}}
\\[3ex]
    \([-1,1]^\dims\)
    & ``Legendre'': \(\mu \sim \unif{[-1,1]^\dims}\)
    & \(\norm{\lever_m^\mu}_\infty \le m\)
    & \tabularnote{\citealp[Lem.~3.1]{chkifaDiscreteLeastSquares2015}, their \(K_L(\Lambda)\) is \(\norm{\lever_m^\mu}_\infty^2\)}
    \\[3ex]
    \([-1,1]^\dims\)
    & ``Cheb.'': \(\mu(dx) = \prod_{i=1}^\dims \frac{dx_i}{\pi\sqrt{1-x_i^2}}\)
    & \(\norm{\lever_m^\mu}_\infty \le \min\set{m^{\frac{\ln 3}{2\ln 2}}, 2^{\frac\dims2} \sqrt{m}}\)
    & \tabularnote{\citealp[Lem.~3.3 + Eq.\ following (3.29)]{chkifaDiscreteLeastSquares2015}, their \(K_T(\Lambda)\) is \(\norm{\lever_m^\mu}_\infty^2\)}
    \\[3ex]
    \([-1,1]^\dims\)
    &``Jacobi'':
    \(\mu(dx) = \prod_{i=1}^\dims \rho(x_i)dx_i\),\par
    \(\rho(x) \propto (1-x)^{\alpha}(1+x)^{\beta}\), \(\alpha, \beta \in \nat_0\)
    & \(\norm{\lever_m^\mu}_\infty \le m^{\max\set{\alpha,\beta}+1}\)
    & \tabularnote{\citealp[Lem.~3.4]{chkifaDiscreteLeastSquares2015}, their \(K_J(\Lambda)\) is \(\norm{\lever_m^\mu}_\infty^2\)}
    \\[\defaultaddspace]
    \bottomrule
    \end{NiceTabular}
}
\end{center}
    \item\textbf{Trigonometric polynomials.}
    Let \(\mu= \unif{\torus}\) be the uniform measure on the circle \(\torus\)
    and \(\modelSpace\) the \emph{trigonometric polynomials} up to degree \(M\), then
    \[
        \norm{\lever_m^\mu}_\infty = \sqrt{m}
        \qquad\text{with}\qquad m=\dim \modelSpace = 2M+1.
    \]
    See e.g. \citet[Sec.~3, p.~830]{cohenStabilityAccuracyLeast2013}. This is also easily
    seen from the fact that the reproducing kernel is the Dirichlet kernel (e.g.\
    Theorem~\ref{thm: ONBs and reproducing kernels}), which is stationary
    and therefore has constant diagonal such that Lemma~\ref{lem: optimality of constant weighted leverage bound} applies.
    \item\textbf{Other.} \citet{chardonSamplingReconstructionSolutions2014} investigate
    \emph{Fourier-Bessel functions} and \emph{plane waves} on the disk.
\end{enumerate}

\subsection{Lebesgue constants}

The upper bound on the Lebesgue constant using the leverage function is at best of
order \(\sqrt{m}\) (Lemma~\ref{lem: optimality of constant weighted leverage bound}) 
and can thus be very loose if the Lebesgue constant itself is very small. Sometimes, it is therefore
beneficial to exert more effort to bound the Lebesgue constant directly.  As we
explained in Section~\ref{sec: strategies} already, the usual strategy is to
find an \(L^2(\mu)\) projection with good Lebesgue constant first and then try to
emulate its behavior with a regression operator---either with exact quadrature or
with randomized regression (cf.~Theorem~\ref{thm: lebesgue constant concentration}).
To explain why we focus on Fourier and Chebyshev projection, we begin with a
summary of known results about the Lebesgue constants of projections. Then,
we present novel bounds on the Lebesgue constants of Fourier and Chebyshev
regression with optimal logarithmic growth.  To contextualize our results, we finally summarize known results
about the Lebesgue constants of interpolation. The asymptotic behavior of the
Lebesgue constant of equispaced interpolation and those using the roots of the Jacobi polynomials
explain why we expect the leverage function bound to be good enough for many
other types of regression.

\subsubsection{Lebesgue constants of projections}

\begin{enumerate}[wide=0pt]
    \item \textbf{Trigonometric polynomials.}
    Let the Trigonometric projection \(\linProj_{\unif{\torus}}\) be the
    \(L^2\)-projection from \(C(\torus)\) into the trigonometric polynomials
    \(\modelSpace_M^\Trig\) up to degree \(M\) with respect to the uniform
    measure on the circle \(\torus\). Then \(\linProj_{\unif{\torus}}\) is the \emph{unique}
    minimizer of the Lebesgue constant among all linear projections
    \(\linProj\colon C(\torus)\to
    \modelSpace_M^\Trig\)
    \citep[Thm.~5.1]{cheneyMinimalPropertyFourier1969,devoreConstructiveApproximation1993}
    and with Fejèr's characterization of the Lebesgue constant of Fourier
    projection \citep[see e.g.][Thm.~2.1]{devoreConstructiveApproximation1993} this
    implies
    \[
        \lebconst{\infty}{\linProj_{\unif{\torus}}}
        = \inf_{\linProj} \lebconst{\infty}{\linProj}
        = \tfrac{4}{\pi^2} \ln M + \bigO(1)
    \]
    \item \textbf{Algebraic polynomials.}
    Via reduction to Fourier projection, one obtains a lower bound on the Lebesgue
    constant of projections \(\linProj\colon C[-1,1]\to \poly_M\) into the
    algebraic polynomials \(\poly_M\) of degree \(M\)
    \citep[Cor.~5.4]{devoreConstructiveApproximation1993} and this
    lower bound is very close to the Lebesgue constant of Chebyshev projection \(\linProj_{\mu_C}\)
    \citep[Thm.~15.3]{trefethenApproximationTheoryApproximation2019}
    \[
        \tfrac{4}{\pi^2} \ln(M+1) + 3
        \overset{\text{Trefethen}}\ge
        \lebconst{\infty}{\linProj_\mu}
        \ge \inf_{\linProj} \lebconst{\infty}{\linProj}
        \overset{\text{DeVore}}\ge
        \tfrac{2}{\pi^2} \ln M + \bigO(1).
    \]
    Chebyshev projection \(\linProj_{\mu_C}\) is the \(L^2(\mu_C)\) projection with respect
    to the Chebyshev measure
    \[
        \mu_C(dx) = \tfrac{dx}{\pi\sqrt{1-x^2}} = \tfrac1\pi\abs{\arccos'(x)}dx.
    \]
    The Chebyshev measure is the pushforward of the uniform measure on the half-circle
    \([0,\pi]\) under the map \(\cos\colon [0,\pi] \to [-1,1]\) and thereby related
    to Fourier projection.

    \item \textbf{Polynomials on the sphere.} Assume \(\domain = \sphere\) for \(\dims \in \nat\setminus\set{1}\)
    and assume that \(\modelSpace\) is a rotation invariant subspace of the \(\dims\)-variate polynomials \(\poly^\dims\), that is
    \(\modelSpace = \set{\model \circ U : \model \in \modelSpace}\) for any rotation \(U\). For example,
    the space \(\poly_M^\dims\) of polynomials of total degree at most \(M\) is rotation invariant.
    Then it is well known \citep[Thm.~6.4 (Berman, Daugavet)]{reimerMultivariatePolynomialApproximation2003}
    and \citep[(6.16)+(6.17)]{reimerMultivariatePolynomialApproximation2003}
    \vspace{-1.2em}
    \[
        \inf_{\linProj} \lebconst{\infty}{\linProj}
        = \lebconst{\infty}{\linProj_{\unif{\sphere}}}
        \overset{\substack{\text{assuming}\\\modelSpace=\poly_M^\dims}}\sim \begin{cases}
            \frac{2}{\pi^{3/2}} \frac{\Gamma(\frac{\dims-2}4) \Gamma(\frac{\dims}4)}{\Gamma(\frac{\dims-1}2)}
            M^{\frac{\dims-2}2} & \dims \ge 3
            \\
            \frac{4}{\pi^2} \ln M & \dims = 2,
        \end{cases}
    \]
    where the infimum is taken over all linear projections \(\linProj\colon
    C(\sphere)\to \modelSpace\). Recall the dimension of the model space behaves like \(m = \dim
    \poly_M^\dims \asymp M^{\dims-1}\) by \eqref{eq: dimension of polynomials on
    the sphere} such that for \(\dims \ge 3\) 
    \[
        \lebconst{\infty}{\linProj_{\unif{\sphere}}} \asymp M^{\frac{\dims-2}{2}} \asymp m^{\frac12(1-\frac{1}{\dims-1})}.
    \]
    Consequently, as \(\dims \to \infty\) the optimal leverage function bound \(\sqrt{m}\) becomes tight.
\end{enumerate}

\subsubsection{Lebesgue constants of regression}
\label{sec: lebesgue constants of regression}

\begin{enumerate}[wide=0pt]
    \item \textbf{Piecewise constant functions.} 
    Let \((\domain, \metric)\) be a metric space and 
    \((A_j)_{j\in I}\) a finite collection of disjoint subsets of \(\domain\)
    with \(\biguplus_{j\in I} A_j =\domain\) defining the
    space of piecewise constant functions \(\modelSpace = \Span\set{\ind{A_j} :
    j\in I}\). Let \(\mu_n=\frac1n\sum_{i=1}^n \omega_i \delta_{x_i}\)
    be the empirical measure of a WLS regression projection \(\linProj_{\mu_n}\) with weights \(\omega_i\ge 0\) and
    regression points \(x_i\in \domain\). Assume \(\mu_n(A_j) > 0\) for all
    \(j\in I\). Since \(\mu_n(A_j)^{-\frac12} \ind{A_j}\) is then an orthonormal
    basis of \(\modelSpace\) we have by Lemma~\ref{lem: reproducing kernel} and
    Remark~\ref{rem: positive reproducing kernels}
    \[
        \reprK_{\modelSpace}^{\mu_n}(x,y)
        = \sum_{j\in I} \frac{\ind{A_j}(x)\ind{A_j}(y)}{\mu_n(A_j)}
        \ge 0
        \qquad
        \overset{\text{Rem.~\ref{rem: positive reproducing kernels}}}\implies
        \qquad
        \lebfun{\infty}{\linProj_{\mu_n}}{x}=1.
    \]
    
    \item \textbf{Fourier and Chebyshev regression.}
    In the following definitions of function spaces we use \(x\) as a placeholder for the variable of the function,
    assume \(k\in \integer\) and define \(\torus \coloneq \real/(2\pi\integer)\) to be the circle:
    \begin{align*}
        \modelSpace_M^\F
        &\coloneq \Span_\complex\set[\big]{e^{ikx} : \abs{k} \le M}
        && x\in \torus
        && \dim_\complex \modelSpace_M^\F = 2M+1
        \\
        \modelSpace_M^\Trig
        &\coloneq \Span_\real\set[\big]{1, \cos(kx), \sin(kx) : 1 \le k \le M}
        && x\in \torus
        && \dim_\real \modelSpace_M^\Trig = 2M+1
        \\
        \modelSpace_M^{\cos}
        &\coloneq \Span_\real\set[\big]{\cos(kx) : 0 \le k \le M}
        && x\in [0,\pi]
        && \dim_\real \modelSpace_M^{\cos} = M+1
        \\
        \poly_M
        &\coloneq \Span_\real\set[\big]{x^k : 0 \le k \le M}
        && x\in [-1,1]
        && \dim_\real \poly_M = M+1.
    \end{align*}
    The trigonometric polynomials \(\modelSpace_M^\Trig\) are the restriction of Fourier polynomials \(\modelSpace_M^\F\) to real-valued functions
    \[
        \modelSpace_M^\Trig
        = \set[\big]{f\in \modelSpace_M^\F : f \text{ is real-valued}}.
    \]
    And with the extension \(f(-x) \coloneq f(x)\) for \(x\in [0,\pi]\) of arbitrary functions \(f\colon [0, \pi] \to \real\)
    on the interval/half-circle \([0,\pi]\) to \(2\pi\)-periodic functions
    \(f\colon \torus \to \real\) we also get with some abuse of notation
    \[
        \modelSpace_M^{\cos}
        = \set[\big]{f \in \modelSpace_M^\Trig: f \text{ is even}}.
    \]
    Finally, the surprising fact that the Chebyshev polynomials \(T_k(x) = \cos(k\cos^{-1}(x))\)
    are indeed polynomials implies that
    \[
        \poly_M = \set[\big]{f\circ \cos^{-1} : f\in \modelSpace_M^{\cos}}.
    \]
    In the previous section on Lebesgue constants of projections, we have seen that
    \(L^2\)-projection with respect to the uniform measure on the circle is optimal.
    The natural discretization of this is the family of discrete measures
    \[
        \mu_n^\alpha \coloneq \frac1n \sum_{k=0}^{n-1} \delta_{\ang_k^\alpha} \quad\text{with equispaced}\quad
        \ang_k^\alpha \coloneq \alpha + \tfrac{2\pi k}{n} \in \torus.
    \]
    \(L^2(\mu_n^\alpha)\)-projection to \(\modelSpace_M^\F\) is OLS regression on
    the equispaced points \(\ang_k^\alpha\). With a slight abuse of terminology, we will
    refer to this regression as the discrete Fourier transform (DFT(\(\complex\))).
    If we restrict ourselves to real-valued functions, we get
    \(L^2(\mu_n^\alpha)\)-projection to \(\modelSpace_M^\Trig\), which we refer to as DFT(\(\real\)).
    Recall that \(\modelSpace_M^{\cos}\) may be viewed as the restriction of the
    trigonometric polynomials to even functions on \([-\pi, \pi]\). For this we
    reflected arbitrary functions from \([0,\pi]\) around zero.  For regression this
    means we need to reflect the regression points around zero
    and obtain equispaced points on the circle. Not every shift \(\alpha\) yields
    a symmetric set of points after reflection, which is why the discrete cosine transforms
    (DCT) use more specific regression points. We will focus on the DCT-1 and DCT-2 nodes given by
    \[
        \ang_k^{(1)} \coloneq k\tfrac{\pi}{n-1} \in [0, \pi]
        \qquad\text{and}\qquad
        \ang_k^{(2)} \coloneq \tfrac{\pi}{2n} + k\tfrac{\pi}{n} \in [0, \pi]
        \qquad\text{for}\qquad k=0,\dots, n-1.
    \]
    Observe that the first and last DCT-1 nodes are reflected to themselves (cf.~Figure~\ref{fig: DCT node reflection}), which
    requires them to be weighted by \(\frac12\) to get an approximation of a uniform measure
    on the circle. The pushforward of the resulting measures on \([0,\pi]\) via \(\cos\) yields the
    Chebyshev measures on \([-1,1]\) used for algebraic polynomial regression. The cosine of the DCT-2 nodes are the roots of the Chebyshev polynomial
    \(T_n\), known as the Chebyshev nodes of the \ul{first} kind (confusingly!), and the cosine of the
    DCT-1 nodes are the extrema of \(T_{n-1}\), known as the Chebyshev nodes of the \ul{second} kind.

    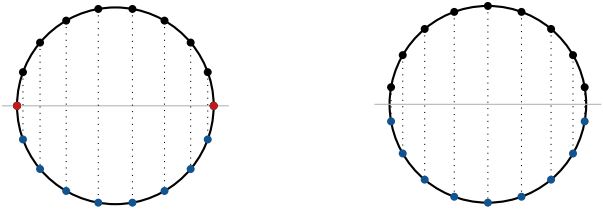
\begin{figure}
        \centering
        \begin{subfigure}[t]{0.33\textwidth}
            \centering
            
\begin{tikzpicture}[scale=1, every node/.style={font=\small}]
    \def\r{1.3}   \def\n{10}     

    \pgfmathtruncatemacro{\nmone}{\n-1}

\draw[thick] (0,0) circle (\r);

\draw[gray!60] (-1.5,0) -- (1.5,0);

    \foreach \k in {0,...,\nmone} {
        \pgfmathsetmacro{\ang}{180*\k/\nmone}

        \coordinate (U\k) at ({\r*cos(\ang)},{\r*sin(\ang)});
        \coordinate (L\k) at ({\r*cos(-\ang)},{\r*sin(-\ang)});

        \draw[dotted, thin] (U\k) -- (L\k); \fill[black] (U\k) circle (1.5pt); \fill[blue] (L\k) circle (1.5pt); }
    \fill[red] (L0) circle (1.5pt);
    \fill[red] (L\nmone) circle (1.5pt);

\end{tikzpicture}             \caption{\label{fig: DCT-1 nodes}The reflection of the {DCT-1} nodes. The red nodes are invariant.}
        \end{subfigure}
        \hspace{0.01\textwidth}
        \begin{subfigure}[t]{0.33\textwidth}
            \centering
            \begin{tikzpicture}[scale=1, every node/.style={font=\small}]
    \def\r{1.3}   \def\n{9}    

    \pgfmathtruncatemacro{\nmone}{\n-1}

\draw[thick] (0,0) circle (\r);

\draw[gray!60] (-1.5,0) -- (1.5,0);

    \foreach \k in {0,...,\nmone} {
        \pgfmathsetmacro{\ang}{180*(\k+0.5)/\n}

        \coordinate (U\k) at ({\r*cos(\ang)},{\r*sin(\ang)});
        \coordinate (L\k) at ({\r*cos(-\ang)},{\r*sin(-\ang)});

        \draw[dotted, thin] (U\k) -- (L\k); \fill[black] (U\k) circle (1.5pt);  \fill[blue]  (L\k) circle (1.5pt);  }
\end{tikzpicture}             \caption{\label{fig: DCT-2 nodes}The reflection of the {DCT-2} nodes.}
        \end{subfigure}
        \hspace{0.01\textwidth}
        \begin{minipage}[c]{0.27\textwidth}
        \caption{
            \label{fig: DCT node reflection}
            The reflection of the DCT-1 and DCT-2 nodes around zero modulo \(2\pi\)
            is an equidistant set of points on the circle.
        }
        \end{minipage}
        \vspace{-1ex}
    \end{figure}
    \begin{theorem}[Lebesgue constants of Fourier and Chebyshev regression]
        \label{thm: lebesgue constant bounds for fourier and chebyshev regression}
        In the following table we list Lebesgue constant bounds for Fourier and Chebyshev regression.
        The double prime \(\sum''\) indicates that the first and
        last term are weighted by \(\frac12\).  We always assume \(n\ge m\coloneq
        \dim\modelSpace\) and recall that sometimes \(m=2M+1\) and sometimes \(m=M+1\)
        depending on the model space.
    \begin{center}
    \begin{tabular}{l l l l}
        \toprule
        Name & \(\modelSpace\)  & WLS measure \(\mu_n\) & Bound on \(\lebconst{\infty}{\linProj_{\mu_n}}\)
        \\
        \midrule
        \(\mathrm{DFT}(\complex)\)
        & \(\modelSpace_M^\F\)
        & \(\frac1n\sum_{k=0}^{n-1} \delta_{\alpha + k\frac{2\pi}{n}}\)
        & \(\frac2\pi\ln(m) + \frac{m}n + 1.08\) 
        \\[1ex]
        \(\mathrm{DFT}(\real)\)
        & \(\modelSpace_M^\Trig\)
        & \(\frac1n\sum_{k=0}^{n-1} \delta_{\alpha + k\frac{2\pi}{n}}\)
        & \(\frac2\pi\ln(m) + \frac{m}n + 1.08\) 
        \\[1ex]
        \(\dct{1}\)
        & \(\modelSpace_M^{\cos}\)
        & \(\frac1{n-1} {\sum''}_{k=0}^{n-1} \delta_{k\frac{\pi}{n-1}}\)
        & \(\frac2\pi\ln(2m-1) + \frac{2m-1}{2n-2} + 1.08\) 
        \\[1ex]
        \(\dct{2}\)
        & \(\modelSpace_M^{\cos}\)
        & \(\frac1n \sum_{k=0}^{n-1} \delta_{\frac{\pi}{2n} + k\frac{\pi}{n}}\)
        & \(\frac2\pi\ln(2m-1) + \frac{2m-1}{2n} + 1.08\) 
        \\[1ex]
        \(T_{n-1}\) extr.
        & \(\poly_M\)
        & \(\frac1{n-1} {\sum''}_{k=0}^{n-1} \delta_{\cos(k\frac{\pi}{n-1})}\)
        & \(\frac2\pi\ln(2m-1) + \frac{2m-1}{2n-2} + 1.08\) 
        \\[1ex]
        \(T_n\) zeros
        & \(\poly_M\)
        & \(\frac1n \sum_{k=0}^{n-1} \delta_{\cos(\frac{\pi}{2n} + k\frac{\pi}{n})}\)
        & \(\frac2\pi\ln(2m-1) + \frac{2m-1}{2n} + 1.08\) 
        \\
        \bottomrule
    \end{tabular}
    \end{center}
    Thus in all cases \(\lebconst{\infty}{\linProj_{\mu_n}} \le \frac2\pi\ln(m) + \bigO(1)\).
    \end{theorem}

    Details about orthonormal bases, reproducing kernels and a sketch of the
    derivation of these bounds are given in Section~\ref{sec: fourier-chebyshev
    regression}. In the interpolation regime \(n=m\) of algebraic polynomials,
    \(\lebconst{\infty}{\linProj_{\mu_n}} \sim \frac{2}{\pi}\ln(m)\), by Example~\ref{ex: lebesgue constant list}
    \ref{it: any points}, so the bounds are asymptotically sharp in
    this regime. For \(n\gg m\) the bounds will likely approach the Lebesgue
    constant of the \(L^2\)-projection, which has slightly smaller constants
    in front of the logarithm.
\end{enumerate}

\subsubsection{Lebesgue constants of interpolation}

The Lebesgue constant is already a well-studied object in approximation
theory in the case of interpolation \(n=m\). To contextualize our results
for regression, we recapitulate well known bounds for interpolation in the
following.

\begin{example}[Lebesgue constants for polynomial interpolation]
    \label{ex: lebesgue constant list}
    Consider the space \(\funSpace= C([-1,1]; \real)\) equipped with the sup-norm \(\norm{\cdot}_\infty\) and let the
    model space be \(\modelSpace_m = \poly_{m-1}\), where \(\poly_{m-1}\) are the polynomials of degree at most \(m-1\).
    Let \(\linProj_m\) be the interpolation operator at \(m\) distinct points \(x_1,\dots, x_m \in [-1, 1]\).
    \begin{enumerate}[label=(\roman*)]
        \item\label{it: equispaced} 
        \textbf{Equispaced interpolation performs badly.}
        If the \(x_1, \dots, x_m\) are \textbf{equispaced}, then 
        for all \(m > 1\) 
        \[
            \lebconst{\infty}{\linProj_m} > \frac{2^{m-1}}{(m-1)^2}
            \quad \text{and} \quad
            \lebconst{\infty}{\linProj_m} \sim \frac{2^m}{e (m-1)\ln(m-1)},
        \]
        in other words the Lebesgue constant is extremely badly behaved and grows
        exponentially in \(m\). For equispaced points, Runge provided an explicit example for a function
        that is badly behaved as proved to exist in \eqref{eq: lebesgue lower bound},
        specifically the ``Runge function'' \(f(x) = \frac{1}{1+25x^2}\) \citep[cf.][p.~120]{trefethenApproximationTheoryApproximation2019}.

        \item\label{it: any points}
        \textbf{Lower bound on best interpolation.}
        For \textbf{any} points \(x_1, \dots, x_m\)
        \[
            \lebconst{\infty}{\linProj_m}
            \ge \tfrac{2}{\pi}\ln(m) + \underbrace{\tfrac{2}\pi(\gamma + \ln(\tfrac{4}\pi))}_{=0.52125...},
        \]
        where \(\gamma \approx 0.577\) is the Euler-Mascheroni constant.
        In other words the Lebesgue constant grows at least logarithmically in \(m\) for any choice of interpolation points.

        \item\label{it: chebyshev points 1st kind} Let \(x_1, \dots, x_m\)
        be the \textbf{Chebyshev points of the first kind} (the zeros of the
        \(m\)-th Chebyshev polynomial \(T_m\)), then the Lebesgue constant
        of the interpolation operator \(\linProj_m^{1\text{\tiny st}}\) satisfies 
        \[
            \lebconst{\infty}{\linProj_m^{1\text{\tiny st}}} = \tfrac{2}{\pi}\ln(m) + \tau_m
            \le \lebconst{\infty}{\linProj_{m+1}^{1\text{\tiny st}}}
            \qquad \forall m \ge 1,
        \]
        with \((\tau_m)_{m\in \nat}\) strictly monotonically \underline{decreasing}, \(0.96 \le \tau_m \le \tau_1=1\) and
        \[
            \lim_m \tau_m = \tfrac2\pi(\gamma + \ln(\tfrac8\pi)) = 0.9625....
        \]
        In other words, the Chebyshev points achieve the lower bound up to less than half a digit.

        \item\label{it: chebyshev points 2nd kind}
        Let \(x_1, \dots, x_m\) be the \textbf{Chebyshev points of the second kind} (the extrema of \(T_{m-1}\)), then
        the interpolation operator \(\linProj_m^{2\text{\tiny nd}}\) satisfies for \(m\ge 2\)
        \[
            \lebconst{\infty}{\linProj_m^{2\text{\tiny nd}}}
            = \tfrac2\pi\ln(m) + \tau_m'
            = \lebconst{\infty}{\linProj_{m-1}^{1\text{\tiny st}}} - \epsilon_m,
        \]
        where \(\tau_m' \in (0.52, 1)\), with the same limit \(\lim_m\tau_m'=0.9625...\) and
        \(\tau_{2k}'\) is strictly monotonically \underline{increasing}.
        Finally, \(\epsilon_m = 0\) for \(m\) even and \(\epsilon_m \in (0,
        \frac1{m^2})\) for \(m\) odd.
        
        For finite \(m\) the Chebyshev points of the second kind thus have slightly 
        smaller Lebesgue constants than the first kind. Asymptotically they are however equivalent.

        \item\label{it: jacobi polynomial roots} Let \(x_1, \dots, x_m\) be
        the zeros of the \textbf{Jacobi polynomial} \(p_m^{(\alpha, \beta)}\) for some
        \(\alpha, \beta > -1\), then the Lebesgue constant of the interpolation operator \(\linProj_m^{(\alpha, \beta)}\) satisfies
        \[
            \lebconst{\infty}{\linProj_m^{(\alpha, \beta)}}
            \asymp \begin{cases}
                m^{\Upsilon + \frac12} & \Upsilon > -\frac12
                \\
                \log(m) & \Upsilon \le -\frac12
            \end{cases}
            \qquad 
            \Upsilon \coloneq \max\set{\alpha, \beta}.
        \]
    \end{enumerate}
\end{example}
\begin{proof}
    The list is inspired by
    \citep[Thm.~15.2]{trefethenApproximationTheoryApproximation2019}.
    The bound in \ref{it: equispaced} was proved in \citep[Thm.~2]{trefethenTwoResultsPolynomial1991} and the
    asymptotic result was independently discovered by \citep{turetskiiBoundingPolynomialsPrescribed1940} and \citep[Satz 2]{schonhageFehlerfortpflanzungBeiInterpolation1961},
    \ref{it: any points} was proved in \citep[Eq.~(41)]{brutmanLebesgueFunctionPolynomial1978}.
    Most of \ref{it: chebyshev points 1st kind} is given by \citep[Theorem 1]{rivlinLebesgueConstantsPolynomial1974},
    the monotonic increase was proved in \citep[Satz 3, Eq.~(5.2)]{ehlichAuswertungNormenInterpolationsoperatoren1966}.
    \ref{it: chebyshev points 2nd kind} is assembled from \citep[Theorem 2]{rivlinLebesgueConstantsPolynomial1974} and
    \citep[Satz 4]{ehlichAuswertungNormenInterpolationsoperatoren1966} and the lower
    bound on \(\tau_m'\) follows from \ref{it: any points}. \citet[Eq.~(16)]{rivlinLebesgueConstantsPolynomial1974} observed that
    Szegő proved \ref{it: jacobi polynomial roots} in his proof of a different result
    \citep[Proof of Theorem~14.4, between (14.4.16) and (14.4.17)]{szegoOrthogonalPolynomials1939}.
\end{proof}

The takeaway from this list of examples is that the Lebesgue constant can be
extremely badly behaved with exponential growth in \(m\). On the other hand
there are also examples of interpolation operators where the Lebesgue constant is
almost constant in \(m\), with logarithmic growth.

In general, the zeros of orthogonal polynomials on bounded intervals are very well understood. The
Jacobi polynomials are a very general class of orthogonal polynomials that
include the
\begin{enumerate}[noitemsep]
    \item Chebyshev polynomials \(T_n\) with \((\alpha, \beta) = (-\frac12, -\frac12)\) 
    \item Legendre polynomials \(P_n\) with \((\alpha, \beta) = (0, 0)\)
    \item Gegenbauer (ultraspherical) polynomials \(C^\lambda_n\) with \((\alpha, \beta) = (\lambda - \frac12, \lambda - \frac12)\) with \(\lambda > -\tfrac12\).
\end{enumerate}

\subsection{Restricted range inequalities}
\label{sec: restricted range inequalities catalogue}

As discussed in Section~\ref{sec: restricted range inqualities}
restricted range inequalities can be used to turn an unbounded
regression problem into a bounded one. Such a restricted range inequality exists
for univariate polynomial spaces. The takeaway of the following well-known result
about them is that the restricted range \(S_\modelSpace\) which determines the
weighted uniform norm with weight \(\weight(x) = \exp(-\abs{x}^\alpha)\) is
given by
\[
    S_\modelSpace = [-C_\alpha m^{\frac1\alpha}, C_\alpha m^{\frac1\alpha}]\qquad m=\dim\modelSpace.
\]
That is, the faster the weight decays, the smaller is the relevant set \(S_\modelSpace\).
Except for the statement of optimality (which follows quickly from a 
result by \citet{szabadosWeightedLagrangeHermite1997}), the
following theorem is in the form of
\citep[Thm.~6.2]{lubinskySurveyWeightedPolynomial2007} and was proved
in \citep[Thm.~4.1]{levinOrthogonalPolynomialsExponential2001}. Similar
statements are in \citep[Thm.~6.1.6 and
6.2.4]{mhaskarIntroductionTheoryWeighted1996}.

\begin{theorem}[Restricted range inequalities for real polynomials]
    \label{thm: restricted range inequalities for polynomials}
    Assume \(\domain=\field=\real\).
    Let \(\weight(x) = \exp(-Q(x))\)
    with \(Q\colon \real \to [0,\infty)\) an even, convex function with
    the properties
    \[
        \lim_{\abs{x}\to\infty} Q(x) = \infty
        \qquad\text{and}\qquad 0=Q(0) < Q(x) < \infty \quad \forall \abs{x}>0.
    \]
    Let \(a_M\) be the Mhaskar-Rakhmanov-Saff (MRS) number uniquely defined by the equation\footnote{
        The derivative \(Q'\) only has to be defined almost everywhere, which it is
        by convexity of \(Q\).
    }
    \begin{equation}
        \label{eq: MRS number}     
        M = \frac2\pi \int_0^1 \frac{a_M t Q'(a_M t)}{\sqrt{1-t^2}} dt.
    \end{equation}
    Then for \(\poly_M\) the space of univariate (real) polynomials of degree at most \(M\) we have:
   \begin{enumerate}[label=(\roman*)]
       \item 
        \textbf{The weighted sup-norm} satisfies an exact restricted range inequality, specifically
        \begin{align}
            \norm{\model}_\weight
            &= \norm{\model \ind{[-a_M,a_M]}}_\weight
            &\forall \model \in \poly_M
            \\
            \norm{\model \ind{\real \setminus [-a_M,a_M]}}_\weight
            &< \norm{\model \ind{[-a_M,a_M]}}_\weight
            &\forall \model \in \poly_M : \model \neq 0
        \end{align}
        \ul{Approximate optimality of \(a_M\):} If \(Q\) is twice continuously differentiable
        on \((0,\infty)\) and there exist constants \(A, B>1\)  such that
        \(A\le \frac{(xQ'(x))'}{Q'(x)} \le B\) for all \(x>0\), then 
        there exists a constant \(c_1>0\) independent of \(M\in \nat\) such that
        \[
            a_M\bigl(1-c_1M^{-\frac23}\bigr)
            \le \inf\set[\Big]{a>0 : \norm{\model}_\weight = \norm{\model \ind{[-a,a]}}_\weight\ \forall \model \in \poly_M}
            \le a_M.
        \]

        \item \textbf{The \(L^p(\mu)\)-norm}
        for \(p\in [1,\infty)\) and \(\mu(dx) \coloneq \weight^p(x) dx\)
        satisfies a restricted range inequality, specifically
        \begin{align}
            \norm{\model}_{L^p(\mu)}
            &< 2^{1/p} \norm{\model \ind{[-a_M,a_M]}}_{L^p(\mu)}
            &\forall \model \in \poly_{\floor{M-\frac2p}}: \model \neq 0.
            \\
            \norm{\model \ind{\real \setminus [-a_M,a_M]}}_{L^p(\mu)}
            &< \norm{\model \ind{[-a_M,a_M]}}_{L^p(\mu)}
            &\forall \model \in \poly_{\floor{M-\frac2p}} : \model \neq 0
        \end{align}
        These equations also hold for \(p\in(0,1)\), where the
        \(L^p(\mu)\)-norm is a quasi-norm.
   \end{enumerate} 
\end{theorem}

\begin{remark}[Weighted \(L^p\)-norms]
    \(\norm{f}_{L^p(\mu)} = \norm{f\weight}_{L^p(dx)}\) for \(\mu(dx) = \weight^p(x) dx\).
\end{remark}
\begin{remark}[Mhaskar-Rakhmanov-Saff number]
    \label{rem: MRS number}
    For \(Q(x) = \abs{x}^\alpha\) with \(\alpha\in [1,\infty)\) such that
    \(\weight(x) = \exp(-\abs{x}^\alpha)\), the Mhaskar-Rakhmanov-Saff (MRS)
    number \(a_M\) is given by
    \begin{equation} 
        a_M = C_\alpha M^{\frac1\alpha} \qquad\text{with}\qquad C_\alpha \coloneq \Bigl(2^{\alpha-2}\frac{\Gamma(\frac{\alpha}2)^2}{\Gamma(\alpha)}\Bigr)^{\frac1\alpha}.
    \end{equation}
    In particular for \(\alpha=2\) we have \(a_M = \sqrt{M}\).
    A simpler alternative to the almost optimal MRS number is the Freud
    number \(q_M\), given by the solution of the equation \(q_M Q'(q_M) = M\).
    It bounds the MRS number as follows \citep[Equation
    (6.1.21)]{mhaskarIntroductionTheoryWeighted1996}
    \[
        q_M \le a_M \le \sqrt{2}q_{2M}.
    \]
\end{remark}
 
	\subsection{Examples}
\label{sec: examples}

\subsubsection{Piecewise constant functions}

Let \(\domain=[0,1]^\dims\) and let \((A_i)_{i\in I}\) be a tiling of \(\domain\) with
side length \(\frac1M\) such that the diameter of each tile in the sup-norm is
\(\diam_\infty(A_i) = \frac1M\) . Let
\[
    \modelSpace_M=\Span_\real\set{\ind{A_i} : i\in I}
    \qquad \text{with}
    \qquad m = \dim\modelSpace_M=\#I = M^\dims
\]
be the space of piecewise constant functions on this tiling.

\begin{example}[Deterministic points]
    Let \(x_1, \dots, x_n\) be selected such that each tile \(A_i\) contains 
    an equal number, \(\frac{n}{m}\in \nat\), of evaluation points and let \(\regOp\) be the 
    OLS regression operator. Then with probability at least \(1-\delta\)
    \[
        \norm{\regOp(y) - f}_\infty
        \le \green{2\error_{\modelSpace_M}^\infty[f]} + \blue{\varepsilon(\delta)\sqrt{\opnorm{\Sigma}_2\tfrac{m}{n}}},
    \]
    with \(\error_{\modelSpace_M}^\infty[f] \le \frac12\modCont(f; \frac1M)\) and \(\varepsilon(\delta)\) as listed in the Table of Corollary~\ref{cor: uniform noise bound}
    depending on the assumptions on the noise.
\end{example}
\begin{proof}
    The approximation error bound follows from \(\lebfun{\infty}{\linProj}{x} =1\) by Section~\ref{sec: lebesgue constants of regression}
    and the stochastic error bound follows from Corollary~\ref{cor: uniform noise bound}
    with \(\inf_{\linProj_\star} \lebconst{\infty}{\linProj_\star} = 1\), \(\Omega_n=\identity\), \(\Gram_{\Omega_n} = \identity\),
    and \(\lever_{\modelSpace_M}^{\Omega_n}(x) = \sqrt{m}\). Indeed, the functions \(\sqrt{m}\ind{A_i}\) are orthonormal
    with respect to \(\scp{\cdot}{\cdot}_{\Omega_n}\) since each \(A_i\) contains \(\frac{n}{m}\)
    evaluation points and therefore
    \[
        \reprK_{\modelSpace_M}^{\Omega_n}(x,x) = \sum_{i\in I} m \ind{A_i}(x)\ind{A_i}(x)
        = m.
    \]
    A different way to see this is to consider \(\reprK_{\modelSpace_M}^{\mu}\)
    with \(\mu=\unif{[0,1]^\dims}\) and observe that we achieve exact quadrature
    with our evaluation point selection and unit weights to get \(\Gram_\mu = \identity\).
\end{proof}

\begin{remark}[Discussion of rates]
    If \(f\) is \(\alpha\)-Hölder continuous, that is \(\modCont(f; h) \in
    \bigO(h^\alpha)\), and the noise has \(p\)-th moments such that
    \(\varepsilon(\delta) \in \bigO((\frac{m}\delta)^{\frac1p})\), then for \(n\ge m\) we
    have with probability at least \(1-\delta\)
    \[
        \norm{\regOp(y) - f}_\infty
        \le \bigO(m^{-\frac{\alpha}{\dims}})
        + \bigO\Bigl((\tfrac{m}{\delta})^{\frac1p}\sqrt{\tfrac{m}n}\Bigr)
    \]
    Balancing the rates of the approximation and stochastic error yields \(n \asymp m^{1+\frac{2}{p} + \frac{2\alpha}{\dims}}\),
    which gives the overall rate
    \[
        \norm{\regOp(y) - f}_\infty \in \bigO((\tfrac1\delta)^{\frac1p} n^{-\gamma})
        \qquad\text{with}\qquad \gamma = \frac1{2+\frac{\dims}\alpha(1+\frac{2}p)}.
    \]
    It is not hard to see that sub-Gaussian, sub-exponential and sub-Weibull noise
    results in the same rate \(\tilde\bigO(n^{-\gamma})\) with \(p=\infty\)
    up to logarithmic factors hidden by \(\tilde\bigO\).  If all moments of the
    noise are available (\(p=\infty\)) and \(f\) is Lipschitz continuous
    (\(\alpha=1\)), we therefore have the rate
    \(\tilde\bigO(n^{-\frac1{2+\dims}})\).
\end{remark}

\subsubsection{Trigonometric and Chebyshev polynomial regression}
\label{sec: examples, trigonometric and chebyshev regression}

\begin{example}[Trigonometric polynomials]
    \label{ex: trigonometric polynomials}
    Let \(\modelSpace_M^\Trig\) be the space of trigonometric polynomials of
    degree \(M\) (cf.~\ref{sec: lebesgue constants of regression}) and let \(\regOp\)
    be the OLS regression operator with equispaced evaluation points
    \(X=(x_1,\dots, x_n)\) on the circle \(\torus = \real / 2\pi \integer\)
    \[
        x_k = \alpha + \tfrac{2\pi}{n} k, \quad k=0,\dots,n-1.
    \]
    Then for \(n \ge m \coloneq \dim \modelSpace_M^\Trig = 2M+1\) we have with probability at least \(1-\delta\)
    \[
        \norm{\regOp(y) - f}_\infty
        \le \green{\underbrace{\Bigl(\tfrac{2}{\pi}\ln(m)+ \tfrac{m}n + 2.08\Bigr) \error_{\modelSpace_M^\Trig}^\infty[f]}_{\text{approx. error}}}
        + \blue{\underbrace{\varepsilon(\delta)(\tfrac2\pi\ln(m) +2.08)\sqrt{\opnorm{\Sigma}_2 \tfrac{m}{n}}}_{\text{stochastic error}}}.
    \]
    where \(\error_{\modelSpace_M^\Trig}^\infty[f] \le \frac\pi{2}(M+1)^{-k}
    \modCont(f^{(k)}; \frac\pi{M+1})\) for \(f\in C^k\) and \(f^{(k)}\) the
    \(k\)-th derivative of \(f\) and \(\varepsilon(\delta)\) is given by:
    \begin{center}
    \makebox[\textwidth][c]{
    \begin{NiceTabular}{l m{6cm} m{5cm}}[tabularnote={
        \footnotesize 
    \(C_\alpha, C_{\alpha}'>0\) are unoptimized constants given in \eqref{eq: sub-Weibull constants}, \(C_{p,\real}\coloneq \frac{14.7p}{\ln p}\) and
    \(C_{p,\complex}=\sqrt2\,C_{p,\real}\).
    }]
        \toprule
        & Assumption (decreasing strength \(\downarrow\)) &  \(\varepsilon(\delta)\)
        \\
        \midrule
        \newtag{\normalfont{(i)}}{it: (uniform concentration) sub-Gaussian, trig example}\tabularnote{\footnotesize Assume \(\field=\real\).\label{footnote: K=R, trig example}}\!\!\!
        & \(\noiseIndep_i \sim \subG{\sigma^2}\)
        & \(
            \sigma\sqrt{2\ln(\frac{2m}\delta)}
        \)
        \\[1ex]
        \newtag{\normalfont{(ii)}}{it: (uniform concentration) exponential, trig example}\ref*{footnote: K=R, trig example}\!\!\!
        & \(\noiseIndep_i \sim \subExp{\nu, \alpha}\)
        & \(
            \max\set[\Big]{\sqrt{2\nu^2\ln(\frac{2m}\delta)}, 2\alpha\ln(\frac{2m}\delta)}
        \)
        \\[1ex]
        \newtag{\normalfont{(iii)}}{it: (uniform concentration) sub-Weibull, trig example}\ref*{footnote: K=R, trig example}\!\!\!
        & \(\noiseIndep_i \sim \subWeibull{\alpha}\), \(\norm{\noiseIndep_i}_{\psi_\alpha}\le K\), \(\alpha\le 2\)
        & \(
            K\Bigl(C_\alpha \sqrt{\ln(\tfrac{2m}{\delta})} + C_{\alpha}' \ln(\tfrac{2m}{\delta})^{\frac1\alpha}\Bigr)
        \)
        \\[1ex]
        \newtag{\normalfont{(iv)}}{it: (uniform concentration) finite p-th moment, trig example}
        & \(\E{\abs{\noiseIndep_i}^2}\le \sigma^2\) and \(\E{\abs{\noiseIndep_i}^p} \le \kappa^p\), \(p> 2\)
        & \(
            C_{p, \field}\max\set{\sigma, \kappa}\bigl(\frac{m}{\delta}\bigr)^{\frac1p}
        \)
        \\[1ex]
        \newtag{\normalfont{(v)}}{it: (uniform concentration) finite second moment, trig example}
        & \(\E{\abs{\noiseIndep_i}^2} \le \sigma^2\)
        & \(
            \sigma\sqrt{\frac{m}{\delta}}
        \)
        \\
        \bottomrule
    \end{NiceTabular}
    }
    \end{center}
\end{example}

\begin{proof}
    The approximation error bound follows from the bound on the Lebesgue
    constant in Section~\ref{sec: lebesgue constants of regression} and
    the best approximation error is bounded in Section~\ref{sec: metric projection error}.
    For the bound on the stochastic error we use Corollary~\ref{cor: uniform noise bound}
    to obtain
    \[
        \norm{\regOp(\noise)}_\weight
        \le \varepsilon(\delta)\inf_{\linProj_\star} \lebconst{\infty}{\linProj_\star}\sqrt{\opnorm{\Gram_{\Omega_n}^{-1}}_2 \opnorm{\Omega_n^{\frac12}\Sigma \Omega_n^{\frac12}}_2} \frac{\norm{\lever_\modelSpace^{\Omega_n}}_\infty}{\sqrt{n}},
    \]
    where \(\Gram_{\Omega_n} = \identity\) by definition
    and \(\inf_{\linProj_\star} \lebconst{\infty}{\linProj_\star} \le \frac2\pi\ln(m) + 2.08\)
    by our Lebesgue constant bounds in Section~\ref{sec: lebesgue constants of regression} again,
    and finally \(\norm{\lever_\modelSpace^{\Omega_n}}_\infty = \sqrt{m}\) since the
    reproducing kernel is stationary (the Dirichlet kernel, Thm.~\ref{thm: ONBs and reproducing kernels}) combined
    with Lemma~\ref{lem: optimality of constant weighted leverage bound}.
\end{proof}

\begin{remark}[Discussion of rates]
    If \(f^{(k)}\) is \(\alpha\)-Hölder continuous, that is \(\modCont(f^{(k)}; h) \in \bigO(h^\alpha)\),
    and the noise has \(p\)-th moments, then for \(n\ge m\) we have with
    probability at least \(1-\delta\)
    \[
        \norm{\regOp(y) - f}_\infty
        \in \bigO(\tfrac{\log(m)}{m^{k+\alpha}})
        + \bigO\Bigl(\log(m)(\tfrac{m}{\delta})^{\frac1p}\sqrt{\tfrac{m}n}\Bigr)
    \]
    Balancing the rates of the approximation and stochastic error yields \(n \asymp m^{1+ \frac2p + 2(k+\alpha)}\),
    which gives the overall rate
    \[
        \norm{\regOp(y) - f}_\infty
        \in \bigO\Bigl(n^{-\gamma} \log(n) (\tfrac1{\delta})^{\frac1p}\Bigr)
        = \tilde \bigO\bigl(n^{-\gamma}\bigr)
        \qquad\text{with}\qquad \gamma = \frac{1}{2+\frac{1+2/p}{k+\alpha}}.
    \]
    Here \(\tilde\bigO\) hides logarithmic factors and keeps \(\delta\) constant.
    It is not hard to see that sub-Gaussian, sub-exponential and sub-Weibull noise
    results in the same rate \(\tilde\bigO(n^{-\gamma})\) with \(p=\infty\).
    For \(p=2\) and Lipschitz continuous functions (\(k=0, \alpha=1\)) we thus have the
    rate \(\tilde \bigO(n^{-1/4})\) and as the smoothness of the function and
    the number of finite moments of the noise \(p\) increases, the rate
    approaches \(\tilde \bigO(n^{-1/2})\).

    We finally want to highlight, that most choices of \(p\), \(k\) and \(\alpha\)
    lead to \(n\) scaling much faster than \(m^2\), which is the regime
    where the Lebesgue constant of randomized regression concentrates around
    the Lebesgue constant of \(L^2\)-projection (Thm.~\ref{thm: lebesgue constant concentration}).
\end{remark}

This rate discussion also applies to the following example of Chebyshev regression.

\begin{example}[Chebyshev regression]
    \label{ex: chebyshev regression}
    Let \(\poly_M\) be the space of univariate algebraic polynomials of
    degree \(M\) and assume \(n \ge m \coloneq \dim \poly_M = M+1\) 
    and let \(\varepsilon(\delta)\) be the same as in the previous example.
    \begin{enumerate}
        \item\textbf{Chebyshev nodes of the first kind.}
        Let \(\regOp\) be the OLS regression on the zeros of \(T_n\) given by
        \[
            x_k^{(z)} = \cos\bigl(\tfrac{\pi}{2n}+k\tfrac{\pi}{n}\bigr)
        \]
        Then with probability at least \(1-\delta\)
        \[
            \norm{\regOp(y) - f}_\infty
            \le \green{\underbrace{\Bigl(\tfrac{2}{\pi}\ln(2m-1)+ 3.08\Bigr) \error_{\poly_M}^\infty[f]}_{\text{approx. error}}}
            + \blue{\underbrace{\varepsilon(\delta)(\tfrac2\pi\ln(m) + 1)\sqrt{\opnorm{\Sigma}_2 \tfrac{2m-1}{n}}}_{\text{stochastic error}}}.
        \]
        \item\textbf{Chebyshev nodes of the second kind.}
        Assume \(n\ge 2\) and let \(\regOp\) be the WLS regression on the
        extrema of \(T_{n-1}\), the Chebyshev nodes
        of the second kind, given by
        \[
            x_k^{(e)} = \cos\bigl(k\tfrac{\pi}{n-1}\bigr), \qquad 
            \omega_k = \begin{cases}
                \frac{n}{n-1} & k\in \set{1,\dots, n-2},\\
                \frac{n}{2(n-1)} & k\in \set{0, n-1}.
            \end{cases}
        \]
        such that \(\frac1n\sum_{k=0}^{n-1} \omega_k \delta_{x_k} = \frac1{n-1}{\sum''}^{n-1}_{k=0} \delta_{x_k}\)
        and \(\frac1n\sum_{k=0}^{n-1} \omega_k = 1\).
        Then with probability at least \(1-\delta\)
        \[
            \norm{\regOp(y) - f}_\infty
            \le \green{\underbrace{\Bigl(\tfrac{2}{\pi}\ln(2m-1)+ 4\Bigr) \error_{\poly_M}^\infty[f]}_{\text{approx. error}}}
            + \blue{\underbrace{\sqrt{2}\varepsilon(\delta)(\tfrac2\pi\ln(m) + 1)\sqrt{\opnorm{\Sigma}_2 \tfrac{2m-1}{n}}}_{\text{stochastic error}}}.
        \]
    \end{enumerate}
\end{example}
\begin{proof}
    The proof is analogous to the previous example, we only need to replace the Lebesgue constant
    and the Leverage function bounds (Sec.~\ref{sec: lebesgue constants of
    regression}, \ref{sec: christoffel function and reproducing kernel bounds}).
    We also use \(\frac{2m-1}{2n} \le 1\) and \(\frac{2m-1}{2n-2}\le 1.5\) for simplicity.
    In the WLS case, we also do not have \(\Omega_n=\identity\) which necessitates the bound
    \[
        \opnorm{\Omega_n^{\frac12}\Sigma \Omega_n^{\frac12}}_2 \le \opnorm{\Omega_n}_2\opnorm{\Sigma}_2
        \le \tfrac{n}{n-1}\opnorm{\Sigma}_2 \le 2 \opnorm{\Sigma}_2
        \qquad n\ge 2.
        \qedhere
    \]
\end{proof}

	\subsection*{Acknowledgements}
	ChatGPT has been an invaluable
	tool in the writing of this paper. Large language models are extremely capable at
	interpolating between two existing fields of mathematics such as
	approximation theory and statistical regression. ChatGPT consequently
	provided many of the ideas for generalizations and sources for related work
	in the literature. Most of the human effort went into generating
	conjectures, and polishing the proofs into human readable form.
	During this work we got the impression that language models are not yet able
	to strike a balance between simple yet sharp bounds or identify what may
	be a useful abstraction that can be reused in multiple proofs such as
	the leverage function. But given a conjecture, they are already extremely capable
	at finding a proof or counterexample.

	\bibliography{references}

	\clearpage \appendix
	\section{Approximation theory}
\label{sec: approximation theory}

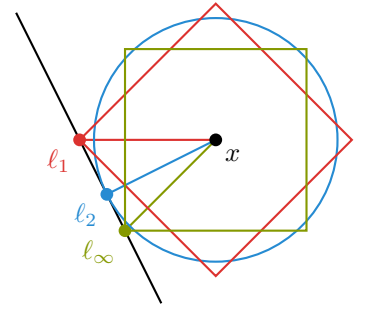
\begin{wrapfigure}{r}{0.2\textwidth}
    \center
    \vspace{-1cm}
    \begin{tikzpicture}[scale=1.2]

\definecolor{lone}{RGB}{220,50,47}
\definecolor{ltwo}{RGB}{38,139,210}
\definecolor{linf}{RGB}{133,153,0}

\draw[thick, domain=-2.2:-0.6, smooth, variable=\x] plot ({\x}, {-2*\x-3});

\coordinate (x) at (0,0);

\coordinate (p2) at (-1.2, -0.6);
\fill[ltwo] (p2) circle (2pt) node[below left] {$\ell_{2}$};
\draw[ltwo, thick] (x) -- (p2);

\coordinate (p1) at (-1.5,0);
\fill[lone] (p1) circle (2pt) node[below left] {$\ell_{1}$};
\draw[lone, thick] (x) -- (p1);

\coordinate (pinf) at (-1,-1);
\fill[linf] (pinf) circle (2pt) node[below left] {$\ell_{\infty}$};
\draw[linf, thick] (x) -- (pinf);

\draw[ltwo, thick] (x) circle (1.3416);

\fill (x) circle (2pt) node[below right] {$x$};

\draw[lone, thick]
    ($(x)+(1.5,0)$) --
    ($(x)+(0,1.5)$) --
    ($(x)+(-1.5,0)$) --
    ($(x)+(0,-1.5)$) -- cycle;

\draw[linf, thick]
    ($(x)+(1,1)$) --
    ($(x)+(-1,1)$) --
    ($(x)+(-1,-1)$) --
    ($(x)+(1,-1)$) -- cycle;

\end{tikzpicture}     \caption{Metric projections onto a 1D subspace}
    \label{fig: metric projections}
\end{wrapfigure}

Much of approximation theory is concerned with quantifying the error made using
a much more tractable \emph{linear} projection (e.g.\ interpolation or regression)
instead of the \emph{metric} projection that minimizes the error in a given
norm. It may therefore be useful to recall the definition of these different
types of projections.

\begin{definition}[Projection types]
    \label{def: projection types}
    Let \(V\) be a vector space
    \begin{enumerate}[label=(\roman*),noitemsep]
        \item \(\linProj \colon V \to V\) is a \emph{linear projection} if \(\linProj\) is a linear map and \(\linProj^2 = \linProj\).
        \item Assume \(V\) is equipped with a semi-norm \(\norm{\cdot}\) and let
        \(U\) be a linear subspace of \(V\). Then the \emph{metric projection}
        of \(v\) onto \(U\) is defined as
        \[
            \metricProj^*_U[v] \coloneq \argmin_{u\in U} \norm{v-u}_*.
        \]
        \item Assume \(V\) is equipped with an inner product and \(U\) is a linear subspace of \(V\).
        Then \(u\) is an \emph{orthogonal projection} of \(v\) onto \(U\) if and only if \(v-u \in U^\perp\). 
    \end{enumerate} 
\end{definition}

\begin{remark}[Projections coincide in Hilbert spaces]
    \label{rem: projections coincide in hilbert spaces}
    Let \(\hilbert\) be a Hilbert space. Then the orthogonal projection of \(v\in \hilbert\)
    onto a closed linear subspace \(U\) of \(\hilbert\) exists and is unique for every \(v\in
    \hilbert\) and the map \(\linProj\colon \hilbert \to U\) that maps \(v\) to
    its orthogonal projection is a linear projection. Moreover the metric
    projection of the norm induced by the inner product is always a singleton
    and coincides with the orthogonal projection.
    See e.g. \citet[Section.~3.3, Theorem 1 \& 2]{luenbergerOptimizationVectorSpace1998}.
\end{remark}

The metric projection minimizes the error in the semi-norm \(\norm{\cdot}_*\)
resulting in the metric projection error as defined in \eqref{eq: metric
projection error}
\[
    \error_\modelSpace^*[f]
    = \inf_{\model\in \modelSpace} \norm{f-\model}_*
    = \norm{f - \metricProj_\modelSpace^*[f]}_*.
\]
Note that \(\metricProj_\modelSpace^*[f]\) is never empty for finite dimensional
model spaces \(\modelSpace\) \citep[Thm.~1.1]{devoreConstructiveApproximation1993}.
If the norm is induced by an inner product, the
metric projection happens to be a linear projection (Remark \ref{rem: projections
coincide in hilbert spaces}). But for most norms, in particular the sup-norm,
this is generally not the case (Figure \ref{fig: metric projections}).
It turns out that the operator norm
of the linear projection \(\linProj \colon \funSpace\to \modelSpace\), is the
quantity that controls exactly how much worse the approximation \(\linProj f\)
is compared to the metric projection. In fact, the operator norm defines
the generalization of Lebesgue constant and function to arbitrary linear projections.
\begin{definition}[Generalized Lebesgue function and constant]
    \label{def: generalized lebesgue function and constant}
    The \emph{generalized Lebesgue function} is at any \(x\in \domain\) given by the operator norm of \(\eval_x\linProj\)
    from the space \((\funSpace, \norm{\cdot}_\weight)\) into \((\field, \abs{\cdot}\weight(x))\),
    that is
    \begin{equation*}
        \tag{generalized Lebesgue function}
        \opnorm{\eval_x \linProj}_{\funSpace,\weight}^{\weight(x)}
        \coloneq \sup_{\substack{f\in \funSpace\\ \norm{f}_\weight \le 1}} \abs{\linProj f(x)} \weight(x)
    \end{equation*}
    And we write \(\opnorm{\eval_x \linProj}_{\funSpace,\weight}\) for \(\weight(x)=1\) and observe that \(\opnorm{\eval_x \linProj}_{\funSpace,\weight}^{\weight(x)}= \opnorm{\eval_x \linProj}_{\funSpace,\weight}\weight(x)\).
    The \emph{generalized Lebesgue constant} is given by the operator norm of \(\linProj\colon \funSpace \to \modelSpace\)
    in the space \((\funSpace, \norm{\cdot}_\weight)\), that is
    \begin{equation*}
        \tag{generalized Lebesgue constant}
        \opnorm{\linProj}_{\funSpace,\weight}
        \coloneq \sup_{\substack{f\in \funSpace\\ \norm{f}_\weight \le 1}} \norm{\linProj f}_\weight
    \end{equation*}
\end{definition}
While we usually omit the dependence of the operator norms on the function space \(\funSpace\)
in the notation, we highlight it here, since the Lebesgue function and constant
only coincide with their generalized versions if the function space \(\funSpace\) is
sufficiently rich (cf.\ Def.~\ref{def: sufficiently rich function space},
Lem.~\ref{lem: lebesgue function is operator norm}, Cor.~\ref{cor: lebesgue
constant is operator norm} and Thm.~\ref{thm: L2 proj Lebesgue function and constant are operator norms}).
Assuming the Lebesgue constant coincides with its generalized version, the
following Lemma  proves the upper bound of Theorem \ref{thm: optimal bias bounds
in the weighted sup-norm} \ref{it: (weighted) uniform bound}.

\begin{mdframed}[style=default]
\begin{lemma}[Lebesgue {\citep[e.g.][Prop.~4.1]{devoreConstructiveApproximation1993}}]
    \label{lem: lebesgue}
    Let \(\linProj\colon \funSpace \to \modelSpace\), be a linear projection
    \textbf{onto} the linear subspace \(\modelSpace\) of the vector space \(\funSpace\) equipped with the semi-norm \(\norm{\cdot}_*\). Then we have
    \begin{equation}
        \label{eq: lebesgue upper bound}
        \norm{\linProj f - f}_*
        \le (1 + \opnorm{\linProj}_*) \error_\modelSpace^*[f],
    \end{equation}
    Moreover, for any \(\epsilon>0\) there exists \(f \in \funSpace\)
    such that
    \begin{equation}
        \label{eq: lebesgue lower bound}        
        \norm{\linProj f - f}_* \ge (\opnorm{\linProj}_* -1 - \epsilon) \error_\modelSpace^*[f].
    \end{equation}
\end{lemma}
\end{mdframed}

Observe that the Lebesgue lemma allows for the generalization of the
Lebesgue constant to semi-norms other than the weighted sup-norm, that is
\(\opnorm{\linProj}_* = \sup_{\norm{f}_* = 1} \norm{\linProj f}_*\).
However, the resulting bounds are not useful for \(L^p(\mu)\)-norms as we discuss
in Section \ref{sec: lebesgue constant for Lp norms}.

\begin{remark}[Gap between upper and lower bound]
    The gap between the upper bound \eqref{eq: lebesgue upper bound} and \eqref{eq: lebesgue lower bound}
    may be closed in the case of regression projections and \(L^2(\mu)\)-projections. Specifically, the lower
    bound may be raised to the upper bound (Cor.~\ref{cor: lebesgue
    constant is operator norm} and Thm.~\ref{thm: L2 proj Lebesgue function and constant are operator norms}).
\end{remark}

While this Lemma is very powerful, its proof is surprisingly simple and instructive and
we make an exception to our general rule of postponing proofs.

\begin{mdframed}[topline=false, bottomline=false, innertopmargin=0pt]
\begin{proof}
    Using \(\model^* \in \metricProj^*_\modelSpace[f]\) we have \(\linProj \model^* = \model^*\)
    as \(\linProj\) is a linear projection onto \(\modelSpace\) and therefore
    acts like the identity on \(\modelSpace\). This implies
    \[
        \norm{f - \linProj f}_* 
        \le \norm{f - \model^*}_* + \norm{\linProj(\model^* - f)}_*
        \le (1+\opnorm{\linProj}_*) \norm{f - \model^*}_*.
    \]
For a fixed \(\epsilon > 0\) pick \(f_\epsilon\in\funSpace\) such that \(\frac{\norm{\linProj f_\epsilon}_*}{\norm{f_\epsilon}_*} \ge \opnorm{\linProj}_* - \epsilon\).
    Then
    \begin{align*}
        \norm{f_\epsilon - \linProj f_\epsilon}_*
        &\ge \norm{\linProj f_\epsilon}_* - \norm{f_\epsilon}_*
        \\
        &\ge (\opnorm{\linProj}_* -\epsilon - 1)\norm{f_\epsilon}_*
        \ge (\opnorm{\linProj}_* -\epsilon - 1)\norm{f_\epsilon - \model^*}_*.
    \end{align*}
    The last inequality uses \(\norm{f_\epsilon - \model^*}_* \le \norm{f_\epsilon - 0}_*\)
    since \(0\in\modelSpace\) as a linear subspace and
    assumes \(\opnorm{\linProj}_* -\epsilon - 1 \ge 0\) without loss of generality
    as the bound is trivial otherwise.
\end{proof}
\end{mdframed}

While the Lebesgue Lemma may in principle be used with any semi-norm, the
corresponding pointwise bounds using the generalized Lebesgue function are only
possible since the norm of the function value \(\abs{f(x)}\weight(x)\) in the
space \((\field, \abs{\cdot}\weight(x))\) is dominated by the weighted sup-norm
\(\norm{f}_\weight\). This is a key property that \(L^p(\mu)\)-norms do not
have. Again, assuming the Lebesgue function coincides with its generalized version, the
following Lemma proves the upper bound of Theorem \ref{thm: optimal bias bounds
in the weighted sup-norm} \ref{it: (weighted) pointwise bound}.

\begin{lemma}[Pointwise bounds]
    \label{lem: pointwise bounds}
    Let \(\linProj\colon \funSpace \to \modelSpace\), be a linear projection
    \textbf{onto} the linear subspace \(\modelSpace\) and \(\error_\modelSpace^\weight[f]\) be the metric projection error from the function
    space \(\funSpace\) to the model space
    \(\modelSpace\) with respect to the weighted sup-norm
    \(\norm{\cdot}_\weight\), then
    \begin{equation}
        \label{eq: pointwise upper bound}
        \abs[\big]{\linProj f(x) - f(x)} \weight(x)
        \le \bigl(1+ \opnorm{\eval_x \linProj}_\weight\weight(x)\bigr)\error_\modelSpace^\weight[f],
    \end{equation}
    Moreover for all \(\epsilon>0\) there exists a function \(f\in \funSpace\) such that
    \begin{equation}
        \label{eq: pointwise lower bound}        
        \abs[\big]{\linProj f(x) - f(x)} \weight(x)
        \ge \bigl(\opnorm{\eval_x \linProj}_\weight\weight(x) - 1 -\epsilon\bigr)\error_\modelSpace^\weight[f],
    \end{equation}
\end{lemma}

\begin{mdframed}[topline=false, bottomline=false,innertopmargin=0pt]
\begin{proof}
We only prove the upper bound to show it is analogous to the proof of
Lemma \ref{lem: lebesgue}.
For \(\model^* \in \metricProj_\modelSpace^\weight[f]\) we have
\begin{align*}
    \abs[\big]{\linProj f(x) - f(x)}\weight(x)
    &\le \underbrace{\abs[\big]{\linProj (f - \model^*)(x)}}_{\le \opnorm{\eval_x \linProj}_\weight \mathrlap{\norm{f-\model^*}_\weight}}\weight(x)
    + \underbrace{\abs[\big]{\model^*(x) - f(x)}\weight(x)}_{\le \norm{\model^* - f}_\weight}
    \\
    &\le \bigl(\opnorm{\eval_x \linProj}_\weight\weight(x)+1\bigr)\error_\modelSpace^\weight[f].
\end{align*}
The proof of the lower bound is analogous to the one in Lemma \ref{lem: lebesgue}.
\end{proof}
\end{mdframed}

That the supremum over the generalized Lebesgue function is the generalized
Lebesgue constant also only holds for the weighted sup-norm.

\begin{lemma}
    \label{lem: operator norm relation}
    \(
        \opnorm{\linProj}_\weight = \sup_{x\in \domain}\opnorm{\eval_x\linProj}_\weight\weight(x)
    \)
\end{lemma}

\begin{mdframed}[topline=false, bottomline=false, innertopmargin=0pt]
\begin{proof}
    By definition
    \begin{align*}
        \opnorm{\linProj}_\weight
        \overset{\text{def.}}= \sup_{\substack{f\in \funSpace\\ \norm{f}_\weight \le 1}} \norm{\linProj f}_\weight
        &= \sup_{\substack{f\in \funSpace\\ \norm{f}_\weight \le 1}} \sup_{x\in \domain} \abs{\linProj f(x)}\weight(x)
        \\
        &= \sup_{x\in \domain} \sup_{\substack{f\in \funSpace\\ \norm{f}_\weight \le 1}} \abs{\linProj f(x)}\weight(x)
        \overset{\text{def.}}= \sup_{x\in \domain} \opnorm{\eval_x \linProj}_\weight\weight(x).
        \qedhere
    \end{align*}
\end{proof}
\end{mdframed}

\subsection{Regression projections}
\label{sec: approximation theory for regression}

In light of Lebesgue's Lemma \ref{lem: lebesgue} and the corresponding pointwise
bounds (Lem.~\ref{lem: pointwise bounds}), the proof of Theorem \ref{thm: optimal
bias bounds in the weighted sup-norm} only
requires us to show that the Lebesgue function and constant coincide with their generalized versions
and tighten the lower bounds. For the Lebesgue function this is done by the following 
Lemma.

\begin{mdframed}[style=default]
\begin{lemma}[The Lebesgue function is the operator norm]
    \label{lem: lebesgue function is operator norm}
    Let \(\linProj\) be a regression projection with Lagrange function \(\ell\). Then we have
    \begin{equation}
        \label{eq: lebesgue function regression upper bound}    
        \opnorm{\eval_x \linProj}_\weight\weight(x) \le \lebfun{\weight}{\linProj}{x}
        \qquad \forall x\in \domain,
    \end{equation}
    where \(\lebfun{\weight}{\linProj}{x}
        \overset{\mathrm{def.}}= \weight(x)\sum_{i=1}^n \frac{\abs{\ell_i(x)}}{\weight(x_i)}
    \) is the Lebesgue function (Def.~\ref{def: lebesgue function and constant}).
    Moreover if the \(x_1, \dots, x_n\) are distinct and the function space \(\funSpace\)
    is sufficiently rich (Def.~\ref{def: sufficiently rich function space}), then
    \begin{equation}
        \label{eq: lebesgue function regression equality}    
        \opnorm{\eval_x \linProj}_\weight \weight(x)
        = \lebfun{\weight}{\linProj}{x} \qquad \forall x\in \domain,
    \end{equation}
    and for any \(x\notin \set{x_1, \dots, x_n}\) there exists a function \(f\in
    \funSpace\) such that
    \begin{equation}
        \label{eq: approximation error regression lower bound}        
        \abs{f(x) - \linProj f(x)}\weight(x)
        = \bigl(1+ \lebfun{\weight}{\linProj}{x}\bigr)\error_\modelSpace^\weight[f].
    \end{equation}
\end{lemma}
\end{mdframed}

\begin{remark}[The assumptions are necessary]
    The assumption \(x\notin \set{x_1, \dots, x_n}\) for \eqref{eq:
    approximation error regression lower bound} is necessary. To see this,
    observe that for interpolation projections (\(n=m\)), we have \(f(x_i) = \linProj f(x_i)\)
    whereas typically \(\lebfun{\weight}{\linProj}{x_i} \ge 1\). Indeed,
    let \(g \in \funSpace\) be a function which
    assumes its weighted sup-norm at \(x_i\), that is \(g(x_i)\weight(x_i) =
    \norm{g}_\weight\) with \(\norm{g}_\weight = 1\) after possibly scaling, then
    \[
        \lebfun{\weight}{\linProj}{x_i}
        \overset{\eqref{eq: lebesgue function regression equality}}= \sup_{\substack{f\in \funSpace\\ \norm{f}_\weight \le 1}} \abs{\linProj f(x_i)}\weight(x_i)
        \ge \abs{\linProj g(x_i)}\weight(x_i)
        = \abs{g(x_i)}\weight(x_i) = 1.
    \]
\end{remark}

The following Corollary, which proves Theorem \ref{thm: optimal bias bounds in the weighted sup-norm} \ref{it: (weighted) uniform bound} with
the Lebesgue Lemma, essentially follows from Lemma \ref{lem: lebesgue
function is operator norm} and \ref{lem: operator norm relation} and the
definition of the Lebesgue constant as the supremum of the
Lebesgue function.

\begin{mdframed}[style=default]
\begin{corollary}[Lebesgue constant is operator norm]
    \label{cor: lebesgue constant is operator norm}
    Let \(\linProj\) be a regression projection and \(\lebconst{\weight}{\linProj}\) its Lebesgue constant, then
    \[
        \opnorm{\linProj}_\weight \le \lebconst{\weight}{\linProj}.
    \]
    If the \(x_i\) are distinct and the function space is sufficiently rich (Def.~\ref{def: sufficiently rich function space}),
    then
    \[
        \opnorm{\linProj}_\weight = \lebconst{\weight}{\linProj}.
    \]
    And if additionally, the \(x_i\) are not isolated points in \(\domain\), \(\modelSpace\subseteq C(\domain; \field)\) and \(\weight\) is continuous,
    then for any \(\epsilon>0\) there exists a function \(f\in \funSpace\) such that
    \begin{equation}
        \label{eq: lower bound tightening of Lebesgue Lemma}
        \norm{f - \linProj f}_\weight
        \ge (\opnorm{\linProj}_\weight + 1 -\epsilon)\error_\modelSpace^\weight[f].
    \end{equation}
\end{corollary}
\end{mdframed}

The key for \eqref{eq: lower bound tightening of Lebesgue Lemma}
is the continuity of the Lebesgue function that follows from the continuity of
the model space, which implies continuity of the Lagrange function (Lemma
\ref{lem: lagrange function properties} \ref{it: lagrange continuity}).  With
this continuity, we can bridge over the points \(x \in \set{x_1, \dots, x_n}\)
where the lower bound of Lemma \ref{lem: lebesgue function is operator norm}
does not hold.

\subsection{\texorpdfstring{\(L^p(\mu)\)}{Lp(µ)}-projection}
\label{sec: approximation theory of Lp projection}

In the following we will show the Lebesgue function and constant for orthogonal
projections \(\linProj_\mu\) coincide with their generalized versions
if the function space is \(\funSpace = L^2(\mu)\).
First, we will prove a kernel representation of the \(L^2(\mu)\)-projection \(\linProj_\mu\).

\begin{mdframed}[style=default]
\begin{theorem}[\(L^2\)-projection Lebesgue function and constant]
    \label{thm: L2 proj Lebesgue function and constant are operator norms}
    For \(\weight\colon \domain \to (0, \infty)\) let \(\mu(dx) = \weight(x)^2 \nu(dx)\)
    for some \(\sigma\)-finite measure \(\nu\) on \(\domain\).
    Then the Lebesgue constant and function of the projection \(\linProj_\mu\)
    (Def.~\ref{def: lebesgue function and constant of L2 projection})
    coincide with their generalized versions on \(L^2(\mu)\). More specifically,
    for \(\modelSpace\subseteq \funSpace \subseteq L^2(\mu)\)
    \begin{alignat*}{3}
        \opnorm{\eval_x \linProj_\mu}_{\funSpace, \weight}^{\weight(x)}
        &\le \opnorm{\eval_x \linProj_\mu}_{L^2(\mu), \weight}^{\weight(x)}
        &&= \lebfun{\weight}{\linProj_\mu}{x}
\\
        \opnorm{\linProj_\mu}_{\funSpace, \weight}
        &\le \opnorm{\linProj_\mu}_{L^2(\mu), \weight}
        &&= \lebconst{\weight}{\linProj_\mu}
\end{alignat*}
    In particular, we have for all \(f\in \funSpace \subseteq L^2(\mu)\)
    \begin{align}
        \label{eq: L2 orth proj, pointwise bound}
        \abs{\linProj_\mu f(x) - f(x)} \weight(x)
        &\le \bigl(1+ \lebfun{\weight}{\linProj_\mu}{x}\bigr)\error_\modelSpace^\weight[f]
        \\
        \label{eq: L2 orth proj, uniform bound}
        \norm{\linProj_\mu f - f}_\weight
        &\le \bigl(1+ \lebconst{\weight}{\linProj_\mu}\bigr)\error_\modelSpace^\weight[f]
    \end{align}
    If \(\funSpace=L^2(\mu)\),
    then \eqref{eq: L2 orth proj, pointwise bound} is sharp if \(\mu(\set{x}) = 0\)
    in the sense that there exist no smaller constant than \(1+\lebfun{\weight}{\linProj_\mu}{x}\) such
    that \eqref{eq: L2 orth proj, pointwise bound} holds uniformly for all
    \(f\in \funSpace\). And \eqref{eq: L2 orth proj, uniform bound} is sharp if
    \(\funSpace=L^2(\mu)\) and \(\mu\) is non-atomic, that is \(\mu(\set{x}) =
    0\) for all \(x\in \domain\).
\end{theorem}
\end{mdframed}

The key to its proof is the following kernel representation of the
\(L^2(\mu)\)-projection \(\linProj_\mu\).

\begin{lemma}[Kernel representation of \(L^2\)-projection]
    \label{lem: kernel representation of L2 proj}
    Let \(\linProj_\mu \colon L^2(\mu) \to \modelSpace\)
    be the \(L^2(\mu)\)-orthogonal projection. Then, using the reproducing kernel
    \(\reprK_\modelSpace^\mu\) (Lem.~\ref{lem: reproducing kernel}), we have
    \[
        \linProj_\mu f(x) = \int \reprK_\modelSpace^\mu(x,y)f(y)\mu(dy).
    \]
\end{lemma}
\begin{proof}
    Let \(\mBasis_1, \dots, \mBasis_m\) be an orthonormal basis of \(\modelSpace\) with respect to the inner product of \(L^2(\mu)\).
    Then the \(L^2(\mu)\)-orthogonal projection is given by
    \[
        \linProj_\mu f(x) = \sum_{k=1}^m \scp{f}{\mBasis_k}_\mu \mBasis_k(x)
        = \int \sum_{k=1}^m \mBasis_k(x)\overline{\mBasis_k(y)} f(y) \mu(dy)
        = \int \reprK_\modelSpace^\mu(x,y)f(y)\mu(dy).
        \qedhere
    \]
\end{proof}

Similar to the regression setting, we can bound the Lebesgue function and constant using
the leverage function \(\lever_\modelSpace^\mu\) (Def.~\ref{def: leverage function}).

\begin{mdframed}[style=default]
\begin{prop}[Leverage function bound on \(L^2\)-projection]
    \label{prop: leverage function bound on L2 projection}
    For the \(L^2(\mu)\)-projection \(\linProj_\mu\), we have 
    \[
        \lebfun{\weight}{\linProj_\mu}{x}
        \le \norm{\tfrac1\weight}_\mu\lever_\modelSpace^\mu(x)\weight(x)
        \qquad
        \text{and}
        \qquad
        \lebconst{\weight}{\linProj_\mu}
        \le \norm{\tfrac1\weight}_\mu\norm{\lever_\modelSpace^\mu}_\weight.
    \]
\end{prop}
\end{mdframed}

Often, \(\norm{\tfrac1\weight}_\mu = \sqrt{\nu(\domain)} = 1\),
e.g.\ in the randomized regression setting (Def.~\ref{def: randomized
regression}).

	\section{Proofs}
\label{sec: proofs and further details}

\subsection{Least squares}

The following result generalizes well-known characterizations of the
solution of least squares problems
\citep[e.g.][Theorem~3.12]{albertRegressionMoorePenrosePseudoinverse1972} to the
GLS case. How to perform this generalization is also well-known folklore.
As we could not find a good reference we provide a proof.

\begin{lemma}[General solution of the least squares problem]
    \label{lem: general least squares solution}
    Let \(\mBasisVec = (\mBasis_1, \dots, \mBasis_m)^\transpose\) be a basis of
    \(\modelSpace\) and let \(\Omega_n\) be a positive semi-definite matrix, \(\modelVander\) the design and \(\Gram\) the
    Gram matrix as defined after Definition \ref{def: least squares}.
    Then
    \[
        \regMin(y) \overset{\mathrm{def.}}=
        \argmin_{\model \in \modelSpace} \abs[\big]{\eval_X[\model] - y}_{\Omega_n}^2
        = \set[\Big]{\mBasisVec^\transpose(\cdot)\beta : \beta \in  \nullspace(\Gram) + \hat\beta},
    \]
    where
    \(\hat\beta= \frac1n\Gram^\dagger \modelVander^\conT \Omega_n y\)
    is the minimizer with minimum Euclidean norm \(\abs{\beta}_2\) and
    \[
        \nullspace(\Gram)
        = \set{(\identity - \Gram^\dagger \Gram) z : z\in \field^m}
    \]
    is the null space of \(\Gram\). If \(\mBasis_1, \dots, \mBasis_m\) is an orthonormal
    basis of \(\modelSpace\) with respect to \(\scp{\cdot}{\cdot}_*\), then
    \(\model^* = \mBasisVec^\transpose(\cdot)\hat\beta \in \regMin(y)\) is the unique minimizer with minimal \(\norm{\cdot}_*\)-norm.
\end{lemma}
\begin{remark}[Minimum norm solution]
    Since the coefficient vector \(\beta\) depends on the choice of basis,
    the minimum norm of \(\abs{\beta}_2\) depends on the choice of basis. If the
    solution is not a singleton (\(\dim \nullspace(\Gram) > 0\)), the minimum norm solution is thereby not
    invariant of the choice of basis. With an inner product
    \(\scp{\cdot}{\cdot}_*\) on \(\modelSpace\) it is however possible to define
    a unique basis-free minimum norm solution.
\end{remark}
\begin{proof}
    Since \(\mBasisVec = (\mBasis_1, \dots, \mBasis_m)^\transpose\) is a basis of \(\modelSpace\)
    \[
        \argmin_{\model\in\modelSpace} \abs[\big]{\eval_X[\model] - y}_{\Omega_n}^2
        = \set[\Big]{\mBasisVec^\transpose(\cdot)\beta : \beta \in \argmin_{\beta\in\field^m} \abs[\big]{\underbrace{\eval_X[\mBasisVec^\transpose(\cdot)\beta]}_{=\modelVander\beta} - y}_{\Omega_n}^2}.
    \]
    It is therefore sufficient to prove
    \[
        \argmin_{\beta\in\field^m} \abs[\big]{\modelVander\beta - y}_{\Omega_n}^2
        = \hat \beta + \nullspace(\Gram).
    \]
    To this end define \(H\coloneq\Omega_n^{1/2}\modelVander\) and \(b \coloneq \Omega_n^{1/2}y\)
    such that
    \begin{align*}
        \argmin_{\beta\in\field^m} \abs[\big]{\modelVander\beta - y}_{\Omega_n}^2
        &= \argmin_{\beta\in\field^m} \abs[\big]{\Omega_n^{1/2}(\modelVander\beta - y)}_2^2
        \\
        &= \argmin_{\beta\in\field^m} \abs{H\beta - b}_2^2,
        \\
        &= H^\dagger b + \set[\Big]{(I-H^\dagger H)z : z\in \field^m},
    \end{align*}
    where the last equality is a standard result \citep[e.g][Thm~3.12]{albertRegressionMoorePenrosePseudoinverse1972}
    or \citep[ch.~3, Cor.~1, p.~105]{ben-israelGeneralizedInverses2003}. The first source
    is a bit more accessible but does not cover the complex case.
    Combining \(H^\dagger H^{\dagger\conT} H^\conT = H^\dagger\) \citep[Eq.~(10)]{penroseGeneralizedInverseMatrices1955} with
    \(H^\dagger H^{\dagger\conT} = (H^\conT H)^\dagger\) \citep[Lem.~1\(\cdot\)5]{penroseGeneralizedInverseMatrices1955}
    we have
    \[
        H^\dagger = (H^\conT H)^\dagger H^\conT.
    \]
    This implies
    \[
        H^\dagger b
        =(H^\conT H)^\dagger H^\conT b
        = (\modelVander^\conT \Omega_n \modelVander)^\dagger \modelVander^\conT \Omega_n y
        = \tfrac1n\Gram^\dagger \modelVander^\conT \Omega_n y
        = \hat\beta,
    \]
    and \(H^\dagger H = (H^\conT H)^\dagger H^\conT H = \Gram^\dagger \Gram\)
    implies
    \[
        \set{( \identity - H^\dagger H) z : z\in \field^m}
        = \nullspace(\Gram).
    \]
    This proves the first claim. For \(\beta = \hat \beta + \tilde \beta\) with
    \(\tilde \in \nullspace(\Gram)\) we have
    \[
        \abs{\beta}_2^2
        = \abs{\hat\beta}_2^2 + \abs{\tilde\beta}_2^2
    \]
    since \(\hat\beta = H^\dagger b\) is orthogonal to any vector \((\identity - H^\dagger H)z \in \nullspace(\Gram)\).
    Consequently, \(\hat\beta\) is the minimum Euclidean norm solution.
    If \(\mBasis_1, \dots, \mBasis_m\) is an orthonormal basis of \(\modelSpace\) with respect to \(\scp{\cdot}{\cdot}_*\), then
    the norm of the coefficients coincides with the norm \(\norm{\cdot}_*\) of the corresponding function in \(\modelSpace\), that is
    \[
        \norm{\mBasisVec(\cdot)^\transpose \beta}_*^2
        = \sum_{i,j=1}^m \overline{\beta_i} \beta_j \scp{\mBasis_j}{\mBasis_i}_*
        = \abs{\beta}_2^2
    \]
    Consequently, \(\model^* = \mBasisVec(\cdot)^\transpose\hat\beta\) is the unique minimizer with minimal \(\norm{\cdot}_*\)-norm.
\end{proof}

\subsubsection{Proof of Theorem \ref{thm: least squares solution}: Least squares characterization}

    By Lemma \ref{lem: general least squares solution} the set \(\regMin(y)\)
    always contains the element
    \begin{equation*}
        \regOp(y)
        \coloneq \mBasisVec(\cdot)^\transpose \hat{\beta}
        = \tfrac1n\mBasisVec(\cdot)^\transpose \Gram^\dagger \modelVander^\conT \Omega_n y
    \end{equation*}
    and is therefore never empty. To prove that \(\linProj = \regOp \circ \eval_X\) is a 
    linear projection, we observe
    \begin{align*}
        \linProj^2 f
        &= \regOp \circ \eval_X \circ \regOp \circ \eval_X[f]
        \\
        &= \tfrac1{n^2}\mBasisVec(\cdot)^\transpose \Gram^\dagger \underbrace{\modelVander^\conT \Omega_n \modelVander}_{=n\Gram} \Gram^\dagger \modelVander^\conT \Omega_n \eval_X[f]
        = \tfrac1n\mBasisVec(\cdot)^\transpose \Gram^\dagger \modelVander^\conT \Omega_n \eval_X[f]
        = \linProj f.
    \end{align*}
    It is a metric projection with respect to \(\norm{\cdot}_{\Omega_n}\) almost by definition as
    \begin{equation}
        \label{eq: lin proj is metric projection}
        \linProj f = \regOp(\eval_X[f])
        \in \argmin_{\model \in \modelSpace} \abs[\big]{\eval_X[\model] - \eval_X[f]}_{\Omega_n}^2
        \overset{\text{def.}}= \argmin_{\model\in\modelSpace} \norm{f - \model}_{\Omega_n}^2.
    \end{equation}
    To prove that \(\eval_X \circ \regOp\) is a linear projection observe that
    \begin{align*}
        (\eval_X \circ \regOp)^2
        &= (\tfrac1n\modelVander \Gram^\dagger \modelVander^\conT \Omega_n)^2
        \\
        &= \tfrac1{n^2}\modelVander \Gram^\dagger \underbrace{\modelVander^\conT \Omega_n \modelVander}_{=n\Gram} \Gram^\dagger \modelVander^\conT \Omega_n
        = \tfrac1n\modelVander \Gram^\dagger \modelVander^\conT \Omega_n
        = \eval_X \circ \regOp,
    \end{align*}
    using the property \(\Gram^\dagger \Gram \Gram^\dagger = \Gram^\dagger\) of the Moore-Penrose pseudoinverse.
\begin{enumerate}[wide=0pt]
    \item[\ref{it: least squares unique} \(\Rightarrow\) \ref{it: norm positive definite}]
    Assume \(\norm{\cdot}_{\Omega_n}\) is not positive definite, then there exists
    \(0\neq \model_0\in \modelSpace\) with \(\norm{\model_0}_{\Omega_n} = 0\). We
    will show for \(\model^* \in \regMin(y)\) that \(\model^* + \model_0 \in
    \regMin(y)\), which implies that \(\regMin(y)\) is not a singleton. Indeed due to
    \(\abs{\eval_X(\model_0)}_{\Omega_n}^2 = \norm{\model_0}_{\Omega_n}^2 = 0\) we have
    \begin{align*}
        \abs[\big]{\eval_X[ \model^* + \model_0 ] - y}_{\Omega_n}^2
        &= \abs[\big]{\eval_X[ \model^* ] - y}_{\Omega_n}^2 + 2\Re\underbrace{
            \scp[\big]{\eval_X[ \model^* ] - y}{\eval_X[ \model_0 ]}_{\Omega_n}
        }_{
            \le \abs{\eval_X[ \model^* ] - y}_{\Omega_n}\norm{\model_0}_{\Omega_n}
        } + \norm{\model_0}_{\Omega_n}^2
        \\
        &\le \abs[\big]{\eval_X[ \model^* ] - y}_{\Omega_n}^2
    \end{align*}
    and since \(\model^*\) is a minimizer, \(\model^* + \model_0\) is also a
    minimizer. This contradicts the uniqueness of the minimizer and therefore
    \(\norm{\cdot}_{\Omega_n}\) must be positive definite.

    \item[\ref{it: norm positive definite} \(\Rightarrow\) \ref{it: gram positive definite}]
    Let \(0\neq a\in \field^m\), then by definition of \(\Gram_{kl} = \scp{\mBasis_l}{\mBasis_k}_{\Omega_n}\) we have
    \[
        a^\conT \Gram a
        = \sum_{k,l=0}^{m-1} \overline{a_k} a_l \scp{\mBasis_l}{\mBasis_k}_{\Omega_n}
        = \scp[\bigg]{\sum_{k=0}^{m-1} a_k \mBasis_k}{\sum_{l=0}^{m-1} a_l \mBasis_l}_{\Omega_n} > 0,
    \]
    where we use that \(\model \coloneq \sum_{k=0}^{m-1} a_k \mBasis_k \neq 0\).

    \item[\ref{it: gram positive definite} \(\Rightarrow\) \ref{it: least squares unique}]
    If \(\Gram\) is invertible, then \(\nullspace(\Gram)=\set{0}\) in Lemma \ref{lem: general least squares solution} and \(\regMin(y)\) is a singleton.

    \item[\ref{it: gram positive definite} \(\Rightarrow\) \ref{it: lin proj surjective}]
    Since \(\linProj\) is a linear projection, it is surjective if and only if it acts
    as the identity on \(\modelSpace\).  To prove \(\linProj\) is the identity
    on \(\modelSpace\) it is sufficient to show that \(\linProj\) acts like the
    identity on the basis \(\mBasis_1, \dots, \mBasis_m\). Observe that
    \(\eval_X[ \mBasis_k ]\) is the \(k\)-th column \((\modelVander)_k\) of
    \(\modelVander\). Since \(\frac1n\Gram^{-1}\modelVander^\conT \Omega_n
    \modelVander = \identity\) by definition
    of \(\Gram=\frac1n\modelVander^\conT \Omega_n \modelVander\) we thus have
    \[
        \linProj \mBasis_k = \regOp \circ \eval_X[ \mBasis_k ]
        = \regOp((\modelVander)_k)
        = \tfrac1n\mBasisVec(\cdot)^\transpose\Gram^{-1} \modelVander^\conT \Omega_n (\modelVander)_k = \mBasisVec(\cdot)^\transpose e_k = \mBasis_k.
    \]
    
    \item[\ref{it: lin proj surjective} \(\Rightarrow\) \ref{it: norm positive definite}]
    We will prove the contrapositive. Assume \(\norm{\cdot}_{\Omega_n}\) is not positive definite, then there exists
    \(0\neq \model \in \modelSpace\) with \(\norm{\model}_{\Omega_n} = 0\) and basis representation
    \(\model = a^\transpose\mBasisVec\). Observe that
    \[
        0 = \norm{\model}_{\Omega_n}^2
        = \scp{\eval_X[a^\transpose\mBasisVec]}{\eval_X[a^\transpose\mBasisVec]}_{\Omega_n}
        = \tfrac1n a^\conT \modelVander^\conT \Omega_n \modelVander a
        = a^\conT \Gram a,
    \]
    implies that \(\Gram a = 0\) since \(\Gram\) is positive semi-definite. But this then implies
    \[
        \linProj \model = \regOp \circ \eval_X[\model]
        = \tfrac1n\mBasisVec(\cdot)^\transpose \Gram^\dagger \modelVander^\conT \Omega_n \modelVander a
        = \mBasisVec(\cdot)^\transpose \Gram^\dagger \Gram a = 0.
    \]
    In particular \(\model \in \modelSpace\) is not in the image of the linear
    projection \(\linProj\) and therefore \(\linProj\) is not surjective.
\end{enumerate}
    That the minimizer in \eqref{eq: LS is metric projection} is unique if the equivalent conditions
    hold follows directly from \ref{it: least squares unique} applied to
    \(\eval_X[f]\) (cf.~\eqref{eq: lin proj is metric projection}).

\subsection{Approximation theory}

\subsubsection{Proof of Lemma \ref{lem: lebesgue function is operator norm}: The Lebesgue function is the operator norm}
\begin{steps}
    \item \textbf{Upper bound.}
    Let \(x \odot y = (x_1y_1, \dots, x_n y_n)\) denote the Hadamard product.
    For any \(f\in \funSpace\) we have
    \begin{align*}
        \abs{\linProj f(x)}
        \overset{\text{def.}}&= \abs{\pi_x \circ \regOp \circ \eval_X[f]}
        \\
        \overset{\text{def.}}&= \abs{\ell(x)^\conT \eval_X[f]}
        \\
        &= \abs[\big]{(\ell(x) \odot \eval_X[\tfrac1\weight])^\conT \eval_X[f\weight]}
        \\
        \overset{\text{Hölder}}&\le \abs[\big]{\ell(x)\odot \eval_X[\tfrac1\weight]}_1 \abs[\big]{\eval_{X}[fw]}_\infty
        \le \abs[\big]{\ell(x)\odot \eval_X[\tfrac1\weight]}_1 \norm[\big]{f\weight}_\infty
    \end{align*}
    Since by definition \(\norm{f}_\weight = \norm{fw}_\infty\)
    this implies
    \[
        \opnorm{\eval_x \linProj}_\weight
        = \sup_{\substack{f\in \funSpace\\ \norm{f}_\weight \le 1}} \abs{\linProj f(x)}
        \le \abs[\big]{\ell(x)\odot \eval_X[\tfrac1\weight]}_1 = \sum_{i=1}^n \frac{\abs{\ell_i(x)}}{\weight(x_i)}.
    \]
    and thus by definition of the Lebesgue function \(\opnorm{\eval_x \linProj}_\weight \weight(x) \le \lebfun{\weight}{\linProj}{x}\).
    \item \textbf{Lower bound.}
    Let us now consider the case where the \(x_i\) are distinct and the function
    space is sufficiently rich. For the lower bound on the Lebesgue function
    we use the complex extension of the sign function, the phase function
    \begin{equation}
        \label{eq: definition of phase}    
        \phase(z) \coloneq \begin{cases}
            z/\abs{z} & z\neq 0
            \\
            1 & z=0.
        \end{cases}
    \end{equation}
    This phase function is applied entry-wise to the lagrange function to obtain
    adversarial labels
    \[
        y \coloneq \phase\bigl(\ell(x)\bigr) \odot \eval_X[\tfrac1\weight]
        =  \bigl(\tfrac{\phase(\ell_1(x))}{\weight(x_1)},\dots, \tfrac{\phase(\ell_n(x))}{\weight(x_n)}\bigr) \in \field^n.
    \]
    Since \(\funSpace\) is sufficiently rich we may choose \(f\in \funSpace\) with \(\eval_X[f] = y\) and
    \[
        \norm{f}_\weight = \norm{f\weight}_\infty \le \abs{\eval_X[f\weight]}_\infty
        = \abs{y\odot \eval_X[\weight]}_\infty = \abs{\phase(\ell(x))}_\infty = 1.
    \]
    This implies
    \begin{equation}
        \label{eq: linear projection achieves Lebesgue function}    
        \linProj f(x) = \ell(x)^\conT \eval_X[f]
        = \ell(x)^\conT y = \abs[\big]{\ell(x) \odot \eval_X[\tfrac1\weight]}_1
    \end{equation}
    and consequently
    \[
        \opnorm{\eval_x \linProj}_\weight
        = \sup_{\norm{g}_{\weight} \le 1} \abs{\linProj g(x)} \ge \abs{\linProj f(x)}
        = \abs{\ell(x) \odot \eval_X[\tfrac1\weight]}_1.
    \]
    Together with the general upper bound this implies \(\opnorm{\eval_x \linProj}_\weight \weight(x) = \lebfun{\weight}{\linProj}{x}\).
    \item \textbf{Worst case function.}
    Finally, let \(x\notin \set{x_1, \dots, x_n}\). Then for
    \[
        z\coloneq
\bigl(\tfrac{\phase(\ell_1(x))}{\weight(x_1)},\dots, \tfrac{\phase(\ell_n(x))}{\weight(x_n)}, \tfrac{-1}{\weight(x)}\bigr) \in \field^{n+1}.
    \]
    and \(\tilde X = (x_1, \dots, x_n, x)\) there exists \(f\in \funSpace\)
    such that \(\eval_{\tilde X}[f] = z\) and
    \[
        \norm{f}_\weight \le \abs{\eval_{\tilde{X}}[f\weight]}_\infty
        = \abs{z \odot \eval_{\tilde X}[w]}_\infty = 1.
    \]
    Observe that \(\eval_{\tilde X}[f] = z\) implies \(f(x) =
    \frac{-1}{\weight(x)}\)  and \(\eval_X[f] = y\). But
    \(\eval_X[f] = y\) implies \(\linProj f(x) = \opnorm{\eval_x \linProj}_\weight\)
    by \eqref{eq: linear projection achieves Lebesgue function} and the following arguments. Therefore
    \begin{align*}
        \abs[\big]{\linProj f(x) - f(x)} \weight(x)
        &= \opnorm{\eval_x \linProj}_\weight\weight(x) + 1
        \\
        &\ge \bigl(\opnorm{\eval_x \linProj}_\weight\weight(x) + 1\bigr)\norm{f}_\weight
        \ge \bigl(\opnorm{\eval_x \linProj}_\weight\weight(x) + 1\bigr)\error_\modelSpace^\weight[f],
    \end{align*}
    where we used \(\norm{f}_\weight \le 1\) and \(\norm{f-0}_\weight \ge \error_\modelSpace^\weight[f]\)
    since \(0\in \modelSpace\). By \eqref{eq: lebesgue function regression equality}
    we are done.
    \qedhere
\end{steps}

\subsubsection{Proof of Corollary \ref{cor: lebesgue constant is operator norm}: Lebesgue constant is operator norm}

Lemma \ref{lem: operator norm relation} in combination with the relationship of the Lebesgue function with its
generalized version \eqref{eq: lebesgue function regression upper bound}
from Lemma \ref{lem: lebesgue function is operator norm} and the definition of the Lebesgue constant
yield
\[
    \opnorm{\linProj}_\weight
    \overset{\text{Lem.~\ref{lem: operator norm relation}}}=
    \sup_{x\in \domain}\opnorm{\eval_x \linProj}_\weight\weight(x)
    \overset{\eqref{eq: lebesgue function regression upper bound}}\le \sup_{x\in \domain} \lebfun{\weight}{\linProj}{x}
    \overset{\text{def.}}= \lebconst{\weight}{\linProj}.
\]
And if the additional assumptions that ensure the Lebesgue function coincides with its generalized version 
holds true (cf.~Lemma \ref{lem: lebesgue function is operator norm} \eqref{eq:
lebesgue function regression equality}), then the above is an equality.

The only thing left to show is the lower bound in \eqref{eq: lower bound tightening of Lebesgue Lemma}.
    For this, let \(\epsilon>0\) and choose \(\tilde x \in \domain\)
    such that
    \[
        \lebfun{\weight}{\linProj}{\tilde x} > \sup_{x\in \domain}\lebfun{\weight}{\linProj}{x} -\epsilon
        \overset{\mathrm{def.}}= \lebconst{\weight}{\linProj} - \epsilon.
    \]
    Since \(\modelSpace\subseteq C(\domain; \field)\) the Lagrange function \(\ell\) is continuous
    by Lemma \ref{lem: lagrange function properties} \ref{it: lagrange continuity} and therefore
    the Lebesgue function
    \[
        \lebfun{\weight}{\linProj}{x} = \weight(x)\sum_{i=1}^n \frac{\abs{\ell_i(x)}}{\weight(x_i)}
    \]
    is continuous in \(x\) since we also assume \(\weight\) is continuous. Consequently,
    there exists a neighborhood \(U\subseteq \domain\) of \(\tilde x\) such that for all \(x\in U\)
    \[
        \lebfun{\weight}{\linProj}{x} > \lebconst{\weight}{\linProj} - 2\epsilon.
    \]
    Since the points \(x_1, \dots, x_n\) are not isolated in \(\domain\) by assumption,
    it is possible to choose \(x^*\in U\) such that \(x^*\notin \set{x_1, \dots,
    x_n}\), possibly \(x^* = \tilde x\). then by the pointwise lower bound
    \eqref{eq: approximation error regression lower bound} in Lemma \ref{lem:
    lebesgue function is operator norm} there exists \(f\in \funSpace\) such
    that
    \begin{align*}
        \norm{\linProj f -f}_\weight
        \overset{\text{def.}}\ge \abs[\big]{\linProj f(x^*) - f(x^*)} \weight(x^*)
        \overset{\eqref{eq: approximation error regression lower bound}}&= \bigl(1+ \lebfun{\weight}{\linProj}{x^*}\bigr)\error_\modelSpace^\weight[f]
        \\
        &> \bigl(1 + \lebconst{\weight}{\linProj} - 2\epsilon\bigr)\error_\modelSpace^\weight[f].
    \end{align*}
    Since \(\epsilon>0\) was arbitrary this implies that there is no constant
    \(c < 1+\lebconst{\weight}{\linProj}\) such that \(\norm{\linProj f -f}_\weight \le
    c\error_\modelSpace^\weight[f]\) for all \(f\in \funSpace\).

\subsubsection{Proof of Theorem \ref{thm: optimal bias bounds in the weighted sup-norm}: Optimal approximation error bound in weighted sup-norm}

This follows essentially from the discussion of generalized Lebesgue functions
and constants in Appendix \ref{sec: approximation theory}.  Here is a formal
proof:

\begin{enumerate}[wide=0pt]
    \item[\ref{it: (weighted) pointwise bound}]
    The bound follows from \eqref{eq: lebesgue function regression upper bound}
    in Lemma \ref{lem: lebesgue function is operator norm} and \eqref{eq:
    pointwise upper bound} in Lemma \ref{lem: pointwise bounds}
    \[
        \abs[\big]{\linProj f(x) - f(x)} \weight(x)
        \overset{\eqref{eq: pointwise upper bound}}\le \bigl(\opnorm{\eval_x \linProj}_\weight\weight(x) + 1\bigr)\error_\modelSpace^\weight[f]
        \overset{\eqref{eq: lebesgue function regression upper bound}}\le \bigl(\lebfun{\weight}{\linProj}{x} + 1\bigr)\error_\modelSpace^\weight[f].
    \]
    The optimality follows from \eqref{eq: approximation error regression lower bound} in Lemma \ref{lem: lebesgue function is operator norm}.

    \item[\ref{it: (weighted) uniform bound}] This follows similarly from Corollary \ref{cor: lebesgue constant is operator norm} and
    Lebesgue's Lemma \ref{lem: lebesgue}.

\end{enumerate}

\subsubsection{Proof of Theorem \ref{thm: L2 proj Lebesgue function and constant are operator norms}: \texorpdfstring{\(L^2\)}{L2}-projection Lebesgue function and constant}
    To prove an upper bound on the general lebesgue function, let
    \(\norm{f}_\weight \le 1\).  Then by definition of \(\norm{\cdot}_\weight\)
    we have \(\abs{f(y)} \le \frac1{\weight(y)}\) and consequently
    \begin{equation}
        \label{eq: L2 orth proj, pointwise bound, proof}    
        \abs{\linProj_\mu f(x)} \weight(x)
        = \weight(x) \abs[\Big]{\int \reprK_\modelSpace^\mu(x,y)f(y)\mu(dy)}
        \le \weight(x)\!\! \int \!\frac{\abs{\reprK_\modelSpace^\mu(x,y)}}{\weight(y)}\mu(dy)
        = \lebfun{\weight}{\linProj_\mu}{x}.
    \end{equation}
    We note that if the right hand side is finite, then \(K_\modelSpace^\mu(x,y)f(y)\) is integrable
    as it is dominated by an integrable \(\abs{\reprK_\modelSpace^\mu(x,y)}/\weight(y)\).
    Whereas if the right hand side is infinite the inequality is trivial.
    Since the inequality holds for all \(f\in L^2(\mu)\) with \(\norm{f}_\weight \le 1\), we get the upper bound
    \begin{equation}
        \label{eq: upper bound on general lebesgue function}
        \opnorm{\eval_x\linProj_\mu}_{L^2(\mu),\weight}^{\weight(x)}
        = \sup_{\norm{f}_\weight \le 1} \abs{\linProj_\mu f(x)}\weight(x)
        \le \lebfun{\weight}{\linProj_\mu}{x}.
    \end{equation}
    To see that the inequality in \eqref{eq: upper bound on general lebesgue function} is an equality
    take a sequence of sets \((E_n)_{n\in \nat}\) with \(\nu(E_n) < \infty\) and \(E_n \uparrow \domain\),
    which is possible by \(\sigma\)-finiteness of \(\nu\).
    Recalling the definition of the phase function \eqref{eq: definition of phase} 
    the sequence of functions
    \begin{equation}
        \label{eq: (L2 Lebesgue function) sequence of f_n}    
        f_n(y) \coloneq \frac{\ind{E_n}(y)}{\weight(y)}\overline{\phase(\reprK_\modelSpace^\mu(x,y))},
    \end{equation}
    are in \(L^2(\mu)\) since \(\nu(dx) = \frac1{\weight(x)^2}\mu(dx)\) and
    \(\nu(E_n) < \infty\). These functions converge against the supremum in
    \eqref{eq: upper bound on general lebesgue function} by
    the kernel representation of \(L^2(\mu)\) projections (Lemma \ref{lem: kernel representation of L2 proj})
    and the monotone convergence theorem (MCT)
    \begin{align*}
        \lim_{n\to \infty} \abs{\linProj_\mu f_n(x)}\weight(x)
        \overset{\text{Lem.~\ref{lem: kernel representation of L2 proj}}}&
        = \lim_{n\to\infty} \weight(x) \abs[\Big]{\int \reprK_\modelSpace^\mu(x,y)f_n(y)\mu(dy)}
        \\
        \overset{\eqref{eq: (L2 Lebesgue function) sequence of f_n}}&= \lim_{n\to \infty} \weight(x) \int_{E_n} \frac{\abs{\reprK_\modelSpace^\mu(x,y)}}{\weight(y)}\mu(dy)
        \overset{\text{MCT}}= \lebfun{\weight}{\linProj_\mu}{x}.
    \end{align*}
    Since we have \(\norm{f_n}_\weight \le 1\), this implies the inequality in
    \eqref{eq: upper bound on general lebesgue function} is an equality.
    We then get \(\lebconst{\weight}{\linProj_\mu} = \opnorm{\linProj_\mu}_{L^2(\mu),
    \weight}\) directly by the representation of the general Lebesgue constant in terms of the
    Lebesgue function in Lemma \ref{lem: operator norm relation}.
    The general statements about the general Lebesgue function and constant,
    Lemma \ref{lem: pointwise bounds} and \ref{lem: lebesgue}, yield the upper bounds
    \eqref{eq: L2 orth proj, pointwise bound} and \eqref{eq: L2 orth proj,
    uniform bound} respectively.
    To get the sharpness of \eqref{eq: L2 orth proj, pointwise bound} we use the
    same sequence of sets \((E_n)_{n\in \nat}\) as above to define the functions
    \[
        g_n(y) \coloneq \begin{cases}
            \frac{\ind{E_n}(y)}{\weight(y)}\overline{\phase(\reprK_\modelSpace^\mu(x,y))} & y\neq x
            \\
            -\frac1{\weight(x)} & y=x.
        \end{cases}
    \]
    These functions are again in \(L^2(\mu)\), since \(\nu\) is \(\sigma\)-finite.
    And \(g_n = f_n\) \(\mu\)-almost everywhere since \(\mu(\set{x}) = 0\).
    Hence we have \(\linProj_\mu g_n = \linProj_\mu f_n\) by the Kernel
    representation of \(\linProj_\mu\) (Lemma \ref{lem: kernel representation of
    L2 proj}). In particular,
    \begin{align*}
        \liminf_{n\to \infty} \bigl(\linProj_\mu g_n(x) - g_n(x)\bigr)\weight(x)
        &= \liminf_{n\to \infty} \weight(x) \int_{E_n} \frac{\abs{\reprK_\modelSpace^\mu(x,y)}}{\weight(y)} \mu(dy) + 1
        \\
        \overset{\text{MCT}}&= 1+ \lebfun{\weight}{\linProj_\mu}{x}
    \end{align*}
    We thus have \(A_n \coloneq \abs{\linProj_\mu g_n(x) - g_n(x)} \weight(x) \to 1+ \lebfun{\weight}{\linProj_\mu}{x} \in [1,\infty]\).
    Since we also have
    \[
        1 = \norm{g_n}_\weight = \norm{g_n-0}_\weight \ge \error_\modelSpace^\weight[g_n]
    \]
    due to \(0\in \modelSpace\) this implies for any \(0\le c < 1 + \lebfun{\weight}{\linProj_\mu}{x}\) there exists \(n\in \nat\) such that
    \(A_n > c \ge c \error_\modelSpace^\weight[g_n]\). This proves 
    \begin{equation}
        \label{eq: L2 orth proj, pointwise bound, sharpness}    
        \inf\set[\Big]{c\ge 0: \abs[\big]{\linProj_\mu f(x) - f(x)}\weight(x) \le c \error_\modelSpace^\weight[f] \text{ for all } f\in \funSpace} \ge 1 + \lebfun{\weight}{\linProj_\mu}{x}
    \end{equation}
    and thereby sharpness of \eqref{eq: L2 orth proj, pointwise bound}. For the uniform bound first
    consider the case \(\lebconst{\weight}{\linProj_\mu} < \infty\) and choose
    \(x_\epsilon\) with \(\lebfun{\weight}{\linProj_\mu}{x_\epsilon} \ge
    \lebconst{\weight}{\linProj_\mu} - \epsilon\) and apply the previous
    argument to get for some \(f\in \funSpace=L^2(\mu)\)
    \[
        \abs{\linProj_\mu f(x_\epsilon) - f(x_\epsilon)}\weight(x_\epsilon)
        \overset{\eqref{eq: L2 orth proj, pointwise bound, sharpness}}\ge (1+\lebfun{\weight}{\linProj_\mu}{x_\epsilon}-\epsilon) \error_\modelSpace^\weight[f]
        \ge (1+\lebconst{\weight}{\linProj_\mu} - 2\epsilon) \error_\modelSpace^\weight[f].
    \]
    In particular, since \(x_\epsilon\) is a specific point, we have by definition
    of the weighted sup-norm
    \[
        \norm{\linProj_\mu f - f}_\weight
        \ge \abs{\linProj_\mu f(x_\epsilon) - f(x_\epsilon)}\weight(x_\epsilon)
        \ge (1+\lebconst{\weight}{\linProj_\mu} - 2\epsilon) \error_\modelSpace^\weight[f].
    \]
    Since \(\epsilon > 0\) was arbitrary, we get the optimality of \(1+\lebconst{\weight}{\linProj_\mu}\)
    in \eqref{eq: L2 orth proj, uniform bound}. For the case of
    \(\lebconst{\weight}{\linProj_\mu} = \infty\) we can choose \(x_n\) with
    \(\lebfun{\weight}{\linProj_\mu}{x_n} \ge n\) and apply the same arguments
    to prove that no finite constant \(c\) can satisfy \(\norm{\linProj_\mu f -
    f}_\weight \le c \error_\modelSpace^\weight[f]\) for all \(f\in \funSpace\).
    
\subsubsection{Proof of Proposition \ref{prop: leverage function bound on L2 projection}: Leverage function bound on \texorpdfstring{\(L^2\)}{L2} projection}

Let \(\mBasis_1,\dots, \mBasis_m\) be an orthonormal basis of \(\modelSpace\)
with respect to \(L^2(\mu)\). For \(f\in \funSpace\subseteq L^2(\mu)\)
let \(c_k \coloneq \scp{f}{\mBasis_k}_\mu\) be the Fourier coefficients of \(f\)
with respect to the basis \(\mBasis_k\). Then
\begin{align*}
    \opnorm{\eval_x \linProj_\mu}_{\mu}
    \overset{\text{def.}}&=
    \sup_{\norm{f}_{\mu} \le 1} \abs{\linProj_\mu f(x)}
    = \sup_{\norm{f}_{\mu} \le 1} \abs[\Big]{\sum_{k=1}^m c_k \mBasis_k(x)}
    \\
    \overset{\text{Cauchy-Schwarz}}&\le
    \sup_{\norm{f}_{\mu} \le 1}\sqrt{\sum_{k=1}^m \abs{c_k}^2}\underbrace{\sqrt{\sum_{k=1}^m \abs{\mBasis_k(x)}^2}}_{=\lever_\modelSpace^\mu(x)}
    \\[-2ex]
    \overset{\text{Bessel-inequality}}&\le
    \sup_{\norm{f}_{\mu} \le 1} \norm{f}_{\mu} \lever_\modelSpace^\mu(x)
    \\
    &\le \lever_\modelSpace^\mu(x).
\end{align*}
Consequently, we have \(\opnorm{\eval_x \linProj_\mu}_{\mu} \le \lever_\modelSpace^\mu(x)\) and
therefore
\begin{align*}
    \lebfun{\weight}{\linProj_\mu}{x}
    = \opnorm{\eval_x \linProj_\mu}_{L^2(\mu),\weight}^{\weight(x)}
    &= \sup_{\substack{f\in L^2(\mu)\\ \norm{f}_\weight \le 1}} \abs{\linProj_\mu f(x)}\weight(x)
    \\
    &\le \sup_{\substack{f\in L^2(\mu)\\ \norm{f}_\weight \le 1}} \opnorm{\eval_x \linProj_\mu}_{\mu}\norm{f}_\mu \weight(x)
    \\
    &\le \lever_\modelSpace^\mu(x) \sup_{\substack{f\in L^2(\mu)\\ \norm{f}_\weight \le 1}} \norm{f}_\mu \weight(x)
    \\
    &\le \norm{\tfrac1\weight}_{\mu} \lever_\modelSpace^\mu(x)\weight(x),
\end{align*}
using \(\norm{f}_\mu \le \norm{\tfrac1\weight}_\mu \norm{f}_\weight\) by Lemma \ref{lem: Lp inclusion}.
Taking the supremum over the bound on the Lebesgue function yields the bound on the Lebesgue constant.

\subsection{Reproducing kernel and leverage function}

\subsubsection{Proof of Theorem \ref{thm: reproducing kernel representation}: Reproducing kernel representation}
Since \(\linProj\) is surjective, \(\norm{\cdot}_{\Omega_n}\) is a norm by Theorem \ref{thm: least squares solution} \ref{it: norm positive definite} \(\Leftrightarrow\) \ref{it: lin proj surjective}.
As \(\norm{\cdot}_{\Omega_n}\) is a norm, there exists a corresponding reproducing kernel \(\reprK_\modelSpace^{\Omega_n}\) and orthonormal basis \(\mBasis_1, \dots, \mBasis_m\)
of the model space \(\modelSpace\) with respect to \(\scp{\cdot}{\cdot}_{\Omega_n}\) by the Gram-Schmidt
orthonormalization procedure. 
Since \(\mBasis_1, \dots, \mBasis_m\) are orthonormal with
respect to \(\scp{\cdot}{\cdot}_{\Omega_n}\), the Gram matrix is the identity
\[
    \Gram_{ij} = \scp{\mBasis_j}{\mBasis_i}_{\Omega_n} = \delta_{ij}.
\]
The unique metric projection \(\regOp(y)\) (Theorem \ref{thm: least squares solution})
is therefore given by
\[
    \regOp(y) = \frac1n\mBasisVec(\cdot)^\transpose \modelVander^\conT \Omega_n y
    = \frac1n\sum_{i=1}^n  \underbrace{\sum_{k=1}^m \mBasis_k(\cdot)\overline{\mBasis_k(x_i)}}_{=\reprK_\modelSpace^{\Omega_n}(\cdot, x_i)} (\Omega_n y)_i
    = \scp[\Bigg]{y}{
        \underbrace{
        \begin{pmatrix}
            \reprK_\modelSpace^{\Omega_n}(x_1,\cdot)\\
            \vdots\\
            \reprK_\modelSpace^{\Omega_n}(x_n,\cdot)
        \end{pmatrix}
    }_{\eqcolon \reprK_\modelSpace^{\Omega_n}(X,\cdot)}
    }_{\Omega_n},
\]
since \((\modelVander)_{ij} = \mBasis_j(x_i)\) and the reproducing kernel is
Hermitian. And by definition of the Lagrange function we thus obtain \(\ell(x)=\frac1n\Omega_n \reprK_\modelSpace^{\Omega_n}(X,x)\).
The representation of the Lebesgue function follows immediately from its
Definition~\ref{def: lebesgue function and constant} and
the representation of the Lagrange function.

\subsubsection{Proof of Lemma \ref{lem: variational characterization}: Variational characterization of leverage}

By the reproducing property of \(K^*_\modelSpace\) and Cauchy-Schwarz
\[
    \abs{\model(x)}
    = \langle \model, \reprK^*_\modelSpace(\cdot, x)\rangle_*
    \le \sqrt{\reprK^*_\modelSpace(x,x)}\norm{\model}_{*}
    = \lever^*_\modelSpace(x) \norm{\model}_{*}.
\]
Since this bound is an equality for \(\model^* = K^*_\modelSpace(\cdot, x)\in \modelSpace\)
we have \eqref{eq: pointwise variational repro kernel bound}.
Equation \eqref{eq: variational weighted sup bound} then follows directly
from \eqref{eq: pointwise variational repro kernel bound} and the definition
of the weighted sup-norm. 

\subsubsection{Proof of Theorem~\ref{thm: leverage function bound}: Leverage function bounds}

\begin{steps}
\item \textbf{First bound on Lebesgue function.}
Recall that we have \(\ell(x) = \frac1n\Omega_n
\reprK_\modelSpace^{\Omega_n}(X,x)\) by Theorem \ref{thm: reproducing kernel
representation} with \(\reprK_\modelSpace^{\Omega_n}(X,x) = (\reprK_\modelSpace^{\Omega_n}(x_i,x))_{i=1}^n\).
Using \(D\coloneq \diag[\frac1{\weight(x_1)}, \dots, \frac1{\weight(x_n)}]\)
and the definition of the Lebesgue function we thus have
\begin{align*}
    \lebfun{\weight}{\linProj}{x}
    &= \weight(x) \sum_{i=1}^n \frac{\abs{\ell_i(x)}}{\weight(x_i)}
    \\
    &= \weight(x) \sup_{\substack{\sigma_i\in\field\\\abs{\sigma_i}=1}}
       \sum_{i=1}^n \frac{\overline{\sigma_i}\ell_i(x)}{\weight(x_i)}
    \\
    &= \weight(x) \sup_{\abs{\sigma_i}=1} \scp[\Big]{\reprK_\modelSpace^{\Omega_n}(X,x)}{D\sigma}_{\Omega_n}
    \\
    \overset{\text{Cauchy-Schwarz}}&\le
    \weight(x)
    \abs[\big]{\reprK_\modelSpace^{\Omega_n}(X,x)}_{\Omega_n}
    \sup_{\abs{\sigma_i}=1} \abs{D\sigma}_{\Omega_n}.
\end{align*}
This yields the bound \eqref{eq: rk bound Omega}, since
\begin{equation}
    \label{eq: Dsigma bound with A_\weight(Omega_n)}
    \sup_{\abs{\sigma_i}=1} \abs{D\sigma}_{\Omega_n}^2
    = \sup_{\abs{\sigma_i}=1}
    \frac1n\sum_{i,j=1}^n \frac{\overline{\sigma_i}(\Omega_n)_{ij}\sigma_j}{\weight(x_i)\weight(x_j)}
    \le \frac1n\sum_{i,j=1}^n \frac{\abs{(\Omega_n)_{ij}}}{\weight(x_i)\weight(x_j)} = A_\weight(\Omega_n)^2,
\end{equation}
and 
\(\norm{f}_{\Omega_n} \overset{\text{def.}}= \abs{\eval_X[f]}_{\Omega_n}\) implies
\[
    \abs[\big]{\reprK_\modelSpace^{\Omega_n}(X,x)}_{\Omega_n}^2
    = \norm[\big]{\reprK_\modelSpace^{\Omega_n}(\cdot, x)}_{\Omega_n}^2
    = \scp{\reprK_\modelSpace^{\Omega_n}(\cdot, x)}{\reprK_\modelSpace^{\Omega_n}(\cdot, x)}_{\Omega_n}
    \overset{\text{reproducing}}= \reprK_\modelSpace^{\Omega_n}(x, x).
\]

\item \textbf{Second bound on Lebesgue function.}
The bound \eqref{eq: rk bound general} follows from \eqref{eq: rk bound Omega}
and the variational characterization of the diagonal of reproducing kernels
(Lemma \ref{lem: variational characterization}). Specifically, by Lemma \ref{lem: gram bound equivalence}
we have \(\norm{\model}_* \le \sqrt{\opnorm{\Gram^{-1}}_2}\norm{\model}_{\Omega_n}\) and therefore
\begin{align}
    \label{eq: leverage function bound}
    \lever_\modelSpace^{\Omega_n}(x)
    \overset{\text{Lem.~\ref{lem: variational characterization}}}&= 
    \sup_{\model\in \modelSpace, \norm{\model}_{\Omega_n} \le 1} \abs{\model(x)}
    \\
    \nonumber
    \overset{\text{Lem.~\ref{lem: gram bound equivalence}}}&\le \sqrt{\opnorm{\Gram^{-1}}_2} \sup_{\model\in \modelSpace, \norm{\model}_* \le 1} \abs{\model(x)}
    \overset{\text{Lem.~\ref{lem: variational characterization}}}= \sqrt{\opnorm{\Gram^{-1}}_2} \lever_\modelSpace^*(x).
\end{align}
\end{steps}

\subsubsection{Generalization of Theorem~\ref{thm: leverage function bound}: Bounds on the Lebesgue function}

While the proof of the general kernel bound \eqref{eq: rk bound general} via the kernel
bound using \(\reprK_\modelSpace^{\Omega_n}\) in \eqref{eq: rk bound Omega} is instructive
and also implies that \(\scp{\cdot}{\cdot}_{\Omega_n}\) minimizes \eqref{eq: rk bound general} over
all inner products \(\scp{\cdot}{\cdot}_*\), it is possible to avoid the
assumption that the Gram matrix is necessarily invertible if one proves
\eqref{eq: rk bound general} directly. The following Theorem takes such a
direct approach and \eqref{eq: trace-like bound on lebesgue function} is a
generalization of
\eqref{eq: rk bound general} with the observation that
\[
    \reprK^*_\modelSpace(x, x) = \sum_{k=1}^m \abs{\mBasis_k(x)}^2= \abs{\mBasisVec(x)}_2^2
\]
by definition of \(\mBasisVec(x) = (\mBasis_1(x), \dots, \mBasis_m(x))^\transpose\) for orthonormal \(\mBasis_k\)
with respect to \(\scp{\cdot}{\cdot}_*\).

\begin{theorem}
    \label{thm: L2 bounds on Lebesgue function}
    Let \(\linProj\) be a GLS regression projection induced by \(\Omega_n\) and \(\lebfun{\weight}{\linProj}{x}\) its Lebesgue function.
    Let \(\mBasisVec(x) = (\mBasis_1(x), \dots, \mBasis_m(x))^\transpose\) be the vector of the basis functions \(\mBasis_1, \dots, \mBasis_m\) of
    \(\modelSpace\) evaluated at \(x\) and \(\Gram_{ij} = \scp{\mBasis_j}{\mBasis_i}_{\Omega_n}\) be the Gram matrix.
    Then for all \(x\in \domain\)
    \begin{align}
        \label{eq: trace-like bound on lebesgue function}
        \lebfun{\weight}{\linProj}{x}
        &\le \weight(x)
        A_\weight(\Omega_n)
        \sqrt{\opnorm{\Gram^\dagger}_2}\abs{\mBasisVec(x)}_2.
        & A_\weight(\Omega_n)\coloneq \sqrt{
            \frac1n\sum_{i,j=1}^n \frac{\abs{(\Omega_n)_{ij}}}{\weight(x_i)\weight(x_j)}
        }
        \\
        \label{eq: operator norm bound on lebesgue function}
        \lebfun{\weight}{\linProj}{x}
        &\le \weight(x)
        A_\weight(\opnorm{\Omega_n}_2\identity)
        \sqrt{\opnorm{\Gram^\dagger}_2}\abs{\mBasisVec(x)}_2.
    \end{align}
\end{theorem}

\begin{remark}[Degenerate Gram matrix]
    It is important to emphasize again, that if the Gram matrix is not invertible,
    the model space \(\modelSpace\) must be redefined to be the image of the
    regression projection \(\linProj\). If one selects \(\Omega_n=0\) for example,
    then \(\Gram^\dagger = 0\) and the bound may look great, but then
    \(\linProj\equiv 0\) and consequently \(\modelSpace = \set{0}\). This causes
    the metric projection error \(\error_\modelSpace^\infty[f]\) to be large
    instead! For the purposes of optimizing this bound it is therefore sensible
    to assume \(\Gram\) to be invertible as was done in Theorem \ref{thm: leverage function bound}.
\end{remark}

\begin{remark}[Quality of the bounds]
    Observe that in the weighted case, the first bound is smaller
    than the second
    \[
        A_\weight(\Omega_n)^2 = \frac1n\sum_{i=1}^n \frac{\omega_i}{\weight(x_i)^2}
        \le \frac1n\sum_{i=1}^n \frac{\opnorm{\Omega_n}_2}{\weight(x_i)^2}
        = A_\weight(\opnorm{\Omega_n}_2\identity)^2.
    \]
    since \(\opnorm{\Omega_n}_2=\max_i \omega_i\) is the largest eigenvalue of \(\Omega_n\).
    In the GLS case it is not as clear which bound is preferable, since
    the second bounds avoids taking the entrywise absolute value. It may therefore be preferable
    if the off-diagonal entries of \(\Omega_n\) exhibit substantial cancellation.
\end{remark}

\begin{proof}
    \begin{steps}
    \item \textbf{Absolute entries bound}
    By Example \ref{ex: lagrange function of least squares} 
    \(\ell(x)^\conT = \frac1n\mBasisVec^\transpose(x) \Gram^\dagger \modelVander^\conT \Omega_n\) and consequently
    \begin{align*}
        \lebfun{\weight}{\linProj}{x}
        &= \weight(x) \sum_{i=1}^n \frac{\abs{\ell_i(x)}}{\weight(x_i)}
        \\
        &= \frac{\weight(x)}n \sum_{i=1}^n \frac{\abs{\mBasisVec^\transpose(x) \Gram^\dagger \modelVander^\conT \Omega_n e_i}}{\weight(x_i)}
        = \sup_{\abs{\sigma_i}=1} \frac{\weight(x)}n
          \abs[\Big]{\mBasisVec^\transpose(x) \Gram^\dagger \modelVander^\conT \Omega_n D\sigma}.
    \end{align*}
    Using \(D = \diag[\frac1{\weight(x_1)}, \dots, \frac1{\weight(x_n)}]\) we can write this as
    \[
        \lebfun{\weight}{\linProj}{x}
        = \sup_{\abs{\sigma_i}=1} \weight(x)
          \abs[\big]{\scp[\big]{D \sigma}{\modelVander \Gram^\dagger \overline{\mBasisVec(x)}}_{\Omega_n}}
        \overset{\text{C.S.}}\le \sup_{\abs{\sigma_i}=1}\weight(x)
          \abs[\big]{\modelVander \Gram^\dagger \overline{\mBasisVec(x)}}_{\Omega_n} \abs{D\sigma}_{\Omega_n}.
    \]
    Using \(\sup_{\abs{\sigma_i}=1} \abs{D\sigma}_{\Omega_n} \le A_\weight(\Omega_n)\) from \eqref{eq: Dsigma bound with A_\weight(Omega_n)}
    and
    \[
        \abs[\big]{\modelVander \Gram^\dagger \overline{\mBasisVec(x)}}_{\Omega_n}^2
        = \mBasisVec^\transpose(x) \Gram^\dagger \underbrace{\tfrac1n\modelVander^\conT \Omega_n \modelVander}_{=\Gram} \Gram^\dagger \overline{\mBasisVec(x)}
        = \mBasisVec^\transpose(x) \Gram^\dagger \overline{\mBasisVec(x)}
        \le \opnorm{\Gram^\dagger}_2 \abs{\mBasisVec(x)}_2^2
    \]
    we obtain the claim.

    \item \textbf{Operator norm bound}
    Observe that by the Cauchy-Schwarz inequality
    \[
        \lebfun{\weight}{\linProj}{x}
        = \weight(x) \sum_{i=1}^n \frac{\abs{\ell_i(x)}}{\weight(x_i)}
        \overset{\text{C.S.}}\le \weight(x) \sqrt{\sum_{i=1}^n \frac1{\weight(x_i)^2}} \abs{\ell(x)}_2.
    \]
    Note that for \(\weight\equiv 1\) this is simply the bound \(\abs{\ell(x)}_1 \le \sqrt{n}\abs{\ell(x)}_2\). 
    By Example \ref{ex: lagrange function of least squares} 
    we have \(\ell(x)^\conT = \frac1n\mBasisVec^\transpose(x) \Gram^\dagger \modelVander^\conT \Omega_n\) and consequently
    \[
        \abs{\ell(x)}_2^2
        = \frac1{n^2}\underbrace{\mBasisVec^\transpose(x) \Gram^\dagger \modelVander^\conT \Omega_n^{\frac12}}_{\eqcolon v^\conT}\Omega_n \underbrace{\Omega_n^{\frac12} \modelVander \Gram^\dagger \overline{\mBasisVec(x)}}_{= v}
        \le \frac{\opnorm{\Omega_n}_2}{n^2} \abs{v}_2^2.
    \]
    Combining this with
    \[
        \abs{v}_2^2
        = \mBasisVec^\transpose(x) \Gram^\dagger \underbrace{\modelVander^\conT \Omega_n \modelVander}_{=n\Gram} \Gram^\dagger \overline{\mBasisVec(x)}
        = n\mBasisVec^\transpose(x) \Gram^\dagger \overline{\mBasisVec(x)}
        \le n\opnorm{\Gram^\dagger}_2\abs{\mBasisVec(x)}_2^2
    \]
    results in \(\abs{\ell(x)}_2 \le \sqrt{\frac{\opnorm{\Omega_n}_2}{n}}\sqrt{\opnorm{\Gram^\dagger}_2}\abs{\mBasisVec(x)}_2\) and consequently
    the claimed GLS bound.
\end{steps}
\end{proof}

\subsubsection{Absolute value around \texorpdfstring{\(\Omega_n\)}{Ω} is necessary}
\label{sec: absolute value around Omega is necessary}

Since \(A_\weight(\Omega_n) = \norm{\tfrac1\weight}_{\Omega_n}\) in the weighted
case, the reader might hypothesize as the authors did that
\(A_\weight(\Omega_n)\) should be replaceable by
\[
    \norm{\tfrac1\weight}_{\Omega_n}
    = \sqrt{\frac1n\sum_{i,j=1}^n \frac{(\Omega_n)_{ij}}{\weight(x_i)\weight(x_j)}}
\]
in the general GLS case. This is not possible however.
Specifically, there is already a counterexample in the unweighted case
with \(\weight=\ones\) and \(\norm{\cdot}_\weight = \norm{\cdot}_\infty\),
where
\(
    \norm{\tfrac1\weight}_{\Omega_n}
    = \norm{\ones}_{\Omega_n}.
\)
Indeed the following bound is \emph{false}
\begin{equation}
    \label{eq: false norm conjecture}
    \lebfun{\infty}{\linProj}{x}
    = \abs{\ell(x)}_1
    \le \norm{\ones}_{\Omega_n}\sqrt{\opnorm{\Gram^{-1}}_2} \abs{\mBasisVec(x)}_2.
\end{equation}
Since \(A_\weight(\Omega_n) = \sqrt{\frac1n\tr(\Omega_n)}\) in the unweighted WLS case, one
may alternatively conjecture the bound
\begin{equation}
    \label{eq: false trace conjecture}    
    \lebfun{\infty}{\linProj}{x}
    = \abs{\ell(x)}_1
    \le \sqrt{\frac{\tr(\Omega_n)}{n} \opnorm{\Gram^{-1}}_2} \abs{\mBasisVec(x)}_2
\end{equation}
for GLS regression, which is also \emph{false}.

\begin{example}[Counterexample to the norm conjecture \eqref{eq: false norm conjecture}]
    We need to construct an example, where
    \[
        \lebfun{\infty}{\linProj}{x}
        = \abs{\ell(x)}_1
        > \norm{\ones}_{\Omega_n}\sqrt{\opnorm{\Gram^{-1}}_2} \abs{\mBasisVec(x)}_2
    \]
    Let \(x_1\neq x_2 \in \domain \coloneq \real\), select any probability
    measure \(\mu\) on \(\real\), \(\modelSpace \coloneq \Span{\set{\ones}}\)
    with \(\ones(x) = 1\). Observe that since
    \(\mu\) is a probability measure \(\mBasisVec(x) = (\ones(x))\) is an orthonormal basis.
    Define
    \[
        \Omega_n = \begin{pmatrix}
            1 & -2
            \\
            -2 & 5
        \end{pmatrix}.
    \]
    such that 
    \[
        \Omega_n \begin{pmatrix}
            1\\1
        \end{pmatrix}
        = \begin{pmatrix}
            -1\\3
        \end{pmatrix}
    \]
    and therefore \(\norm{\ones}_{\Omega_n}^2 = 1 = \Gram\). This implies
    \[
        \norm{\ones}_{\Omega_n}\sqrt{\opnorm{\Gram^{-1}}_2} \abs{\mBasisVec(x)}_2
        = 1.
    \]
    On the other hand,
    \[
        \ell(x)^\conT = \tfrac1n\mBasisVec^\transpose(x) \Gram^{-1} \modelVander^\conT \Omega_n
        = \tfrac1{2} (1,1) \Omega_n
        = \tfrac1{2}(-1, 3)
    \]
    such that
    \[
        \abs{\ell(x)}_1
        = \tfrac12 + \tfrac32 = 2 > 1 =  
        \norm{\ones}_{\Omega_n}\sqrt{\opnorm{\Gram^{-1}}_2} \abs{\mBasisVec(x)}_2.
    \]
\end{example}

\begin{example}[Counterexample to the trace conjecture \eqref{eq: false trace conjecture}]
    We need to construct an example, where
    \[
        \lebfun{\infty}{\linProj}{x}
        = \abs{\ell(x)}_1
        > \sqrt{\frac{\tr(\Omega_n)}{n} \opnorm{\Gram^{-1}}_2} \abs{\mBasisVec(x)}_2
    \]
    Let \(x_1\neq x_2 \in \domain \coloneq \real\), select any probability
    measure \(\mu\) on \(\real\), \(\modelSpace \coloneq \Span{\set{\ones}}\)
    with \(\ones(x) = 1\) and define
    \[
        \Omega_{n,\epsilon} = \begin{pmatrix}
            1 & 1-\epsilon \\
            1-\epsilon & 1
        \end{pmatrix}
        \quad\text{with}\quad \tr(\Omega_{n,\epsilon}) = 2.
    \]
    Then
    \begin{alignat*}{3}
        \modelVander \overset{\text{def.}}&= (\ones(x_1), \ones(x_2))^\transpose 
        &&=(1, 1)^\transpose
        \\
        \Gram
        \overset{\text{def.}}&= \tfrac12\modelVander^\conT \Omega_{n,\epsilon} \modelVander
        &&= 2-\epsilon
        \\
        \abs{\mBasisVec(x)}_2^2
        \overset{\text{def.}}&= \ones(x)^2
        &&= 1,
        \\
        \ell(x)^\conT
        \overset{\text{def.}}&= \tfrac12\mBasisVec(x)^\transpose \Gram^{-1} \modelVander^\conT \Omega_{n,\epsilon}
        &&= \tfrac1{2}(1, 1)
    \end{alignat*}
    Due to
    \(\opnorm{\Gram^{-1}}_2 = \tfrac1{2-\epsilon}\)
    we consequently have
    for \(\epsilon \in (0,1)\)
    \[
        \sqrt{\frac{\tr(\Omega_{n,\epsilon})}{2} \opnorm{\Gram^{-1}}_2} \abs{\mBasisVec(x)}_2 = \tfrac1{\sqrt{2-\epsilon}} < 1 = \abs{\ell(x)}_1.
    \]
\end{example}

\subsubsection{Proof of Lemma~\ref{lem: reproducing kernel}: Reproducing kernel}

This is a standard result in the theory of reproducing kernel Hilbert spaces.
To allow the reader to skip the general theory of reproducing kernel Hilbert spaces in
infinite dimension we give a direct proof that avoids the use of the Riesz representation theorem.

We use \(\reprK_\modelSpace^*\) to denote the explicit kernel representation and show
that it has the reproducing property, thereby proving existence. Afterwards we
prove uniqueness. Let \(\model \in \modelSpace\) have the basis representation
\(\model = \sum_{j=0}^{m-1} c_j \mBasis_j\) for some \(c_0, \dots, c_{m-1}\in
\field\). Then
\[
    \scp{\model}{\reprK_\modelSpace^*(\cdot, y)}_*
    = \sum_{j,k=0}^{m-1} c_j\mBasis_k(y)\underbrace{\scp{\mBasis_j}{\mBasis_k}_*}_{=\delta_{jk}}
= \sum_{k=0}^{m-1} c_k \mBasis_k(y)
    = \model(y).
\]
For uniqueness, assume that \(K\) is another kernel with the desired properties.
Then, using the reproducing property of \(\reprK_\modelSpace^*\) and \(K\) and the Hermitian kernel
property of \(\reprK_\modelSpace^*\) we have
\begin{align*}
    K(x,y)
    \overset{\text{repr.}}&= \scp{K(\cdot, y)}{\reprK_\modelSpace^*(\cdot, x)}_*
    \\
    \overset{\text{Herm.}}&= \overline{\scp{\reprK_\modelSpace^*(\cdot, x)}{K(\cdot, y)}_*}
    \overset{\text{repr.}}= \overline{\reprK_\modelSpace^*(y, x)} \overset{\text{Herm.}}= \reprK_\modelSpace^*(x, y).
\end{align*}

\subsubsection{Proof of Lemma~\ref{lem: optimality of constant weighted leverage bound}: Leverage lower bound}

This result is well known \citep[e.g.][]{narayanChristoffelFunctionWeighted2016}. Since the
proof is very short and instructive we include it here for completeness:
Let \(\mBasis_1, \dots, \mBasis_m\) be an orthonormal basis of \(\modelSpace\) with respect to \(\scp{\cdot}{\cdot}_\mu\), then
\[
    \sup_{x\in \domain} \weight^2(x) K^\mu_m(x,x)
    \ge \frac1{\nu(\domain)}\int_\domain K^\mu_m(x,x) d\mu(x)
    = \frac1{\nu(\domain)}\sum_{k=1}^{m} \underbrace{\int_\domain \abs{\mBasis_k(x)}^2 d\mu(x)}_{=\norm{\mBasis_k}_{\mu}^2 = 1}
    = \frac{m}{\nu(\domain)}.
\]
If \(\weight^2(x)K^\mu_m(x,x)\) is constant in \(x\), then the first inequality is an equality.
Conversely, the first inequality is strict if \(\weight^2(x)K^\mu_m(x,x)\) is not \(\nu\)-almost surely constant in \(x\).

\subsection{Restricted range inequalities}

\subsubsection{Proof of Lemma \ref{lem: restricted lebesgue constant}: Restricted Lebesgue constant}

In the case where \(\linProj\) is an \(L^2(\mu)\)-projection
define \(\nu(dx) = \frac1{\weight^2(x)}\mu(dx)\). Since \(\mu\) is
\(\sigma\)-finite, we have that \(\nu\) is also \(\sigma\). Indeed, let \(E_n
\uparrow \domain\) with \(\mu(E_n) < \infty\). Define
\[
    E_{n,m} \coloneq \set{x\in E_n : \tfrac1{\weight^2(x)} \le m}
\]
then \(\nu(E_{n,m}) \le m \mu(E_n) < \infty\) and \(\bigcup_{n,m} E_{n,m} = \domain\)
such that an enumeration and partial union of the countable sets proves the \(\sigma\)-finiteness of \(\nu\).
We thus have \(\mu(dx) = \weight^2(x)\nu(dx)\) for some \(\sigma\)-finite \(\nu\)
and Theorem \ref{thm: L2 proj Lebesgue function and constant are operator norms}
therefore implies the operator theoretic characterization of the Lebesgue
function
\begin{equation}
    \label{eq: (restricted lebesgue const) operator norm characterization}    
    \lebfun{\weight}{\linProj}{x}
    = \opnorm{\eval_x\linProj}_{\weight} \weight(x)
    = \sup_{\substack{f \in \funSpace \\ \norm{f}_\weight \le 1}} \abs{\linProj f(x)}\weight(x).
\end{equation}
This identity then implies
\begin{align*}
    \lebconst{\weight}{\linProj}
    \overset{\text{def.}}= \sup_{x\in \domain} \lebfun{\weight}{\linProj}{x}
    \overset{\eqref{eq: (restricted lebesgue const) operator norm characterization}}&
    = \sup_{\substack{f \in \funSpace \\ \norm{f}_\weight \le 1}} \sup_{x\in \domain} \abs{\linProj f(x)}\weight(x)
    \\
    &\le \sup_{\substack{f \in \funSpace \\ \norm{f}_\weight \le 1}} C\sup_{x\in S_\modelSpace} \abs{\linProj f(x)}\weight(x)
    \overset{\eqref{eq: (restricted lebesgue const) operator norm characterization}}
    = C \sup_{x\in S_\modelSpace} \lebfun{\weight}{\linProj}{x},
\end{align*}
where the inequality is the application of the restricted range inequality (Definition \ref{def: restricted range inequality})
to \(\norm{\linProj f}_\weight = \sup_{x\in \domain}\abs{\linProj f(x)}\weight(x)\),
since \(\linProj f \in \modelSpace\). Equality in the case of \(C=1\) follows from
\(S_\modelSpace \subseteq \domain\) and the definition of the supremum.

If we were willing to assume the function space \(\funSpace\) is sufficiently rich and
the evaluation points \(X=(x_1, \dots, x_n)\) distinct, then Lemma \ref{lem: lebesgue function is operator norm}
would imply \eqref{eq: (restricted lebesgue const) operator norm characterization} for
the case where \(\linProj\) is a regression projection too and we would have finished the
proof. To avoid this assumption we will make a more direct argument. Observe that
by Definition \ref{def: lebesgue function and constant} we have
\begin{align*}
    \lebconst{\weight}{\linProj}
    \overset{\text{def.}}= \sup_{x\in \domain} \lebfun{\weight}{\linProj}{x}
    \overset{\text{def.}}&= \sup_{x\in \domain} \weight(x) \sum_{i=1}^n \frac{\abs{\ell_i(x)}}{\weight(x_i)}
    \\
    &= \sup_{\substack{\sigma\in\field^n\\\abs{\sigma_i}=1}} \sup_{x\in \domain}
    \weight(x) \abs[\Big]{\sum_{i=1}^n \frac{\sigma_i\ell_i(x)}{\weight(x_i)}}
    =  \sup_{\substack{\sigma\in\field^n\\\abs{\sigma_i}=1}} \norm{\psi_\sigma}_\weight
\end{align*}
with \(\psi_\sigma(x) \coloneq \sum_{i=1}^n \frac{\sigma_i\ell_i(x)}{\weight(x_i)}\).
Since \(\ell_i = \regOp(e_i)\in \modelSpace\) we have that \(\psi_\sigma \in \modelSpace\) as a linear combination.
Hence we may apply the restricted range inequality to \(\psi_\sigma\) and obtain
\[
    \lebconst{\weight}{\linProj}
    = \sup_{\substack{\sigma\in\field^n\\\abs{\sigma_i}=1}} \norm{\psi_\sigma}_\weight
    \le C \sup_{\substack{\sigma\in\field^n\\\abs{\sigma_i}=1}} \norm{\psi_\sigma\ind{S_\modelSpace}}_\weight
    \overset{\text{def.}}= C \sup_{x\in S_\modelSpace} \weight(x)\sum_{i=1}^n \frac{\abs{\ell_i(x)}}{\weight(x_i)}
    = C \sup_{x\in S_\modelSpace} \lebfun{\weight}{\linProj}{x}.
\]

\subsubsection{Proof of Theorem \ref{thm: restricted range inequalities for polynomials}: Restricted range inequalities for polynomials}

Except for the optimality this result follows directly from \citep[Thm.~6.2]{lubinskySurveyWeightedPolynomial2007}.
We will deduce the approximate optimality of the MRS number \(a_n\) from the following Lemma due
to \citet[Lem.~2]{szabadosWeightedLagrangeHermite1997}.
Indeed, assume that the restricted range inequality holds with \(a>0\).
Since \(p_k \in \poly_k\) as it has exactly degree \(k\), we then have that
\[
    \norm{p_k}_\weight = \norm{p_k\ind{[-a,a]}}_\weight = \sup_{\xi \in [-a,a]} \abs{p_k(\xi)}\weight(\xi).
\]
Since \([-a, a]\) is compact there must exist a maximizer \(\xi_k\in [-a,a]\). Since
\(Q\) and therefore \(\weight\) is even, we have that \(p_k(-x) = (-1)^k
p_k(x)\) \citep[Thm.~1.17]{gautschiOrthogonalPolynomialsComputation2004} and
thus we may assume \(\xi_k \ge 0\). By Lemma \ref{lem: szabados lemma} \citep[Lem.~2]{szabadosWeightedLagrangeHermite1997} we
therefore have
\[
    a_k\bigl(1-c_1k^{-\frac23}\bigr) \le \xi_k \le a
\]
for sufficiently large \(k\). After possibly replacing the constant \(c_1\)
for small \(k\) this proves the lower bound on any \(a\) that satisfies the restricted range inequality.
The upper bound on the infimum is trivial since \(a_k\) satisfies the restricted range inequality.

\begin{lemma}[{\citet[Lem.~2]{szabadosWeightedLagrangeHermite1997}}]
    \label{lem: szabados lemma}
    Assume the setting of the approximate-optimality statement of Theorem
    \ref{thm: restricted range inequalities for polynomials}.  Let
    \((p_k)_{k=0}^\infty\) be orthonormal polynomials with respect to
    \(L^2(\mu)\) with \(\mu(dx) = \weight^2(x) dx\) and \(p_k\) of degree \(k\).
    Let \(\xi_k \ge 0\) be such that
    \[
        \abs{p_k(\xi_k)}\weight(\xi_k) = \norm{p_k}_\weight.
    \]
    Then for some constant \(c_1>0\) independent of \(k\) we have for all sufficiently large \(k\in \nat\)
    \[
        a_k\bigl(1-c_1k^{-\frac23}\bigr) \le \xi_k \le a_k.
    \]
\end{lemma}

\subsection{\texorpdfstring{\(L^p(\mu)\)}{Lp(µ)} norm bounds}

\subsubsection{Proof of Theorem~\ref{thm: Lp approximation is hopeless}: Impossibility of adversarial \texorpdfstring{\(L^p(\mu)\)}{Lp(µ)} bounds}
\label{sec: lebesgue constant for Lp norms}

While the Lebesgue Lemma (Lemma \ref{lem: lebesgue}) applies to arbitrary
norms, the resulting bounds are not useful for \(L^p(\mu)\)-norms.
In fact, for atomless \(\mu\) the following Lemma shows that the generalized
Lebesgue constant \(\norm{\linProj}_{L^p(\mu)}\) and Lebesgue function
\(\opnorm{\eval_x \linProj}_{L^p(\mu)}\) are generally infinite. The technical
assumptions are that the function space is big enough to contain bump functions
and that the model space is not in the equivalence class of zero for the
semi-norm. We will consider this degenerate case in Lemma \ref{lem: model space
in equivalence class of zero}.

\begin{mdframed}[style=default]
\begin{lemma}[\(L^p(\mu)\)-operator norm of regression projections]
    \label{lem: Lp operator norm of regression projections}
    Assume \(\funSpace\subseteq \mathcal{L}^p(\mu)\) and let \(\linProj\colon \funSpace\to \modelSpace\) be a regression projection with distinct evaluation points
    \((x_1,\dots, x_n)\in \domain^n\) with \(\domain \subseteq \real^\dims\) and
    let \(\mu\) be a measure on \(\domain\) that is finite on bounded measurable sets with
    \(\mu(\set{x_1, \dots, x_n}) = 0\),
    assume \(C_c^\infty(\domain; \field) \subseteq \funSpace\) and \(p\in [1, \infty)\), then
    the following holds:
    \begin{enumerate}[label=(\roman*),noitemsep]
        \item If there exists \(\model \in \modelSpace\) such that \(\norm{\model}_* \neq
        0\), then 
        \[
            \opnorm{\linProj}_{L^p(\mu)}^* = \sup_{f\in \funSpace} \frac{\norm{\linProj f}_*}{\norm{f}_{L^p(\mu)}} = \infty.
        \]

        \item If there exists \(\model \in \modelSpace\) such that \(\abs{\model(x)} \neq 0\),
        then
        \[
            \opnorm{\eval_x \linProj}_{L^p(\mu)} = \sup_{f\in \funSpace} \frac{\abs{\linProj f(x)}}{\norm{f}_{L^p(\mu)}} = \infty.
        \]
    \end{enumerate}
\end{lemma}
\end{mdframed}
\begin{remark}[Possible generalization]
    \label{rem: generalization of Lp operator norm lemma}
    To avoid the assumptions that none of the \(x_i\) are atoms of \(\mu\),
    a sufficient condition for the same conclusion is that the following subset
    of the model space
    \[
        \tilde \modelSpace \coloneq \modelSpace \cap \set[\Big]{f\in \funSpace : f(x_i) = 0 \text{ for all } i \in \set{1,\dots, n} \text{ with }\mu(\set{x_i}) > 0}
    \]
    is not in the equivalence class of zero for the seminorm \(\norm{\cdot}_*\).
    That is, there exists \(\model \in \tilde\modelSpace\) such that
    \(\norm{\model}_* > 0\) or \(\abs{\model(x)} > 0\) respectively.
    Since \(\tilde \modelSpace = \modelSpace\) under the assumptions of the Lemma above,
    this is a generalization of the above Lemma.
\end{remark}

\begin{proof}
    We will prove the generalization explained in Remark \ref{rem: generalization of Lp operator norm lemma}.
    Let \(\model \in \tilde\modelSpace\) with \(\norm{\model}_* >0\) or \(\abs{\model(x)} > 0\) respectively, and let \(I = \set{i : \mu(x_i) = 0}\).
    For the euclidean distance \(\abs{x-x_i}\) use the bump functions
    \[
        g_i(x) \coloneq  g(\abs{x-x_i}/\epsilon)\in C_c^\infty(\domain;\real) \qquad g(x)\coloneq \exp(1-\tfrac{1}{1-x^2})\ind{\abs{x}< 1}(x) 
    \]
    to construct
    \[
        f_\epsilon(x) \coloneq \sum_{i\in I} \model(x_i) g_i(x) \in C_c^\infty(\domain; \field)
    \]
    with the properties
    \begin{enumerate}[label=(\roman*),noitemsep]
        \item \(f_\epsilon(x_i) = \model(x_i)\) for all  \(i\in\set{1,\dots, n}\) for sufficiently small \(\epsilon\)
        \item \(\abs{f_\epsilon(x)} \le h(x) \coloneq \sum_{i\in I} \abs{\model(x_i)} \ind{B_{\epsilon_0}(x_i) \cap \domain}(x)\) for \(\epsilon<\epsilon_0\) and \(x\in \domain\)
        \item\label{it: limit of f_epsilon} \(\lim_{\epsilon \to 0} f_\epsilon(x) = \sum_{i\in I} \model(x_i)\delta_{xx_i} \) for all \(x\in \domain\),
        where \(\delta_{xy}\) is the Kronecker delta.
    \end{enumerate}
    Since \(h\in \mathcal L^p(\mu)\) as the union of \(B_{\epsilon_0}(x_i) \cap \domain\) is a bounded measurable set, we thus have by dominated convergence
    \[
        \lim_{\epsilon\to 0}\norm{f_\epsilon}_{L^p(\mu)}^p
        \overset{\text{DCT}}= \int \lim_{\epsilon \to 0} \abs{f_\epsilon (x)}^p\mu(dx)
        \overset{\mathrm{\ref{it: limit of f_epsilon}}}= \sum_{i\in I} \abs{\model(x_i)}^p \mu(\set{x_i}) \overset{\text{def. }I}= 0.
    \]
    Whereas for all \(\epsilon\) we have \(\linProj f_\epsilon = \linProj \model = \model\) since \(f_\epsilon\) coincides with \(\model\) on the evaluation
    points \(X=(x_1, \dots, x_n)\) and \(\linProj = \regOp \circ \eval_X\). Consequently
    \[
        \sup_{f\in \funSpace} \frac{\norm{\linProj f}_*}{\norm{f}_{L^p(\mu)}}
        \ge \lim_{\epsilon \to 0} \frac{\norm{\linProj f_\epsilon}_*}{\norm{f_\epsilon}_{L^p(\mu)}}
        = \lim_{\epsilon \to 0} \frac{\norm{\model}_*}{\norm{f_\epsilon}_{L^p(\mu)}} = \infty,
    \]
    for the choice \(\model \in \tilde \modelSpace\) with \(\norm{\model}_* > 0\) and
    similarly
    \[
        \sup_{f\in \funSpace} \frac{\abs{\linProj f(x)}}{\norm{f}_{L^p(\mu)}}
        \ge \lim_{\epsilon \to 0} \frac{\abs{\linProj f_\epsilon(x)}}{\norm{f_\epsilon}_{L^p(\mu)}}
        = \lim_{\epsilon \to 0} \frac{\abs{\model(x)}}{\norm{f_\epsilon}_{L^p(\mu)}} = \infty,
    \]
    for the choice \(\model \in \tilde \modelSpace\) with \(\abs{\model(x)} > 0\).
\end{proof}

What happens if the model space is in the equivalence class of zero for
a semi-norm? In that case any map into the model space achieves the
metric projection error.

\begin{lemma}[Model space in equivalence class of zero]
    \label{lem: model space in equivalence class of zero}
   If \(\norm{\model}_* = 0\) for all \(\model \in \modelSpace\), then
   for all maps \(\linProj\colon \funSpace\to \modelSpace\) we have
   for all \(f\in \funSpace\)
   \[
        \norm{\linProj f - f}_* = \norm{f}_* = \error_\modelSpace^*[f].
   \]
\end{lemma}
\begin{proof}
    For any \(\model \in \modelSpace\) we have
    \[
        \norm{f}_* + \underbrace{\norm{\model}_*}_{=0}
        \ge \norm{f-\model}_*
        \ge \norm{f}_* - \underbrace{\norm{\model}_*}_{=0},
    \]
    Consequently, \(\norm{f}_* = \norm{f-\model}_*\) for all \(\model \in \modelSpace\) and 
    in particular
    \[
        \norm{f-\linProj f}_*
        = \norm{f}_*
        = \inf_{\model \in \modelSpace} \norm{f-\model}_*
        = \error_\modelSpace^*[f].
        \qedhere
    \]
\end{proof}

\paragraph*{Proof of Theorem \ref{thm: Lp approximation is hopeless}.}
    If \(\norm{\model}_{L^p(\mu)} = 0\) for all \(\model \in \modelSpace\), then this is Lemma \ref{lem: model space in equivalence class of zero}.
    For \(p=q\) the other case follows directly from Lemma \ref{lem: Lp operator norm of regression projections}
    and \eqref{eq: lebesgue lower bound} in Lemma \ref{lem: lebesgue}  and the measure \(\mu\) only needs to be
    finite on bounded measurable sets.
    In general, if there exists \(\model \in \modelSpace\) with \(\norm{\model}_{L^p(\mu)} > 0\), then by
    Lemma \ref{lem: Lp operator norm of regression projections}
    \[
        \opnorm{\linProj}_{L^q(\mu)}^{L^p(\mu)}
        = \sup_{f\in \funSpace} \frac{\norm{\linProj f}_{L^p(\mu)}}{\norm{f}_{L^q(\mu)}} = \infty.
    \]
    That is for any \(n \in \nat\) there exists \(f_n \in \funSpace\) such that
    \[
        \norm{\linProj f_n}_{L^p(\mu)} \ge n
        \quad \text{and} \quad
        \norm{f_n}_{L^q(\mu)} \le 1.
    \]
\begin{casebycase}
    \item \(p \le q\): Then by Lemma \ref{lem: Lp inclusion}
    we have
    \[
        -\norm{f_n}_{L^p(\mu)} \ge -\mu(\domain)^{\frac1p - \frac1q}\norm{f_n}_{L^q(\mu)}
    \]
    and consequently
    \begin{align*}
        \norm{\linProj f_n - f_n}_{L^p(\mu)}
        &\ge \norm{\linProj f_n }_{L^p(\mu)} - \norm{f_n}_{L^p(\mu)}
        \\
        &\ge \norm{\linProj f_n }_{L^p(\mu)} - \mu(\domain)^{\frac1p - \frac1q}\norm{f_n}_{L^q(\mu)} \ge n-1
        \\
        &\ge (n - \mu(\domain)^{\frac1p - \frac1q})\norm{f_n}_{L^q(\mu)}
        \\
        &\ge \underbrace{(n - \mu(\domain)^{\frac1p - \frac1q})}_{\to \infty}\error_\modelSpace^{L^q(\mu)}[f_n],
    \end{align*}
    where we used in the last inequality that \(\error_\modelSpace^{L^q(\mu)}[f_n] \le \norm{f_n-0}_{L^q(\mu)}\) since \(0\in \modelSpace\).
    This implies the claim
    \[
        \inf_{c>0} \set[\Big]{\forall f\in \funSpace: \norm{\linProj f - f}_{L^p(\mu)} \le c\, \error_\modelSpace^{L^q(\mu)}[f]}
        =\infty.
    \]
    \item \(p > q\) and let \(c\) be such that for all \(f\in \funSpace\)
    \[
        \norm{\linProj f - f}_{L^p(\mu)} \le c\, \error_\modelSpace^{L^q(\mu)}[f].
    \]
    Then by Lemma \ref{lem: Lp inclusion} we have for all \(f\in \funSpace\)
    \[
        \norm{\linProj f - f}_{L^q(\mu)}
        \le \mu(\domain)^{\frac1q - \frac1p}\norm{\linProj f - f}_{L^p(\mu)}
        \le c \mu(\domain)^{\frac1q - \frac1p} \error_\modelSpace^{L^q(\mu)}[f].
    \]
    But by the previous case this implies \(c \mu(\domain)^{\frac1q - \frac1p} = \infty\) and thus \(c = \infty\).
    \qed
\end{casebycase}

\subsubsection{Proof of Theorem~\ref{thm: mixed norm bound}: Mixed norm bound}

    We observe that \(\linProj\) is an orthogonal projection with respect to
    \(\norm{\cdot}_{\Omega_n}\). For any \(\model\in \modelSpace\) we therefore have
    \begin{equation}
        \label{eq: orthogonal projection Pythagorean} 
        \norm{f- \model}_{\Omega_n}^2
        = \norm{f - \linProj f}_{\Omega_n}^2
        + \norm{\linProj f - \model}_{\Omega_n}^2,
    \end{equation}
    since \(f - \linProj f \perp \modelSpace\) and \(\linProj f - \model \in \modelSpace\).
    This implies for any \(\model\in \modelSpace\) that
    \begin{align}
        \nonumber
        \norm{\linProj f - \model}_{\Omega_n}^2
        \overset{\eqref{eq: orthogonal projection Pythagorean}}\le \norm{f - \model}_{\Omega_n}^2
        \overset{\mathrm{def.}}&= \frac1n\sum_{i,j=1}^n (\Omega_n)_{ij} (f(x_i) - \model(x_i))(f(x_j) - \model(x_j))
        \\
        \label{eq: Omega norm bound by infinity norm}
        &\le \frac1n\sum_{i,j=1}^n \frac{\abs{(\Omega_n)_{ij}}}{\weight(x_i)\weight(x_j)}  \norm{f-\model}_\weight^2
    \end{align}
    Since \(\norm{\cdot}_{L^p(\mu)} \le \sqrt{\gamma} \norm{\cdot}_{\Omega_n}\) only holds on
    \(\modelSpace\) we cannot directly apply it to \(\linProj f -f\), but we
    may apply this assumption to \(\linProj f - \model\). With
    the above we thus have for any \(\model\in \modelSpace\)
    \begin{align*}
        \norm{\linProj f - f}_{L^p(\mu)}
        \overset{\Delta}&\le 
        \underbrace{\norm{\linProj f - \model}_{L^p(\mu)}}_{
            \le \sqrt{\gamma} \norm{\linProj f - \model}_{\Omega_n}
        }
        + \underbrace{\norm{\model - f}_{L^p(\mu)}}_{
            \le \norm[\big]{\tfrac1\weight}_{L^p(\mu)} \mathrlap{\norm{\model-f}_\weight \quad (\mathrm{Lem.~\ref{lem: Lp inclusion}})}
        }
        \\
        \overset{\eqref{eq: Omega norm bound by infinity norm}}&\le
        \Biggl(\norm[\big]{\tfrac1\weight}_{L^p(\mu)}  + \sqrt{\gamma} \sqrt{\frac1n\sum_{i,j=1}^n \frac{\abs{(\Omega_n)_{ij}}}{\weight(x_i)\weight(x_j)}}\Biggr) \norm{f - \model}_\weight
    \end{align*}
    Minimizing this upper bound over \(\model\in \modelSpace\) yields the desired claim.
    \qed

\subsection{Noise bounds}

\begin{lemma}[Bounds on the coefficients of the weighted regression of the noise]
    \label{lem: bounds on beta}
    Let \(\beta(x) \coloneq \weight(x)A^\conT\ell(x)\), then the weighted regression of the noise at \(x\) is a linear combination of the independent noise variables \(\noiseIndep_i\)
   \[
        S_n \coloneq \weight(x)\regOp(\noise)(x)
        = \sum_{i=1}^n \overline{\beta_i(x)} \noiseIndep_i.
   \]
   Moreover the coefficients \(\beta_i(x)\) satisfy the following bounds for any \(p\ge 2\)
   \begin{equation}
        \label{eq: bounds on beta}
        \abs{\beta(x)}_p \le \abs{\beta(x)}_2
        \le \frac1{\sqrt{n}}\sqrt{\opnorm{\Gram_*^{-1}}_2\opnorm{\Omega_n^{\frac12}\Sigma \Omega_n^{\frac12}}_2}\, \lever_\modelSpace^*(x)\weight(x) \eqqcolon c(x).
   \end{equation}
\end{lemma}
    Observe that \(c(x)\) is the factor in the pointwise error bound in the statement of the theorem.
    Also note that the first bound on \(\abs{\beta}_2\) is very sharp and an equality for
    \(\Omega_n\Sigma = \identity\). And the second bound is sharp for a well conditioned Gram matrix.

\begin{proof}
    By definition of the Lagrange function \(\ell\)
    our random variable of interest is of the form
    \begin{equation}
        S_n\coloneq \weight(x)\regOp(\noise)(x)
        = \weight(x)\ell(x)^\conT \noise
        = \underbrace{\weight(x)\ell(x)^\conT A}_{= \beta(x)^\conT} \noiseIndep
        = \sum_{i=1}^n \overline{\beta_i} \noiseIndep_i.
    \end{equation}
    We further have \(\abs{\beta(x)}_2 = \weight(x)\abs{\ell(x)}_\Sigma\) 
    by homogeneity of the norm and by the kernel representation of
    \(\ell(x)\) (Theorem \ref{thm: reproducing kernel representation})
    \[
        \abs{\ell(x)}_\Sigma^2
        = \frac1{n^2}\begin{pmatrix}
            \reprK_\modelSpace^{\Omega_n}(x, x_1)
            \\
            \vdots
            \\
            \reprK_\modelSpace^{\Omega_n}(x, x_n)
        \end{pmatrix}^\conT
        \Omega_n \Sigma \Omega_n \begin{pmatrix}
            \reprK_\modelSpace^{\Omega_n}(x, x_1)
            \\
            \vdots
            \\
            \reprK_\modelSpace^{\Omega_n}(x, x_n)
        \end{pmatrix}
        \le \frac1{n}\opnorm{\Omega_n^{\frac12} \Sigma \Omega_n^{\frac12}}_2 \abs[\Bigg]{
            \begin{pmatrix}
                \reprK_\modelSpace^{\Omega_n}(x, x_1)
                \\
                \vdots
                \\
                \reprK_\modelSpace^{\Omega_n}(x, x_n)
            \end{pmatrix}
        }_{\Omega_n}^2.
    \]
    Using the notation \(\reprK_\modelSpace^{\Omega_n}(x, x_{[1\inter n]}) = (\reprK_\modelSpace^{\Omega_n}(x, x_1), \dots, \reprK_\modelSpace^{\Omega_n}(x,x_n))^\transpose\)
    and an orthonormal basis \(\mBasis_1, \dots, \mBasis_m\) of \(\modelSpace\) with respect to \(\scp{\cdot}{\cdot}_{\Omega_n}\) we have
    \begin{align*}
        \abs{\reprK_\modelSpace^{\Omega_n}(x,x_{[1\inter n]})}_{\Omega_n}^2
        &= \frac1n\sum_{i,j=1}^n (\Omega_n)_{ij} \reprK_\modelSpace^{\Omega_n}(x, x_i) \overline{\reprK_\modelSpace^{\Omega_n}(x, x_j)}
        \\
        &= \frac1n\sum_{i,j=1}^n (\Omega_n)_{ij} \sum_{k=1}^m \mBasis_k(x)\overline{\mBasis_k(x_i)}
        \sum_{l=1}^m \overline{\mBasis_l(x)} \mBasis_l(x_j)
        \\
        &= \sum_{k,l=1}^m \mBasis_k(x)\overline{\mBasis_l(x)} \scp{\mBasis_l}{\mBasis_k}_{\Omega_n}
        =\sum_{k=1}^m \mBasis_k(x)\overline{\mBasis_k(x)}
        = \reprK_\modelSpace^{\Omega_n}(x,x).
    \end{align*}
    Put together this implies the claimed bound
    \begin{align*}
        \abs{\beta(x)}_2
        &= \weight(x)\abs{\ell(x)}_\Sigma
        \\
        \nonumber
        &\le \sqrt{\tfrac1n \opnorm{\Omega_n^{\frac12}\Sigma \Omega_n^{\frac12}}_2}\, \lever_\modelSpace^{\Omega_n}(x)\weight(x)
        \le  \tfrac1n\sqrt{\opnorm{\Gram_*^{-1}}_2\opnorm{\Omega_n^{\frac12}\Sigma \Omega_n^{\frac12}}_2}\, \lever_\modelSpace^*(x)\weight(x)\eqcolon c(x),
    \end{align*}
    where we recall \(\lever_\modelSpace^{\Omega_n}(x) \le \sqrt{\opnorm{\Gram_*^{-1}}_2}\lever_\modelSpace^*(x)\) from Lemma \ref{thm: leverage function bound}
    and the \(\ell^p\) norm inequality \(\abs{\beta}_p \le \abs{\beta}_2\) for \(p\ge 2\) (Lemma \ref{lem: Lp inclusion}).
\end{proof}

\begin{lemma}\label{lem: bound on operator norm of Omega Sigma}
    \(\opnorm{\Omega_n^{\frac12}\Sigma \Omega_n^{\frac12}}_2 \le \opnorm{\Omega_n\Sigma}_2\)
    for any positive semi-definite \(\Sigma\) and \(\Omega_n\).
\end{lemma}
\begin{proof}
    Let \(B=\Omega_n^{\frac12}\in \field^{n\times n}\) and \(A=B\Sigma \in \field^{n\times n}\). Then the eigenvalues of
    \(AB=\Omega_n^{\frac12} \Sigma \Omega_n^{\frac12}\) and \(BA=\Omega_n \Sigma\)
    coincide by \citet[Thm.~1.3.22]{hornMatrixAnalysis2017} and since
    \(\Omega_n^{\frac12} \Sigma \Omega_n^{\frac12}\) is positive semi-definite,
    its largest eigenvalue coincides with its spectral (operator)-norm such that
    \[
        \opnorm{\Omega_n^{\frac12} \Sigma \Omega_n^{\frac12}}_2
        = \lambda_{\max}(\Omega_n^{\frac12} \Sigma \Omega_n^{\frac12})
        = \lambda_{\max}(\Omega_n\Sigma)
        \le \opnorm{\Omega_n\Sigma}_2,
    \]
    where the last inequality uses \citet[Thm.~5.6.9]{hornMatrixAnalysis2017}.
\end{proof}

\subsubsection{Proof of Theorem \ref{thm: pointwise concentration}: Pointwise concentration}
\label{sec: proof of pointwise concentration}

We will now prove all the concentration bounds one by one using Lemma \ref{lem: bounds on beta}.
We will start with the easiest case of finite second moments.

    \begin{enumerate}[wide=0pt]
        \item[\ref{it: (concentration) finite second moment}] \textbf{Finite second moment.}
        If \(\E{\abs{\noiseIndep_i}^2} \le \sigma^2\), then independence and
        centeredness give
        \[
            \E{\abs{S_n}^2}
            = \sum_{i=1}^n \abs{\beta_i(x)}^2\E{\abs{\noiseIndep_i}^2}
            \le \sigma^2\abs{\beta(x)}_2^2
            \overset{\text{Lem.~\ref{lem: bounds on beta}}}\le \sigma^2 c(x)^2.
        \]
        Markov's inequality applied to \(\abs{S_n}^2\) then yields the first
        bound by \(\prob{\abs{S_n} \ge \varepsilon c(x)} \le
        (\varepsilon c(x))^{-2}\E{\abs{S_n}^2}\).
        
        \item[\ref{it: (concentration) subgaussian}] \textbf{Sub-Gaussian.}
        Here and in the sub-exponential and sub-Weibull cases below we assume
        \(\field=\real\), as required by the scalar concentration inequalities
        being applied.
        Assume the \(\noiseIndep_i\) are sub-Gaussian with variance proxy \(\sigma^2\).
        Then \(\beta_i \noiseIndep_i\) is sub-Gaussian with variance proxy \(\sigma^2 \beta_i^2\).
        And since the variance proxy of a sum of independent sub-Gaussian random variables is the
        sum of variance proxies, \(S_n = \beta^\transpose\noiseIndep\) is sub-Gaussian 
        with variance proxy \(\sigma^2\abs{\beta}_2^2\). Then by the standard Hoeffding concentration inequality for
        sub-Gaussian random variables, e.g.\ \citep[(2.9)]{wainwrightHighDimensionalStatisticsNonAsymptotic2019} or \citep[Thm.~2.6.3]{vershyninHighDimensionalProbabilityIntroduction2018}
        \[
            \prob{\abs{S_n} \ge \epsilon}
            \le 2\exp\bigl(- \tfrac{\epsilon^2}{2\sigma^2\abs{\beta(x)}_2^2}\bigr)
            \overset{\text{Lem.~\ref{lem: bounds on beta}}}\le
            2\exp\bigl(- \tfrac{\epsilon^2}{2\sigma^2 c(x)^2}\bigr).
        \]
        The choice \(\epsilon = \varepsilon c(x)\) then yields the claimed bound.
    
        \item[\ref{it: (concentration) subexponential}] \textbf{Sub-exponential.}
        Assume the \(\noiseIndep_i\) are sub-exponential with parameters \((\nu, \alpha)\) \citep[Def.~2.7]{wainwrightHighDimensionalStatisticsNonAsymptotic2019}.
        Then \(\beta(x)^\transpose \noiseIndep\) is sub-exponential with parameters \((\nu\abs{\beta(x)}_2, \alpha\abs{\beta(x)}_\infty)\). By
        Bernstein's inequality
        \citep[Prop.~2.9]{wainwrightHighDimensionalStatisticsNonAsymptotic2019} we thus have
        \begin{align*}
            \prob{\abs{S_n} \ge \epsilon}
            &\le \begin{cases}
                2\exp\bigl(-\tfrac{\epsilon^2}{2\nu^2\abs{\beta(x)}_2^2}\bigr)
                & 0\le \epsilon \le \frac{\nu^2\abs{\beta(x)}_2^2}{\alpha\abs{\beta(x)}_\infty}
                \\
                2\exp\bigl(-\tfrac{\epsilon}{2\alpha\abs{\beta(x)}_\infty}\bigr)
                & \epsilon > \frac{\nu^2\abs{\beta(x)}_2^2}{\alpha\abs{\beta(x)}_\infty}
            \end{cases}
            \\
            &= 2\exp\Bigl(-\min\set[\Big]{\frac{\epsilon}{\nu^2\abs{\beta(x)}_2^2}, \frac1{\alpha\abs{\beta(x)}_\infty}}\frac\epsilon2\Bigr)
            \\
            \overset{\text{Lem.~\ref{lem: bounds on beta}}}&\le 
            2\exp\Bigl(-\min\set[\Big]{\frac{\epsilon}{\nu^2c(x)^2}, \frac{1}{\alpha c(x)}}\frac\epsilon2\Bigr).
        \end{align*}
        The choice \(\epsilon = \varepsilon c(x)\) yields the claimed bound again.

        \item[\ref{it: (concentration) sub-Weibull}] \textbf{Sub-Weibull.}
        The goal is to apply the concentration inequality Theorem 3.1 of \citep{kuchibhotlaMovingSubGaussianityHighdimensional2022}.
        In their notation we choose \(X_i = \noiseIndep_i\) and \(a_i = \beta_i(x)\). 
        Their Theorem 3.1 (3.1) with \(b\coloneq (a_1\norm{X_1}_{\psi_\alpha}, \dots,
        a_n\norm{X_n}_{\psi_\alpha}) = (\beta_1(x)\norm{\noiseIndep_1}_{\psi_\alpha}, \dots, \beta_n(x)\norm{\noiseIndep_n}_{\psi_\alpha})\) then implies
        \begin{equation}
            \label{eq: sub-Weibull concentration} 
            \prob[\Big]{\abs{S_n} \ge 2e C(\alpha) \abs{b}_2 \sqrt{t} + 2eL_n^*(\alpha) t^{\frac1\alpha} \abs{b}_{p(\alpha)}} 
            \le 2e^{-t} \qquad \text{for all } t \ge 0,
        \end{equation}
        with \(p(\alpha) = \infty\) for \(\alpha \le 1\) and \(p(\alpha) = \alpha/(\alpha-1)\) for \(\alpha > 1\)
        and 
        \begin{align*}
            C(\alpha) &\coloneq \max\set{\sqrt{2}, 2^{\frac1\alpha}}\begin{cases}
                \sqrt{8} e^3 (2\pi)^{\frac14} e^{\frac1{24}} (\tfrac{e^{2/e}}{\alpha})^{\frac1\alpha} & \alpha < 1
                \\
                4e + 2(\ln 2)^{\frac1\alpha} & 1 \le \alpha
            \end{cases}
            \\
            L_n^*(\alpha) &\coloneq C(\alpha)\tfrac{4^{1/\alpha}}{\sqrt{2}}
            \begin{cases}
                1 & \alpha \le 1
                \\
                \tfrac{4e}{C(\alpha)} & 1 < \alpha
            \end{cases}
            \\
            &\le \tfrac{4^{1/\alpha}}{\sqrt{2}}C(\alpha) \qquad (C(\alpha) \ge 4e \text{ for } 1 < \alpha).
        \end{align*} 
        By \(\norm{\noiseIndep_i}_{\psi_\alpha} \le K\) and the definition of \(b\) we 
        thus have for \(\alpha \le 2\) such that \(p(\alpha) \ge 2\)
        \begin{equation}
            \label{eq: sub-Weibull bound on b}    
            \abs{b}_{p(\alpha)} \le \abs{b}_2 \le K \abs{\beta(x)}_2
            \overset{\text{Lem.~\ref{lem: bounds on beta}}}\le Kc(x).
        \end{equation}
        This implies
        \begin{align*}
            2e C(\alpha) \abs{b}_2 \sqrt{t} + 2eL_n^*(\alpha) t^{\frac1\alpha} \abs{b}_{p(\alpha)}
            &\le 2e C(\alpha) Kc(x) \bigl(\sqrt{t} + \tfrac{4^{1/\alpha}}{\sqrt{2}} t^{\frac1\alpha}\bigr)
            \\
            &\le \underbrace{2^{\frac2\alpha+\frac12}e C(\alpha)}_{\eqcolon C_2(\alpha)} Kc(x) \bigl(\sqrt{t} + t^{\frac1\alpha}\bigr)
            \overset{(*)}\le c(x) \varepsilon,
        \end{align*}
        where the inequality \((*)\) follows from the choice
        \(u\coloneq\frac{\varepsilon}{C_2(\alpha)K}\) and \(t \coloneq \min\set{(\frac{u}{2})^2, (\frac{u}{2})^\alpha}\)
        such that \(\sqrt{t} + t^{\frac1\alpha}\le u\).
        Put together with \eqref{eq: sub-Weibull concentration} this yields
        \begin{align*}
            \prob{\abs{S_n} \ge \varepsilon c(x)}
            &\le 2\exp(-\min\set{(\tfrac{u}{2})^2, (\tfrac{u}{2})^\alpha})
            \\
            &= 2\exp\Bigl(-\min\set[\Big]{\bigl(\tfrac{\varepsilon}{2C_2(\alpha)K}\bigr)^2, \bigl(\tfrac{\varepsilon}{2C_2(\alpha) K}\bigr)^\alpha}\Bigr)
            \\
            &\le 2\exp\Bigl(-\min\set{\tfrac1{(2C_2(\alpha))^{2}}, \tfrac1{(2C_2(\alpha))^{\alpha}}}\min\set[\big]{\bigl(\tfrac{\varepsilon}{K}\bigr)^2, \bigl(\tfrac{\varepsilon}{K}\bigr)^\alpha}\Bigr).
        \end{align*}
        With \(c_\alpha \coloneq \min\set{\tfrac1{(2C_2(\alpha))^{2}}, \tfrac1{(2C_2(\alpha))^{\alpha}}}\) we have the claim,
        where we recall
        \begin{equation}
            \label{eq: sub-Weibull constant}
            C_2(\alpha)
            = 2^{\frac3\alpha+\frac12}e \begin{cases}
                \sqrt{8} e^3 (2\pi)^{\frac14} e^{\frac1{24}} (\tfrac{e^{2/e}}{\alpha})^{\frac1\alpha} & \alpha < 1
                \\
                4e + 2(\ln 2)^{\frac1\alpha} & \alpha \in [1,2].
            \end{cases}
        \end{equation}

        \item[\ref{it: (concentration) finite p-th moment}] \textbf{Finite \(p\)-th moment.}
        We want to apply the Rosenthal inequality to
        the sum of indpendent random variables  \(X_i \coloneq \overline{\beta_i(x)}\noiseIndep_i\)
        to obtain 
        \begin{equation}
            \label{eq: Rosenthal bound}    
            \norm[\Big]{\sum_{i=1}^n X_i}_{L^p(\pSym)}
            \le C_{p, \field} \max\set[\Big]{
                \Bigl(\sum_{i=1}^n \norm{X_i}_{L^2(\pSym)}^2\Bigr)^{1/2},
                \Bigl(\sum_{i=1}^n \norm{X_i}_{L^p(\pSym)}^p\Bigr)^{1/p}}.
        \end{equation}
        \begin{casebycase}
            \item \textbf{\(\field=\real\):}
            Our \(X_i\) are only centered and not symmetrized. We 
            therefore apply Theorem 4.1 from
            \citet{johnsonBestConstantsMoment1985} together
            with their Remark on symmetrization below the theorem to get
            the claimed bound \eqref{eq: Rosenthal bound} with \(C_{p, \real} = \frac{14.7p}{\ln p}\) for \(p\ge 2\).
            Note that we also use
            \[
                \norm[\big]{\sum_{i=1}^n X_i}_{L^2(\pSym)} = \Bigl(\sum_{i=1}^n \norm{X_i}_{L^2(\pSym)}^2\Bigr)^{1/2},
            \]
            which follows from independence and centeredness of the \(X_i\).

            \item\textbf{\(\field=\complex\):}
            The goal is to apply the same bound separately to the real and imaginary
            parts of the independent centered variables \(X_i\).
            For this we use Lemma \ref{lem: convexity of p/2 power}
            to obtain
            \[
                \abs{X_i}^p\le 2^{p/2-1}(\abs{\Re X_i}^p+ \abs{\Im X_i}^p).
            \]
            This implies
            \begin{align*}
                \norm[\Big]{\sum_{i=1}^n X_i}_{L^p(\pSym)}^p
                &\le 2^{p/2-1}\Bigl(\norm[\Big]{\sum_{i=1}^n \Re X_i}_{L^p(\pSym)}^p
                + \norm[\Big]{\sum_{i=1}^n \Im X_i}_{L^p(\pSym)}^p\Bigr)
                \\
                &\le 2^{p/2} \Bigl(\frac{14.7p}{\ln p}\Bigr)^p \max\set[\Big]{
                    \Bigl(\sum_{i=1}^n \norm{X_i}_{L^2(\pSym)}^2\Bigr)^{p/2},
                    \Bigl(\sum_{i=1}^n \norm{X_i}_{L^p(\pSym)}^p\Bigr)
                }
            \end{align*}
            where we used \(\norm{\Re X_i}_{L^2(\pSym)}^2 \le \norm{X_i}_{L^2(\pSym)}^2\) and \(\norm{\Re X_i}_{L^p(\pSym)}^p \le \norm{X_i}_{L^p(\pSym)}^p\) (and similarly for the imaginary part)
            in the last inequality. Taking the \(p\)-th root and defining \(C_{p, \complex} \coloneq 2^{1/2}C_{p, \real}\) yields \eqref{eq: Rosenthal bound}.
        \end{casebycase}
        What is left to do is to bound the terms in the Rosenthal bound \eqref{eq: Rosenthal bound}.
        We have
        \[
            \sum_{i=1}^n \norm{X_i}_{L^2(\pSym)}^2
            = \sum_{i=1}^n \E{\abs{X_i}^2}
            = \sum_{i=1}^n \abs{\beta_i(x)}^2 \E{\abs{\noiseIndep_i}^2}
            \le \sigma^2 \abs{\beta(x)}_2^2.
        \]
        And, using \(\abs{x}_p \le \abs{x}_2\) for \(p\ge 2\), we have for the \(p\)-th moments
        \[
            \sum_{i=1}^n \norm{X_i}_{L^p(\pSym)}^p
            = \sum_{i=1}^n \E{\abs{X_i}^p}
            \le \kappa^p \abs{\beta(x)}_p^p \le \kappa^p \abs{\beta(x)}_2^p.
        \]
        Inserted into \eqref{eq: Rosenthal bound} this yields for \(\sum_{i=1}^n X_i = \beta(x)^\conT \noiseIndep\)
        \begin{equation}
            \label{eq: Rosenthal bound on noise}
            \norm{\beta(x)^\conT \noiseIndep}_{L^p(\pSym)}
            = \norm[\Big]{\sum_{i=1}^n X_i}_{L^p(\pSym)}
            \le C_{p,\field} \max\set{\kappa, \sigma}\abs{\beta(x)}_2
        \end{equation}
        Lemma \ref{lem: bounds on beta} implies \(\abs{\beta(x)}_2 \le c(x)\) and
        together with Markov's inequality we thereby get the claimed bound
        \[
            \prob{\abs{\beta(x)^\conT \noiseIndep} \ge \varepsilon c(x)}
            \le \frac{\E{\abs{\beta(x)^\conT \noiseIndep}^p}}{\varepsilon^p c(x)^p}
            \le \Bigl(C_{p,\field} \frac{\max\set{\kappa, \sigma}}{\varepsilon}\Bigr)^p.
            \qedhere
        \]
\end{enumerate}

\begin{lemma}\label{lem: convexity of p/2 power}
    For \(a,b \ge 0\) and \(p\ge 2\) we have \((a+b)^{p/2} \le 2^{p/2-1}(a^{p/2}+b^{p/2})\).
\end{lemma}
\begin{proof}
    Convexity of \(x\mapsto x^{p/2}\) for \(p\ge 2\) implies
    \(
        \bigl(\frac{a+b}{2}\bigr)^{p/2} \le \frac{a^{p/2}+b^{p/2}}{2}
    \). Multiplying by \(2^{p/2}\) yields the claim.
\end{proof}

\subsubsection{Proof of Lemma \ref{lem: pointwise to uniform conversion}: Pointwise to uniform conversion}

This proof is based on a proof by \citet{matsudaPolynomialApproximationNoisy2025} specialized to
Chebyshev polynomials.

\begin{steps}
    \item \textbf{Sup-norm to finite maximum.}
    Let \(\linProj_\star\) be a regression projection with \(m\) distinct regression points \(Z=(z_1, \dots, z_m)\).
    We will show for every \(\model\in \modelSpace\) that
    \begin{equation}
        \label{eq: bound on w norm of regressed noise}
        \norm{\model}_\weight
        \le \lebconst{\weight}{\linProj_\star} \max_{i=1,\dots,m} \abs{\model(z_i)\weight(z_i)},
    \end{equation}
    Since \(\linProj_\star \model = \model\), we have
    using the Lagrange function \(\ell_\star\) of \(\linProj_\star\)
    \[
        \abs{\model(x)} = \abs{\linProj_\star \model(x)} = \abs{\ell_\star(x)^\conT\eval_Z[\model]}
        = \abs[\bigg]{\sum_{i=1}^m \overline{\ell_{\star, i}(x)}\model(z_i)}
        \le \max_{i=1,\dots, m} \abs{\model(z_i)\weight(z_i)} \sum_{i=1}^m \frac{\abs{\ell_{\star, i}(x)}}{\weight(z_i)}.
    \]
    By definition of the Lebesgue function (Def.~\ref{def: lebesgue function and constant}) this implies
    \[
        \abs{\model(x)\weight(x)}
        \le \lebfun{\weight}{\linProj_\star}{x}
        \max_{i=1,\dots, m} \abs{\model(z_i)\weight(z_i)}.
    \]
    Taking the supremum over \(x\in \domain\)  yields \eqref{eq: bound on w norm of regressed noise}
    by definition of the Lebesgue constant.

    \item\label{step: (pointwise to uniform) union bound} \textbf{Union bound.} Since \(\regOp(\noise) \in \modelSpace\)
    we thus have
    \begin{align*}
        \prob[\Big]{\norm{\regOp(\noise)}_\weight > \lebconst{\weight}{\linProj_\star} \epsilon}
        \overset{\eqref{eq: bound on w norm of regressed noise}}&\le \prob[\Big]{\max_{i=1,\dots, m} \abs{\regOp(\noise)(z_i)\weight(z_i)} > \epsilon}
        \\
        \overset{\text{union}}&\le \sum_{i=1}^m \prob[\Big]{\abs{\regOp(\noise)(z_i)\weight(z_i)} > \epsilon}
        \le m \delta.
    \end{align*}
    The last inequality follows from the assumption of the theorem.

    \item \textbf{Limit argument.}
    Let \(c_\star \coloneq \inf_{\linProj_\star}\lebconst{\weight}{\linProj_\star}\). Then
    by the continuity of probability measures we have
    \[
        \prob[\Big]{\norm{\regOp(\noise)}_\weight >  c_\star\epsilon}
        = \prob[\Big]{\bigcup_{c>c_\star}\norm{\regOp(\noise)}_\weight > c \epsilon}
        = \lim_{c \downarrow c_\star}\prob[\Big]{\norm{\regOp(\noise)}_\weight > c \epsilon} 
        \le m\delta,
    \]
    where the last inequality follows from \(c>c_\star\) such that there
    exists a regression \(\linProj_\star\) with \(\lebconst{\weight}{\linProj_\star} < c\) and the bound of \ref{step: (pointwise to uniform) union bound}.
\end{steps}

\subsubsection{Proof of Corollary \ref{cor: uniform noise bound}: Uniform noise bound}

This is a direct consequence of Theorem \ref{thm: pointwise concentration} and Lemma \ref{lem: pointwise to uniform conversion},
we only need to invert \(\delta(\epsilon)\) to obtain \(\epsilon(\delta)\) and replace \(\delta\) by \(\frac{\delta}{m}\)
to obtain the desired bound with probability at least \(1-\delta\). The inversion of \(\delta(\epsilon)\)
in the sub-Gaussian, and finite moment cases is straightforward. In the sub-exponential case it requires a bit more care.
Observe that
\begin{align*}
    &\varepsilon(\delta) \coloneq \max\set[\Big]{\sqrt{2\nu^2\ln(\tfrac{2}{\delta})}, 2\alpha \ln(\tfrac{2}{\delta})}
    \\
    &\implies 2\ln(\tfrac{2}{\delta}) \le \tfrac{\varepsilon(\delta)^2}{\nu^2}
    \quad\text{and}\quad
    2\ln(\tfrac{2}{\delta}) \le \tfrac{\varepsilon(\delta)}{\alpha}
    \\
    &\implies \delta \ge 2\exp(-\min\set{\tfrac{\varepsilon(\delta)^2}{2\nu^2}, \tfrac{\varepsilon(\delta)}{2\alpha}})
    = \delta(\varepsilon(\delta)).
\end{align*}
Consequently, this selection of \(\varepsilon(\delta)\) is sufficient to guarantee the \(\varepsilon(\delta)\)
deviation bound with probability at least \(1-\delta(\varepsilon(\delta)) \ge
1-\delta\). Replacing \(\delta\) by \(\frac{\delta}{m}\) then yields the desired
bound after applying Lemma \ref{lem: pointwise to uniform conversion}. 
In the sub-Weibull case we can improve the constants by opening up the proof of
the sub-Weibull concentration bound. This is because the original form of the
bound matches the \(\varepsilon(\delta)\) form of the bound and not the
\(\delta(\varepsilon)\) form of the bound. Specifically, recall
\(\abs{S_n} = \abs{\weight(x)\regOp(\noise)(x)}\) (Lemma \ref{lem: bounds on beta}). Then
the concentration bound of \eqref{eq: sub-Weibull concentration} together
with \(t=\ln(\frac{2}{\delta})\) and the bound \(\abs{b}_{p(\alpha)} \le \abs{b}_2 \le Kc(x)\) of \eqref{eq: sub-Weibull bound on b} yields
with probability at least \(1-\delta\)
\begin{align*}
    \abs{\weight(x)\regOp(\noise)(x)}
    &\le 2e Kc(x) \Bigl(C(\alpha) \sqrt{\ln(\tfrac{2}{\delta})} + L_n^*(\alpha) \ln(\tfrac{2}{\delta})^{\frac1\alpha}\Bigr),
    \\
    &= K c(x)  \Bigl(\underbrace{2eC(\alpha)}_{\eqcolon C_\alpha} \sqrt{\ln(\tfrac{2}{\delta})} + \underbrace{2eL_n^*(\alpha)}_{\eqcolon C_{\alpha}'} \ln(\tfrac{2}{\delta})^{\frac1\alpha}\Bigr)
\end{align*}
with \(c(x)\) defined in Lemma \ref{lem: bounds on beta} and \(C(\alpha), L_n^*(\alpha)\) defined after \eqref{eq: sub-Weibull concentration}
such that
\begin{equation}
    \label{eq: sub-Weibull constants}
\begin{aligned} 
    C_\alpha &\coloneq 2e\max\set{\sqrt{2}, 2^{\frac1\alpha}}\begin{cases}
        \sqrt{8} e^3 (2\pi)^{\frac14} e^{\frac1{24}} (\tfrac{e^{2/e}}{\alpha})^{\frac1\alpha} & \alpha < 1
        \\
        4e + 2(\ln 2)^{\frac1\alpha} & 1 \le \alpha
    \end{cases} 
    \\
    C_\alpha' &\coloneq C_\alpha\tfrac{4^{1/\alpha}}{\sqrt{2}}
    \begin{cases}
        1 & \alpha \le 1
        \\
        \tfrac{8e^2}{C_\alpha} & 1 < \alpha.
    \end{cases}
\end{aligned}
\end{equation}

\subsubsection{Proof of Theorem \ref{thm: L2 concentration}: \texorpdfstring{\(L^2(\mu)\)}{L2(µ)}-bounds on the noise term}

By definition of \(\regOp\) \eqref{eq: definition least squares reg op} and \(\noise= A\noiseIndep\) we have
\begin{equation}
    \label{eq: reg op noise} 
    \regOp(\noise)
    = \mBasisVec(\cdot)^\transpose B_\mu \noiseIndep
    \qquad \text{with}\qquad
    B_\mu \coloneq \tfrac1n\Gram_\mu^\dagger \modelVander^\conT \Omega_n A \in \field^{m\times n}
\end{equation}
Since the basis \(\mBasisVec = (\mBasis_1, \dots, \mBasis_m)^\transpose\) is
orthonormal with respect to \(\scp{\cdot}{\cdot}_\mu\) this implies 
\begin{equation}
    \label{eq: parseval identity}    
    \norm{\regOp(\noise)}_\mu^2
    = \norm{\mBasisVec(\cdot)^\transpose B_\mu \noiseIndep}_\mu^2
    \overset{\text{orth.}}= \abs{B_\mu \noiseIndep}_2^2
    = \tr(B_\mu \noiseIndep\noiseIndep^\conT B_\mu^\conT).
\end{equation}
The concentration bounds for \(\norm{\regOp(\noise)}_\mu^2\) thereby follow from
concentration bounds for the quadratic form \(\abs{B_\mu \noiseIndep}_2^2\),
which yield the following (sharper) bounds.

\begin{mdframed}[style=default]
\begin{theorem}[Sharper \(L^2\)-bounds on the noise term]
    \label{thm: L2 concentration sharper}
    Assume the setting of Theorem~\ref{thm: L2 concentration} and let \(H_\mu \coloneq B_\mu B_\mu^\conT \) with \(B_\mu\) as in \eqref{eq: reg op noise}.
    Then we have the following sharper bounds on the noise term:
    \begin{center}
    \makebox[\textwidth][c]{
    \begin{NiceTabular}{l m{6.15cm} m{5.4cm}}[
        tabularnote={\footnotesize
        \(C_\alpha>0\) is a constant that only depends on \(\alpha\) and \(C_{p,\real}=\frac{14.7p}{\ln p}\), \(C_{p,\complex}=\sqrt2\,C_{p,\real}\).
    }]
        \toprule
        & Assumption (decreasing strength \(\downarrow\)) & With probability at least \(1-\delta\)\newline \(\norm{\regOp(\noise)}_\mu \le \dots\)
        \\
        \midrule
        \newtag{\normalfont{(i)}}{it: (L2 concentration sharper) subgaussian}\tabularnote{Assume \(\field=\real\).\label{footnote: K=R, L2 stronger}}\!
        & \(\noiseIndep_i \sim \subG{\sigma^2}\)
        & \(\sigma\Bigl(\sqrt{\tr(H_\mu)} + \sqrt{2\opnorm{H_\mu}_2 \ln(\tfrac1\delta)}\Bigr)\)
        \\[1ex]
        \newtag{\normalfont{(ii)}}{it: (L2 concentration sharper) sub-Weibull}\ref*{footnote: K=R, L2 stronger}\!
        & \(\noiseIndep_i \sim \subWeibull{\alpha}\), \(\alpha\le 2\), \(\norm{\noiseIndep_i}_{\psi_\alpha}\le K\)
        & \(C_\alpha K \Bigl(\sqrt{\tr(H_\mu)} + \sqrt{\opnorm{H_\mu}_2} \ln(\tfrac2\delta)^{\frac1\alpha}\Bigr)\)
        \\[1ex]
        \newtag{\normalfont{(iii)}}{it: (L2 concentration sharper) finite p-th moment}
        & \(\E{\abs{\noiseIndep_i}^2}\le \sigma^2\) and \(\E{\abs{\noiseIndep_i}^p} \le \kappa^p\) for \(p> 2\)
        & \(C_{p,\field} \max\set{\kappa, \sigma}\bigl(\frac{1}{\delta}\bigr)^{\frac1p}\sqrt{\tr(H_\mu)}\)
        \\[1ex]
        \newtag{\normalfont{(iv)}}{it: (L2 concentration sharper) finite second moment}
        & \(\E{\abs{\noiseIndep_i}^2} \le \sigma^2\)
        & \(\sigma\sqrt{\tfrac{1}{\delta}\tr(H_\mu)}\).
        \\
        \bottomrule
    \end{NiceTabular}
    }
    \end{center}
    Moreover, subject to assumption \ref{it: (L2 concentration sharper) finite p-th moment}
    (resp.\ \ref{it: (L2 concentration sharper) finite second moment}) we have the moment bounds
    \[
        \E{\norm{\regOp(\noise)}_\mu^p}
        \le C_{p,\field}^p \max\set{\kappa^p, \sigma^p}
        \tr(H_\mu)^{\frac{p}2},
        \qquad\text{resp.}\qquad
        \E{\norm{\regOp(\noise)}_\mu^2}
        \le \sigma^2 \tr(H_\mu).
    \]
\end{theorem}
\end{mdframed}

The more user friendly bounds in Theorem \ref{thm: L2 concentration}
follow directly from the following bounds on the trace and operator norm of
\(H_\mu\) in terms of the Gram matrix \(\Gram_\mu\) and the structure of the noise
\(\Sigma\).

\begin{lemma}
    \label{lem: (L2-noise bound) trace and op-norm bounds}
    We have for \(H_\mu \coloneq B_\mu B_\mu^\conT\)
    the spectral domination \(H_\mu \precsim \frac1n \opnorm{\Omega_n^{\frac12} \Sigma \Omega_n^{\frac12}}_2 \Gram_\mu^\dagger\) and
    therefore for \(Q_\mu \coloneq B_\mu^\conT B_\mu\)
    \begin{alignat}{3}
        \label{eq: (L2-noise bound) bound on trace}
        \tr(Q_\mu) &= \tr(H_\mu) &&\le \tfrac1n \opnorm{\Omega_n^{\frac12} \Sigma \Omega_n^{\frac12}}_2 \tr(\Gram_\mu^\dagger)
        \\
        \label{eq: (L2-noise bound) bound on operator norm}
        \opnorm{Q_\mu}_2 &= \opnorm{H_\mu}_2 &&\le \tfrac1n \opnorm{\Omega_n^{\frac12} \Sigma \Omega_n^{\frac12}}_2 \opnorm{\Gram_\mu^\dagger}_2.
    \end{alignat}
    If \(\Omega_n = \diag[\omega_1,\dots, \omega_n]\) we moreover have the alternative trace bound
    \[
        \tr(H_\mu) \le \frac{\opnorm{\Sigma}_2\opnorm{\Gram_\mu^\dagger}_2^2}{n} \frac1n\sum_{i=1}^n \omega_i^2 \lever_\modelSpace^\mu(x_i)^2.
    \]
    For \(\opnorm{\Gram_\mu - \identity}_2 \le \epsilon \in [0,1)\), \(\ones\in \modelSpace\) and \(\omega_i = \frac{m}{\lever_\modelSpace^\mu(x_i)^2}\) this further reduces to
    \[
        \tr(H_\mu)
        \le \opnorm{\Sigma}_2\tfrac{1+\epsilon}{(1-\epsilon)^2}\tfrac{m}{n}  \mu(\domain).
    \]
\end{lemma}

\begin{proof}
    The equality \(\tr(Q_\mu) = \tr(H_\mu)\) follows directly from the cyclic property of the trace.
    For the spectral norm equality, observe that the largest eigenvalue of
    \(B_\mu B_\mu^\conT\) and \(B_\mu^\conT B_\mu\) are both
    the square of the largest singular value of \(B_\mu\).
    Since both \(Q_\mu\) and \(H_\mu\) are positive semi-definite and their largest eigenvalues
    coincide with their spectral norms, which implies
    \(\opnorm{Q_\mu}_2 = \opnorm{H_\mu}_2\).

    To get the spectral domination, observe that for any matrices \(V\in \field^{n\times m}\) and positive
    semi-definite \(B\in \field^{n\times n}\) we have
    \[
        v^\conT V^\conT B V v
        \le \opnorm{B}_2 v^\conT V^\conT V v \qquad \forall v\in \field^m.
    \]
    Consequently, \(V^\conT B V \precsim \opnorm{B}_2 V^\conT V\) and thus in our case
    \begin{align*}
        H_\mu = B_\mu B_\mu^\conT
        \overset{\text{def.}}&= \tfrac1{n^2}\underbrace{\Gram_\mu^\dagger \modelVander^\conT \Omega_n^{\frac12}}_{\eqcolon V^\conT}
        \Omega_n^{\frac12} \Sigma \Omega_n^{\frac12}
        \underbrace{\Omega_n^{\frac12} \modelVander \Gram_\mu^\dagger}_{\eqcolon V}
        \\
        &\precsim \tfrac1{n^2}\opnorm{\Omega_n^{\frac12} \Sigma \Omega_n^{\frac12}}_2 V^\conT V
        = \tfrac1{n}\opnorm{\Omega_n^{\frac12} \Sigma \Omega_n^{\frac12}}_2 \Gram_\mu^\dagger,
    \end{align*}
    where the last equation uses \(\Gram_\mu = \frac1n\modelVander^\conT \Omega_n \modelVander\) to obtain \(\frac1nV^\conT V = \Gram_\mu^\dagger\).

    Hence using the standard basis vectors \(e_1, \dots, e_m\) of \(\field^m\) we have
    the trace bound
    \begin{equation}
        \label{eq: bound on trace of BB transpose} 
        \tr(H_\mu)
        = \sum_{i=1}^m e_i^\conT H_\mu e_i
        \le \tfrac1{n}\opnorm{\Omega_n^{\frac12} \Sigma \Omega_n^{\frac12}}_2 \tr(\Gram_\mu^\dagger)
    \end{equation}
    and the operator norm bound follows directly from the spectral domination
    and positive semi-definiteness of \(H_\mu\) and \(\Gram_\mu^\dagger\).

    For the alternative trace bound in the WLS setting, we first move out
    \(\Sigma\) without the surrounding \(\Omega_n^{\frac12}\) matrices
    \[
        H_\mu = \tfrac1{n^2}\Gram_\mu^\dagger \modelVander^\conT \Omega_n \Sigma \Omega_n \modelVander \Gram_\mu^\dagger
        \precsim \tfrac{\opnorm{\Sigma}_2}{n^2} \Gram_\mu^\dagger \modelVander^\conT \Omega_n^2 \modelVander \Gram_\mu^\dagger.
    \]
    Using \(\tr(AB) \le \tr(A)\opnorm{B}_2\) for positive semi-definite \(A,B\) by Lemma \ref{lem: trace-schatten norm bound}
    we thus have
    \[
        \tr(H_\mu) \le \frac{\opnorm{\Sigma}_2}{n^2} \tr(\Gram_\mu^\dagger \modelVander^\conT \Omega_n^2 \modelVander \Gram_\mu^\dagger)
        \overset{\text{Lem.~\ref{lem: trace-schatten norm bound}}}\le
        \frac{\opnorm{\Sigma}_2\opnorm{\Gram_\mu^\dagger}_2^2}{n^2}  \tr(\modelVander^\conT \Omega_n^2 \modelVander).
    \]
    And we finish the proof using the basis representation of the reproducing kernel (Lem~\ref{lem: reproducing kernel})
    and the definition of the leverage function (Def.~\ref{def: leverage function}) to get
    \[
        \tr(\modelVander^\conT \Omega_n^2 \modelVander)
        = \sum_{i=1}^n \omega_i^2 \sum_{k=1}^m \overline{\mBasisVec_k(x_i)}\mBasisVec_k(x_i)
        = \sum_{i=1}^n \omega_i^2 \lever_\modelSpace^\mu(x_i)^2.
    \]
    If we assume the additional assumptions including \(\omega_i = \frac{m}{\lever_\modelSpace^\mu(x_i)^2}\), we further have
    \[
        \frac1n\sum_{i=1}^n \omega_i^2 \lever_\modelSpace^\mu(x_i)^2
        = m\norm{\ones}_{\Omega_n}^2
        \le m (1+\epsilon) \norm{\ones}_\mu^2
        = m(1+\epsilon) \mu(\domain).
    \]
    where inequality follows from \(\abs{\norm{\ones}_{\Omega_n}^2 - \norm{\ones}_{\mu}^2} \le \epsilon \norm{\ones}_{\mu}^2\)
    by Theorem \ref{thm: approximately identity gram matrix} combined with \(\ones\in \modelSpace\) and \(\opnorm{\Gram - \identity}_2 \le \epsilon\).
    Lemma \ref{lem: bounds on inverse gram matrix} further implies \(\opnorm{\Gram}_2 \le \frac{1}{1-\epsilon}\) such that put together
    \[
        \tr(H_\mu) \le \opnorm{\Sigma}_2\tfrac{1+\epsilon}{(1-\epsilon)^2}\tfrac{m}{n}  \mu(\domain).
        \qedhere
    \]
\end{proof}

What is left to do, is to prove the claimed sharper \(L^2\)-bounds on the noise
term of Theorem \ref{thm: L2 concentration sharper}.

\begin{proof}[Proof of Theorem \ref{thm: L2 concentration sharper}]
    \begin{enumerate}[wide=0pt]
        \item[\ref{it: (L2 concentration sharper) finite second moment}] \textbf{Finite second moment.}
        We have
        \[
            \E{\norm{\regOp(\noise)}_\mu^2}
            \overset{\eqref{eq: parseval identity}}= \E{\tr(B_\mu \noiseIndep\noiseIndep^\conT B_\mu^\conT)}
            = \tr(B_\mu \E{\noiseIndep\noiseIndep^\conT} B_\mu^\conT)
            \le \sigma^2 \tr(H_\mu)
        \]
        The concentration bound then follows directly from Markov's inequality.

        \item[\ref{it: (L2 concentration sharper) subgaussian}] \textbf{Sub-Gaussian.}
        We will apply the concentration inequality by \citet[Thm.~2.1]{hsuTailInequalityQuadratic2012}.
        This result is stated for real random vectors, which is why this case
        assumes \(\field=\real\).
        Since the \(\noiseIndep_i\) are independent sub-Gaussian random variables with
        variance proxy \(\sigma^2\) \citep[Def.~2.2]{wainwrightHighDimensionalStatisticsNonAsymptotic2019}, we have for any \(\alpha\in \real^n\)
        \[
            \E{\exp(\alpha^\transpose \noiseIndep)}
            \overset{\text{indep.}}= \prod_{i=1}^n \E{\exp(\alpha_i \noiseIndep_i)}
            \overset{\text{def.}}\le \prod_{i=1}^n \exp\Bigl(\frac{\sigma^2\alpha_i^2}{2}\Bigr)
            = \exp\Bigl(\frac{\sigma^2}{2}\abs{\alpha}_2^2\Bigr).
        \]
        Consequently, \(\noiseIndep\) is a sub-Gaussian random vector in the
        sense of \citet[Thm.~2.1]{hsuTailInequalityQuadratic2012}
        and for \(Q_\mu \coloneq B_\mu^\conT B_\mu\) we have
        \[
            \prob[\Big]{\abs{B_\mu \noiseIndep}_2^2 > \sigma^2\Bigl(\tr(Q_\mu) + 2\sqrt{\tr(Q_\mu^2)t} + 2\opnorm{Q_\mu}_2 t\Bigr)}
            \le e^{-t}.
        \]
        We can bound the contents using Lemma \ref{lem: (L2-noise bound) trace and op-norm bounds} and \ref{lem: trace-schatten norm bound} to get
        \begin{align*}
            \tr(Q_\mu) + 2\sqrt{\tr(Q_\mu^2)t} + 2\opnorm{Q_\mu}_2 t
            \overset{\text{Lem.~\ref{lem: trace-schatten norm bound}}}&\le \tr(Q_\mu) + 2\sqrt{ \tr(Q_\mu)\opnorm{Q_\mu}_2t} + 2\opnorm{Q_\mu}_2 t
            \\
            &\le \Bigl(\sqrt{\tr(Q_\mu)} + \sqrt{2\opnorm{Q_\mu}_2 t}\Bigr)^2
        \end{align*}
        With \(\tr(Q_\mu) = \tr(H_\mu)\) and \(\opnorm{Q_\mu}_2 = \opnorm{H_\mu}_2\) by Lemma \ref{lem: (L2-noise bound) trace and op-norm bounds}
        and the selection of \(t\coloneq\ln(\frac1\delta)\) and  \(\norm{\regOp(\noise)}_\mu = \abs{B_\mu \noiseIndep}_2\) as shown in \eqref{eq: parseval identity} we thus have
        \[
            \prob[\Big]{
                \norm{\regOp(\noise)}_\mu > \sigma 
                \Bigl(\sqrt{\tr(H_\mu)} + \sqrt{2\opnorm{H_\mu}_2 \ln(\tfrac1\delta)}\Bigr)
            }
            \le \delta.
        \]
        Or in other words, with probability \(1-\delta\) we have the claim.

        \item[\ref{it: (L2 concentration sharper) sub-Weibull}] \textbf{Sub-Weibull.} 
        To apply the generalized Hanson--Wright inequality by \citet{sambaleNotesConcentrationASubexponential2023},
        we need to bound the second moment of the sub-Weibull random variables \(\noiseIndep_i\) first.
        \begin{steps}
        \item \textbf{Bound on the second moment of sub-Weibull random variables.}
        We start with a proof that \(\noiseIndep_i\in \subWeibull{\alpha}\) with \(\norm{\noiseIndep_i}_{\psi_\alpha}\le K\)
        implies
        \begin{equation}
            \label{eq: sub-Weibull second moment bound}
            \E{\noiseIndep_i^2} \le c_\alpha K^2
            \qquad \text{for all } i=1,\dots, n
        \end{equation}
        with the constant \(c_\alpha \coloneq 2\Gamma(\frac{2}{\alpha}+1)\).
        For this, we follow \citet[Proof
        of Thm.~2.1]{vladimirovaSubWeibullDistributionsGeneralizing2020}.
        Specifically, their \(4\Rightarrow 1\) implies
        \[
            \prob{\abs{\noiseIndep_i} \ge t}
            = \prob[\Big]{\exp\Bigl(\bigl(\tfrac{\abs{\noiseIndep_i}}{K}\bigr)^\alpha\Bigr) \ge \exp\bigl(\tfrac{t^\alpha}{K^\alpha}\bigr)}
            \overset{\text{Markov}}\le \frac{\E{\exp\bigl(\bigl(\frac{\abs{\noiseIndep_i}}{K}\bigr)^\alpha\bigr)}}{\exp\bigl(\frac{t^\alpha}{K^\alpha}\bigr)}
            \le 2\exp(-\tfrac{t^\alpha}{K^\alpha}),
        \]
        where the last inequality is due to \(\norm{\noiseIndep_i}_{\psi_\alpha}
        \le K\) for all \(i=1,\dots, n\).
        And their \(1\Rightarrow 2\)
        implies
        \begin{align*}
            \E{\abs{\noiseIndep_i}^p}
            = \int_0^\infty \prob{\abs{\noiseIndep_i}^p \ge t} dt
            \overset{t=u^p}&= \int_0^\infty p u^{p-1} \prob{\abs{\noiseIndep_i} \ge u} du
            \\
            &\le \int_0^\infty 2p u^{p-1} \exp(-\tfrac{u^\alpha}{K^\alpha}) du
            \\
            \overset{s=(u/{K})^\alpha}&= 2 K^p \frac{p}{\alpha} \int_0^\infty s^{\frac{p}{\alpha}-1} e^{-s} ds = K^p \underbrace{2\Gamma(\tfrac{p}{\alpha}+1)}_{\eqcolon c_{\alpha,p}}.
        \end{align*}
        Note that \(K\) is raised to the power of \(p\), which was overlooked by \citet[Proof of Thm.~2.1]{vladimirovaSubWeibullDistributionsGeneralizing2020}.
        The case \(p=2\) with \(c_\alpha \coloneq c_{\alpha, 2}\) is the claim.
        \item\label{step: (L2 concentration) sub-Weibull, concentration} \textbf{Concentration bound.}
        The concentration bound is now essentially an application of Theorem~2.1
        of \citet{sambaleNotesConcentrationASubexponential2023}.
        Specifically for \(Q_\mu=B_\mu^\conT B_\mu\), their Theorem~2.1 implies
        \[
            \prob[\Big]{\abs[\Big]{\overbrace{\noiseIndep^\transpose Q_\mu \noiseIndep}^{=\abs{B_\mu \noiseIndep}_2^2} - \overbrace{\sum_{i=1}^n \E{\noiseIndep_i^2}(Q_\mu)_{ii}}^{=\E{\noiseIndep^\transpose Q_\mu \noiseIndep}}} \ge t}
            \le 2\exp\Bigl(-\frac1{C_\alpha} \min\set[\Big]{\frac{t^2}{K^4 \norm{Q_\mu}_F^2}, \Bigl(\frac{t}{K^2 \opnorm{Q_\mu}_2}\Bigr)^{\frac\alpha{2}}}\Bigr),
        \]
        where \(\norm{Q_\mu}_F^2=\tr(Q_\mu^2)\) and
        \(C_\alpha\) is a constant depending only on \(\alpha\).
        Now we have
        \begin{align}
            \label{eq: (sub-Weibull L2 concentration) u bound}
            &u\le \frac1{C_\alpha} \min\set[\Big]{\frac{t^2}{K^4 \norm{Q_\mu}_F^2}, \Bigl(\frac{t}{K^2 \opnorm{Q_\mu}_2}\Bigr)^{\frac\alpha{2}}}
            \\
            \label{eq: (sub-Weibull L2 concentration) t bound}
            &\iff t \ge K^2 \max\set[\Big]{\norm{Q_\mu}_F \sqrt{C_\alpha u}, \opnorm{Q_\mu}_2 (C_\alpha u)^{\frac2\alpha}}
        \end{align}
        For the choice \(\tilde{C}_\alpha \coloneq \max\set{C_\alpha^{1/2}, C_\alpha^{2/\alpha}, c_\alpha}\), we have
        that
        \[
            t_u \coloneq \tilde{C_\alpha} K^2 \max\set[\Big]{\norm{Q_\mu}_F \sqrt{u}, \opnorm{Q_\mu}_2 u^{\frac2\alpha}}
        \]
        implies \eqref{eq: (sub-Weibull L2 concentration) t bound} and thus \eqref{eq: (sub-Weibull L2 concentration) u bound}
        such that
        \[
            \prob[\Big]{\abs[\Big]{\abs{B_\mu \noiseIndep}_2^2 - \sum_{i=1}^n \E{\noiseIndep_i^2}(Q_\mu)_{ii}} \ge t_u}
            \overset{\eqref{eq: (sub-Weibull L2 concentration) u bound}}\le 2\exp(-u).
        \]
        Thus with probability \(1-2\exp(-u)\) we have
        \[
            \abs{B_\mu \noiseIndep}_2^2
            \le \sum_{i=1}^n \underbrace{\E{\noiseIndep_i^2}}_{\le c_\alpha K^2} (Q_\mu)_{ii} + t_u
            \le \tilde{C}_\alpha K^2 \Bigl(\tr(Q_\mu) + \max\set[\Big]{\norm{Q_\mu}_F \sqrt{u}, \opnorm{Q_\mu}_2 u^{\frac2\alpha}}\Bigr).
        \]
        Observe that \(u=\ln(\frac2\delta)\) has the property \(u\ge \ln(2)\) for \(\delta\in (0,1]\), such that for \(\alpha\in (0,2]\)
        \begin{equation}
            \label{eq: (sub-Weibull L2 concentration) u bound 2}    
            u = u^{1-\frac2\alpha} u^{\frac2\alpha} \le \ln(2)^{1-\frac2\alpha} u^{\frac2\alpha}.
        \end{equation}
        Using Lemma \ref{lem: trace-schatten norm bound} for
        \(\norm{Q_\mu}_F = \tr(Q_\mu^2)\le \sqrt{\opnorm{Q_\mu}_2 \tr(Q_\mu)}\)
        and Lemma \ref{lem: (L2-noise bound) trace and op-norm bounds} to substitute \(Q_\mu\) with \(H_\mu\)
        we thus have that with probability \(1-\delta\)
        and \(u=\ln(\frac2\delta)\) for some constant \(C_\alpha>0\)
        \begin{align*}
            \tr(Q_\mu) + \max\set[\Big]{\norm{Q_\mu}_F \sqrt{u}, \opnorm{Q_\mu}_2 u^{\frac2\alpha}}
            &\le \tr(H_\mu) + \max\set[\Big]{\sqrt{\opnorm{H_\mu}_2 \tr(H_\mu)u}, \opnorm{H_\mu}_2 u^{\frac2\alpha}} 
            \\
            \overset{\eqref{eq: (sub-Weibull L2 concentration) u bound 2}}&\le C_\alpha\Bigl(\tr(H_\mu) + 2\sqrt{\opnorm{H_\mu}_2\tr(H_\mu) u^{\frac2\alpha}} +  \opnorm{H_\mu}_2 u^{\frac2\alpha}\Bigr)
            \\
            &= C_\alpha\Bigl(\sqrt{\tr(H_\mu)} + \sqrt{\opnorm{H_\mu}_2} u^{\frac1\alpha}\Bigr)^2.
        \end{align*}
        Consequently, we have for some constant \(C_\alpha>0\) depending only on
        \(\alpha\) (but not necessarily the same as before), that with probability
        \(1-\delta\)
        \[
            \norm{\regOp(\noise)}_\mu
            \overset{\eqref{eq: parseval identity}}= \abs{B_\mu \noiseIndep}_2
            \le C_\alpha K \Bigl(\sqrt{\tr(H_\mu)} + \sqrt{\opnorm{H_\mu}_2} \ln(\tfrac2\delta)^{\frac1\alpha}\Bigr).
        \]
        \end{steps}

        \item[\ref{it: (L2 concentration sharper) finite p-th moment}] \textbf{Finite \(p\)-th moment.}
        Recall that \(\regOp(\noise)(x) = \mBasisVec(x)^\transpose B_\mu \noiseIndep\). For
        \(\beta(x) \coloneq B_\mu^\conT \overline{\mBasisVec(x)}\) we have
        \(\regOp(\noise)(x)=\beta(x)^\conT\noiseIndep\).
        \begin{equation}
            \label{eq: rosenthal bound for L2 norm}
            \E{\abs{\regOp(\noise)(x)}^p}^{\frac1p}
            = \norm{\beta(x)^\conT \noiseIndep}_{L^p(\pSym)}
            \overset{\eqref{eq: Rosenthal bound on noise}}\le C_{p,\field} \max\set{\kappa, \sigma}\abs{\beta(x)}_2
        \end{equation}
        with \(C_{p,\real}=\frac{14.7p}{\ln p}\) and \(C_{p,\complex}=\sqrt{2}\,C_{p,\real}\).
        We now want to apply the Integral Minkowski inequality \citep[6.19]{follandRealAnalysisModern1999}
with \(q=\frac{p}{2}\ge 1\) to get
        {
\begin{align*}
            \E{\norm{\regOp(\noise)}_\mu^p}^{\frac2p}
            &= \Bigl(\int \Bigl(\int \abs{\regOp(\noise)(x)}^2 \mu(dx)\Bigr)^{\frac{p}2} d\pSym \Bigr)^{\frac2p}
            \\
            \overset{\text{Mink.}}&\le
            \int \Bigl(\int \abs{\regOp(\noise)(x)}^p d\pSym \Bigr)^{\frac2p}\mu(dx)
            \begin{aligned}[t]
                &= \int \E{\abs{\regOp(\noise)(x)}^p}^{\frac2p} d\mu(x)
                \\
                \overset{\eqref{eq: rosenthal bound for L2 norm}}&\le
                C_{p,\field}^2 \max\set{\kappa, \sigma}^2 \int \abs{\beta(x)}_2^2 d\mu(x).
            \end{aligned}
        \end{align*}
        }
        Using linearity and the cyclic property of the trace, the remaining
        integral is given by
        \[
            \int \abs{\beta(x)}_2^2 d\mu(x)
            = \int \underbrace{\abs{B_\mu^\conT\overline{\mBasisVec(x)}}_2^2}_{=\tr\bigl(B_\mu^\conT \overline{\mBasisVec(x)} \mathrlap{\mBasisVec(x)^\transpose B_\mu\bigr)}} d\mu(x)
            \overset{\substack{\text{cyclic}\\\text{linear}}}= \tr\Bigl(B_\mu B_\mu^\conT \underbrace{\int \overline{\mBasisVec(x)}\mBasisVec(x)^\transpose d\mu(x)}_{=\identity}\Bigr)
            = \tr(B_\mu B_\mu^\conT).
        \]
        With the application of Lemma \ref{lem: (L2-noise bound) trace and op-norm bounds} we thus have
        \[
            \E{\norm{\regOp(\noise)}_\mu^p}^{\frac2p}
            \le C_{p,\field}^2 \max\set{\kappa, \sigma}^2 \tr(B_\mu B_\mu^\conT)
\]
        Raising both sides to the power \(\frac{p}2\) yields the claimed moment bound by \(H_\mu = B_\mu B_\mu^\conT\).
        Markov's inequality then implies the claimed concentration bound.
        \qedhere
    \end{enumerate}
\end{proof}

In the proof we used the following lemma to bound the Frobenius norm of \(H_\mu\) in terms of its trace and operator norm.

\begin{lemma}[Trace-Schatten norm inequality]
    \label{lem: trace-schatten norm bound}
    Let \(\norm{A}_{[p]} \coloneq \bigl(\sum_{i=1}^n \sigma_i(A)^p\bigr)^{\frac1p}\) be the Schatten \(p\)-norm
    of a matrix \(A\in \complex^{n\times n}\) where \(\sigma_1(A) \ge \dots \ge \sigma_n(A) \ge 0\) are its singular values.
    Then we have for \(\frac1p + \frac1q =1\) with \(p,q\in [1,\infty]\)
    \[
        \abs{\tr(AB)} \le \norm{A}_{[p]} \norm{B}_{[q]}
        \qquad \forall A,B\in \complex^{n\times n}.
    \]
    In particular, for \(A\) and \(B\) positive definite we have
    \[
        \abs{\tr(AB)} \le \tr(A) \opnorm{B}_2.
    \]
\end{lemma}
\begin{proof}
    Von Neumann's trace inequality \citep{mirskyTraceInequalityJohn1975} implies 
    \(\abs{\tr(AB)} \le \sum_{i=1}^n \sigma_i(A)\sigma_i(B)\).
    Hölder's inequality then gives the claim for the Schatten \(p\)-norms.
    The special case for positive definite \(A\) and \(B\) follows from the fact
    that the singular values of a positive definite matrix are its eigenvalues
    and thereby \(\tr(A) = \norm{A}_{[1]}\) and \(\opnorm{B}_2 = \norm{B}_{[\infty]}\).
\end{proof}

\subsection{Spectral bounds on the Gram matrix}
\label{sec: spectral bounds on the gram matrix}

\subsubsection{Proof of Lemma \ref{lem: gram bound equivalence}: Lower spectral bounds}

Since \(\Gram_*\) is positive semidefinite, we have
for \(\lambda_{\min}(\Gram_*)\) the smallest eigenvalue of \(\Gram_*\),
\begin{align*}
    \opnorm{\Gram_*^{-1}}_2
    = \tfrac1{\lambda_{\min}(\Gram_*)}
    \le \gamma
    &\iff \tfrac1\gamma \le \lambda_{\min}(\Gram_*)
    \\
    &\iff \tfrac1\gamma \identity \preceq \Gram_*
    \iff \identity \preceq \gamma\Gram_*.
\end{align*}
where ``\(\Leftarrow\)'' in the penultimate equivalence follows from \citep[Cor.~7.7.4]{hornMatrixAnalysis2017}.
Let \(a\in \field^m\) be the coefficients of the basis representation
\(\model = \sum_{k=1}^{m} a_k \mBasis_k\), where \(\mBasis_k\) is the orthogonal
basis with respect to \(\scp{\cdot}{\cdot}_*\). Then we get 
the following isometries between \(\field^m\) and \(\modelSpace\):
\begin{equation}
    \label{eq: isometries between Rm and model space for gram bound equivalence}    
    \abs{a}_2^2
    = \norm{\model}_*^2,
    \qquad \text{and} \qquad
    \norm{\model}_{\Omega_n}^2
    = a^\conT \Gram_* a,
\end{equation}
where we used \((\Gram_*)_{ij} = \scp{\mBasis_j}{\mBasis_i}_{\Omega_n}\) and
sesquilinearity for the second equality. Consequently,
\begin{align*}
    \tfrac1\gamma \identity \preceq \Gram_*
    \overset{\text{def.}}&\iff a^\conT (\Gram_* - \tfrac1\gamma\identity) a \ge 0 \text{ for all } a\in \field^m
    \\
    \overset{\eqref{eq: isometries between Rm and model space for gram bound equivalence}}&\iff
    \norm{\model}_{\Omega_n}^2 - \tfrac1\gamma \norm{\model}_*^2 \ge 0 \text{ for all } \model \in \modelSpace
    \\
    &\iff \gamma\norm{\model}_{\Omega_n}^2 \ge \norm{\model}_*^2 \text{ for all } \model \in \modelSpace.
    \qedHere
\end{align*}

\subsubsection{Proof of Theorem \ref{thm: approximately identity gram matrix}: Symmetric spectral bounds}

    Let \(a, b\in \field^m\) be the coefficients of the basis representations
    \(\model = \sum_{k=1}^{m} a_k \mBasis_k\) and \(\tilde\model =
    \sum_{k=1}^{m} b_k \mBasis_k\).
    Then we get the following isometries between \(\field^m\) and \(\modelSpace\):
    \begin{equation}
        \label{eq: isometries between Rm and model space}    
        \scp{\model}{\tilde \model}_*
        = \scp{a}{b}
        \qquad \text{and} \qquad
        \scp{\model}{\tilde \model}_{\Omega_n} = b^\conT \Gram_* a.
    \end{equation}
    The first equality follows directly from the orthonormality of the basis \(\mBasis_k\),
    and the second equality follows from the definition of the Gram matrix \((\Gram_*)_{ij} = \scp{\mBasis_j}{\mBasis_i}_{\Omega_n}\).
\begin{enumerate}[wide=0pt, noitemsep]
    \item[\ref{it: approximately identity Gram matrix} \(\Rightarrow\) \ref{it: approximately equal inner product}:] 
    We have
    \begin{align*}
        \abs{\scp{\model}{\tilde\model}_{\Omega_n} - \scp{\model}{\tilde\model}_*}
        &= \abs{b^\conT (\Gram_* - \identity) a}
        \\
        \overset{\text{Cauchy-Schwarz}}&\le \abs{b}_2 \abs[\big]{(\Gram_* - \identity) a}_2
        \le \opnorm{\Gram_* - \identity}_2 \norm{\model}_* \norm{\tilde\model}_*,
    \end{align*}
    where we used the isometries \eqref{eq: isometries between Rm and model space} in the first equality and the last equality.
    \item[\ref{it: approximately equal inner product} \(\Rightarrow\) \ref{it: approximately equal norms}:] trivial by setting \(\model = \tilde\model\).
    \item[\ref{it: approximately equal norms} \(\Rightarrow\) \ref{it: approximately identity Gram matrix}:]
    Observe that \(M = \Gram_* - \identity\) is a Hermitian matrix such that the operator
    norm is given by the largest absolute eigenvalue and thus
    \[
        \opnorm{M}_2 = \sup_{a\in \field^m\setminus\set{0}} \frac{\abs{a^\conT M a}}{\abs{a}_2^2},
    \]
    However
    \[
        \abs{a^\conT M a}
        = \abs[\big]{a^\conT \Gram_* a - \abs{a}_2^2}
        \overset{\eqref{eq: isometries between Rm and model space}}=
        \abs[\big]{\norm{\model}_{\Omega_n}^2 - \norm{\model}_*^2}
        \overset{\mathrm{\ref{it: approximately equal norms}}}\le
        \epsilon \norm{\model}_*^2
        \overset{\eqref{eq: isometries between Rm and model space}}= \epsilon \abs{a}_2^2,
    \]
    Consequently, we have \(\opnorm{\Gram_* - \identity}_2 = \opnorm{M}_2 \le \epsilon\).
    \qedhere
\end{enumerate}

\subsubsection{Proof of Lemma \ref{lem: bounds on inverse gram matrix}: Bounds on the inverse Gram matrix}
    By Theorem \ref{thm: approximately identity gram matrix},
    the assumption implies for all \(\model \in \modelSpace\)
    \[
        \norm{\model}_*^2 - \norm{\model}_{\Omega_n}^2
        \le \epsilon \norm{\model}_*^2
        \quad\implies\quad
        \norm{\model}_*^2 \le \tfrac1{1-\epsilon} \norm{\model}_{\Omega_n}^2,
    \]
    Lemma \ref{lem: gram bound equivalence} then implies \(\opnorm{\Gram_*^{-1}}_2 \le \frac1{1-\epsilon}\).
    And from \(\Gram_*^{-1} - \identity = \Gram_*^{-1}(\identity - \Gram_*)\)
    the second claim follows by sub-multiplicativity of the operator norm, that is
    \[
        \opnorm{\Gram_*^{-1} - \identity}_2
        \le \opnorm{\Gram_*^{-1}}_2 \opnorm{\identity - \Gram_*}_2
        \le \tfrac{\epsilon}{1-\epsilon}.
        \qedhere
    \]

\subsubsection{Proof of Lemma \ref{lem: gram concentration}: Gram concentration}

    A proof of the unweighted case may be found in \citep[Thm.~1]{cohenStabilityAccuracyLeast2013}. 
    The weighted case was already considered in \citep{cohenOptimalWeightedLeastsquares2017}. Both papers
    originally contained a mistake in the constant \(c(\epsilon)\), which
    was corrected in the unweighted paper \citep{cohenCorrectionStabilityAccuracy2019}.
    The bound \(c(\epsilon) \ge C \epsilon^2\) for \(\epsilon \in [0,1)\) is our contribution.
    For completeness sake we also provide a corrected proof of the weighted case.

    \noindent Observe that \(\E{\Gram_\mu} = \identity\) by
    \begin{align*}
        \E{(\Gram_\mu)_{ij}} = \E{\scp{\mBasis_j^\mu}{\mBasis_i^\mu}_{\Omega_n}}
        &= \E[\Big]{\frac1{n}\sum_{k=1}^n \weight^2(x_k) \overline{\mBasis_i^\mu(x_k)}\mBasis_j^\mu(x_k)}
        \\
        &= \int \overline{\mBasis_i^\mu(x)}\mBasis_j^\mu(x)\weight^2(x)\nu(dx)
        = \int \overline{\mBasis_i^\mu(x)}\mBasis_j^\mu(x)\mu(dx)
        = \delta_{ij}.
    \end{align*}
    The Gram matrix is further of the form \((\Gram_\mu)_{ij} = \sum_{k=1}^n X_k\)
    with independent positive semi-definite rank-one matrices \((X_k)_{ij} \coloneq \frac{\omega_k}{n}
    \overline{\mBasis_i^\mu(x_k)}\mBasis_j^\mu(x_k)\)
    with the property
    \[
        \lambda_{\max}(X_k)
        = \opnorm{X_k}_2
        = \frac{\omega_k}n\sum_{i=1}^m \abs{\mBasis_i^\mu(x_k)}^2
        = \frac{\weight^2(x_k)}{n} K_\modelSpace^\mu(x_k, x_k)
        \le \frac{\norm{\lever_\modelSpace^\mu}_\weight^2}{n} \eqcolon R
    \]
    We may consequently apply the Matrix Chernoff bound \citep[Thm.~1.1]{troppUserFriendlyTailBounds2012} to obtain
    \begin{align*}
        \prob{\lambda_{\min}(\Gram_\mu) \le 1-\epsilon}
        &\le m \Bigl[\frac{e^{-\epsilon}}{(1-\epsilon)^{1-\epsilon}}\Bigr]^{1/R}
        && \epsilon \in [0,1],
        \\
        \prob{\lambda_{\max}(\Gram_\mu) \ge 1+\epsilon}
        &\le m \Bigl[\frac{e^{\epsilon}}{(1+\epsilon)^{1+\epsilon}}\Bigr]^{1/R}
        && \epsilon \ge 0,
    \end{align*}
    where we used that \(\lambda_{\min}(\E{\Gram_\mu}) = \lambda_{\max}(\E{\Gram_\mu}) = 1\).
    To unify the bounds observe that for \(g(\epsilon) \coloneq \epsilon - (1+\epsilon)\ln(1+\epsilon)\)
    we have \(g(\epsilon) \ge g(-\epsilon)\) for all \(\epsilon \in [0, 1)\), since \(h(\epsilon) \coloneq g(\epsilon) - g(-\epsilon)\) is
    zero at \(\epsilon=0\) and increasing for \(\epsilon \in (0,1)\). Consequently,
    \[
        \tfrac{e^{-\epsilon}}{(1-\epsilon)^{1-\epsilon}} = e^{g(-\epsilon)} \le e^{g(\epsilon)} = \tfrac{e^{\epsilon}}{(1+\epsilon)^{1+\epsilon}}
        \qquad \forall \epsilon \in [0,1)
    \]
    and therefore
    \[
        \prob{\opnorm{\Gram_\mu - \identity}_2 > \epsilon}
        \le 2m \Bigl[\frac{e^{\epsilon}}{(1+\epsilon)^{1+\epsilon}}\Bigr]^{1/R}
        = 2m \exp\bigl( - \tfrac{c(\epsilon)}{\norm{\lever_\modelSpace^\mu}_\weight^2}n\bigr)
    \]
    using \(c(\epsilon) = -g(\epsilon)\) and the definition of \(R\). By Lemma \ref{lem: slowly decreasing functions}
    the function \(h(\epsilon) \coloneq c(\epsilon)/\epsilon^2\) is non-increasing and thus for all \(\epsilon \in [0,1)\)
    \[
        c(\epsilon)
        = \epsilon^2 h(\epsilon) \ge \epsilon^2 h(1)
        = \epsilon^2 (2\ln 2 - 1)
        = C \epsilon^2 
    \]
    reinserted into the equation above yields the desired result. With
    \(\epsilon \coloneq 1- \frac1\eta\) we have \(\eta = \frac1{1-\epsilon}\) and the second claim follows from 
    first and Lemma \ref{lem: bounds on inverse gram matrix} due to
    \[
        \prob{\opnorm{\Gram_\mu^{-1}}_2 > \eta}
        = \prob{\opnorm{\Gram_\mu^{-1}}_2 > \tfrac1{1-\epsilon}}
        \overset{\text{Lem.~\ref{lem: bounds on inverse gram matrix}}}\le \prob{\opnorm{\Gram_\mu - \identity}_2 > \epsilon}.
    \]
    What is left is to prove the Lemma we used to bound \(c(\epsilon)\) by \(C\epsilon^2\).

\begin{lemma}
    \label{lem: slowly decreasing functions}
    The function \(h(t) = \frac{c(t)}{t^2}\) is decreasing in \(t\ge 0\) for
    \(c(t) = (1+t)\ln(1+t) - t\). In particular,
    \begin{align*}
        h(t) &\in [h(1), h(0)] = \bigl[2\ln 2-1, \tfrac12\bigr] \subseteq [0.38, 0.5]
        && t \in [0,1]
    \end{align*}
    where the functions \(h\) is continuously extended to \(t=0\).
\end{lemma}
\begin{proof}
    Since \(c'(t) = \ln(1+t)\) we have
    \[
        h'(t)
        = \frac{c'(t)}{t^2} - \frac{2c(t)}{t^3}
        = \frac{t\ln(1+t) - 2[(1+t)\ln(1+t) - t]}{t^3}
        = \frac{2t - (t+2)\ln(1+t)}{t^3}
    \]
    Consequently \(h'(t) = -H(t)/t^3\) with \(H(t) = (t+2)\ln(1+t) - 2t\) and it is sufficient to prove \(H(t) > 0\) for \(t>0\). Observe that \(H(0) = 0\) 
    such that \(H'(t) > 0\) for \(t>0\) implies \(H(t) > 0\) for \(t>0\). And indeed, we have for \(t>0\)
    \[
        H'(t) = \ln(1+t) + \tfrac{t+2}{1+t} - 2 = \ln(1+t) - \tfrac{t}{1+t} > 0,
    \]
    since
    \[
        \ln(1+t) = \int_0^t \frac{1}{1+s} ds > \int_0^t \frac{1}{1+t} ds = \frac{t}{1+t}.
    \]
    We consequently have proven that \(h'(t) < 0\) for all \(t > 0\).
    For the lower bounds we simply evaluate the functions \(h\) at \(t=1\). For the upper bound we need
    to compute the limit as \(t\to 0\). We have by L'Hôpital's rule
    \[
        \lim_{t \to 0}h(t) 
        \overset{\text{L'Hôpital}}= \lim_{t \to 0} \frac{c'(t)}{2t}
        = \lim_{t \to 0} \frac{\ln(1+t)}{2t}
        \overset{\text{L'Hôpital}}= \lim_{t \to 0} \frac{1}{2(1+t)}
        = \frac12.
        \qedhere
    \]
\end{proof}

\subsection{Exact quadrature}

\begin{lemma}\label{lem: A_w and A_1 for exact quadrature}
    Assume the WLS setting \(\Omega_n = \diag[\omega_1,\dots, \omega_n]\), let \(\mu\) be a measure on \(\domain\) 
    \(\Gram_\mu\) the associated Gram matrix.
    \begin{enumerate}[label={(\roman*)},nosep]
        \item \label{it: A_1 approximate Gram identity}
        If \(\ones \in \modelSpace\) and \(\opnorm{\Gram_\mu -\identity}_2 \le \epsilon\) for \(\epsilon\in [0,1)\), then
        \[
            A_\ones(\Omega_n)
            = \norm{\ones}_{\Omega_n}
            \le \sqrt{1+\epsilon}\norm{\ones}_\mu
            = \sqrt{(1+\epsilon)\mu(\domain)}.
        \]
        \item \label{it: A_w exact quadrature}
        If \(\weight(x) = \frac{\sqrt{m}}{\lever_\modelSpace^\mu(x)}\) and \(\Gram_\mu = \identity\), then
        \[
            A_\weight(\Omega_n)
            = \norm{\tfrac1\weight}_{\Omega_n}
            = \norm{\tfrac1\weight}_\mu
            = 1.
        \]
\end{enumerate}
\end{lemma}
\begin{proof}
    \ref{it: A_1 approximate Gram identity} By Theorem \ref{thm: approximately identity gram matrix} we have on the event \(\opnorm{\Gram_\mu -\identity}_2 \le \epsilon\)
    that \(\abs{\norm{\model}_{\Omega_n}^2 - \norm{\model}_\mu^2} \le \epsilon \norm{\model}_\mu^2\) for all \(\model \in \modelSpace\),
    consequently
    \begin{equation}
        \label{eq: bound on sum of omega_i}    
        A_\ones(\Omega_n)^2
        = \frac1n\sum_{i,j=1}^n \frac{\abs{(\Omega_n)_{ij}}}{1 \cdot 1}
        = \frac1n\sum_{i=1}^n \omega_i = \norm{\ones}_{\Omega_n}^2
        \overset{\text{Thm.~\ref{thm: approximately identity gram matrix}}}\le (1+\epsilon) \norm{\ones}_\mu^2.
    \end{equation}
    and we always have \(A_\weight(\Omega_n)= \norm{\tfrac1\weight}_{\Omega_n}\)
    in the case of WLS regression. 

    \ref{it: A_w exact quadrature} We use the definition of the weight and leverage function to get
    \[
        \frac1{\weight^2(x)}
        = \frac{\lever_\modelSpace^\mu(x)^2}{m}
        = \frac{\reprK_\modelSpace^\mu(x,x)}{m}
        = \frac1m \sum_{j=1}^m \abs{\mBasis_j(x)}^2
    \]
    for an orthonormal basis \(\mBasis_1, \dots, \mBasis_m\) of \(\modelSpace\) with respect to \(\scp{\cdot}{\cdot}_\mu\).
    By Corollary \ref{cor: exact quadrature} and \(\Gram_\mu =\identity\) this basis is also orthonormal with respect to \(\scp{\cdot}{\cdot}_{\Omega_n}\) and thus
    \[
        \norm{\tfrac1\weight}_{\Omega_n}
        \overset{\text{def.}}= \sum_{i=1}^n \frac{\omega_i}{\weight^2(x_i)}
        = \frac1m \sum_{j=1}^m \sum_{i=1}^n \omega_i \abs{\mBasis_j(x_i)}^2
        \overset{\text{def.}}= \frac1m \sum_{j=1}^m \norm{\mBasis_j}_{\Omega_n}^2
        = 1.
    \] 
    We similarly have
    \[
        \norm{\tfrac1\weight}_\mu^2
        = \int \frac1{\weight^2(x)} d\mu(x)
        = \frac1m \sum_{j=1}^m \int \abs{\mBasis_j(x)}^2 d\mu(x)
        = 1.\qedhere
    \]
\end{proof}

\subsubsection{Proof of Theorem \ref{thm: specialized L2 noise bounds for exact quadrature}: \texorpdfstring{\(L^2\)}{L2}-bounds for exact quadrature}

Recall that the proof of Theorem \ref{thm: L2 concentration} was
the application of Lemma~\ref{lem: (L2-noise bound) trace and op-norm bounds} to bound
\(\tr(H_\mu)\) and \(\opnorm{H_\mu}_2\) in Theorem \ref{thm: L2 concentration
sharper}. This Lemma \ref{lem: (L2-noise bound) trace and op-norm bounds} provides an alternative bound for the trace in the case of WLS
regression, specifically
\[
    \tr(H_\mu)
    \overset{\text{Lem.~\ref{lem: A_w and A_1 for exact quadrature}}}\le \frac{\opnorm{\Sigma}_2\opnorm{\Gram_\mu^\dagger}_2^2}{n}
    \frac1n \sum_{i=1}^n \omega_i^2\lever_\modelSpace^\mu(x_i)^2
    \overset{\substack{\Gram_\mu=\identity\\ \text{def. }\omega_i}}=\frac{\opnorm{\Sigma}_2}{n}
    \underbrace{\frac1n \sum_{i=1}^n \omega_i m}_{=m \norm{\ones}_{\Omega_n}}
\]
By exact quadrature \(\norm{\ones}_{\Omega_n} = \norm{\ones}_\mu = \sqrt{\mu(\domain)} = 1\), 
which consequently implies
\[
    \tr(H_\mu) \le \opnorm{\Sigma}_2 \frac{m}{n}.
\]
Using this trace bound instead of the general bound from Lemma~\ref{lem:
(L2-noise bound) trace and op-norm bounds} in Theorem \ref{thm: L2 concentration sharper}
together with the usual bound on the operator norm yields the claim.

\subsubsection{Proof of Theorem \ref{thm: feasibility of exact quadrature}: Feasibility of exact quadrature}
\begin{itemize}[wide=0pt,nosep]
    \item \(\dim \modelSpace \le n_\mathrm{exact}\): For \(X=(x_1,\dots, x_n)\) with \(n < \dim \modelSpace\) the
    evaluation map \(\eval_X:\modelSpace\to \field^n\) needs to have a
    non-trivial kernel, which implies that there exists \(0\neq \model
    \in \modelSpace\) with \(\model(x_i)=0\) for all \(i=1,\dots,n\). Exact quadrature
    on \(\modelSpace \cdot \overline{\modelSpace}\) 
    then implies the contradiction
    \[
        0 = \frac1n\sum_{i=1}^n \omega_i \abs{\model(x_i)}^2 = \norm{\model}_\mu^2 > 0.
    \]
    \item \(n_\mathrm{exact} \le \dim(\modelSpace \cdot \overline{\modelSpace})\)
    is the Richter-Tchakaloff theorem \citep[e.g.][Thm.~1.24]{schmudgenMomentProblem2017}.
    \item \(\dim(\modelSpace \cdot \overline{\modelSpace}) \le m^2\)
    follows from the number of all possible basis products.
    In the real case, symmetry removes half of the mixed products yielding the bound \(\dim(\modelSpace \cdot \overline{\modelSpace}) \le \frac{m(m+1)}2\).
\end{itemize}

\subsubsection{Proof of Example~\ref{ex: exact quadrature for monomials of degree one}: Exact quadrature for monomials of degree one}

For the dimension of the product space we need to prove that all products
\(x^{(i)} \overline{x^{(j)}}\) are linearly independent (in the real case up to the
symmetry \(x^{(i)} x^{(j)} = x^{(j)} x^{(i)}\)). Assume there exists a linear combination
\[
    0 = \sum_{i,j=1}^m c_{ij} x^{(i)} \overline{x^{(j)}}
\]
In the real case, we can assume \(c_{ij} = c_{ji}\).
We get that the diagonal elements must be zero, \(c_{kk}=0\), by inserting
standard basis vectors \(e_k\). For the off-diagonal elements, inserting
\(e_k + e_l\) yields
\[
    0 = \sum_{i,j=1}^m c_{ij} (e_k^{(i)} + e_l^{(i)}) \overline{(e_k^{(j)} + e_l^{(j)})} 
    = c_{kk} + c_{kl} + c_{lk} + c_{ll}
    = c_{kl} + c_{lk}.
\]
In the real case, \(x^{(i)} x^{(j)} = x^{(j)} x^{(i)}\) so we would be done at this point
since \(c_{kl} = c_{lk}\). In the complex case, we can insert \(e_k + i e_l\) to get
\[
    0 = \sum_{i,j=1}^m c_{ij} (e_k^{(i)} + i e_l^{(i)}) \overline{(e_k^{(j)} + i e_l^{(j)})}
    = c_{kk} - i c_{kl} + i c_{lk} + c_{ll}
    = - i c_{kl} + i c_{lk}.
\]
Adding the second equation multiplied by \(i\) to the first yields \(c_{kl} = c_{lk} = 0\).
To get exact quadrature, it is sufficient to get exact quadrature on the basis functions
\(x^{(i)} \overline{x^{(j)}}\) of the product space. In other words, we need to match the moment matrix
\[
    M = \int x x^\conT \,\mu(dx),
\]
which is positive definite by the assumption that \(\norm{\cdot}_\mu\) is positive definite on \(\modelSpace\).
Consequently, there exists a spectral decomposition
\[
    M = \frac1m\sum_{i=1}^m \omega_i x_i x_i^\conT
\]
with positive eigenvalues \(\frac{\omega_i}m>0\) and orthonormal eigenvectors
\(x_i\in \complex^m\).  WLS with evaluation points \(x_1,
\dots, x_m\) and weights \(\omega_1, \dots, \omega_m\) therefore yields
exact quadrature and thus \(n_\mathrm{exact} \le m\). The equality \(n_\mathrm{exact} = m\) then follows from
Theorem \ref{thm: feasibility of exact quadrature}.

\subsubsection{Proof of Example \ref{ex: exact quadrature for monomials of degree one, reduced domain}: Minimal number of samples is upper bound}

Since the points in \(\domain\) are the only points used in the proof that
\(\dim(\modelSpace\cdot\overline{\modelSpace}) = m^2\) in Example \ref{ex: exact quadrature for monomials of degree one},
we get the same result for the reduced domain. As \(\modelSpace\cdot\overline{\modelSpace}\) is a subspace
of the vector space of all functions from \(\domain\) to \(\complex\), i.e.~\(\complex^{\domain}\) and
\(\dim \complex^{\domain} = m^2\) since \(\domain\) has \(m^2\) elements, we have \(\modelSpace\cdot\overline{\modelSpace} = \complex^{\domain}\)
by a dimension argument. This implies that \(\ind*{x}\in \modelSpace\cdot\overline{\modelSpace}\) for all \(x \in \domain\).
Exact quadrature must therefore satisfy
\[
    \mu(\set{x}) = \int \ind*{x}\,d\mu = \frac1n\sum_{i=1}^n \weight_i \ind*{x}(x_i)
\]
for all \(x\in \domain\) and by assumption \(\mu(\set{x}) >0\) for all \(x\in \domain\).
Therefore, each \(x\in \domain\) must equal one of the evaluation points \(x_i\)
such that necessarily \(n \ge  \#\domain = m^2\).

\subsubsection{Proof of Proposition \ref{prop: interpolation case}: The interpolation case \texorpdfstring{\(n=m\)}{n=m}}

Since \(\linProj = \regOp \circ \eval_X\) is the identity on \(\modelSpace\), we know that
the restriction of \(\eval_X\) to \(\modelSpace\), \(\eval_X^\modelSpace \coloneq {\eval_X}_{\restriction_\modelSpace}\)
is injective and by \(\dim \modelSpace=n\) bijective. And since \(\linProj_{\restriction\modelSpace} = \id\) we have \(\regOp = (\eval_X^\modelSpace)^{-1}\). This
implies that
\[
    \ell_i(x) = \ell(x)^\conT e_i \overset{\text{def.}}= \regOp(e_i)(x) = (\eval_X^\modelSpace)^{-1}[e_i](x).
\]
Consequently, \(\ell_i = (\eval_X^\modelSpace)^{-1}[e_i]\) and since \(f(x_j) = \eval_X[f]_j\), we have
\[
    \ell_i(x_j) = (\eval_X^\modelSpace)^{-1}[e_i](x_j) = \eval_X[(\eval_X^\modelSpace)^{-1}[e_i]]_j = (e_i)_j = \delta_{ij}.
\]
\ref{it: (interpolation) exact quadrature} \(\Rightarrow\) \ref{it: (interpolation) lagrange}: Since \(\ell_i \in \modelSpace\) we have by exact quadrature and \(n=m\)
\[
    \scp{\ell_i}{\ell_j}_\mu
    = \scp{\ell_i}{\ell_j}_{\Omega_n}
    = \frac1n\sum_{k=1}^n \omega_k \ell_i(x_k) \overline{\ell_j(x_k)}
    = \frac1m\sum_{k=1}^m \omega_k \delta_{ik}\delta_{jk} = \frac{\omega_i}m \delta_{ij}.
\]
\ref{it: (interpolation) lagrange} \(\Rightarrow\) \ref{it: (interpolation) kernel, weights}:
Since \(\linProj \model = \model\) for all \(\model \in \modelSpace\), we have
due to \(\eval_X[f] = \sum_{i=1}^m f(x_i) e_i\) and linearity
\[
    \model
    = \regOp \circ \eval_X[\model]
    = \sum_{i=1}^m \regOp(e_i)\model(x_i)  = \sum_{i=1}^m \model(x_i) \ell_i.
\]
Due to \(\ell_i \neq 0\) as \(\ell_i(x_i) = 1\), we have \(\frac{\omega_i}m = \norm{\ell_i}_\mu^2 > 0\) by positive definiteness of the norm.
Orthogonality of the \(\ell_i\) then implies
\[
    \scp{\model}{\tfrac{m}{\omega_j}\ell_j}_\mu
    = \sum_{i=1}^m \model(x_i) \scp{\ell_i}{\tfrac{m}{\omega_j}\ell_j}_\mu
    = \sum_{i=1}^m \model(x_i)  \delta_{ij}
    = \model(x_j).
\]
Consequently, \(\frac{m}{\omega_j}\ell_j\) is the unique Riesz representer of the
evaluation functional \(\eval_{x_j}\). Since the defining reproducing property 
implies the same for \(\reprK_\modelSpace^\mu(\cdot, x_j)\), we have
\(\frac{m}{\omega_j}\ell_j = \reprK_\modelSpace^\mu(\cdot, x_j)\). Observe that
this also implies the additional claim 
\(\ell_j = \tfrac{\omega_j}{m}\reprK_\modelSpace^\mu(\cdot, x_j)\)
once we have shown the full equivalence.
To finish this part of the proof, we now observe
\[
    \reprK_\modelSpace^\mu(x_i, x_j)
    = \tfrac{m}{\omega_j}\ell_j(x_i)
    = \tfrac{m}{\omega_j}\delta_{ij}.
\]
This implies the claim \(\reprK_\modelSpace^\mu(x_i, x_j) = 0\) for all \(i\neq j\) and \(\omega_i = \frac{m}{\reprK_\modelSpace^\mu(x_i, x_i)} = \frac{m}{\lever_\modelSpace^\mu(x_i)^2}\).
\\
\ref{it: (interpolation) kernel, weights} \(\Rightarrow\) \ref{it: (interpolation) orthonormal kernel}: By the reproducing
property of \(\reprK_\modelSpace^\mu\) (Lemma \ref{lem: reproducing kernel}) we have
\[
    \scp{\reprK_\modelSpace^\mu(\cdot, x_j)}{\reprK_\modelSpace^\mu(\cdot, x_i)}_\mu
    = \reprK_\modelSpace^\mu(x_i, x_j) = \lever_\modelSpace^\mu(x_i)^2 \delta_{ij}.
\]
For the last inequality we used the definition of the leverage function (Def.~\ref{def: leverage function}) for \(i=j\) and
the assumption for \(i\neq j\). Thus, the functions
\[
    \mBasis_i \coloneq  \tfrac{1}{\lever_\modelSpace^\mu(x_i)}\reprK_\modelSpace^\mu(\cdot, x_i)
    \overset{\text{def.~}\omega_i}= \sqrt{\tfrac{\omega_i}{m}}\reprK_\modelSpace^\mu(\cdot, x_i)
\]
form an orthonormal basis of \(\modelSpace\) w.r.t.\ \(\scp{\cdot}{\cdot}_\mu\).
\\
\ref{it: (interpolation) orthonormal kernel} \(\Rightarrow\) \ref{it: (interpolation) exact quadrature}: 
Since \(\mBasis_i\) is an orthonormal basis (ONB) of \(\modelSpace\), the basis representation
of any \(\model, \tilde\model \in \modelSpace\) implies
\[
    \scp{\model}{\tilde\model}_\mu
    \overset{\text{ONB}}= \sum_{i=1}^m \scp{\model}{\mBasis_i}_\mu \overline{\scp{\tilde\model}{\mBasis_i}_\mu}
    \overset{\text{repr.}}= \sum_{i=1}^m \frac{\omega_i}{m}\model(x_i)\overline{\tilde\model(x_i)}
    \overset{\text{def.}}= \scp{\model}{\tilde\model}_{\Omega_n}.
\]
The second inequality uses the reproducing property of \(\reprK_\modelSpace^\mu\) and \(\mBasis_i = \sqrt{\frac{\omega_i}{m}}\reprK_\modelSpace^\mu(\cdot, x_i)\).

\subsection{Proof of Theorem \ref{thm: lebesgue constant concentration}: Lebesgue constant concentration}

    Let \(\mBasisVec = (\mBasis_1, \dots, \mBasis_m)^\transpose\) be a basis-function vector of \(\modelSpace\),
    that is orthonormal with respect to \(L^2(\mu)\). Let \(\Gram_\mu\) be the associated Gram matrix and
    \(\modelVander\) the design matrix. Then the Lagrange function
    \(\ell(x) = \frac1n\Omega_n \modelVander\Gram_\mu^{-1}\mBasisVec(x)\)
    has components \(\ell_i(x) = \frac{\omega_i}{n} \mBasisVec^\conT(x_i) \Gram_\mu^{-1} \mBasisVec(x)\) such that
    the Lebesgue function (Definition \ref{def: lebesgue function and constant}) may be expressed as
    \[
        \lebfun{\weight}{\linProj}{x}
        \overset{\text{def.}}= \weight(x) \sum_{i=1}^n \frac{\abs{\ell_i(x)}}{\weight(x_i)}
        = \frac1n\sum_{i=1}^n \weight(x) \weight(x_i) \abs{\mBasisVec^\conT(x_i) \Gram_\mu^{-1} \mBasisVec(x)},
    \]
    using \(\omega_i = \weight^2(x_i)\). On the continuous side we use
    \(\mu(dx) = \weight^2(x)\nu(dx)\) to get
    \[
        \lebfun{\weight}{\linProj_\mu}{x}
        \overset{\text{def.}}= \weight(x) \int \frac{\abs{\reprK_\modelSpace^\mu(x, y)}}{\weight(y)}\mu(dy)
        = \int \weight(x)\weight(y) \abs{\reprK_\modelSpace^\mu(x, y)}\nu(dy).
    \]
    Since the \(x_i\) are iid samples drawn from \(\nu\) and \(\reprK_\modelSpace^\mu(x, x_i) = \mBasisVec^\conT(x_i) \mBasisVec(x)\)
    this is almost a Monte Carlo approximation of the integral if not for the inverse Gram matrix. 

\begin{steps}
    \item\label{step: (random to continuous) first bound on difference}\textbf{A first bound on the difference.}
    The following arguments will let us assume that \(\weight\equiv 1\) without loss of generality.
    Define the weighted model space \(\modelSpaceW \coloneq \set{\weight \model : \model \in
    \modelSpace}\) with basis \(\mBasisW_i(x)\coloneq \weight(x)
    \mBasis_i(x)\) and basis-function vector \(\mBasisVecW(x) = (\mBasisW_1(x),
    \dots, \mBasisW_m(x))^\transpose\). Observe that the \(\mBasisW_i\) are
    orthogonal with respect to \(L^2(\nu)\) since \(\mu(dx) = \weight^2(x)\nu(dx)\).
    In particular, we have
    \[
        \weight(x)\weight(y)\reprK_\modelSpace^\mu(x,y)
        = \sum_{k=1}^m \weight(x)\mBasis_k(x)\overline{\weight(y)\mBasis_k(y)}
        = \sum_{k=1}^m \mBasisW_k(x)\overline{\mBasisW_k(y)}
        = \reprK_\modelSpaceW^\nu(x,y).
    \]
    and similarly due to \((\Gram_\mu)_{ij} = \scp{\mBasis_j}{\mBasis_i}_{\Omega_n} = \sum_{k=1}^n \frac{\weight^2(x_k)}{n} \overline{\mBasis_i(x_k)}\mBasis_j(x_k) = \scp{\mBasisW_j}{\mBasisW_i}_\identity \eqcolon (\Gram_\nu)_{ij}\)
    \[
        \weight(x) \weight(x_i) \mBasisVec^\conT(x_i) \Gram_\mu^{-1} \mBasisVec(x)
        = \mBasisVecW^\conT(x_i) \Gram_\mu^{-1} \mBasisVecW(x)
        = \mBasisVecW^\conT(x_i) \Gram_\nu^{-1} \mBasisVecW(x).
    \]
    Put together, we consequently have
    \[
        \lebfun{\weight}{\linProj}{x}
        = \frac1n\sum_{i=1}^n \abs{\mBasisVecW^\conT(x_i) \Gram_\nu^{-1} \mBasisVecW(x)}
        \qquad\text{and}\qquad
        \lebfun{\weight}{\linProj_\mu}{x}
        = \int \abs{\reprK_\modelSpaceW^\nu(x, y)}\nu(dy).
    \]
    
        Using \(\mBasisVecW^\conT(x_i)\mBasisVecW(x) = \reprK_\modelSpaceW^\nu(x, x_i)\), the triangle inequality and \(\abs[\big]{\abs{a}- \abs{b}}
        \le \abs{a-b}\) we have
    \begin{align}
        \label{eq: Gram inverse distortion}
        &\abs{\lebfun{\weight}{\linProj}{x}
        - \lebfun{\weight}{\linProj_\mu}{x}}
        \\
        \nonumber
        &\le \underbrace{\frac1n \sum_{i=1}^n
        \abs[\Big]{\mBasisVecW^\conT(x_i) \Gram_\nu^{-1} \mBasisVecW(x) - \mBasisVecW^\conT(x_i) \mBasisVecW(x)}}_{
            \eqcolon \Delta^{\Gram}(x)
        }
        + \underbrace{\abs[\Bigg]{\frac1n \sum_{i=1}^n \abs[\big]{\reprK_\modelSpaceW^\nu(x, x_i)} - \int \abs{\reprK_\modelSpaceW^\nu(x,y)}\nu(dy)}}_{
            \eqcolon \Delta^{\text{MC}}(x)
        }.
    \end{align}

    \item\label{step: (random to continuous) bound on gram distortion} \textbf{Bound on Gram distortion \(\Delta^\Gram\).}
    Observe that \(\Delta^\Gram(x)\) is essentially an \(L^1\)-norm of a discrete probability measure
    that is thus bounded by the \(L^2\)-norm
    \begin{align*}
        \Delta^\Gram(x)
        &\le \Bigl(\frac1n \sum_{i=1}^n \bigl[\mBasisVecW^\conT(x_i) (\Gram_\nu^{-1} -\identity)\mBasisVecW(x)\bigr]^2\Bigr)^{1/2}
        \\
        &= \Bigl(\mBasisVecW^\conT(x)(\Gram_\nu^{-1} -\identity)\underbrace{\frac1n \sum_{i=1}^n \mBasisVecW(x_i) \mBasisVecW^\conT(x_i)}_{=\Gram_\nu}(\Gram_\nu^{-1} -\identity)\mBasisVecW(x)\Bigr)^{1/2}
        \\
        \overset{\text{Cauchy-Schwarz}}&\le 
        \Bigl(\abs[\big]{\mBasisVecW(x)}_2 \abs[\big]{(\identity - \Gram_\nu)(\Gram_\nu^{-1}-\identity) \mBasisVecW(x)}_2 \Bigr)^{1/2}
        \\
        &\le \opnorm{\identity - \Gram_\nu}_2^{1/2} \opnorm{\Gram_\nu^{-1}-\identity}_2^{1/2} \abs{\mBasisVecW(x)}_2
    \end{align*}
    On the event \(\opnorm{\Gram_\nu - \identity}_2 \le \delta\) we also have by
    Lemma \ref{lem: bounds on inverse gram matrix} \(\opnorm{\Gram_\nu^{-1} - \identity}_2 \le\frac{\delta}{1-\delta}\) and therefore
    \[
        \Delta^\Gram(x)
        \le \tfrac{\delta}{\sqrt{1-\delta}} \lever_\modelSpaceW^\nu(x)
        \qquad\text{and}\qquad
        \sup_{x\in \domain} \Delta^\Gram(x)
        \le \tfrac{\delta}{\sqrt{1-\delta}} \norm{\lever_\modelSpace^\mu}_\weight,
    \]
    since \(\abs{\mBasisVecW(x)}_2 = \lever_\modelSpaceW^\nu(x) = \weight(x) \lever_\modelSpace^\mu(x)\).
    For \(\eta = \frac{\delta}{\sqrt{1-\delta}}\) the only positive solution is
    \begin{equation}
        \label{eq: def delta_eta} 
        \delta_\eta =\tfrac{\eta}2\bigl(\sqrt{\eta^2 + 4} - \eta\bigr)
        = \tfrac{2\eta}{\sqrt{\eta^2 + 4} + \eta}.
    \end{equation}
    and therefore by Lemma \ref{lem: gram concentration}
    \[
    \left.\begin{aligned}[c]
        \prob[\Big]{\Delta^\Gram(x) > \eta \lever_\modelSpaceW^\nu(x)} 
        \\
        \prob[\Big]{\sup_{x\in \domain}\Delta^\Gram(x) > \eta \norm{\lever_\modelSpace^\mu}_\weight}
    \end{aligned}
    \right\} \le \prob{\opnorm{\Gram_\nu - \identity}_2 > \delta_\eta}
    \overset{\text{Lem.~\ref{lem: gram concentration}}}\le 
    2m \exp\Bigl(-g(\eta)\frac{n}{\norm{\lever_\modelSpace^\mu}_\weight^2}\Bigr)
    \]
    with \(g(t) \coloneq c(\delta_t)\) where \(c(t) = (1+t)\ln(1+t) - t\) is the function from Lemma \ref{lem: gram concentration}.

    \begin{lemma}[Properties of the function \(g\)]
        \label{lem: property of g}
        The function \(g(t)\) is increasing in \(t\ge 0\) from \(g(0)= 0\)
        and sub-homogeneous on \([0,1]\), that is \(g(\alpha \epsilon) \ge g(\alpha)\epsilon^2\) for all \(\epsilon, \alpha \in [0,1]\).
    \end{lemma}
    \begin{remark}
        For \(\epsilon = 1\) we have \(g(\alpha \epsilon) =
        g(\alpha)\epsilon^2\) and \(g(\alpha)\) is therefore the largest
        possible constant \(a\) such that \(a\epsilon^2\) forms
        a lower bound for all \(\epsilon \in [0,1]\).
    \end{remark}
    \begin{proof}
        \(g(0)=0\) is an immediate consequence of the definition as \(\delta_0 = 0\). 

        Next, we will show that the function \(g\) is increasing. Observe that \(\delta_t\) is increasing, since the derivative of \(\delta_t\) is
        \[
            \tfrac{d}{dt}\delta_t
            = \tfrac{2}{\sqrt{t^2 + 4} + t} - \tfrac{2t^2}{(\sqrt{t^2 + 4} + t)^2\sqrt{t^2 + 4}}
            = \tfrac{2(t^2 + 4 + t\sqrt{t^2 + 4}) - 2t^2}{(\sqrt{t^2 + 4} + t)^2\sqrt{t^2 + 4}},
        \]
        and the numerator is greater than \(8\) using \(t\sqrt{t^2 + 4} \ge t^2\).
        Since \(c'(t) = \ln(1+t) > 0\), \(c\) is increasing and therefore \(g(t) = c(\delta_t)\) is increasing as well.

        Finally, we prove the sub-homogeneity property.
        By Lemma \ref{lem: slowly decreasing functions} we have that \(h(t) =
        \frac{c(t)}{t^2}\) is
        decreasing. This implies that \(h(\delta_t)\) is decreasing, as \(\delta_t\) is increasing
        and therefore \(g(t)/t^2\) is decreasing since
        \[
            \tfrac{g(t)}{t^2}
            \overset{\text{def.~}g}= \tfrac{c(\delta_t)}{t^2}
            \overset{\text{def.~}h}= h(\delta_t) \tfrac{\delta_t^2}{t^2}
            \overset{\text{\eqref{eq: def delta_eta}}}= h(\delta_t) \bigl(\tfrac{2}{\sqrt{t^2 + 4} + t}\bigr)^2.
        \]
        As \(\alpha \epsilon \le \alpha\) we have \(\frac{g(\alpha \epsilon)}{(\alpha \epsilon)^2} \ge \frac{g(\alpha)}{\alpha^2}\) and thus
        \(g(\alpha \epsilon) \ge g(\alpha)\epsilon^2\).
    \end{proof}

    \item\label{step: (random to continuous) pointwise bound on MC}
    \textbf{Hoeffding bound on pointwise Monte Carlo error \(\Delta^{\text{MC}}\).}
    For the pointwise Monte Carlo error we want to apply Hoeffding's inequality
    to the real valued random variables \(\abs{\reprK_\modelSpaceW^\nu(x, x_i)}\)
    and must therefore show that these summands are bounded. Cauchy-Schwarz implies
    \[
        \abs{\reprK_\modelSpaceW^\nu(x, x_i)} \le \lever_\modelSpaceW^\nu(x)\lever_\modelSpaceW^\nu(x_i) \le \lever_\modelSpaceW^\nu(x)\norm{\lever_\modelSpace^\mu}_\weight,
    \]
    using \(\lever_\modelSpaceW^\nu(x_i) = \weight(x_i) \lever_\modelSpace^\mu(x_i)\) in the last inequality.
    The summands in \(\Delta^{\text{MC}}\) are therefore contained in the
    interval \([0,
    \lever_\modelSpaceW^\nu(x)\norm{\lever_\modelSpace^\mu}_\weight]\).  As they
    are equal in expectation to the second term
    \(\E{\abs{\reprK_\modelSpaceW^\nu(x, x_i)}} = \int
    \abs{\reprK_\modelSpaceW^\nu(x,y)}\nu(dy)\) we may therefore apply
    Hoeffding's inequality \citep[Thm.~2]{hoeffdingProbabilityInequalitiesSums1963} to get
    \[
        \prob[\big]{\Delta^{\text{MC}} > \delta \lever_\modelSpaceW^\nu(x)}
        = \prob[\Big]{
            \abs[\Big]{\frac1n \sum_{i=1}^n \abs{\reprK_\modelSpaceW^\nu(x, x_i)} - \E{\abs{\reprK_\modelSpaceW^\nu(x, x_i)}}}
            > \delta\lever_\modelSpaceW^\nu(x)
        }
        \le 2\exp\bigl(-\tfrac{2n\delta^2}{\norm{\lever_\modelSpace^\mu}_\weight^2}\bigr)
    \]
    Using \(\alpha \in (0,1)\) we consequently have for all \(\epsilon \in [0,1]\)
    by the stochastic triangle inequality
    \begin{align*}
        \nonumber
        &\prob[\big]{\abs{\lebfun{\weight}{\linProj}{x} - \lebfun{\weight}{\linProj_\mu}{x}} > \epsilon\lever_\modelSpaceW^\nu(x)}
        \\
        \nonumber
        &\le \prob{\Delta^\Gram(x) > \alpha\epsilon \lever_\modelSpaceW^\nu(x)}
        + \prob{\Delta^{\text{MC}}(x) > (1-\alpha)\epsilon \lever_\modelSpaceW^\nu(x)}
        \\
        &\le 2m \exp\Bigl(-g(\alpha\epsilon)\frac{n}{\norm{\lever_\modelSpace^\mu}_\weight^2}\Bigr)
        + 2\exp\Bigl(-2(1-\alpha)^2\epsilon^2\frac{n}{\norm{\lever_\modelSpace^\mu}_\weight^2}\Bigr)
    \intertext{
        with \(g(t) \coloneq c(\delta_t)\) as defined in \ref{step: (random to continuous) bound on gram distortion}. 
        Using \(g(\alpha \epsilon) \ge g(\alpha)\epsilon^2\) from Lemma \ref{lem: property of g}
        and joining the two terms we get
    }
        &\le 2(m+1) \exp\Bigl(-\underbrace{\min\set[\big]{g(\alpha), 2(1-\alpha)^2}}_{\coloneq \kappa(\alpha)}\frac{n\epsilon^2}{\norm{\lever_\modelSpace^\mu}_\weight^2}\Bigr).
    \end{align*}
    We now want to maximize the rate \(\kappa(\alpha)\) over \(\alpha\).
    As \(g\) is increasing (Lemma \ref{lem: property of g}) and the second term decreasing in \(\alpha\),
    a unique solution exists and can be found by balancing the terms in
    \(\kappa(\alpha)\). Numerical optimization suggests a maximum of more than \(0.1177\in [\frac19, \frac18)\)
    at a point close to \(\alpha \coloneq 3/4\) (Figure \ref{fig: bound optimization}).
Consequently, \(\max_\alpha \kappa(\alpha) \ge 0.1177 \ge \frac19\) proves 
    for all \(\epsilon \in [0,1]\)
    \begin{equation}
        \label{eq: pointwise lebesgue function concentration}
        \prob[\big]{\abs{\lebfun{\weight}{\linProj}{x} - \lebfun{\weight}{\linProj_\mu}{x}} > \epsilon\lever_\modelSpaceW^\nu(x)}
        \le 2(m+1) \exp\Bigl(-\frac{n\epsilon^2}{9\norm{\lever_\modelSpace^\mu}_\weight^2}\Bigr).
    \end{equation}
    Remembering \(\lever_\modelSpaceW^\nu(x) = \weight(x)\lever_\modelSpace^\mu(x)\) and reordering
    proves the high probability bound. The condition on \(n\) stems from
    the assumption \(\epsilon\in [0,1]\).

    \item
    \textbf{Uniform bound on the Monte Carlo error \(\Delta^{\text{MC}}\).}
    The uniform bound requires arguments from the  `uniform law of large numbers' theory \citep[e.g.][Ch.~4]{wainwrightHighDimensionalStatisticsNonAsymptotic2019}.
    Using \(\reprK_\modelSpaceW^\nu(x,x_i) = \mBasisVecW^\conT(x_i)\mBasisVecW(x) = \scp{\mBasisVecW(x)}{\mBasisVecW(x_i)}\)
    and \(\abs{\mBasisVecW(x)}_2 \le \lever_\modelSpaceW^\nu(x) \le \norm{\lever_\modelSpace^\mu}_\weight\)
    we have with \(Y\sim\nu\)
    \begin{align*}
        \sup_{x\in \domain} \Delta^{\text{MC}}(x)
        &= \sup_{x\in \domain}
        \abs[\bigg]{\frac1n \sum_{i=1}^n \abs[\big]{\reprK_\modelSpaceW^\nu(x, x_i)} - \int \abs{\reprK_\modelSpaceW^\nu(x,y)}\nu(dy)}
        \\
        &\le
        \sup_{\abs{v}_2 \le \norm{\lever_\modelSpace^\mu}_\weight}
        \abs[\bigg]{\frac1n \sum_{i=1}^n \abs[\big]{\scp{v}{\mBasisVecW(x_i)}} - \E{\abs{\scp{v}{\mBasisVecW(Y)}}}}
        \eqcolon F(\mBasisVecW(x_1), \dots, \mBasisVecW(x_n)),
    \end{align*}
    With
    \begin{align*}
        F_+ &\coloneq \sup_{\abs{v}_2 \le \norm{\lever_\modelSpace^\mu}_\weight}Z_v
        &
        Z_v &\coloneq \frac1n \sum_{i=1}^n \abs[\big]{\scp{v}{\mBasisVecW(x_i)}} - \E{\abs{\scp{v}{\mBasisVecW(Y)}}}
        \\
        F_- &\coloneq \sup_{\abs{v}_2 \le \norm{\lever_\modelSpace^\mu}_\weight}(-Z_v)
    \end{align*}
    we observe using \(\max\set{a,-a} = \abs{a}\) that
    \[
        F\coloneq F(\mBasisVecW(x_1), \dots, \mBasisVecW(x_n))
        \overset{\text{def.}}= \sup_{\abs{v}_2 \le \norm{\lever_\modelSpace^\mu}_\weight}\abs{Z_v}
        = \max\set{F_+, F_-}.
    \]
    We will first bound the expectation of \(F_+\) and \(F_-\) and then apply McDiarmid's inequality to
    get a high probability bound.
    \begin{steps}
    \item\textbf{Bound on the expectation of \(F_\pm\).} The first step to tackle the
    expectation of \(F_\pm\) is to apply a symmetrization argument. This approach is standard,
    but we could not find a fitting result in the literature, e.g.\
    \citep[Lem.~11.4]{boucheronConcentrationInequalitiesNonasymptotic2013} assumes centered random variables
    and \citep[Prop.~4.11]{wainwrightHighDimensionalStatisticsNonAsymptotic2019} would result
    in an absolute value around the average value inside the sup. So we
    do this manually. Recall that \(x_1,\dots, x_n\) are iid sampled from
    \(\nu\). Let \(y_1,\dots, y_n\)
    be an independent copy of \(x_1,\dots, x_n\) and \(\epsilon_1,\dots,
    \epsilon_n \sim\unif{\set{\pm 1}}\) be independent Rademacher random variables. Since \(\sup\)
    is a convex function, we have by Jensen's inequality applied to
    \(\eSym_{y}\) that
    \begin{align*}
        \E{F_\pm}
        &= \E[\bigg]{\sup_{\abs{v}_2 \le \norm{\lever_\modelSpace^\mu}_\weight} \pm\frac1n \sum_{i=1}^n \abs[\big]{\scp{v}{\mBasisVecW(x_i)}} - \eSym_y[\abs[\big]{\scp{v}{\mBasisVecW(y_i)}}]}
        \\
        \overset{\text{Jensen}}&\le
        \E[\bigg]{\sup_{\abs{v}_2 \le \norm{\lever_\modelSpace^\mu}_\weight} \pm\frac1n \sum_{i=1}^n \abs[\big]{\scp{v}{\mBasisVecW(x_i)}} - \abs[\big]{\scp{v}{\mBasisVecW(y_i)}}}
        \\
        \overset{x_i\sim y_i}&=
        \E[\bigg]{\sup_{\abs{v}_2 \le \norm{\lever_\modelSpace^\mu}_\weight} \frac1n \sum_{i=1}^n \epsilon_i\bigl(\abs[\big]{\scp{v}{\mBasisVecW(x_i)}} - \abs[\big]{\scp{v}{\mBasisVecW(y_i)}}\bigr)}
        \\
        &\le
        \E[\bigg]{\sup_{\abs{v}_2 \le \norm{\lever_\modelSpace^\mu}_\weight} \frac1n \sum_{i=1}^n \epsilon_i\abs[\big]{\scp{v}{\mBasisVecW(x_i)}}
        + \sup_{\abs{v}_2 \le \norm{\lever_\modelSpace^\mu}_\weight} \frac1n \sum_{i=1}^n -\epsilon_i\abs[\big]{\scp{v}{\mBasisVecW(y_i)}}}
        \\
        &\le 2\E[\bigg]{\sup_{\abs{v}_2 \le \norm{\lever_\modelSpace^\mu}_\weight} \frac1n \sum_{i=1}^n \epsilon_i\abs[\big]{\scp{v}{\mBasisVecW(x_i)}}}.
    \end{align*}
    Observe that we do not have an absolute value around the average value which citing
    \citep[Prop~4.11]{wainwrightHighDimensionalStatisticsNonAsymptotic2019} would have caused.
    This is important for the application of the contraction principle in the next step.
    So far we have only worked with the real random variables \(\abs{\scp{v}{\mBasisVecW(x_i)}}\).
    However, as we are removing the absolute value in the following, we need to
    distinguish between the real and complex case.
    \begin{casebycase}
        \item \(\field = \real\): 
        We may use the standard contraction principle \citep[Thm.~11.6]{boucheronConcentrationInequalitiesNonasymptotic2013}
        as the absolute value is a \(1\)-Lipschitz function that vanishes
        at \(0\), such that
        \begin{align*}
            \E{F_\pm}
            &\le 2\E[\Big]{
                \sup_{\abs{v}_2 \le \norm{\lever_\modelSpace^\mu}_\weight}
                \abs[\Big]{
                    \frac1n \sum_{i=1}^n \epsilon_i \scp{v}{\mBasisVecW(x_i)}
                }
            }
            \\
            &= 2\E[\Big]{
                \sup_{\abs{v}_2 \le \norm{\lever_\modelSpace^\mu}_\weight}
                \abs[\Big]{
                    \scp[\Big]{v}{\frac1n \sum_{i=1}^n \epsilon_i\mBasisVecW(x_i)}
                }
            }
            \\
            &= 2\norm{\lever_\modelSpace^\mu}_\weight
            \E[\Big]{
                \abs[\Big]{\frac1n \sum_{i=1}^n \epsilon_i\mBasisVecW(x_i)}_2
            }
            \\
            \overset{\text{Jensen}}&\le \frac{2\norm{\lever_\modelSpace^\mu}_\weight}n
            \E[\Big]{
                \abs[\Big]{\sum_{i=1}^n \epsilon_i\mBasisVecW(x_i)}_2^2
            }^{\frac12}.
        \end{align*}
        Since the Rademacher random variables are independent with mean zero and
        variance \(\E{\epsilon_i^2} = 1\), the mixed terms vanish and we get
        \[
            \E{F_{\pm}}
            \le \frac{2\norm{\lever_\modelSpace^\mu}_\weight}n
            \Bigl(
                \sum_{i=1}^n \E[\big]{\abs[\big]{\mBasisVecW(x_i)}_2^2}
            \Bigr)^{\frac12}.
        \]

        \item \(\field = \complex\): In this case we treat
        \(\scp{v}{\mBasisVecW(x_i)} = \Re(\scp{v}{\mBasisVecW(x_i)}) +
        i\Im(\scp{v}{\mBasisVecW(x_i)})\) as a two-dimensional real vector and
        the absolute value \((a+ib) \mapsto \sqrt{a^2 + b^2}\) is the
        \(1\)-Lipschitz function on \(\real^2\) that vanishes at \(0\).
        This requires a vector-valued version of the contraction principle,
        which is provided by \citet[Cor.~1,
        p.~6]{maurerVectorContractionInequalityRademacher2016}. With it we get
        with \(\epsilon_{i,1}, \epsilon_{i,2}\) independent Rademacher random variables and \(\zeta_i = \epsilon_{i,1} + i\epsilon_{i,2}\).
        \begin{align*}
            \E{F_\pm}
            &\le 2\E[\bigg]{\sup_{\abs{v}_2 \le \norm{\lever_\modelSpace^\mu}_\weight} \frac1n \sum_{i=1}^n \epsilon_i\abs[\big]{\scp{v}{\mBasisVecW(x_i)}}}.
            \\
            &\le 2\sqrt{2}\E[\Big]{
                \sup_{\abs{v}_2 \le \norm{\lever_\modelSpace^\mu}_\weight}
                \abs[\Big]{
                    \frac1n \sum_{i=1}^n
                    \underbrace{
                    \epsilon_{i,1} \Re(\scp{v}{\mBasisVecW(x_i)})
                    + \epsilon_{i,2} \Im(\scp{v}{\mBasisVecW(x_i)})
                    }_{=\Re\bigl(\overline{\zeta_i} \scp{v}{\mBasisVecW(x_i)}\bigr)}
                }
            }
            \\
            &= 2\sqrt{2}\E[\Big]{
                \sup_{\abs{v}_2 \le \norm{\lever_\modelSpace^\mu}_\weight}
                \abs[\Big]{
                    \Re\Bigl(\scp[\Big]{v}{\frac1n \sum_{i=1}^n\zeta_i \mBasisVecW(x_i)}\Bigr)
                }
            }
            \\
            &= 2\sqrt{2}\norm{\lever_\modelSpace^\mu}_\weight\E[\Big]{
                \abs[\Big]{\frac1n \sum_{i=1}^n\zeta_i \mBasisVecW(x_i)}_2
            }
            \\
            \overset{\text{Jensen}}&\le \frac{2\sqrt{2}\norm{\lever_\modelSpace^\mu}_\weight}n
            \E[\Big]{\abs[\Big]{\sum_{i=1}^n \zeta_i \mBasisVecW(x_i)}_2^2}^{\frac12}.
        \end{align*}
        Since the \(\zeta_i\) are independent random variables in \(\complex\) with mean zero,
        the mixed terms vanish and we get
        \[
            \E[\Big]{\abs[\Big]{\sum_{i=1}^n \zeta_i \mBasisVecW(x_i)}_2^2}
            = \sum_{i=1}^n \E{\abs{\zeta_i \mBasisVecW(x_i)}_2^2}
            = \sum_{i=1}^n \underbrace{\E{\abs{\zeta_i}_2^2}}_{=2} \E{\abs{\mBasisVecW(x_i)}_2^2},
        \]
        using the scaling property of the complex two norm \(\abs{\zeta_i\mBasisVecW(x_i)}_2 = \abs{\zeta_i}\abs{\mBasisVecW(x_i)}_2\)
        and independence.
    \end{casebycase}
    In conclusion, we have for both the real and complex case
    \[
        \E{F_{\pm}}
        \le \frac{c_\field\norm{\lever_\modelSpace^\mu}_\weight}n
        \Bigl(
            \sum_{i=1}^n \E[\big]{\abs[\big]{\mBasisVecW(x_i)}_2^2}
        \Bigr)^{\frac12}.
    \]
    with \(c_\real= 2\) and \(c_\complex = 4\).
    And since \(\mBasisW_1, \dots, \mBasisW_m\) are orthonormal with respect to \(L^2(\nu)\) we have
    that each summand is equal to \(m\), indeed
    \[
        \E{\abs{\mBasisVecW(x_i)}_2^2}
        = \E[\Big]{\sum_{k=1}^m \abs{\mBasisW_k(x_i)}^2}
        = \sum_{k=1}^m \int \abs{\mBasisW_k(x)}^2 \nu(dx)
        = m.
    \]
    Put together, this implies
    \[
        \E{F_{\pm}}
        \le c_\field \norm{\lever_\modelSpace^\mu}_\weight \sqrt{\tfrac{m}{n}}.
    \]

    \item \textbf{Concentration with McDiarmid's inequality.}
    We will now show concentration of both \(F_+\) and \(F_-\) around their expectations using McDiarmid's inequality.
    To apply McDiarmid's inequality, we need to show that \(F_\pm=F_\pm(\mBasisVecW(x_1), \dots, \mBasisVecW(x_n))\) has the bounded difference property. Recall that
    any of its inputs \(\mBasisVecW(x_i)\) have the property \(\abs{\mBasisVecW(x_i)}_2 \le \norm{\lever_\modelSpace^\mu}_\weight\).
    We must therefore show that there exists a constant \(c > 0\) such that for all \(y_i \in V \coloneq \set{y \in \field^m : \abs{y}_2 \le \norm{\lever_\modelSpace^\mu}_\weight}\)
    \[
        \sup_{y_i' \in V}
        \!\abs[\big]{
            F(y_1, \dots, y_{i-1}, y_i, y_{i+1}, \dots, y_n)
            - F(y_1, \dots, y_{i-1}, y_i', y_{i+1}, \dots, y_n)
        } \le c.
    \]
    Due to the symmetry in the definition of \(F_\pm\), we may only consider \(i=1\) without loss of
    generality. Indeed, recall that for
    \[
        Z_v(y_1) = Z_v(y_1, \dots, y_n)
        \coloneq \frac1n \sum_{i=1}^n \abs[\big]{\scp{v}{y_i}} - \E{\abs{\scp{v}{\mBasisVecW(Y)}}}
    \]
    we have \(F_\pm(y_1) \coloneq F_\pm(y_1, \dots, y_n) = \sup_{v\in V} \pm Z_v(y_1, \dots, y_n)\). Then
    \[
        \sup_{y_1'\in V}
        \abs[\big]{F_+(y_1) - F_+(y_1')}
        = \sup_{y_1'\in V}
        \abs[\Big]{\sup_{v\in V} \inf_{v'\in V}Z_v(y_1) -  Z_{v'}(y_1')}
        \le \sup_{y_1'\in V} \sup_{v\in V}
        \abs[\big]{Z_v(y_1) - Z_v(y_1')}
    \]
    and similarly
    \[
        \sup_{y_1'\in V}
        \abs[\big]{F_-(y_1) - F_-(y_1')}
        = \sup_{y_1'\in V}
        \abs[\Big]{\sup_{v'\in V} \inf_{v\in V}Z_{v'}(y_1') -  Z_{v}(y_1)}
        \le \sup_{y_1'\in V} \sup_{v\in V}
        \abs[\big]{Z_v(y_1) - Z_v(y_1')}
    \]
    thus we have by \(\abs{\abs{a} - \abs{b}} \le \abs{a-b}\) and the Cauchy-Schwarz inequality
    \begin{align*}
        \abs[\big]{Z_v(y_1) -Z_v(y_1')}
        = \tfrac1n\abs[\big]{\abs{\scp{v}{y_1}} - \abs{\scp{v}{y_1'}}}
        &\le  \tfrac1n\abs{\scp{v}{y_1 - y_1'}}
        \\
        \overset{\smash{\text{C.S.}}}&\le \tfrac1n\abs{v}_2 \abs{y_1 - y_1'}_2,
        \le \tfrac2n \norm{\lever_\modelSpace^\mu}_\weight^2,
    \end{align*}
    where we use \(v, y_1, y_1' \in V\) and therefore \(\abs{v}_2, \abs{y_1}_2, \abs{y_1'}_2 \le \norm{\lever_\modelSpace^\mu}_\weight\)
    in the last inequality. Put together, this implies
    \[
        \sup_{y_1'\in V}
        \abs[\big]{F_\pm(y_1, \dots, y_n) - F_\pm(y_1', y_2, \dots, y_n)}
        \le \tfrac{2}n\norm{\lever_\modelSpace^\mu}_\weight^2.
    \]
    McDiarmid's bounded difference inequality \citep[e.g.][Thm.~6.2]{boucheronConcentrationInequalitiesNonasymptotic2013}
    then implies
    \[
        \prob[\big]{F_\pm(\mBasisVecW(x_1), \dots, \mBasisVecW(x_n)) - \E{F_\pm(\mBasisVecW(x_1), \dots, \mBasisVecW(x_n))} > t}
        \le \exp\Bigl(-\frac{nt^2}{2\norm{\lever_\modelSpace^\mu}_\weight^4}\Bigr).
    \]

    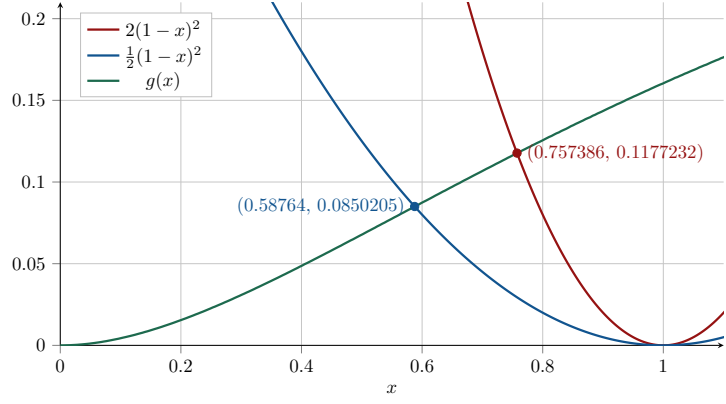
\begin{figure}
        \centering
        \resizebox{0.7\linewidth}{!}{ \begin{tikzpicture}
\begin{axis}[
    width=14cm,
    height=8cm,
    xmin=0, xmax=1.1,
    ymin=0, ymax=0.21,
    axis lines=left,
    xlabel={$x$},
    ylabel={},
    grid=both,
    major grid style={gray!45},
    minor grid style={gray!20},
    tick align=outside,
    xtick={0,0.2,0.4,0.6,0.8,1},
    ytick={0,0.05,0.10,0.15,0.20},
    yticklabel style={
        /pgf/number format/fixed,
        /pgf/number format/precision=2
    },
    samples=300,
    domain=0:1.2,
    enlargelimits=false,
    clip=true,
    legend style={
        at={(0.03,0.97)},
        anchor=north west,
        fill=white,
        draw=gray!50
    },
    declare function={
        delta(\x) = 2*\x/(sqrt(\x^2 + 4) + \x);
        cfun(\x) = (1+\x)*ln(1+\x) - \x;
        g(\x) = cfun(delta(\x));
    },
    every node/.style={font=\small},
    scale=1,
]

\addplot[
    very thick,
    red!75!black
]
{2*(1-x)^2};
\addlegendentry{$2(1-x)^2$}

\addplot[
    very thick,
    blue
]
{0.5*(1-x)^2};
\addlegendentry{$\frac12(1-x)^2$}

\addplot[
    very thick,
    green!90!white
]{g(x)};
\addlegendentry{$g(x)$}

\addplot[
    only marks,
    mark=*,
    mark size=2.2pt,
    red!75!black
]
coordinates {(0.757386,0.1177232)};

\addplot[
    only marks,
    mark=*,
    mark size=2.2pt,
    blue
]
coordinates {(0.58764,0.0850205)};

\node[
    anchor=west,
    text=red!75!black
] at (axis cs:0.757386,0.1177232) {$\,(0.757386,\,0.1177232)$};

\node[
    anchor=east,
    text=blue
] at (axis cs:0.58764,0.0850205) {$(0.58764,\,0.0850205)\,$};

\end{axis}
\end{tikzpicture} }
        \caption{Maximization of \(\kappa(\alpha) = \min\set{g(\alpha), 2(1-\alpha)^2}\) and \(\tilde\kappa(\theta) = \min\set{g(\theta), \frac12(1-\theta)^2}\).}
        \label{fig: bound optimization}
    \end{figure}

    \item\textbf{Assembly and optimization.}
    Put together we have
    \begin{align*}
        &\prob[\big]{\abs{\lebconst{\weight}{\linProj} - \lebconst{\weight}{\linProj_\mu}} > \epsilon \norm{\lever_\modelSpace^\mu}_\weight}
        \\
        &\le
        \prob[\big]{\abs{\lebconst{\weight}{\linProj} - \lebconst{\weight}{\linProj_\mu}} > \epsilon \norm{\lever_\modelSpace^\mu}_\weight}
        \\
        &\le \prob[\Big]{\sup_{x\in \domain} \abs{\lebfun{\weight}{\linProj}{x} - \lebfun{\weight}{\linProj_\mu}{x}} > \epsilon \norm{\lever_\modelSpace^\mu}_\weight}
        \\
        &\le \prob[\Big]{\sup_{x\in \domain}\Delta^\Gram(x) > \alpha\epsilon \norm{\lever_\modelSpace^\mu}_\weight}
        + \prob[\Big]{\sup_{x\in \domain}\Delta^{\text{MC}}(x) > (1-\alpha)\epsilon \norm{\lever_\modelSpace^\mu}_\weight}
        \\
        &\le 2m \exp\Bigl(-g(\alpha\epsilon)\frac{n}{\norm{\lever_\modelSpace^\mu}_\weight^2}\Bigr)
        + \prob[\Big]{F(\mBasisVecW(x_1), \dots, \mBasisVecW(x_n)) > (1-\alpha)\epsilon \norm{\lever_\modelSpace^\mu}_\weight}
    \end{align*}
    with \(g(t) = c(\delta_t)\), \(c(t) = (1+t)\ln(1+t) - t\) and \(\delta_t\) as defined in \eqref{eq: def delta_eta}.
    Using \(F=\max\set{F_+, F_-}\) we have
    \[
        \prob[\Big]{F > (1-\alpha)\epsilon \norm{\lever_\modelSpace^\mu}_\weight}
        \le 
        \prob[\Big]{F_+ > (1-\alpha)\epsilon \norm{\lever_\modelSpace^\mu}_\weight}
        + \prob[\Big]{F_- > (1-\alpha)\epsilon \norm{\lever_\modelSpace^\mu}_\weight}
    \]
    And with the bound on the expectations of \(F_\pm\) and McDiarmid's inequality we have
    \begin{align*}
        \prob[\Big]{F > (1-\alpha)\epsilon \norm{\lever_\modelSpace^\mu}_\weight}
        &\le \prob[\bigg]{F_\pm - \E{F_\pm} > (1-\alpha)\epsilon \norm{\lever_\modelSpace^\mu}_\weight - c_\field\norm{\lever_\modelSpace^\mu}_\weight\sqrt{\tfrac{m}{n}}}
        \\
        &\le \exp\Bigl(-\frac{n\bigl((1-\alpha)\epsilon - c_\field\sqrt{\frac{m}{n}}\bigr)_+^2}{2\norm{\lever_\modelSpace^\mu}_\weight^2}\Bigr),
    \end{align*}
    where \(x_+ = \max\set{x,0}\) denotes the positive part of \(x\). Note that the negative case
    results in the vacuous bound \(\exp(-0) = 1\).
    Put together we thus have
    \begin{align*}
        &\prob[\Big]{\abs{\lebconst{\weight}{\linProj} - \lebconst{\weight}{\linProj_\mu}} > \epsilon \norm{\lever_\modelSpace^\mu}_\weight}
        \\
        &\le 2m \exp\Bigl(-g(\alpha\epsilon)\frac{n}{\norm{\lever_\modelSpace^\mu}_\weight^2}\Bigr)
        + 2\exp\Bigl(-\frac{n\bigl((1-\alpha)\epsilon - c_\field\sqrt{\frac{m}{n}}\bigr)_+^2}{2\norm{\lever_\modelSpace^\mu}_\weight^2}\Bigr).
    \intertext{
        with \(g(t) \coloneq c(\delta_t)\) as defined in \ref{step: (random to continuous) bound on gram distortion}. 
        Using \(g(\alpha \epsilon) \ge g(\alpha)\epsilon^2\) from Lemma \ref{lem: property of g}
        and joining the two terms we get
    }
        &\le
        2(m+1) \exp\Bigl(-\underbrace{\min\set[\Big]{g(\alpha), \tfrac12\Bigl(1-\alpha - c_\field \sqrt{\tfrac{m}{n\epsilon^2}}\Bigr)_+^2}}_{=h(\alpha, \frac{n\epsilon^2}{m})} \frac{n\epsilon^2}{\norm{\lever_\modelSpace^\mu}_\weight^2}\Bigr)
    \end{align*}
    Ideally, we would maximize the rate function
    \[
        h(\alpha, \rho) \coloneq \min\set[\big]{g(\alpha), \tfrac12\bigl(1-\alpha - \tfrac{c_\field}{\sqrt{\rho}}\bigr)_+^2}
        \qquad \text{with}\qquad \rho \coloneq \tfrac{n\epsilon^2}{m}
    \]
    over \(\alpha \in (0,1)\). To avoid the second term being zero, which
    results in a vacuous bound, we need \(0<\alpha < 1
    -\frac{c_\field}{\sqrt{\rho}} \eqcolon a\). For this interval to not be empty, we
    need \(a>0\) and this requires \(\rho > c_\field^2\), which we will now assume
    without loss of generality since the case \(\rho<c_\field^2\) results in a vacuous
    bound in the theorem statement of \(\prob{\dots} \le \exp(0) = 1\). To ensure \(\alpha < a\), we reparametrize
    \(\alpha = \theta a\) with \(\theta \in [0,1]\) to get
    \[
        (1-\alpha - \tfrac{4}{\sqrt{\rho}})^2
        = (a- \theta a)^2 = a^2 (1-\theta)^2.
    \]
    Together with \(g(\alpha) = g(\theta a) \ge g(\theta)a^2\) using Lemma \ref{lem: property of g}, this implies
    \[
        h(\alpha, \rho) \ge a^2 \min\set[\big]{g(\theta), \tfrac12(1-\theta)^2} \eqcolon a^2 \tilde\kappa(\theta).
    \]
    Note that applying Lemma \ref{lem: property of g} twice loses no power,
    since \(g(\alpha \epsilon) = g(\theta a \epsilon) \ge g(\theta) (a \epsilon)^2\).
    We are left with a similar one parameter optimization problem as in the
    pointwise case. This time the second term is halved and not
    doubled. Numerical optimization yields a maximum of \(0.0850 \ge
    \frac1{12}\) (Figure \ref{fig: bound optimization}). Inserting this into the bound yields the uniform bound
    \begin{equation}
        \label{eq: uniform leb constant concentration}    
        \prob[\Big]{\abs{\lebconst{\weight}{\linProj} - \lebconst{\weight}{\linProj_\mu}} > \epsilon \norm{\lever_\modelSpace^\mu}_\weight}
        \le 2(m+1) \exp\Bigl(- \frac{a^2}{12} \frac{n\epsilon^2}{\norm{\lever_\modelSpace^\mu}_\weight^2}\Bigr),
    \end{equation}
    where we recall \(a^2 \coloneq (1- \frac{c_\field}{\sqrt{\rho}})^2\) and \(\rho\coloneq\frac{\epsilon^2n}{m}\) by definition with \(c_\real =2\) and \(c_\complex = 4\).
\end{steps}
    \item \textbf{Reordering for high-probability bound.} To get a high probability bound
    from Equation \eqref{eq: uniform leb constant concentration}, we set
    \[
        2(m+1) \exp\Bigl(- \frac{(1-\frac{c_\field}{\sqrt{\rho}})^2}{12} \frac{n\epsilon^2}{\norm{\lever_\modelSpace^\mu}_\weight^2}\Bigr)
        \overset{!}= \delta,
    \]
    and reorder to get, up to the requirement that \(\epsilon\in[0,1]\), that
    \begin{equation}
        \label{eq: (leb const concentration) epsilon equation}    
        \epsilon^2 = \tfrac{12}{(1-\frac{c_\field}{\sqrt{\rho}})^2}
        \tfrac{\norm{\lever_\modelSpace^\mu}_\weight^2}{n}
        \ln\bigl(2\tfrac{m+1}{\delta}\bigr) \overset{!}\in [0,1].
    \end{equation}
    The difficulty is that \(\rho\) itself is defined in terms of \(\epsilon\), which creates a circular definition. 
    To skip the pain of solving this fixed point equation, we begin by postulating
    \[
        \rho \coloneq \Bigl(c_\field + \sqrt{12\ln(2\tfrac{m+1}{\delta}) \tfrac{\norm{\lever_\modelSpace^\mu}_\weight^2}{m}}\Bigr)^2,
    \]
    which is not dependent on \(\epsilon\) such that we may take Equation
    \eqref{eq: (leb const concentration) epsilon equation} as definition. To see that
    this choice of \(\rho\) and \(\epsilon\) satisfies \(\rho = \frac{\epsilon^2n}{m}\)
    observe that
    \[
        \sqrt{\rho}(1-\tfrac{c_\field}{\sqrt{\rho}}) 
        = \sqrt{\rho} - c_\field
        \overset{\text{def.\ }\rho}= \sqrt{12\ln(2\tfrac{m+1}{\delta}) \tfrac{\norm{\lever_\modelSpace^\mu}_\weight^2}{m}}
    \]
    implies
    \[
        \rho
        = \frac{12\ln(2\tfrac{m+1}{\delta}) }{(1-\tfrac{c_\field}{\sqrt{\rho}})^2}\frac{\norm{\lever_\modelSpace^\mu}_\weight^2}{m}
        \overset{\text{\eqref{eq: (leb const concentration) epsilon equation}}}= \frac{\epsilon^2n}{m}.
    \]
    To get rid of \(\rho\) in \(\epsilon\)
    we define with \(c_\real = 2\) and \(c_\complex = 4\)
    \begin{equation}
        \label{eq: constant C_field,m definition}    
        C_{\field, m, \delta} \coloneq \sqrt{12}\Bigl(1-\tfrac{c_\field}{c_\field + \sqrt{12\ln(2\frac{m+1}{\delta})}}\Bigr)^{-1} \ge \sqrt{12}(1-\tfrac{c_\field}{\sqrt{\rho}})^{-1}
    \end{equation}
    using \(\sqrt{\rho} \ge c_\field + \sqrt{12\ln(2\tfrac{m+1}{\delta})}\) due to \(\norm{\lever_\modelSpace^\mu}_\weight^2 \ge m\).
    The assumption \(\epsilon \le 1\) is therefore satisfied if
    \[
        \epsilon^2 \le C_{\field, m, \delta}^2 \tfrac{\norm{\lever_\modelSpace^\mu}_\weight^2}{n} \ln(2\tfrac{m+1}{\delta}) \le 1,
    \]
    which yields the lower bound on \(n\). By definition of \(\epsilon\) we
    consequently have with probability at least \(1-\delta\)
    \[
        \abs{\lebconst{\weight}{\linProj} - \lebconst{\weight}{\linProj_\mu}}
        \le \epsilon \norm{\lever_\modelSpace^\mu}_\weight
        \le C_{\field,m, \delta} \tfrac{\norm{\lever_\modelSpace^\mu}_\weight^2}{\sqrt{n}} \sqrt{\ln(2\tfrac{m+1}{\delta})}.
    \]
    We have \(C_{\field, m, \delta}\to \sqrt{12}\) as \(m\to \infty\) almost by definition. To get
    universal bounds, observe that \(m\ge 1\) and \(\delta \in [0,1]\) implies
    \[
        12\ln(2\tfrac{m+1}{\delta})
        \ge 12 \ln(4) \approx 16.635... \ge 16
    \]
    Using \(c_\real = 2\) and \(c_\complex = 4\) we thus have
    \[
        C_{\field, m, \delta} \le \sqrt{12}\bigl(1-\tfrac{c_\field}{c_\field + 4}\bigr)^{-1} 
        = \begin{cases}
            \sqrt{4\cdot 12}=4\sqrt{3} = C_\complex
            & \field = \complex
            \\
            \sqrt{12}\frac{3}{2} = \sqrt{27} = 3\sqrt{3} = C_\real
            & \field = \real.
        \end{cases}
    \]
\end{steps}

\subsection{Metric projection error}

\begin{lemma}[Lipschitz approximation]
    \label{lem: lipschitz approximation}
    Assume that \(\domain\subseteq \real^\dims\) is a convex domain equipped with the sup-norm
    metric. For \(f\colon \domain \to \real\) define for \(h>0\)
    \[
        q(x) \coloneq \inf_{y\in \domain}\set[\big]{f(y) + \tfrac{\modCont(f; h)}{h} \abs{x-y}_\infty} + \tfrac{\modCont(f; h)}{2}
    \]
    Then \(\Lip(q) \le \frac{\modCont(f; h)}{h}\) and \(\norm{f-q}_\infty \le \frac12\modCont(f; h)\).
\end{lemma}

\begin{proof}
    Since \(x\mapsto \abs{x-y}_\infty\) is \(1\)-Lipschitz continuous, we clearly have
    \(\Lip(q) \le \frac{\modCont(f; h)}{h}\). What is left to prove is the approximation
    property for which we assume without loss of generality \(\modCont(f;h) < \infty\).
    Since \(\domain\) is convex, we can find a chain of points starting in
    \(x\) and ending in \(y\) each with distance at most \(h\) to get with
    a telescoping sum and the definition of the modulus of continuity
    \[
        \abs{f(x) - f(y)}
        \le \ceil[\big]{\tfrac{\abs{x-y}_\infty}{h}}\modCont(f; h)
        \le \bigl(1+\tfrac{\abs{x-y}_\infty}{h}\bigr)\modCont(f; h).
    \]
    Since this inequality holds for all \(y\in \domain\), dropping the absolute
    value around the first term and rearranging yields
    \[
        f(x) - \modCont(f; h) \le \inf_{y\in \domain}\set[\big]{ f(y) + \tfrac{\modCont(f; h)}{h} \abs{x-y}_\infty}
        \le f(x),
    \]
    where the last inequality follows by choosing \(y=x\). The middle term becomes \(q(x)\) when we add \(\frac12\modCont(f; h)\) to all sides
    such that \(f(x) - \frac12\modCont(f; h) \le q(x) \le f(x) + \frac12\modCont(f; h)\) and since this holds for all \(x\in \domain\) we get \(\norm{f-q}_\infty \le \frac12\modCont(f; h)\).
\end{proof}

\subsubsection{Proof of Corollary~\ref{cor: close to optimal jackson bound}: Close to optimal Jackson bound}

We first use a variant of Trick 2 (Rem.~\ref{rem: derivative lifting}) to reduce the problem to the case \(k=0\):
\[
    \error_M^\infty[f]
    \overset{\text{Trick 2}}= \inf_{p \in \poly_M} \error_M^\infty[f-p]
    \overset{\eqref{eq: sinwell favard bound}}\le
    \inf_{p \in \poly_M}
    \favConst_k \tfrac{(b-a)^k}{2^k(m)_k}\norm{f^{(k)}-p^{(k)}}_\infty
    =
    \favConst_k \tfrac{(b-a)^k}{2^k(m)_k}\error_{M-k}^\infty[f^{(k)}].
\]
Since by \eqref{eq: sinwell favard bound} we have \(\error_{M-k}^\infty[q]
\le \eta_{M-k}\norm{q'}_\infty\) with \(\eta_{M-k} = \frac{\pi(b-a)}{4(m-k)}\) for \(q\) Lipschitz
continuous, we get with \(\delta = \frac{b-a}{m-k}\) in Rem.~\ref{rem: modulus of continuity from lipschitz approximation}
the bound 
\(
    \error_{M-k}^\infty[f^{(k)}]
    \le (\tfrac12 + \tfrac{\pi}{4})\modCont(f^{(k)}; \tfrac{b-a}{m-k})
\), which combined with the previous inequality yields the upper bound.
For the lower bound on \(c_k\) observe that the inequality implies for Lipschitz
continuous \(f^{(k)}\) with \(\modCont(f^{(k)}; \delta) \le \delta \norm{f^{(k+1)}}_\infty\) that
\[
    \error_M^\infty[f]
    \le \tfrac{c_k}{2^k} \tfrac{(b-a)^{k+1}}{(m)_{k+1}} \norm{f^{(k+1)}}_\infty
\]
\citet[Theorem~1]{sinwelUniformApproximationDifferentiable1981} shows this necessitates \(c_k \ge \frac{\favConst_{k+1}}{2}\).

\subsection{Fourier-Chebyshev regression}
\label{sec: fourier-chebyshev regression}

Schematically we obtain the Lebesgue constants like this: First, we calculate the orthonormal
basis of the spaces for the various measures. With these we get
the reproducing kernels. The reproducing kernel representation
of the Lebesgue function (Theorem~\ref{thm: reproducing kernel representation})
then gives a representation of the Lebesgue function
\begin{equation}
    \label{eq: leb fun rep}    
    \lebfun{\infty}{\linProj_{\mu_n}}{x}
    = \frac1n\sum_{i=1}^n \omega_i\abs{\reprK^{\mu_n}_\modelSpace(x,x_i)}
\end{equation}
To bound this representation we replace the reproducing kernel by a rectified
function essentially of the form \(\min\set{m, \frac1{\metric_\torus(x,x_i)}}\), where
\(\metric_\torus(x,x_i)\) is the distance of points. The sum over the \(1/{\metric_\torus(x,x_i)}\) terms gives the
logarithmic term as the points are equispaced such that the distances grow linearly.

The reproducing kernels for Fourier and Chebyshev regression are given in terms
of the (modified) Dirichlet kernel.
\begin{definition}[Dirichlet kernel]
    \label{def: dirichlet kernel}
    The Dirichlet kernel is a stationary kernel defined as
    \[
        \dirichlet_M(\ang) \coloneq \sum_{k=-M}^M e^{ik\ang}
        = 1+ 2\sum_{k=1}^M \cos(k\ang)
        = \frac{\sin((M+\frac12)\ang)}{\sin(\frac12\ang)}.
    \]
    The modified Dirichlet kernel is defined as
    \[
        \tilde\dirichlet_M(\ang)
        \coloneq \dirichlet_{M-1}(\ang) + \cos(M\ang)
        = 1+ 2\sum_{k=1}^{M-1} \cos(k\ang) + \cos(M\ang)
        = \frac{\sin(M\ang)}{\tan(\frac12\ang)}.
    \]
\end{definition}
The cosine representation follows directly from \(e^{ik\ang} + e^{-ik\ang} = 2\cos(k\ang)\)
and the quotient identities may be obtained from the geometric sum formula applied to
\(\omega = e^{i\ang}\), see also \citep[I, p.~2]{zygmundTrigonometricSeries2003} and \citep[II, (5.2)]{zygmundTrigonometricSeries2003}.

Using this Dirichlet kernel we can now give the orthonormal bases and
reproducing kernels for the Fourier, trigonometric, cosine and Chebyshev
polynomial spaces.

\begin{theorem}[ONBs and reproducing kernels]
    \label{thm: ONBs and reproducing kernels}
    For the measures defined by
    \[
        \textstyle
        \mu_n^\alpha
        = \frac1n \sum_{k=0}^{n-1} \delta_{\alpha + \frac{2\pi k}{n}},
        \qquad
        \mu_n^{(1)} = \frac1{n-1} {\sum''}_{k=0}^{n-1} \delta_{\tfrac{\pi k}{n-1}},
        \qquad
        \mu_n^{(2)} = \frac1n \sum_{k=0}^{n-1} \delta_{\frac{\pi}{2n} + k\frac{\pi}{n}}
    \]
    with \(\mu_n^{(e)} = \mu_n^{(1)} \circ \cos^{-1}\) and \(\mu_n^{(z)} = \mu_n^{(2)} \circ \cos^{-1}\) the pushforwards of the DCT-1 and DCT-2 measures via \(\cos\),
    we have the following orthonormal bases and reproducing kernels.
\begin{center}
    \makebox[\textwidth][c]{
    \resizebox{1\linewidth}{!}{
    \begin{NiceTabular}{c c m{6cm} m{5cm}}
        \toprule
        \(\modelSpace\)
        & \(\mu\)
        & \begin{center}ONB w.r.t. \(\scp{\cdot}{\cdot}_\mu\)\end{center}
        & \begin{center}\(\reprK^\mu_\modelSpace(x,x')\)\end{center}
        \\
        \specialrule{\lightrulewidth}{\aboverulesep}{1pt}\midrule
        \Block{2-1}{\(\modelSpace_{M}^\F\)}
        & \(\unif{\torus}\)
        & \Block{2-1}{\(\set{e^{ikx}}_{\abs{k} \le M}\) (\(k\in \integer\))}
        & \Block{2-1}{\(\dirichlet_M(x-x')\)}
        \\ \cmidrule{2-2}
        & \(\mu_n^\alpha\), \(n \ge m\)
\\ \cmidrule{1-4}
        \Block{2-1}{\(\modelSpace_{M}^\Trig\)}
        & \(\unif{\torus}\)
        & \Block{2-1}{\(\set{1} \cup \bigcup_{k=1}^M\set{\sqrt{2}\cos(kx), \sqrt{2}\sin(kx)}\)}
        & \Block{2-1}{\(\dirichlet_M(x-x')\)}
        \\ \cmidrule{2-2}
        & \(\mu_n^\alpha\), \(n \ge m\)
\\ \cmidrule{1-4}
        \Block{4-1}{\(\modelSpace_{M}^{\cos}\)}
        & \(\unif{[0, \pi]}\)
        & \Block{3-1}{\(\set{1} \cup \bigcup_{k=1}^M\set{\sqrt{2}\cos(kx)}\)}
        & \Block{3-1}{\(\frac12(\dirichlet_M(x-x') + \dirichlet_M(x+x'))\)}
        \\ \cmidrule{2-2}
        & \(\mu_n^{(2)}\), \(n \ge m\)
        \\ \cmidrule{2-2}
        & \(\mu_n^{(1)}\), \(n > m\)
        \\ \cmidrule{2-4}
        & \(\mu_n^{(1)}\), \(n = m\)
        & \(\set{1} \cup \bigcup_{k=1}^{M-1}\set{\sqrt{2}\cos(kx)} \cup \set{\cos(Mx)}\)
        & \(\frac12(\tilde\dirichlet_M(x-x') + \tilde\dirichlet_M(x+x'))\)
        \\ \cmidrule{1-4}
        \Block{4-1}{\(\poly_M\)}
        & \(\frac{dx}{\pi\sqrt{1-x^2}}\)& \Block{3-1}{\(\set{\cheb_0, \hat\cheb_1, \dots, \hat\cheb_M}\)
        \par\footnotesize with \(\hat\cheb_k = \sqrt{2}\cheb_k\) for \(k\ge 1\)
        \par\footnotesize and \(\cheb_k(x) = \cos(k\cos^{-1}(x))\)
        }
        & \Block{3-1}{\(\frac12(\dirichlet_M(\ang-\ang') + \dirichlet_M(\ang+\ang'))\),
        \par\footnotesize with \(\ang = \cos^{-1}(x), \ang'=\cos^{-1}(x')\)
        }
        \\ \cmidrule{2-2}
        & \(\mu_n^{(z)}\), \(n \ge m\)
        \\ \cmidrule{2-2}
        & \(\mu_n^{(e)}\), \(n > m\)
        \\ \cmidrule{2-4}
        & \(\mu_n^{(e)}\), \(n = m\)
        & \begin{center}
            \(\set{\cheb_0, \hat\cheb_1, \dots, \hat\cheb_{M-1}, \cheb_M}\)
        \end{center}
        & \(\frac12(\tilde\dirichlet_M(\ang-\ang') + \tilde\dirichlet_M(\ang+\ang'))\),
        \par \footnotesize with \(\ang = \cos^{-1}(x), \ang'=\cos^{-1}(x')\)
        \\
        \bottomrule
    \end{NiceTabular}
    }}
\end{center}
\end{theorem}

In view of the representation of the Lebesgue function in \eqref{eq: leb fun rep} 
the following Lemma essentially provides the necessary bounds. We recall that
this Lemma works by replacing the Dirichlet kernel with a rectified function.

\begin{lemma}[Dirichlet kernel bound]
    \label{lem: dirichlet kernel bound}
    Let \(\ang_k = \alpha + k\frac{2\pi}{n}\) for \(k\in \set{0,\dots, n-1}\). Then
    \begin{align*}
        \frac1n \sum_{k=0}^{n-1}\abs{\dirichlet_M(\ang_k - \ang)}
        &\le \tfrac{2}{\pi}\ln(\mathfrak{m}) + \tfrac{\mathfrak{m}}n + \epsilon(\mathfrak{m})
        & \mathfrak{m} &\coloneq 2M+1
        \\
        \frac1n \sum_{k=0}^{n-1}\abs{\tilde\dirichlet_M(\ang_k - \ang)}
        &\le \tfrac2\pi\ln(\mathfrak{m}) + \tfrac{\mathfrak{m}}n + \tilde\epsilon(\mathfrak{m})
        & \mathfrak{m} &\coloneq 2M
    \end{align*}
    where \(\epsilon(\mathfrak{m})\) and \(\tilde\epsilon(\mathfrak{m})\) are given explicitly in
    \eqref{eq: (dirichlet kernel bound) explicit epsilon(m) definition}
    and 
    \eqref{eq: (dirichlet kernel bound) explicit tilde epsilon(m) definition},
    \(\epsilon(\mathfrak{m})\) is strictly monotonically increasing, \(\tilde \epsilon(\mathfrak{m})\) strictly monotonically decreasing in \(\mathfrak{m}\) with
    \begin{alignat*}{4}
        \epsilon(1) &=1  & \lim_{\mathfrak{m} \to \infty} \epsilon(\mathfrak{m}) &= \tfrac2\pi(1+\ln(2)) &&\approx 1.0779
        \\
        \tilde\epsilon(2) &\approx 0.6614 \qquad\qquad\qquad & \lim_{\mathfrak{m} \to \infty} \tilde\epsilon(\mathfrak{m}) &= \tfrac2\pi&&\approx 0.6366
    \end{alignat*}
\end{lemma}

\subsubsection{Proof of Thm.~\ref{thm: lebesgue constant bounds for fourier and chebyshev regression}: Lebesgue constant for Fourier-Chebyshev regression}

Lemma \ref{lem: dirichlet kernel bound} combined with \eqref{eq: leb fun rep}
and \(\lebconst{\infty}{\linProj_\mu} = \sup_{x\in \torus} \lebfun{\infty}{\linProj_\mu}{x}\)
immediately yields the Lebesgue constants for the Fourier polynomials \(\modelSpace_M^\F\) and the
trigonometric polynomials \(\modelSpace_M^\Trig\) as the dimension of these
spaces is \(2M+1 = \mathfrak{m}\) and \(\epsilon(\mathfrak{m}) \le 1.08\).

Clearly the polynomial case works like the DCT since both the reproducing kernel and the
evaluation points are just mapped through a cosine
function. To turn the DCT-1 \(\ang_k^{(1)} = \frac{\pi k}{n-1}\) nodes and the DCT-2 nodes \(\ang_k^{(2)} = \frac{\pi}{2n} + k\frac{\pi}{n}\) into equispaced points
as in Lemma~\ref{lem: dirichlet kernel bound} we reflect them around
zero (cf.~Figure~\ref{fig: DCT node reflection}) to get
\begin{align*}
    S_1 &\coloneq \set{\ang_0^{(1)}, \dots, \ang_{n-1}^{(1)}} \cup \set{-\ang_1^{(1)}, \dots, -\ang_{n-2}^{(1)}} & \#S_1 &= 2n-2,
    \\
    S_2 &\coloneq \set{\ang_0^{(2)}, \dots, \ang_{n-1}^{(2)}} \cup \set{-\ang_0^{(2)}, \dots, -\ang_{n-1}^{(2)}} & \#S_2 &= 2n.
\end{align*}
Using \eqref{eq: leb fun rep}, we obtain in both cases (except for DCT-1 with \(n=m\))
\[
    \lebfun{\infty}{\linProj_{\mu_n}}{x}
    = \frac1n\sum_{i=1}^{n-1} \omega_i\abs{\reprK^{\mu_n^{(1)}}_\modelSpace(x,\ang_i)}
    \le \frac12\Bigl(\frac1n\sum_{i=1}^{n-1} \omega_i\abs{\dirichlet_M(x-\ang_i)} + \frac1n\sum_{i=1}^{n-1} \omega_i\abs{\dirichlet_M(x+\ang_i)}\Bigr)
\]
In the case of DCT-1 the weights are given by \(\omega_i =\frac{n}{n-1}\) for \(i\notin \{0, n-1\}\) and half these
weights for \(i\in \set{0,n-1}\) resulting in
\[
    \frac1n\sum_{i=1}^{n-1} \omega_i\abs{\dirichlet_M(x-\ang_i)}
    = \frac1{n-1}\sideset{}{''}\sum_{i=1}^{n-1}\abs{\dirichlet_M(x-\ang_i)}
    = \frac1{\#S_1} \sum_{\ang \in S_1} \abs{\dirichlet_M(x-\ang)}
\]
As the term with \(x+\ang_i\) is just the same with \(-x\) and symmetry of \(\dirichlet_M\)
we may apply Lemma~\ref{lem: dirichlet kernel bound} to get
\[
    \lebfun{\infty}{\linProj_{\mu_n}}{x}
    \le \tfrac2\pi \ln(\mathfrak{m}) + \tfrac{\mathfrak{m}}{\#S_1} + 1.08
    = \tfrac2\pi \ln(2m-1) + \tfrac{2m-1}{2n-2} + 1.08,
\]
where we use \(\#S_1 = 2n-2\) and \(\mathfrak{m} = 2M+1 = 2m-1\). In the case
\(n=m\) we simply use the bound on the modified Dirichlet kernel \(\tilde\dirichlet_M\) instead 
with \(2M\le 2M+1\) and \(\tilde{\epsilon}(\mathfrak m) \le 0.66\le 1.08\). The DCT-2 case is
analogous with \(\#S_2 = 2n\) and the same \(\mathfrak{m} = 2M +1 = 2m-1\).

\subsubsection{Proof of Theorem~\ref{thm: ONBs and reproducing kernels}: Orthonormal bases and reproducing kernels}
    \begin{enumerate}
        \item[\(\modelSpace_M^\F\):]
        The orthonormality of \(e^{ik\cdot}\) with respect to \(\unif{\torus}\) well known and
        straightforward to verify by integration, see e.g.\
        \citep[p.~77]{steinFourierAnalysisIntroduction2003}.
        The discrete orthogonality is 
        also well-known and proven by applying the geometric sum formula to \(\omega = e^{i\frac{2\pi}{n}}\), see e.g.\
        \citet{strangDiscreteCosineTransform1999} for
        the case \(\alpha=0\) and Lemma~\ref{lem: fourier basis functions} \ref{it: (fourier) n>2M} for the general case.
        The reproducing kernel formula follows directly from its
        basis representation in Lemma~\ref{lem: reproducing kernel} and the
        definition of \(\dirichlet_M\).

        \item[\(\modelSpace_M^\Trig\):] The orthonormality follows from the previous case. Indeed, since
        the subspaces
        \[
            U_k = \Span_\complex\set{e^{ik\cdot}, e^{-ik\cdot}} = \Span_\complex\set{\cos(k\cdot), \sin(k\cdot)}
        \]
        are orthogonal we only need to check orthonormality within each \(U_k\).
        In the case \(k=0\) the subspace is one dimensional with basis function
        \(1\). For \(k>0\) we use
        \[
            \begin{pmatrix}
                e^{ik\ang} \\ e^{-ik\ang}
            \end{pmatrix}
            = \underbrace{\frac1{\sqrt{2}}
            \begin{pmatrix}
                1 & i \\ 1 & -i
            \end{pmatrix}
            }_{\eqcolon V}
            \begin{pmatrix}
                \sqrt{2}\cos(k\ang) \\ \sqrt{2}\sin(k\ang)
            \end{pmatrix}
        \]
        Since \(V\) is a unitary basis change matrix that preserves inner products,
        we conclude the claimed orthonormality with respect to \(\mu \in \set{\unif{\torus}, \mu_n^\alpha}\).
        This orthonormality of course remains when the space is restricted to real-valued functions.
        Similarly, since the unique reproducing kernel of \(\modelSpace_M^\F\)
        is real-valued by its cosine sum representation (Def.~\ref{def:
        dirichlet kernel}) it satisfies the defining properties of a reproducing kernel
        on the real-valued subspace \(\modelSpace_M^\Trig\) (Lemma~\ref{lem: reproducing kernel}).
        Uniqueness then implies the Dirichlet kernel is the reproducing kernel of \(\modelSpace_M^\Trig\).

        \item[\(\poly_M\):] Before we cover the case \(\modelSpace_M^{\cos}\) let us
        quickly explain why this also implies the claims for \(\poly_M\).
        Since \(\poly_M = \set{\model \circ \cos^{-1} : \model
        \in \modelSpace_M^{\cos}}\) and \(\frac{dx}{\pi\sqrt{1-x^2}} =
        \frac1\pi\abs{\arccos'(x)}dx\) is the pushforward of \(\unif{[0,\pi]}\) via
        \(\cos\) just like \(\mu_n^{(e)}\) and \(\mu_n^{(o)}\), the
        orthonormality of \(\hat T_k = \sqrt{2}\cos(k \cdot) \circ \cos^{-1}\)
        is a direct consequence of those for \(\modelSpace_M^{\cos}\).
        Similarly, we get the reproducing kernel by a change of variables.
        See also \citep[Sec.~4.6.1]{masonChebyshevPolynomials2002}

        \item[\(\modelSpace_M^{\cos}\):] Since for even functions \(f,g\colon [-\pi, \pi] \to \real\) we have
        \[
            \scp{f}{g}_{\unif{[0,\pi]}} = \frac1\pi \int_0^\pi f(x)g(x) dx = \frac1{2\pi} \int_{-\pi}^\pi f(x)g(x) dx = \scp{f}{g}_{\unif{\torus}},
        \]
        the orthonormality with respect to \(\unif{[0,\pi]}\) follows from the orthonormality of the trigonometric basis.
        Let \(\ang_k^{(1)} = \frac{\pi k}{n-1}\) and \(\ang_k^{(2)} = \frac{\pi}{2n} + k\frac{\pi}{n}\) be the DCT-1 and DCT-2 nodes, respectively contained in \([0,\pi]\).
        Reflecting these nodes around zero yields the sets of equidistant points in \(\torus\) (cf.~Figure~\ref{fig: DCT node reflection})
        \begin{align*}
            S_1 &\coloneq \set{\ang_0^{(1)}, \dots, \ang_{n-1}^{(1)}} \cup \set{-\ang_1^{(1)}, \dots, -\ang_{n-2}^{(1)}} & \#S_1 &= 2n-2,
            \\
            S_2 &\coloneq \set{\ang_0^{(2)}, \dots, \ang_{n-1}^{(2)}} \cup \set{-\ang_0^{(2)}, \dots, -\ang_{n-1}^{(2)}} & \#S_2 &= 2n.
        \end{align*}
        The reflected set of nodes in the DCT-1 case is smaller since the first and last nodes are reflected to themselves.
        We then get for functions \(f,g\colon [0,\pi] \to \real\) extended to even functions on \(\torus\) that
        \begin{align}
            \label{eq: discrete inner product to reflected, DCT-1}
            \scp{f}{g}_{\mu_n^{\mathrm{(1)}}}
            &= \frac1{n-1}\sideset{}{''}\sum_{k=0}^{n-1} f(\ang_k^{\mathrm{(1)}})g(\ang_k^{\mathrm{(1)}})
            = \frac1{\# S_1}\sum_{\ang \in S_1} f(\ang)g(\ang)
            \\
            \label{eq: discrete inner product to reflected, DCT-2}
            \scp{f}{g}_{\mu_n^{\mathrm{(2)}}}
            &= \frac1n \sum_{k=0}^{n-1} f(\ang_k^{\mathrm{(2)}})g(\ang_k^{\mathrm{(2)}})
            = \frac1{\# S_2}\sum_{\ang \in S_2} f(\ang)g(\ang),
        \end{align}
        since the modified sum \(\sum''\) halves the duplicate points in \(S_1\).
        To deduce the orthonormality of \(\set{1} \cup \bigcup_{k=1}^M\set{\sqrt{2}\cos(kx)}\) with respect to \(\mu_n^{(1)}\) and \(\mu_n^{(2)}\),
        we may therefore use their orthonormality with respect to \(\mu_n^\alpha\) in \(\modelSpace_M^\Trig\) assuming
        \(\#S_i \ge \dim \modelSpace_M^\Trig = 2M+1\). Observe that the assumption \(n \ge m=M+1\) implies \(\#S_2 =2n \ge 2(M+1) \ge 2M+1\) such that
        the DCT-2 nodes always satisfy this assumption. For the DCT-1 nodes,
        we however get \(\#S_1 = 2M\) in the case \(n=m\). Lemma~\ref{lem: fourier basis functions} \ref{it: (fourier) n=2M}
        combined with the arguments for the space \(\modelSpace_M^\Trig\) then implies
        that \(\set{1} \cup \bigcup_{k=1}^{M-1}\set{\sqrt{2}\cos(kx)}\) are orthonormal and
        orthogonal to the subspace \(\Span_\complex\set{e^{iMx}, e^{-iMx}}\) that contains \(\cos(Mx)\), which
        is of length one by Lemma~\ref{lem: fourier basis functions} with \(\alpha=\beta=0\).
        This proves the proposed basis is orthonormal.

        By Lemma~\ref{lem: reproducing kernel}, we have that in all cases
        except for the DCT-1 and \(n=m\) case, that the reproducing kernel is given by
        \[
            \reprK_{\modelSpace_M^{\cos}}(x, x') 
            \overset{\text{Lem.~\ref{lem: reproducing kernel}}}= 1+ 2\sum_{k=1}^M \underbrace{\cos(kx)\cos(kx')
            }_{
                \frac12(\cos(k(x-x')) \mathrlap{+ \cos(k(x+x')))}
            }
            = \tfrac12\bigl(\dirichlet_M(x-x') + \dirichlet_M(x+x')\bigr),
        \]
        where we used a product identity for cosines that may be easily be verified using the
        exponential representation of the cosine and the cosine sum representation of the Dirichlet kernel.
        The case of the DCT-1 nodes with \(n=m\) is similar
        \begin{align*}
            &\reprK_{\modelSpace_M^{\cos}}(x, x') 
            \\
            &= \Bigl(1+ 2\sum_{k=1}^{M-1}\cos(kx)\cos(kx')\Bigr)
            + \cos(Mx)\cos(Mx')
            \\
            &= \tfrac12\bigl(\dirichlet_{M-1}(x-x') + \dirichlet_{M-1}(x+x')\bigr)
            + \tfrac12(\cos(M(x-x')) + \cos(M(x+x')))
            \\
            &= \tfrac12\bigl(\tilde\dirichlet_M(x-x') + \tilde\dirichlet_M(x+x')\bigr).
        \end{align*}
    \end{enumerate}

\subsubsection{Some Fourier analysis lemmas and remarks}

In the derivation of the orthonormal basis and reproducing kernels we used the
following Lemma.

\begin{lemma}[Discrete Fourier basis functions]
    \label{lem: fourier basis functions}
    For the discrete uniform measure \(\mu_n^\alpha\) on the circle \(\torus\)
    as defined in Theorem~\ref{thm: ONBs and reproducing kernels} we have for
    \(l, l' \in \integer\)
    \begin{equation}
        \label{eq: aliasing}
        \scp{e^{il\cdot}}{e^{il'\cdot}}_{\mu_n^\alpha}
        = e^{i(l-l')\alpha} \ind{l - l'\in n\integer}.
    \end{equation}
    In particular,
    \begin{enumerate}[label=(\roman*),nosep]
        \item\label{it: (fourier) n>2M} if \(n \ge 2M+1\), then \(\set{e^{ik\cdot}}_{\abs{k} \le M}\) is an orthonormal basis  w.r.t. \(\mu_n^\alpha\),
        \item\label{it: (fourier) n=2M} if \(n=2M\), then \(\set{e^{ik\cdot}}_{\abs{k} \le M-1}\) are orthonormal and orthogonal to the subspace
        \(\Span_\complex\set{e^{iM\cdot}, e^{-iM\cdot}}\) w.r.t. \(\mu_n^\alpha\).
        Moreover, for any \(\alpha, \beta\in \real\) and function variable \(x\)
        \[
            \norm{\cos(M(x-\beta))}_{\mu_n^\alpha}
            = \abs{\cos(M(\alpha-\beta))}
            \overset{\text{in general}}\neq \tfrac1{\sqrt{2}}
            = \norm{\cos(M(x-\beta))}_{\unif{\torus}}.
        \]
    \end{enumerate}
\end{lemma}

While the result is standard, at least for the case \(\alpha=0\) \citep[e.g.][]{strangDiscreteCosineTransform1999},
we state its short proof here for the readers convenience.

\begin{proof}
    Define \(\omega \coloneq e^{i\frac{2\pi}n}\).
    Then \(\omega^{l-l'} = 1\) if and only if \(l-l' \in n\integer\). This implies
    \begin{align*}
        \scp{e^{il\ang}}{e^{il'\ang}}_{\mu_n^\alpha}
        = e^{i(l-l')\alpha} \frac1n \sum_{k=0}^{n-1}  \omega^{k(l-l')}
        &= \begin{cases}
            e^{i(l-l')\alpha} & l-l' \in n\integer
            \\
            \frac{1- \omega^{n(l-l')}}{n(1-\omega^{l-l'})} = 0 & l-l' \notin n\integer
        \end{cases}
        \\
        &= e^{i(l-l')\alpha} \ind{l-l' \in n\integer}.
    \end{align*}
    And for \(n=2M\) we have 
    \[
        \norm{\cos(M(\ang - \beta))}_{\mu_{n}^\alpha}^2
        = \frac1{n}\sum_{k=0}^{n-1} \cos^2\bigl(M(\alpha - \beta + \tfrac{2\pi k}{n})\bigr)
        = \cos^2(M(\alpha - \beta)).
        \qedhere
    \]
\end{proof}

\begin{remark}[Aliasing]
    \label{rem: aliasing}
    Equation \eqref{eq: aliasing} means that for \(\alpha=0\), the function
    \(e^{i(l+n)x}\) in \(x\) is indistinguishable from \(e^{ilx}\) with respect
    to \(\norm{\cdot}_{\mu_n^0}\).  That is the basis functions
    \(e^{i(l+n)x}\) `aliases' to the basis function \(e^{ilx}\).  For
    \(\alpha\neq 0\) similar aliasing occurs up to a complex phase factor
    \(e^{i(l-l')\alpha}\in \complex\) of unit length.
\end{remark}
\begin{remark}[Nyquist frequency]
    For the basis function \(e^{ikx}\), \(k\) is the frequency
    of cycles in \([-\pi, \pi) \simeq \torus\). The frequency \(M=\frac{n}2\)
    for \(n\) equispaced samples is the largest frequency, where the basis
    function is not fully aliased to a lower frequency.  This frequency \(M\) is
    called the `Nyquist frequency'. The `Nyquist sampling rate' \(n=2M\)
    is necessary to avoid (fully) aliasing the frequency \(M\).
\end{remark}

Recall that our problems with the DCT-1 nodes occurred essentially due to
the fact that the dimension of the Fourier and trigonometric spaces must
necessarily be odd \(2M+1\), while only an even number of \(2M\) samples was
available. This motivates the following concept of half-degree spaces.

\begin{remark}[Half-degree spaces]
    \label{rem: half-degree spaces}
    If we want to define a half-degree space \(\modelSpace_{M-\frac12}^\F\) of even dimension \(2M\)
    with 
    \(
        \modelSpace_{M-1}^\F \subseteq \modelSpace_{M-\frac12}^\F \subseteq \modelSpace_{M}^\F,
    \)
    and additionally require that \(\modelSpace_{M-\frac12}^\F\) is invariant under complex conjugation (due to Lemma~\ref{lem: projection to spaces closed under complex conjugation}),
    then we necessarily must choose \(\beta\in \torus\) such that
    \[
        \modelSpace_{M-\frac12}^{\F} \coloneq \modelSpace_{M-\frac12}^{\F, \beta}
        \coloneq \Span_\complex \Bigl(\modelSpace_{M-1}^\F \cup \set{\cos(M(x-\beta))}\Bigr).
    \]
    The half-degree Trigonometric polynomials are then the real-valued restriction
    \[
        \modelSpace_{M-\frac12}^{\Trig, \beta} 
        = \Span_\real\Bigl(\set{1} \cup \bigcup_{k=1}^{M-1}\set{\cos(kx), \sin(kx)} \cup \set{\cos(M(x-\beta))}\Bigr).
    \]
    Let \(\model, \model'\in \modelSpace_M^\F\) be equivalent, if \(\norm{\model - \model'}_{\mu_{2M}^\alpha} = 0\).
    Then the spaces \(\modelSpace_{M-\frac12}^{\F, \alpha}\) and \(\modelSpace_{M-\frac12}^{\Trig, \alpha}\)
    are the minimum \(\unif{\torus}\)-norm representatives of the equivalence classes.
\end{remark}

The following lemma explains the need for invariance under complex conjugation used
in the definition of half-degree spaces and also implicitly in the symmetric
definition of the Fourier spaces \(\modelSpace_M^\F\).

\begin{lemma}
    \label{lem: projection to spaces closed under complex conjugation}
    A complex-linear orthogonal projection \(\linProj\) of complex functions
    onto a model space \(\modelSpace\) projects real-valued functions to
    real-valued functions if and only if \(\modelSpace\) is closed under complex
    conjugation.
\end{lemma}
\begin{proof}
    Assume \(\modelSpace\) is closed under complex conjugation
    and let \(f= \linProj f + r\). Since \(r\perp \modelSpace\) we have \(\overline{r} \perp \modelSpace\)
    since for all \(\model \in \modelSpace\)
    \[
        \scp{\overline{r}}{\model} = \overline{\scp{r}{\overline{\model}}} = 0,
    \]
    where we use that \(\overline{\model} \in \modelSpace\) by the closure under complex conjugation.
    Since \(\linProj f\in \modelSpace\) we have \(\overline{\linProj f} \in \modelSpace\) and as an orthogonal
    projection \(\linProj(\modelSpace^\perp) = \set{0}\). Thus
    \[
        \linProj\overline{f}
        = \linProj (\underbrace{\overline{\linProj f}}_{\in \modelSpace} + \underbrace{\overline{r}}_{\in \modelSpace^\perp})
        = \overline{\linProj f}
    \]
    For real-valued functions with \(f= \overline{f}\) this implies \(\linProj f
    = \overline{\linProj f}\) and thus \(\linProj f\) is real-valued.
    Assume \(\linProj\) projects real-valued functions to real-valued functions. We will
    show for any \(\model \in \modelSpace\) that \(\overline{\model} \in \modelSpace\).
    Since \(\linProj\) is complex-linear, we have for \(\model = u + iv\) with
    real-valued \(u, v\) that \(\linProj u\) and \(\linProj v\) are real-valued and thus
    \[
        \linProj \overline{\model} = \linProj(u - iv) = \linProj u - i \linProj v = \overline{\linProj u + i \linProj v} = \overline{\linProj \model} = \overline{\model},
    \]
    where we used that \(\model\in \modelSpace\) and therefore \(\linProj \model = \model\) in the last equation. Thus \(\overline{\model} = \linProj \overline{\model} \in \modelSpace\).
\end{proof}

\subsubsection{Proof of Lemma~\ref{lem: dirichlet kernel bound}: Dirichlet kernel bound}
The proofs for the Dirichlet and modified Dirichlet kernel are fairly similar. We will start with the modified Dirichlet kernel
since the bound \(\tan(x) \ge x\) for \(x\ge 0\) allows an easy shortcut (Remark~\ref{rem: shortcut dirichlet kernel bound}) if
one were only interested in the asymptotic behavior of the bound and not the constant.
Since the Dirichlet kernel differs from the modified Dirichlet kernel by at most
one, this would also yield the asymptotic behavior for the Dirichlet kernel.

\noindent\textbf{Step 1: Bound using a rectified function.}
Using the cosine representation of the modified Dirichlet kernel (Definition~\ref{def: dirichlet kernel})
\[
    \textstyle
    \abs{\tilde\dirichlet_M(\ang)}
    = \abs[\big]{1 + 2\sum_{k=1}^{M-1} \cos(k\ang) + \cos(M\ang)}
    \le 2M = \mathfrak{m},
\]
since \(\abs{\cos(x)} \le 1\). Using the sine and tangent representation
we have by \(\abs{\sin(x)} \le 1\)
\[
    \abs{\tilde\dirichlet_M(\ang)}
    = \abs[\big]{\tfrac{\sin(M\ang)}{\tan(\frac12\ang)}}
    \le \tfrac1{\abs{\tan(\frac12\ang)}}
    = \tfrac1{\tan(\frac12\abs{\ang})}
    \qquad \ang \in (-\pi, \pi).
\]
Put together we get
\[
    \abs{\tilde\dirichlet_M(\ang)}
    \le \min\set[\big]{\mathfrak{m}, \tfrac1{\tan(\frac12\abs{\ang})}}
    \eqcolon g(\ang)
    \qquad \ang \in (-\pi, \pi).
\]
Observe that \(\tilde{\dirichlet}_M\) and \(g\) are even functions defined on \([-\pi, \pi]\) with the continuous extension
of \(g\) to zero on the boundaries. As the Dirichlet kernel is \(2\pi\)-periodic, we consequently have
for the arc-distance \(\metric_\torus(\ang, \ang') \coloneq \min_{k\in \integer} \abs{\ang - \ang' + 2\pi k} \in [0, \pi]\)
\begin{equation}
    \label{eq: modified dirichlet kernel bound}
    \abs{\tilde\dirichlet_M(\ang - \ang')}
    \le \min\set[\big]{\mathfrak{m}, \tfrac1{\tan(\frac12\metric_\torus(\ang, \ang'))}}
    = g(\metric_\torus(\ang, \ang'))
    \qquad \ang, \ang' \in \real \text{ or } \torus.
\end{equation}
Since the points \(\ang_k\) are equally spaced at distance \(\frac{2\pi}{n}\), the number of points in a
radius \(r\) Ball in arc-distance around \(\ang\), or equivalently the number of points in a
circular arc of length \(2r\) is bounded by \(1+\frac{n}{\pi}r\), that is
\begin{equation}
    \label{eq: maximum equidistant point in ball}    
    \#\ball{\ang}{r} \coloneq \#\set{k: \metric_\torus(\ang, \ang_k) \le r} \le 1 + \tfrac{n}{\pi}r.
\end{equation}
Let \(r_k \coloneq \metric_\torus(\ang, \ang_k)\) and order these distances
\(
    r_{(0)} \le r_{(1)} \le \dots \le r_{(n-1)}.
\)
Since the rank \(k\) implies there are at least \(k\) additional points in the ball
of radius \(r_{(k)}\) around \(\ang\) we thus have
\(k+1 \le \#\ball{\ang}{r_{(k)}}\) and combined with \eqref{eq: maximum equidistant point in ball} we obtain
\begin{equation}
    \label{eq: arcdistance lower bound}        
    r_{(k)} \ge \tfrac{\pi}{n}k.
\end{equation}
Using the bound \eqref{eq: modified dirichlet kernel bound} we consequently have
by \(r_k = \metric_\torus(\ang, \ang_k)\) and the reordering \(r_{(k)}\)
\begin{equation}
    \label{eq: modified dirichlet kernel sum bound}    
    \frac1n\sum_{k=0}^{n-1}\abs{\tilde\dirichlet_M(\ang_k - \ang)}
    \overset{\eqref{eq: modified dirichlet kernel bound}}\le
    \frac1n\sum_{k=0}^{n-1} g(r_{(k)})
    \overset{\eqref{eq: arcdistance lower bound}}\le \frac1n\sum_{k=0}^{n-1} g(\tfrac{\pi}{n}k)
    \le \frac1n \int_{-1}^{n-1} g(\tfrac{\pi}{n}x) dx,
\end{equation}
where we used that \(g\) is non-increasing on \([0, \pi]\) in the last two inequalities.

\begin{remark}[Shortcut]
    \label{rem: shortcut dirichlet kernel bound}
    At the cost of replacing \(\epsilon(\mathfrak{m})\) by \(\tfrac2\pi[1+\ln(\tfrac\pi2)]\le 1\), the remaining proof may be simplified using
    the bound \(\tan(x) \ge x\) for \(x\ge 0\) to get \(g(\tfrac{\pi}n x) \le \min\set{\mathfrak{m}, \frac{2n}{\pi}\frac1x} \).
    The point is then that the integral of \(\frac1x\) is 
    the logarithm and thus yields the desired bound.
\end{remark}

\smallskip\noindent\textbf{Step 2: Integral bound.}
To get the sharper bounds, we observe that \(g\) is equal to the constant \(\mathfrak{m}\) if
\[
    \mathfrak{m} \le \tfrac1{\tan(\frac\pi{2n}x)}
    \iff x \le \tfrac{2n}{\pi}\arctan(\tfrac1{\mathfrak{m}}).
\]
Thus
\begin{align*}
    \frac1n \int_{-1}^{n-1} g(\tfrac{\pi}{n}x) dx
    &= \frac1n \int_{-1}^{\frac{2n}{\pi}\arctan\frac1{\mathfrak{m}}} \mathfrak{m} dx
    + \frac1n \int_{\frac{2n}{\pi}\arctan\frac1{\mathfrak{m}}}^{n-1} \frac1{\tan(\frac\pi{2n}x)} dx
    \\
    \overset{y=\frac\pi{2n}x}&= \mathfrak{m}[\tfrac{2}{\pi}\arctan(\tfrac1{\mathfrak{m}}) + \tfrac1n]
    + \frac{2}{\pi}\int_{\arctan\frac1{\mathfrak{m}}}^{\frac{\pi}{2n}(n-1)} \frac1{\tan y} dy
\end{align*}
Using \(\int \frac1{\tan y} dy = \ln(\sin y)\) we obtain
\begin{align*}
    \int_{\arctan\frac1{\mathfrak{m}}}^{\frac{\pi}{2}-\frac{\pi}{2n}} \frac1{\tan y} dy
    &= \ln(\cos(\tfrac{\pi}{2n})) - \ln(\sin(\arctan\tfrac1{\mathfrak{m}}))
    \\
    &\le - \ln\bigl(\sin(\arctan\tfrac1{\mathfrak{m}})\bigr) \qquad \bigl(\cos(\tfrac{\pi}2n) \in (0,1]\bigr)
\end{align*}
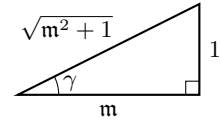
\begin{wrapfigure}{r}{0.3\textwidth}
    \centering
    \vspace{-0.3cm}
    \begin{tikzpicture}[scale=0.6, every node/.style={font=\small}]
\coordinate (A) at (0,0);
    \coordinate (B) at (4,0);
    \coordinate (C) at (4,2);

\draw[thick] (A) -- (B) -- (C) -- cycle;

\draw (3.7,0) -- (3.7,0.3) -- (4,0.3);

\draw (0.9,0) arc[start angle=0,end angle=26.565,radius=0.9];
    \node at (1.15,0.25) {$\gamma$};

\node[below] at (2,0) {$\mathfrak{m}$};
    \node[right] at (4,1) {$1$};
    \node[above left] at (2.3,1) {$\sqrt{\mathfrak{m}^2+1}$};
\end{tikzpicture}     \vspace{-0.2cm}
    \caption{\label{fig: sin arctan geometry} \(\gamma =\arctan\tfrac1{\mathfrak{m}}\)}
\end{wrapfigure}
For the angle \(\gamma \coloneq \arctan\tfrac1{\mathfrak{m}}\) we have by definition
\[
    \tan \gamma = \tfrac1{\mathfrak{m}} = \tfrac{\text{opposite}}{\text{adjacent}}
\]
and therefore by the Pythagorean theorem (Figure~\ref{fig: sin arctan geometry})
\[
    \sin(\arctan(\tfrac1{\mathfrak{m}})) = \tfrac{\text{opposite}}{\text{hypotenuse}} =  \tfrac1{\sqrt{1+\mathfrak{m}^2}}.
\]
This implies 
\[
    - \ln(\sin(\arctan(\tfrac1{\mathfrak{m}})))
    = \tfrac12\ln(1+\mathfrak{m}^2)
    = \tfrac12\ln(\mathfrak{m}^2(1+\tfrac1{\mathfrak{m}^2}))
    = \ln(\mathfrak{m}) + \tfrac12\ln(1+\tfrac1{\mathfrak{m}^2}).
\]
Reinserting our bounds on the integral into \eqref{eq: modified dirichlet kernel sum bound} we get the claim
\begin{align}
    \nonumber
    \frac1n \sum_{k=0}^{n-1}\abs{\tilde\dirichlet_M(\ang_k - \ang)}
    &\le \mathfrak{m}\bigl[\tfrac{2}{\pi}\arctan(\tfrac1{\mathfrak{m}}) + \tfrac1n\bigr] + \tfrac2\pi\bigl[ \ln(\mathfrak{m}) + \tfrac12\ln(1+\tfrac1{\mathfrak{m}^2})\bigr]
    \\
    \label{eq: (dirichlet kernel bound) explicit tilde epsilon(m) definition}
    &= \tfrac2\pi\ln(\mathfrak{m}) + \tfrac{\mathfrak{m}}n + \underbrace{\tfrac2\pi\bigl[\mathfrak{m}\arctan(\tfrac1{\mathfrak{m}}) + \tfrac12\ln(1+\tfrac1{\mathfrak{m}^2})\bigr]}_{\eqcolon \tilde\epsilon(\mathfrak{m})}
\end{align}
\textbf{Step 3: Properties of \(\tilde\epsilon(\mathfrak{m})\).}
We have
\begin{align*}
    \tfrac{d}{d\mathfrak{m}} \arctan(\tfrac1{\mathfrak{m}})
    &= \tfrac1{1+(\frac1{\mathfrak{m}})^2}(-\tfrac1{\mathfrak{m}^2})
    = -\tfrac1{\mathfrak{m}^2 + 1}
    \\
    \tfrac{d}{d\mathfrak{m}} \ln(1+\tfrac1{\mathfrak{m}^2})
    &= \tfrac1{1+\frac1{\mathfrak{m}^2}}(-\tfrac2{\mathfrak{m}^3})
    = - \tfrac2{\mathfrak{m}(\mathfrak{m}^2+1)}
\end{align*}
such that
\begin{align*}
    \tfrac{d}{d\mathfrak{m}} \tilde\epsilon(\mathfrak{m})
    &= \tfrac2\pi\bigl[\arctan(\tfrac1{\mathfrak{m}})-\tfrac{\mathfrak{m}}{\mathfrak{m}^2+1} - \tfrac1{\mathfrak{m}(\mathfrak{m}^2+1)}\bigr]
    \\
    &= \tfrac2\pi\bigl[\arctan(\tfrac1{\mathfrak{m}}) - \tfrac{\mathfrak{m}^2+1}{\mathfrak{m}(\mathfrak{m}^2+1)}\bigr]
    = \tfrac2\pi\bigl[\arctan(\tfrac1{\mathfrak{m}}) - \tfrac1{\mathfrak{m}}\bigr] < 0,
\end{align*}
where the inequality follows from \(\arctan(x) < x\) for \(x>0\) (\(\tan\) is increasing and \citep[4.3.80]{abramowitzHandbookMathematicalFunctions1964}).
Consequently, \(\tilde\epsilon(\mathfrak{m})\) is strictly decreasing with \(\tilde\epsilon(2) \approx 0.6614\).
The limit of \(\ln(1+\frac1{\mathfrak{m}^2})\) is zero and the limit of \(\mathfrak{m}\arctan\frac1{\mathfrak{m}}\) is
determined by the Taylor expansion of \(\arctan\) around zero
\[
    \mathfrak{m}\arctan \tfrac1{\mathfrak{m}} = \mathfrak{m} \underbrace{\arctan(0)}_{=0} + \mathfrak{m} \underbrace{\arctan'(0)}_{=1} \tfrac1{\mathfrak{m}} + \bigO(\tfrac1{\mathfrak{m}})
    \to 1 \qquad \text{as } \mathfrak{m}\to \infty.
\]
Consequently, \(\tilde\epsilon(\mathfrak{m}) \to \frac2\pi\) as \(\mathfrak{m}\to \infty\).

\smallskip\noindent\textbf{Step 4: Dirichlet kernel bound.}
The bound for the Dirichlet kernel is very similar. We start with
\[
    \dirichlet_M(\ang) \le \min\set[\big]{1+2M, \tfrac1{\abs{\sin(\frac12\ang)}}} \eqcolon f(\ang)
\]
stemming from the two representations of the Dirichlet kernel. Using that \(\sin\)
is a periodic even function that is positive on \([0, \frac{\pi}{2}]\) we obtain
\[
    \dirichlet_M(\ang - \ang') \le \min\set[\big]{1+2M, \tfrac1{\sin(\frac12\metric_\torus(\ang, \ang'))}}
    \eqcolon f(\metric_\torus(\ang, \ang')).
\]
And with the same arguments as for the modified Dirichlet kernel we can treat the
equidistant points \(\ang_k\) to obtain
\[
    \frac1n \sum_{k=0}^{n-1}\abs{\dirichlet_M(\ang - \ang_k)}
    \le \frac1n \sum_{k=0}^{n-1} f(r_{(k)})
    \overset{\eqref{eq: arcdistance lower bound}}\le \frac1n \sum_{k=0}^{n-1} f(\tfrac{\pi}{n}k)
    \le \frac1n \int_{-1}^{n-1} f(\tfrac{\pi}{n}x) dx.
\]
While the shortcut \(\tan(x) \ge x\) would yield the same asymptotic behavior
for the modified Dirichlet kernel, the Jordan inequality \(\sin(x) \ge
\frac{2}{\pi}x\) would cause the constant in front of the logarithm to be
one instead of \(\frac2\pi\) which forces us to do the sharper analysis of
the integral for the correct asymptotic behavior.  The analysis is very
similar to the modified Dirichlet kernel.
The function takes the constant value \(1+2M = \mathfrak{m}\) if
\[
    \mathfrak{m} \le \tfrac1{\sin(\frac\pi{2n}x)} \iff x \le \tfrac{2n}{\pi}\arcsin(\tfrac1{\mathfrak{m}}).
\]
This leads to
\[
    \frac1n \int_{-1}^{n-1} f(\tfrac{\pi}{n}x) dx
    = \mathfrak{m}\bigl[\tfrac{2}{\pi}\arcsin(\tfrac1{\mathfrak{m}}) + \tfrac1n\bigr] + \frac{2}{\pi}\int_{\arcsin\frac1{\mathfrak{m}}}^{\frac{\pi}{2n}(n-1)} \frac1{\sin y} dy,
\]
where we now use \(\int \frac1{\sin y} dy = \ln(\tan\frac{y}2)\), which follows from
\[
    \frac{d}{dy} \ln(\tan\tfrac{y}2)
    = \frac1{\tan\frac{y}2}\cdot \frac1{\cos^2(\frac{y}2)}\frac12
    = \frac1{2\sin(\frac{y}2)\cos(\frac{y}2)}
    = \frac1{\sin y}
\]
to obtain
\[
    \int_{\arcsin\frac1{\mathfrak{m}}}^{\frac{\pi}{2}\frac{n-1}{n}} \frac1{\sin y} dy
    = \ln\bigl(\underbrace{\tan(\tfrac{\pi}{4}\tfrac{n-1}{n})}_{\le \tan\tfrac{\pi}4 = 1}\bigr) - \ln\bigl(\tan(\tfrac12\arcsin\tfrac1{\mathfrak{m}})\bigr)
    \le - \ln\bigl(\tan(\tfrac12\arcsin\tfrac1{\mathfrak{m}})\bigr).
\]
Then with \(\gamma \coloneq \arcsin \frac1{\mathfrak{m}}\) we have by definition \(\sin \gamma = \frac1{\mathfrak{m}}\)
and \(\cos \gamma = \sqrt{1- \frac1{\mathfrak{m}^2}}\) such that with the
half-angle formula for \(\tan\) \citep[Eq.~4.3.22]{abramowitzHandbookMathematicalFunctions1964}
\[
    \tan \frac\gamma2
    = \frac{\sin \gamma}{1+\cos \gamma}
    = \frac{1}{\mathfrak{m}\Bigl(1 + \sqrt{1 - \frac1{\mathfrak{m}^2}}\Bigr)},
\]
which implies
\[
    - \ln\bigl(\tan(\tfrac12\arcsin\tfrac1{\mathfrak{m}})\bigr)
    = \ln(\mathfrak{m}) + \ln\Bigl(1 + \sqrt{1 - \tfrac1{\mathfrak{m}^2}}\Bigr).
\]
Put together we consequently have
\begin{align}
    \nonumber
    \frac1n \sum_{k=0}^{n-1}\abs{\dirichlet_M(\ang - \ang_k)}
    &\le \mathfrak{m}\bigl[\tfrac{2}{\pi}\arcsin(\tfrac1{\mathfrak{m}}) + \tfrac1n\bigr]
    + \tfrac{2}{\pi}\bigl[\ln(\mathfrak{m}) + \ln(1 + \sqrt{1 - \tfrac1{\mathfrak{m}^2}})\bigr]
    \\
    \label{eq: (dirichlet kernel bound) explicit epsilon(m) definition}
    &= \tfrac{2}{\pi}\ln(\mathfrak{m}) + \tfrac{\mathfrak{m}}n + \underbrace{\tfrac{2}{\pi}\bigl[\mathfrak{m}\arcsin(\tfrac1{\mathfrak{m}}) + \ln(1 + \sqrt{1 - \tfrac1{\mathfrak{m}^2}})\bigr]}_{\eqcolon \epsilon(\mathfrak{m})}
\end{align}
\noindent\textbf{Step 5: Properties of \(\epsilon(\mathfrak{m})\).}
We have
\begin{alignat*}{3}
    \tfrac{d}{d\mathfrak{m}} \arcsin(\tfrac1{\mathfrak{m}})
    &= \tfrac1{\sqrt{1-(\frac1{\mathfrak{m}})^2}}(-\tfrac1{\mathfrak{m}^2})
    = -\tfrac1{\mathfrak{m}\sqrt{\mathfrak{m}^2-1}}
    \\
    \tfrac{d}{d\mathfrak{m}} \ln(1 + \sqrt{1 - \tfrac1{\mathfrak{m}^2}})
    &= \tfrac1{1 + \sqrt{1 - \frac1{\mathfrak{m}^2}}}\cdot\cancel{\tfrac12}(1 - \tfrac1{\mathfrak{m}^2})^{-\frac12}\cdot\tfrac{\cancel{2}}{\mathfrak{m}^3}
    = \tfrac1{\mathfrak{m}(\mathfrak{m} + \sqrt{\mathfrak{m}^2 -1})\sqrt{\mathfrak{m}^2 - 1}}
\end{alignat*}
Since \((\mathfrak{m}-\sqrt{\mathfrak{m}^2-1})(\mathfrak{m}+\sqrt{\mathfrak{m}^2-1}) = 1\) we have \(\frac1{\mathfrak{m} + \sqrt{\mathfrak{m}^2 -1}} = \mathfrak{m} - \sqrt{\mathfrak{m}^2-1}\) and thus
\[
    \tfrac{d}{d\mathfrak{m}} \ln(1 + \sqrt{1 - \tfrac1{\mathfrak{m}^2}})
    = \tfrac{\mathfrak{m} - \sqrt{\mathfrak{m}^2-1}}{\mathfrak{m}\sqrt{\mathfrak{m}^2 - 1}} = \tfrac{1}{\sqrt{\mathfrak{m}^2-1}} - \tfrac1{\mathfrak{m}}.
\]
This implies
\[
    \tfrac{d}{d\mathfrak{m}} \epsilon(\mathfrak{m})
    = \tfrac{2}{\pi}\bigl[\arcsin(\tfrac1{\mathfrak{m}}) + \cancel{\mathfrak{m}(-\tfrac1{\mathfrak{m}\sqrt{\mathfrak{m}^2-1}})} + \cancel{\tfrac{1}{\sqrt{\mathfrak{m}^2-1}}} - \tfrac1{\mathfrak{m}}\bigr]
    = \tfrac{2}{\pi}\bigl[\arcsin(\tfrac1{\mathfrak{m}}) - \tfrac1{\mathfrak{m}}\bigr] \ge 0,
\]
where the last inequality holds for all \(\mathfrak{m} \ge 1\) and is a strict inequality
for \(\mathfrak{m} >1\) as \(\arcsin(x) > x\) for \(x \in (0,1)\) (\(\sin\) is increasing on \([0, \frac{\pi}2]\) \citep[4.3.79]{abramowitzHandbookMathematicalFunctions1964} and \citep[4.3.80]{abramowitzHandbookMathematicalFunctions1964}).
Consequently, \(\epsilon(\mathfrak{m})\) is strictly increasing in \(\mathfrak{m}\). \(\epsilon(1)\) may be calculated
by inserting \(\mathfrak{m}=1\) into \eqref{eq: (dirichlet kernel bound) explicit epsilon(m) definition}
and the limit follows from a Taylor expansion of \(\arcsin\) around zero.
\qed

 	\section{Auxiliary material}

\subsection{Notation for various semi-norms}
\label{sec: notation for norms}

In general we use the notation \(\abs{\cdot}\) for semi-norms in finite
dimension like \(\field^\dims\), \(\norm{\cdot}\) for semi-norms on function
spaces and \(\opnorm{\cdot}\) for operator norms.

\begin{itemize}
    \item The \emph{\(p\)-norm} on \(\field^\dims\) is defined as
    \[
        \abs{x}_p = \Biggl(\sum_{i=1}^\dims \abs{x_i}^p\Biggr)^{1/p}.
    \]
    \item The \emph{sup-norm} and \emph{weighted sup-norm} on function spaces are defined as
    \[
        \norm{f}_\infty = \sup_{x\in \domain} \abs{f(x)}
        \qquad\text{and}\qquad
        \norm{f}_w = \sup_{x\in \domain} \abs{f(x)} w(x),
    \]
    for a weight function \(w\colon \domain \to (0,\infty)\).
    \item The \emph{\(L^p(\mu)\)-norms} of some measure \(\mu\) on \(\domain\) are given by
    \[
        \norm{f}_{L^p(\mu)} \coloneq \begin{cases}
            \displaystyle
            \left(\int_\domain |f(x)|^p \mathrm{d}\mu(x)\right)^{1/p} & p < \infty \\
            \displaystyle \esssup_{x\in\domain; \mu} |f(x)| & p = \infty
         \end{cases},
    \]
    where \(\norm{\cdot}_\mu \coloneq \norm{\cdot}_{L^2(\mu)}\) is induced by the inner product \(\scp{f}{g}_{\mu} = \int f\overline{g}\, d\mu\).

    \item We define the \emph{operator norm} of a linear operator \(T\) as
    \[
        \opnorm{T}_*^\dagger
        \coloneq \opnorm{T}_{\funSpace, *}^\dagger
        \coloneq \sup_{\substack{f\in \funSpace\\ \norm{f}_* \le 1}} \norm{Tf}_\dagger
        = \sup_{\substack{f\in \funSpace\\ \norm{f}_* = 1}} \norm{Tf}_\dagger
        = \sup_{\substack{f \in \funSpace\\ \norm{f}_* \in (0, \infty)}}\frac{\norm{Tf}_\dagger}{\norm{f}_*}.
    \]
    In the typical case \(\dagger = *\)
    we omit \(\dagger\). If the operator \(T\) maps to \(\field\) we
    implicitly assume \(\norm{\cdot}_\dagger = \abs{\cdot}\) if
    we omit \(\dagger\).

    \item A positive semi-definite matrix \(\Omega_n \in \field^{n\times n}\) defines the semi-inner
    product and semi-norm
    \[
        \scp{a}{b}_{\Omega_n} = \frac1n b^\conT \Omega_n a \qquad \text{and} \qquad \abs{a}_{\Omega_n} = \sqrt{\scp{a}{a}_{\Omega_n}}.
    \]
    on the space \(\field^n\) and the semi-inner product and semi-norm
    \[
        \scp{f}{g}_{\Omega_n} = \scp{\eval_X[f]}{\eval_X[g]}_{\Omega_n}
        \qquad \text{and} \qquad
        \norm{f}_{\Omega_n} = \sqrt{\scp{f}{f}_{\Omega_n}}
    \]
    on the function space \(\funSpace\) using the evaluation functional
    \(\eval_X\) defined in Definition~\ref{def: regression setting}.
\end{itemize}

\begin{lemma}[Inclusion of \(L^p\) spaces]
    \label{lem: Lp inclusion}
    For \(1\le p \le q \le \infty\) we have
    \[
        \norm{f}_{L^p(\mu)} \le \mu(\domain)^{\frac1p - \frac1q}\norm{f}_{L^q(\mu)}
    \quad\text{and}\quad
        \norm{f}_{L^\infty(\mu)}
        \le \norm{f}_\infty,
    \]
    If \(\mu(\domain)\) is finite we thus have
    \(L^\infty(\mu) \subseteq L^q(\mu) \subseteq L^p(\mu)\) and the
    bounded (measurable) functions are contained in all \(L^p(\mu)\) spaces.
    Moreover for a weight function \(\weight \colon \domain \to (0,\infty)\) we have
    \[
        \norm{f}_{L^p(\mu)} \le \norm[\big]{\tfrac1\weight}_{L^p(\mu)} \norm[\big]{f}_\weight.
    \]
    Sequence spaces have the opposite inclusion, \(\ell^p \subseteq \ell^q\) for \(p\le q\), that is for \(x \in \ell^p\) and \(p\le q\) we have \(x\in \ell^q\) with
    \[
        \abs{x}_q \le \abs{x}_p.
    \]
\end{lemma}

These inclusions are well-known but we provide a proof for completeness.

\begin{proof}
    If \(\mu(\domain) = \infty\) the first inequality is trivial. For \(\mu(\domain) <\infty\)
    and \(p\le q < \infty\) it follows from Jensen's inequality as \(x^{q/p}\)
    is convex such that
    \[
        \Bigl(\frac1{\mu(\domain)}\int \abs{f(x)}^p d\mu\Bigr)^{q/p}
        \le \frac1{\mu(\domain)}\int \abs{f(x)}^q d\mu.
    \]
    Raising both sides to the power of \(1/q\) and multiplying by
    \(\mu(\domain)^{1/p}\) yields the first inequality. 
    The case \(q=\infty\) and the second inequality then follow from 
    \[
        f \le \esssup_{x\in \domain; \mu} f(x)
        = \norm{f}_{L^\infty(\mu)}
        \le \sup_{x\in \domain} f(x)
        = \norm{f}_\infty,
    \]
    where the first inequality only holds \(\mu\)-almost everywhere, which is sufficient
    to apply monotonicity of the integral.  The weighted inequality follows from
    \[
        \norm{f}_{L^p(\mu)} = \Bigl(\int (\tfrac{1}{\weight})^p \abs{f\weight}^p \, d\mu\Bigr)^{\frac1p}
        \le \norm[\big]{\tfrac1\weight}_{L^p(\mu)} \norm[\big]{f\weight}_\infty
        = \norm[\big]{\tfrac1\weight}_{L^p(\mu)} \norm[\big]{f}_\weight.
    \]
    By the scaling property of \(p\)-norms we assume without loss of generality
    \(\abs{x}_p = 1\). Then the \(x_i\) are bounded by \(1\) and thus
    \(\abs{x_i}^q \le \abs{x_i}^p\) for all \(i\). Summing over \(i\) yields
    \(\abs{x}_q^q \le \abs{x}_p^p = 1\) and thus \(\abs{x}_q \le 1 =
    \abs{x}_p\).
\end{proof}

\subsection{Lagrange function}
\label{sec: lagrange function}

In general the Lagrange function has the following nice properties.

\begin{lemma}[Properties of the general Lagrange function]
    \label{lem: lagrange function properties}
    The Lagrange function \(\ell\) of a regression has the following properties:
    \begin{enumerate}[label=(\roman*),noitemsep]
        \item\label{it: lagrange basis} The Lagrange function components \(\ell_i\) are in the model space, \(\ell_i \in \modelSpace\) for all \(i=1,\dots, n\).
        \item\label{it: lagrange continuity}
        If \(\modelSpace \subseteq C(\domain; \field)\) then the Lagrange function
        \(\ell(x)\) is continuous in \(x\).

        \item\label{it: lagrange projection}
        Assume \(\Image(\regOp) = \Image(\linProj)\). Then
        \(\ell(X)^\conT = \eval_X \circ \regOp\) is a linear projection in \(\field^n\).
        A sufficient condition is that \(\eval_X[\funSpace] = \field^n\).
    \end{enumerate}
\end{lemma}
\begin{proof}
    \ref{it: lagrange basis} We have \(\ell_i(\cdot) = \ell(\cdot)^\conT e_i = \regOp(e_i) \in \modelSpace\),
    which also directly implies \ref{it: lagrange continuity}.

    \noindent\ref{it: lagrange projection}
    Since \(\regOp(y) \in \Image(\regOp) = \Image(\linProj)\) and \(\linProj\) acts as the identity
    on its image, we have \(\linProj \circ \regOp = \regOp\). Thus
    \[
        (\eval_X \circ \regOp)^2
        = \eval_X \circ (\regOp \circ \eval_X) \circ \regOp
        = \eval_X \circ \linProj \circ \regOp
        = \eval_X \circ \regOp.
    \]
    For the sufficient condition observe that \(\Image(\regOp) \supseteq \Image(\linProj)\)
    by definition of \(\linProj = \regOp \circ \eval_X\). It is thus sufficient to show for any
    \(\model \in \Image(\regOp)\) that there exists \(f \in \funSpace\) such that \(\regOp(\eval_X[f]) = \model\).
    But this is true since there exists \(y\in \field^n\) such that \(\regOp(y) = \model\),
    and since \(\eval_X[\funSpace] = \field^n\) there exists \(f\in \funSpace\) such that \(\eval_X[f] = y\).
\end{proof}

Our motivation for the definition of the `Lagrange function' are the `Lagrange polynomials' 
from Approximation theory that coincide with the Lagrange function in the case
of polynomial interpolation as we can see in the following example.

\begin{example}[Lagrange polynomial]
    \label{ex: lagrange polynomial}
    If \(\domain = \real\) and \(\modelSpace=\poly_{n-1}\), where \(\poly_{n-1}\) is the space of polynomials of degree at most \(n-1\),
    and \(\linProj p = p\) for all \(p\in \poly_{n-1}\), then the components \(\ell_j\) of the Lagrange function are
    `Lagrange polynomials' \(\ell_j\in \poly_{n-1}\) that satisfy:
    \begin{enumerate}[label=(\roman*)]
        \item\label{it: langrange poly, dirac delta} \(\ell_j(x_i) = \delta_{ij}\) for all \(i,j =1\dots, n\).
        \item\label{it: lagrange poly, formula} For all \(x\) and all \(j=1,\dots, n\) we have
        \[
            \ell_j(x) = \prod_{i\neq j} \frac{x-x_i}{x_j-x_i}
            = \frac{\omega(x)}{(x-x_j)\omega'(x_j)}
        \]
        where \(\omega(x) = \prod_{i=1}^n (x-x_i)\) is the node polynomial and therefore
        \(\omega'(x_j) = \prod_{i\neq j} (x_j-x_i)\).
    \end{enumerate}
\end{example}
\begin{proof}
    Clearly the explicit formula in
    \ref{it: lagrange poly, formula} yields a polynomial \(p_j\) of degree
    \(n-1\) that has the property \ref{it: langrange poly, dirac delta}.
    What is left to show is that these polynomials \(p_j\) are indeed the components \(\ell_j\)
    of the Lagrange function. Observe that by \ref{it: langrange poly, dirac delta} we have
    \(\eval_X[p_j] = e_j\). Thus
    \[
        \ell_j(x)
        = \ell^\conT(x) \eval_X[p_j]
        = \eval_{x} \circ \regOp \circ \eval_X[p_j]
        = \eval_{x} \circ \linProj p_j = p_j(x),
    \]
    where we have used that \(\linProj\) is the identity on the space of polynomials \(\poly_{n-1}\)
    in the last step.
\end{proof}

\end{document}